\documentclass{amsart}
\usepackage[utf8]{inputenc} 
\usepackage{url}            
\usepackage[colorinlistoftodos]{todonotes}
\usepackage{float}
\usepackage{amsrefs}
\usepackage[colorlinks,linkcolor=blue,citecolor=blue]{hyperref}
\usepackage{graphicx}
\usepackage{amsmath,amssymb,graphicx}
\usepackage{comment}
\usepackage{color}
\usepackage{dsfont}
\usepackage{ tipa }
\usepackage{amssymb,amsmath,amsthm}
\usepackage[english]{babel}

\usepackage{macros}

\title[Classifying hyperbolic ergodic stationary measures]{Classifying hyperbolic ergodic stationary measures on compact complex surfaces with large automorphism groups}

\author{
 Megan Roda }

\date{\today} 
\address{
Department of Mathematics,
University of Chicago, Chicago, IL 60637
} 

\email{mroda@uchicago.edu}

\begin{document}

\begin{abstract}

Let $X$ be a compact complex surface. Consider a finitely supported probability measure $\mu$ on $\aut(X)$ such that $\Gamma_{\mu} = \langle \supp(\mu)\rangle<\aut(X)$ is non-elementary. We do not assume that $\Gamma_{\mu}$ contains any parabolic elements. In this paper, we study and classify hyperbolic, ergodic $\mu$-stationary probability measures.

\end{abstract}
\maketitle
\tableofcontents

\section{Introduction}
\label{sec:intro}

In this paper, we use techniques from homogeneous and smooth dynamics to further the classification of stationary measures on compact complex surfaces.

\subsection{Main theorem and background}

\begin{theorem}\label{bigthmeasy}
    Let $X$ be a compact complex surface, $\mu$ a finitely supported probability measure on the group of biholomorphisms $\aut(X)$ such that $\Gamma_{\mu}= \langle \supp(\mu) \rangle$ is non-elementary, and $\nu$ a hyperbolic ergodic $\mu$-stationary probability measure on $X$. Assume there is no $\Gamma_{\mu}$-invariant algebraic curve and that the Lyapunov exponents sum to be non-negative. Then, $\nu$ is invariant, and exactly one of the following statements hold
    \begin{enumerate}
        \item $\nu$ is finitely supported,
        \item the supports of the stable and unstable conditional measures are real 1-dimensional curves, the measures are absolutely continuous with respect to Lebesgue, and they are jointly integrable, or
        \item $\nu$ is absolutely continuous with respect to volume.
    \end{enumerate}
\end{theorem}
By $\mu$-stationary, we mean that 
 $$\int f_*\nu d\mu(f) = \nu.$$
Hyperbolicity is the assumption that there are no zero Lyapunov exponents. We assume this as it is conjectured that all measures in this setting are hyperbolic. Non-elementary means that the induced subgroup, $\Gamma^* < GL(H^2(X;\bbZ))$, acting on cohomology of $X$ contains a non-Abelian free group. It turns out that this forces $X$ to be K\"ahler and projective, see \cite{CDsurvey}. Further, the latter reference shows that $X$ must be one of four types of objects: a K3 surface, an Enriques surface, an Abelian surface or a rational surface. In the case of K3 surfaces, Enriques surfaces, and Abelian surfaces, we always have a finite $\aut(X)$-invariant volume form which forces the Lyapunov exponents to sum to zero. Hence these three cases are covered by Theorem \ref{bigthmeasy}. In the case of rational surfaces, some, such as the Coble family of examples carry non-elementary actions and a finite $\aut(X)$-invariant volume form \cite{CDsurvey} and so Theorem \ref{bigthmeasy} applies. This is not true in general for all rational surfaces, see for example the family of Lesieutre examples in \cites{Les,CDsurvey} which carry non-elementary actions but no $\aut(X)$-invariant volume form.

In all cases of our theorem we have homogeneity, meaning that the stable and unstable conditional measures are supported on the orbit of some algebraic group and are the pushforward of the Haar measure of that algebraic group onto the orbit (under the orbit map). 

Benoist and Quint \cite{BQ} give a similar classification result in the homogeneous setting. When $G$ is a connected almost simple Lie group with a probability measure $\mu$ whose support generates a Zariski-dense subgroup, they give a classification of $\mu$-stationary measures on the homogeneous space $G/\Lambda$ where $\Lambda$ is a lattice. They also cover the Abelian surfaces case of our theorem.

In \cite{CDinv}, Cantat and Dujardin classify all invariant probability measures (with some further analysis) in our setting with the added assumption that $\Gamma_{\mu}$ contains a parabolic element. Even if $\aut(X)$ contains a parabolic element, it is not a given that $\Gamma_{\mu}$ contains a parabolic; being non-elementary comes down to having two loxodromic elements that generate a non-Abelian free group. In Appendix \ref{sec:example} we demonstrate that it is even possible to have $X$ such that $\aut(X)$ that does not contain any parabolic elements. Hence our result, which does not assume the existence of a parabolic element, extends the classification of \cite{CDinv} to this setting.

One can see \cite{CDsurvey} for more about what has been done to classify invariant probability measures for non-elementary groups when we do not assume the existence of a parabolic element. The latter reference also provides an analysis of some rational surfaces that carry non-elementary actions.

\subsection{Strategy}

Our main work is to classify the measures using the `Benoist Quint' argument developed in \cite{BQ} (see also \cite{EM} and \cite{BRH}). Alternatives 1 and 2 of Theorem \ref{bigthmeasy} are shown using this technique given that we have already demonstrated invariance of the stationary measure. With an understanding of the notation given in Section \ref{sec:breakdown}, one can read the proof sketches of alternatives 1 and 2 in Sections \ref{sec:outline2} and \ref{sec:outline1}.

To achieve the invariance of the stationary measures, we work to establish when quantitative non-integrability (QNI) conditions hold in Section \ref{sec:twolem} in order to use the results of \cite{BEF}.

Sections \ref{sec:prelim} and \ref{sec:sec3} establish the notation, background and tools needed to prove everything in Sections~\ref{sec:twolem}, \ref{sec:Sect3} and \ref{sec:alt1}.

\subsection*{Acknowledgments}
The author would like to thank her advisor, Alex Eskin, for his mentorship and for many helpful discussions throughout the course of this work. The author would also like to thank Serge Cantat for providing her with an outline of how one might give the example in Appendix \ref{sec:example}.

\section{Preliminaries}
\label{sec:prelim}

Let $X$ be a compact K\"ahler surface and let $\aut(X)$ be the group of biholomorphisms of $X$ (i.e. holomorphic diffeomorphisms). Note that a K\"ahler manifold is a Riemannian manifold. We fix some Riemannian metric on $X$ and let $\Vert\cdot \Vert $ be the norm on $T_xX$ with respect to this metric.

 Let $\mu$ be a finitely supported probability measure on $\aut(X)$, and $\nu$ an ergodic Borel $\mu$-stationary measure. By $\mu$-stationary, we mean that the measure $\nu$ is invariant on average. To be more precise, this means that
 \begin{equation}
     \int f_*\nu d\mu(f) = \nu. \label{stationary}
 \end{equation}
 By ergodic we mean that for any measurable set $A$ such that $\nu(A\Delta f(A))=0$ for $\mu$-a.e. $f\in \aut(X)$, we have that $\nu(A)=0$ or 1. Equivalently, $\nu$ being ergodic means that it cannot be written as a proper convex combination of two distinct Borel $\mu$-stationary measures. 

Let $\Gamma_{\mu}$ be the subgroup generated by the support of $\mu$. We require that $\Gamma_{\mu}$ be non-elementary, which means that the induced subgroup $\Gamma^* < GL(H^2(X;\bbZ))$ acting on cohomology of $X$, contains a non-Abelian free group. By \cite{CDsurvey}, $X$ must be projective. Further, it is worth noting that the latter reference shows that amongst compact complex surfaces, only K\"ahler surfaces can carry non-elementary actions;  \cite{C} showed that only K\"ahler surfaces can carry automorphisms of positive topological entropy and \cite{CDsurvey} demonstrate that these elements are present when a surface carries a non-elementary action. In the end, \cite[Theorem 1.3]{CDsurvey} gives us that only four types of objects can carry these kinds of actions: complex K3 surface, Enriques surface, Abelian surface, or rational surface. 

It is worth remarking that there are two different definitions of K3 surface. Firstly, there is a more general definition of a complex analytic K3 which is a simply connected compact complex manifold of dimension 2 with a nowhere vanishing holomorphic 2-form. Not all of these are projective. There is a also the notion of an algebraic K3 surface defined over~$\bbC$, this set of surfaces can be embedded in the set of complex analytic K3s using Serre's GAGA principle. The image of these algebraic K3 surfaces in the set of complex analytic K3s turn out to be precisely the projective ones, and hence these are the ones we are interested in. One can see~\cite{Huy} as a standard reference for more details. We will work with examples of this kind in Appendix~\ref{sec:example}.

The elements of $\aut(X)$ are classified as loxodromic (hyperbolic), parabolic or elliptic. An element $f$ is loxodromic if $f^*$, its image in~$GL(H^2(X,\bbZ))$, has spectral radius strictly greater than 1. The Gromov-Yomdin Theorem (see~\cites{G1,G2,Yom}) implies that these are the elements of positive topological entropy. An element is elliptic if $f^*$ has finite order. If it is neither of those two, it is parabolic; this means that it is virtually unipotent, i.e. $(f^*)^k$ is unipotent but not the identity for some~$k$.

One can see this classification of automorphisms of $X$ using hyperbolic geometry~\cite{Csurvey}. Since we are in a setting where $X$ is K\"ahler, the Hodge decomposition tells us that the de Rham cohomology splits into a direct sum of Dolbeault cohomology groups:
$$H^2(X;\bbC) = \bigoplus_{p+q=2} H^{p,q}(X;\bbC) = H^{2,0}(X;\bbC) \oplus H^{1,1}(X;\bbC) \oplus H^{0,2}(X;\bbC).$$
Classes in these $H^{p,q}(X;\bbC)$ (Dolbeault cohomology groups) are represented by forms of type $(p,q)$ that are $\overline{\partial}$-closed. We take $H^{1,1}(X;\bbR)$ with its natural non-degenerate intersection form, $\langle \cdot | \cdot \rangle$, of signature $(1, h^{1,1}(X)-1)$, it is isometric to the standard quadratic form on Minkowski space of dimension $h^{1,1}(X)$. Taking the connected component of the hyperboloid $\{v \in H^{1,1}(X;\bbR) : \langle v|v\rangle = 1\}$ that intersects the K\"ahler cone (it is unique), we have a model of hyperbolic space $\bbH_X$. The automorphisms of $X$ act by isometries with respect to the intersection form, preserve the K\"ahler cone, and hence, preserve $\bbH_X$. Viewing elements of $\aut(X)$ as elements of $\isom(\bbH_X)$ we can classify them as elliptic, parabolic and hyperbolic and this is equivalent to the definitions we gave before. 

As explained before, hyperbolic elements are present when a surface carries a non-elementary action, but parabolic elements do not necessarily have to be present. We give an example of a K3 surface $X$ where $\aut(X)$ does not contain any parabolic elements in Appendix \ref{sec:example}.

\subsection{Setup of the random dynamical system}

Now we can setup our dynamical system, which is indeed `random' as we will draw diffeomorphisms to act on $X$ using the law given by $\mu$. Let $\Omega = \supp(\mu)^{\bbZ}$. We typically write an element of $\Omega$ as $\omega = (\dots, f_{-2},f_{-1},f_0,f_1,f_2,\dots)$. We define the dynamics associated to a given element of $\Omega$ by
 \begin{align}\label{forwards}
     f_{\omega}^n &\defeq f_{n-1}\circ \dots \circ f_0  \ \ \text{for }n\geq 1, \\
     f_{\omega}^n &\defeq (f_{-n})^{-1}\circ \dots \circ (f_{-1})^{-1} \ \ \text{for }n<0,\label{backwards}
 \end{align}
 and $f_{\omega}^0$ is the identity.
Let $Y\defeq X\times \Omega$. We use the canonical projection map $\pi_X: Y\to X$ to make the identification between $X \times \{\omega\} \subset Y$ and $X$ itself, denoting this $X_{\omega}$. Also let $\sigma:\Omega \to \Omega$ denote the two-sided shift map. We define the (invertible) skew product map by
 \begin{align}
&F(x,\omega) = (f_{\omega}^1(x),\sigma(\omega)),\label{fskew} \\
&F^{-1}(x,\omega) = (f_{\omega}^{-1}(x),\sigma^{-1}(\omega)).\label{bskew}
 \end{align}
Occasionally, for $n\in \bbZ$ we will write $F_{\omega}^n(x)$ which denotes the $X$-component of $F^n(x,\omega)$.

We define the tangent bundle of $Y$ as $TY\defeq TX\times \Omega$ and denote $DF$ as the tangent map $$D_{(x,\omega)}F : \{\omega\}\times T_xX\to \{\sigma(\omega)\}\times T_{f_{\omega}^1x}X,$$ induced by $D_xf_{\omega}^1$. 
Similarly we denote $DF^{-1}$ the tangent map induced by $D_xf_{\omega}^{-1}$.

Additionally, we have a distance metric $d_Y$ on $Y$, where 
\begin{align} d_Y((x,\omega),(x',\omega')) = \max \{d_X(x,x'), \ d_{\Omega}(\omega,\omega') \}.\end{align}

Here $d_{\Omega}$ is the distance metric on $\Omega$ given by
\begin{align}
d_{\Omega}(\omega,\omega') = \sum_{-\infty}^{\infty} \frac{1}{2^{|n|}d(\omega_n,\omega_n')},
\end{align}
where $\omega_n$ is the $n$th entry of $\omega\in \Omega$ and $d$ is the discrete metric. 

For $\omega\in \Omega$, let $\omega^+$ be the non-negative indexed part of the bi-infinite vector $\omega$ and $\omega^-$ be the negative indexed part of the vector $\omega$. We then decompose $\Omega$ as $\Omega = \Omega^- \times \Omega^+$, where $\Omega^+$ is the space of 1-sided infinite vectors $\omega^+$ and $\Omega^-$ is the 1-sided infinite vectors $\omega^-$.

\begin{lemma}[Proposition 1.2, \cite{K}, \cite{LQ}]\label{m}
    There exists a unique $F$-invariant Borel probability measure $m$ on $Y$ such that the image of $m$ under the canonical projection map $\pi_{X\times \Omega^+}: Y \to X\times \Omega^+$ is $\nu \times \mu^{\bbN}$. Additionally, under the canonical projection maps $\pi_{\Omega} : Y\to \Omega$ and $\pi_X: Y\to X$ the image of $m$ is $\mu^{\bbZ}$ and $\nu$ respectively. This measure $m$ is given by the weak-$*$ limit
    $$m = \lim_{n\to \infty}(F^n)_*(\nu\times \mu^{\bbZ}).$$
    \end{lemma}

\begin{remark}\label{nuinvar}
    If $\nu$ is invariant under $\supp(\mu)$, then $m = \nu \times \mu^{\bbZ}$.
\end{remark}

\begin{remark}
    \label{skewergodic}
Additionally, $\nu$ is ergodic $\mu$-stationary if and only if $m_+\defeq \nu \times \mu^{\bbN}$ is ergodic under the non-invertible skew product map $F$ on $X\times \aut(X)^{\bbN}$. Further this happens if and only if $m$ is ergodic under the invertible skew product map $F$ (see \cite{BQ2} Propositions 1.8 and 1.9, \cite{CD} Section 7.1).
\end{remark}

\subsection{Oseledets' theorem and the global stable and unstable manifolds}

We will assume that $\mu$ satisfies the following moment condition:\\ let  $\Vert f\Vert_{C^1}\defeq \max_{x\in X}\Vert D_xf\Vert $, we ask that 
\begin{equation}\label{IC}
    \int \left( \log \Vert f\Vert_{C^1} + \log\Vert f^{-1}\Vert_{C^1} \right) d\mu(f) < \infty.
\end{equation}

Note that the finiteness of the above integral does not depend on the choice of Riemannian metric.

\begin{remark}\label{higher}
The above moment condition is required for Oseledets' theorem, but it also implies higher moment conditions, i.e. replacing $C^1$ with $C^k$ (this norm computed in charts). This follows from the Cauchy estimates \cite[Lemma 4.1]{CD}. This is needed to use Pesin theory.
\end{remark}

This moment condition given by Equation (\ref{IC}) and Kingsman's subadditive ergodic theorem give us that the following limits exist $m$-almost surely:
\begin{align}\label{lya}
    \lambda^+ = \lim_{n\to \infty}\frac{1}{n}\log \Vert D_xf_{\omega}^n\Vert,\
    \lambda^- =\lim_{n\to \infty}\frac{1}{n}\log \Vert (D_xf_{\omega}^n)^{-1}\Vert^{-1}. 
\end{align}
We will assume that $\lambda^+>0>\lambda^-$ (i.e. hyperbolicity).
 Applying Oseledets' theorem~\cite[Chapter III]{K}, we have that for $m$-a.e. $(x,\omega)\in Y$, there exists a decomposition $$T_xX_{\omega}\cong T_xX = E^s(x,\omega)\oplus E^u(x,\omega),$$ such that for $v\in E^s(x,\omega)\setminus \{0\}$ we have that 
 \begin{equation}
     \lim_{n\to \pm \infty}\frac{1}{n} \log \Vert D_xf_{\omega}^n(v)\Vert = \lambda^-,
 \end{equation}
 and for $v\in E^u(x,\omega)\setminus \{0\}$ we have
 \begin{equation}
     \lim_{n\to \pm \infty} \frac{1}{n} \log \Vert D_xf_{\omega}^n(v)\Vert = \lambda^+.
 \end{equation}
 Additionally, we have that the map $(x,\omega)\mapsto E^{s/u}(x,\omega)$ is measurable and that this decomposition is $DF$-equivariant in the sense that $$D_{(x,\omega)}F(E^{s/u}(x,\omega)) = E^{s/u}(F(x,\omega)).$$ 

We will denote $\Lambda$ to be the conull subset of $Y$ where this theorem holds true, it is often called the Oseledets' set. In the above, $\lambda^+$ and $\lambda^-$ are called Lyapunov exponents and $E^{s/u}$ are called the Oseledets' subspaces. 

The work of \cite{CD} summarizes the conclusions of Pesin theory in our setting in Section 7 of their paper. We get from this the existence of local and global stable and unstable manifolds for $m$-a.e. $(x,\omega)\in Y$. The global stables and unstables, denoted $W^s(x,\omega)$ and $W^u(x,\omega)$ respectively, are given by
\begin{align}
    W^s(x,\omega) &= \{(y,\omega)\in X_{\omega} : \limsup_{n\to \infty} \frac{1}{n}\log d_X(f^n_{\omega}(y),f_{\omega}^n(x))<0\},\label{stable}\\
    W^u(x,\omega) &=  \{(y, \omega)\in X_{\omega} : \limsup_{n\to -\infty} \frac{1}{|n|}\log d_X(f^n_{\omega}(y),f_{\omega}^n(x))<0\},\label{unstable}
\end{align}
where $d_X$ is the Riemannian distance computed in $X$. We view these manifolds as subsets of $X_{\omega}$ which means they live in the skew product. Further, \cite[Proposition 7.10]{CD} gives us that these global stables and unstable manifolds are biholomorphic to~$\bbC$ for $m$-a.e. $(x,\omega)$. In fact, they are parameterized by injectively immersed entire curves.

\subsection{Suspension flow}
\label{sec:susp}
We will establish two `suspension' flows; the `standard' flow and the `time changed' flow. The latter will be constructed in Section \ref{sec:flow} where we have to use normal form coordinates. The space we will be flowing on is $Z = Y\times [0,1[$ which is called the `suspension'. 

The  `standard' flow is denoted $F^t:Z\to Z$ for $t\in \bbR$. For $t\geq 0$, we let $$F^t(x,\omega,k) = (F^n(x,\omega),k+t \mod 1),$$ where $F$ is as defined in Equation \eqref{fskew} and $n = \lfloor k+t\rfloor$. For $t<0$ the flow is given by $$F^t(x,\omega,k) = ( F^{n}(x,\omega) , k+t \mod 1),$$ where $F^{-1}$ is as in Equation \eqref{bskew} and $n = \lfloor k+t \rfloor$ (note that for example $\lfloor -0.8\rfloor = -1$). The natural invariant measure for this flow is $\hat{m} \defeq m\times dt$ where $dt$ is the Lebesgue measure on $[0,1[$. 

Note that the $[0,1[$-component of an element of $Z$ does not play a role in determining the global stable or unstable manifold, hence $W^u(x,\omega,k) \defeq W^u(x,\omega)\subset X$.

\subsection{Conditional measures and the results of \cite{BEF}}
\label{sec:cond}
The following definition is from \cite[Section IV.2]{LQ}, but see also \cite[Section B.4]{EM}, \cite[Section 6.3]{BRH}, \cite[Section 9.3]{MT}, and \cite{BEF}.

\begin{definition}[\cite{LQ}, Section IV.2]\label{part}
    A measurable partition $\eta^s$ of $Y$ is subordinate to the stable manifolds if for $m$-a.e. $(x,\omega),$ we have the following properties:
    \begin{enumerate}
        \item $\eta^s_{\omega}(x)\defeq \{y\in X : (y,\omega)\in \eta^s(x,\omega)\} \subset W^s(x,\omega)$,
        \item $\eta^s_{\omega}(x)$ contains an open neighbourhood of $x$ in $W^s(x,\omega)$ in the submanifold topology.
    \end{enumerate}
    Where $\eta^s(x,\omega)$ is an atom of the partition. A similar definition holds for the unstable manifolds, denote this partition $\eta^u$.
\end{definition}

\begin{proposition}[\cite{LQ}, Proposition 2.1, \cite{CD}, Section 7.6] \label{2.1}
    Such a measurable partition as described in Definition \ref{part} exists and we can also ask that it has the following properties:
\begin{enumerate}
    \item $F^{-1}\eta^s \leq \eta$ where $F^{-1}\eta^s(x,\omega) \defeq F^{-1}(\eta^s(F(x,\omega)))$, 
    \item $\{X\times \{\omega\}\} : \omega \in \Omega\} \leq \eta^s$,
   
\end{enumerate}
where if two partitions $P$ and $Q$ are such that $P\leq Q$ then for all $B\in Q$, there exists $A\in P$ such that $B\subset A$. There is a similar statement for the unstable version of Definition \ref{part}. 
\end{proposition}
    In \cite[Section 7.6]{CD}, they also quote the above proposition and call this partition a stable (or unstable) Pesin partition. Property 1. in Proposition \ref{2.1} tells us that $\eta^s(x,\omega) \subset F^{-1}\eta^s(F(x,\omega))$. 
Associated to measurable partitions are families of conditional measures supported on the atoms of the partition. Here we let $\{m_{x,\omega}^{\eta^s}\}_{(x,\omega)\in Y}$ denote our conditional measures supported on $\eta^s(x,\omega)$. Letting $\calA^s$ be the $\sigma$-algebra of $m$-measurable subsets which are unions of the $\eta^s(x,\omega)$, then these measures are such that for any $\varphi : Y\to \bbR$
\begin{align}
    \bbE[\varphi | \calA^s](x,\omega) = \int_{Y}\varphi dm_{x,\omega}^{\eta^s}.
\end{align}
Note that though the stable or unstable manifolds forms a partition of a full-measure subset of $Y$ this partition is generally not measurable. This is why we have to work with partitions that are subordinate to the stable or unstable manifolds.

We define $\{m^{s}_{x,\omega}\}_{(x,\omega)\in Y}$ such that $m^{s}_{x,\omega}(\eta^s(x,\omega)) = 1$ and for any compact set $K \subset W^s(x,\omega)$ we take $t>0$ such that 
$$F^t(K\times \{\omega\})=F_{\omega}^t(K)\times \{\sigma^t(\omega)\} \subset \eta^s(F^t(x,\omega)),$$ 
where $F^t_{\omega}(x)$ is the $X$ component of $F^t(x,\omega)$.
Then, we let  
\begin{align}
    m^{s}_{x,\omega}(K) \defeq \frac{(F_{\omega}^{-t})_*m_{F_{\omega}^tx,\sigma^t(\omega)}^{\eta^s}(K)}{(F_{\omega}^{-t})_*m^{\eta}_{F_{\omega}^tx,\sigma^t(\omega)}(\eta^s(x,\omega))}.
\end{align}
This is a locally finite, infinite measure on the (entire) stable leaf which we call a leafwise conditional measure or stable conditional measure. The right hand side is well-defined because $m$ is $F$-invariant. The measures $m^s_{x,\omega}$ are supported on the stable leaves. Let $\{m^{u}_{x,\omega}\}_{(x,\omega)\in Y}$ denote the unstable analogues which we will refer to as unstable conditional measures.

\subsection{Compatible families and quantitative non-integrability}
\label{sec:QNI+}

We shall follow the definitions of \cite{BEF} in this section.

According to Appendix C of \cite{BEF}, unstable manifolds admit a `subresonant structure'. We denote $G^{ssr}(W^u(x,\omega,k))$ as the set of strictly subresonant transformations on $W^u(x,\omega,k)$. In our setting, these maps are linear transformations on $W^u(x,\omega,k)$ in normal form coordinates (see Section \ref{sec:NFC}). The groups $G^{ssr}(W^u(x,\omega,k))$ are unipotent algebraic groups. We will study the action of connected unipotent algebraic subgroups $U^+(x,\omega,k)\subset G^{ssr}(W^u(x,\omega,k))$. We denote $U^+[q,\omega,k]\defeq U^+(x,\omega,k)q$ as the orbit of the action of $U^+(x,\omega,k)$ on $q\in W^u(x,\omega,k)$. 

We define the `combinatorial' unstable and stable manifolds as follows
\begin{align}\label{Wu1}
    \hat{W}^u(x,\omega,k) \defeq  W^u(x,\omega,k) \times W^+(\omega,k),\\
    \hat{W}^s(x,\omega,k) \defeq  W^s(x,\omega,k) \times W^-(\omega,k),\label{Ws1}
\end{align}
where
\begin{align*}
    W^+(\omega,k) &\defeq \{ (\omega',k')\in \Omega\times [0,1[ : k'=k, \ \text{and for }n\in \bbN\text{ large, }f'_{-n} = f_{-n}   \},\\
     W^-(\omega,k) &\defeq \{ (\omega',k')\in \Omega\times [0,1[ : k'=k, \ \text{and for }n\in \bbN\text{ large, }f'_{n} = f_{n}   \},
\end{align*}
and where $\omega = (\dots, f_{-2},f_{-1},f_0,f_1,f_2,\dots)$.

\begin{definition}[Random compatible family of subgroups]\label{compat}
We say that a measurable family of connected algebraic subgroups $(x,\omega)\mapsto U^+(x,\omega,k)\subset G^{ssr}(W^u(x,\omega,k))$ is compatible with $\hat{m}$ if we have the following:
\begin{enumerate}
    \item The subgroups are $F^t$-equivariant.
    \item For $\mu^{\bbZ}$-almost every $\omega$ we require that the sets $U^+[x,\omega,k]\cap \eta^u_{\omega}(x)$ form a partition of $X$. Further, the conditional measures along $U^+[x,\omega,k]\cap \eta^u_{\omega}(x)$ are proportional to the restriction of the `Haar measure' which in this setting is the unique $U^+(x,\omega,k)-$invariant measure on $U^+[x,\omega,k]$.
    \item Given $(x',\omega',k') \in \hat{W}^u(x,\omega,k)$, we have the natural identification of the group $G^{ssr}(x,\omega,k)$ with the group $G^{ssr}(x',\omega',k')$, viewing both as acting on $W^u(x,\omega,k) =W^u(x',\omega',k')$. Further, there is a conull set $X_0\subset X$ such that if the points $(x,\omega,k),(x',\omega',k')\in X_0$ and $(x',\omega',k')\in U^+[x,\omega,k]$, then $U^+(x',\omega',k') = U^+(x,\omega,k)$
    \item $U^+(x,\omega,k)$ is `maximal', i.e. it contains the largest connected subgroup of $G^{ssr}(x,\omega,k)$ that preserves the Haar measure on $U^+[x,\omega,k]$.

\end{enumerate}
\end{definition}

\begin{remark}
    We allow $U^+(x,\omega,k) = \{e\}$.
\end{remark}

Now let 
\begin{align}\label{combo+}
    W^+_1(\omega,k) \defeq  \{ (\omega',k')\in \Omega\times [0,1[ : k'=k, (\omega')^- = \omega^-\},\\
      W^-_1(\omega,k) \defeq  \{ (\omega',k')\in \Omega\times [0,1[ : k'=k, (\omega')^+ = \omega^+\}.
\end{align}

Recall $\eta^s$ from Definition \ref{part} and the unstable counterpart $\eta^u$. We will define
\begin{align*}
    \eta^u_{\omega, t}(x) &= F^{-t}(\eta^u_{\sigma^t(\omega)}(F^t_{\omega}(x)) \subset X_{\omega},\\
    \eta^s_{\omega, t}(x) &= F^{t}(\eta^s_{\sigma^{-t}(\omega)}(F^{-t}_{\omega}(x)) \subset X_{\omega},
\end{align*}
where $F^t_{\omega}(x)$ is the $X$-component of $F^t(x,\omega)$.

It is easy to see that $\eta^u_{\omega,t}(x)$ and $\eta^s_{\omega,t}(x)$ are smaller than (and live inside) $\eta^u_{\omega}(x)$ and $\eta^s_{\omega}(x)$ resp.

Now let
 \begin{align}\label{etau1}
     \eta^u_t[x,\omega,k] &\defeq \eta^u_{\omega, t}(x) \times W_1^+(\omega,k) \subset W^u(x,\omega,k) \times W_1^+(\omega,k) \subset \hat{W}^u(x,\omega,k),\\
     \eta^s_t[x,\omega,k] &\defeq \eta^s_{\omega, t}(x) \times W^-_1(\omega,k) \subset W^s(x,\omega,k) \times W_1^-(\omega,k)\subset \hat{W}^s(x,\omega,k).\label{etas1}
 \end{align}

The natural measure on $W_1^+(\omega,k)$ is $$\tilde{\mu}^u_{\omega,k} \defeq \delta_{\omega^-} \times \mu^{\bbN} \times \delta_k,$$ on $\Omega^- \times \Omega^+ \times [0,1[$. Similarly for $W_1^-$ we have $\tilde{\mu}_{\omega,k}^s$. Let $\tilde{m}^u_{x,\omega}$ be the restriction of the conditional measure of $m^u_{x,\omega}$ to $U^+[x,\omega,k]\cap \eta^u_{\omega}(x)$. Note that $\tilde{m}^u_{x,\omega}$ is proportional to the `Haar' measure as in (2) of Definition \ref{compat}.

Consider the measures 
\begin{align}\label{measu}
    \hat{m}^u_{x,\omega,k} \defeq \tilde{m}^u_{x,\omega} \times \tilde{\mu}^u_{\omega,k},\\
    \hat{m}^s_{x,\omega,k} \defeq  m^s_{x,\omega} \times \tilde{\mu}^s_{\omega,k}.\label{meass}
\end{align}
 They live on $\hat{W}^u(x,\omega,k)$ (or more specifically $U^+[x,\omega,k]\times W^+_1(\omega,k)$) and $\hat{W}^s(x,\omega,k)$ respectively.

\begin{definition}[Random QNI]\label{QNIdef}
    Let $(x,\omega,k)\mapsto U^+(x,\omega,k)$ be a measurable family of connected algebraic subgroups satisfying Definition \ref{compat}, we define
    \begin{align*}
        \hat{U}[x,\omega,k] \defeq  U^+[x,\omega,k] \times W_1^+ \subset \hat{W}^u(x,\omega,k).
    \end{align*}
    We say that $U^+(x,\omega,k)$ satisfies the random QNI condition if the following holds: \\
    $\bullet$ there exists $\alpha_0>0$ and,\\
    $\bullet$ for every $\delta>0$ there exists a compact set $K\subset Z$ with $\hat{m}(K)>1-\delta$ and constants $C=C(\delta)>0$ and $\ell_0 = \ell_0(\delta)>0$ such that:\\
    If $\ell>\ell_0$ and $\hat{x}_{1/2} = (x_{1/2},\omega_{1/2},k_{1/2})\in K$, $\hat{x}_1 = F^{\ell/2}(\hat{x}_{1/2})$, $\hat{x} = F^{-\ell/2}(\hat{x}_{1/2})\in K$ then\\
    $\bullet$ there is a subset $S = S(\hat{x}_{1/2},\ell) \subset \eta^u_{\ell/2}[\hat{x}_{1/2}]\cap \hat{U}^+[\hat{x}_{1/2}] \subset \hat{W}^u(\hat{x}_{1/2})$ such that  $\hat{m}^u_{\hat{x}_{1/2}}(S)\geq (1-\delta)\hat{m}^u_{\hat{x}_{1/2}}(\eta^u_{\ell/2}[\hat{x}_{1/2}]\cap \hat{U}^+[\hat{x}_{1/2}])$ and for each $\hat{y}_{1/2}\in S$ we have\\
    $\bullet$ that there is a subset $S' = S'(\hat{y}_{1/2},\ell))\subset \eta^s_{\ell/2}[\hat{x}_{1/2}]\subset \hat{W}^s(\hat{x}_{1/2})$ such that  $\hat{m}^s_{\hat{x}_{1/2}}(S') \geq (1-\delta)\hat{m}^s_{\hat{x}_{1/2}}(\eta^s_{\ell/2}[\hat{x}_{1/2}])$ so that if $\hat{z}_{1/2}\in S'$, then
    \begin{align}
        d_X(U^+_{q}[\hat{z}_{1/2}],W^{cs}_{q}(\hat{y}_{1/2}))\geq Ce^{-\alpha_0\ell},
    \end{align}
    where $U^+_q[x]$ is $U^+[x]$ intersected with $W^u_q(x,\omega)$, i.e. the local stable manifold, see Theorem \ref{locstablemfld}.

\end{definition}

Now we state a theorem from \cite{BEF}: 
\begin{theorem}[\cite{BEF}, Corollary 4.1.5]\label{4.1.5}
    Let $\mu$ be a finitely supported probability measure on $\aut(X)$, $\nu$ a $\mu$-stationary ergodic probability measure, $\hat{m}$ as in Section \ref{sec:susp}, then there exists a measurable family of connected algebraic subgroups $(x,\omega,k)\mapsto U^+(x,\omega,k)$ satisfying Definition \ref{compat} that does not satisfy Definition \ref{QNIdef}, i.e. random QNI does not hold. 
\end{theorem}

\subsection{Proof of Theorem \ref{bigthmeasy}}
\label{sec:breakdown}

In this section, we will give the proof of Theorem~\ref{bigthmeasy} except for Lemma \ref{0.1},  Proposition \ref{finite}, Lemma~\ref{0.5} and Proposition~\ref{full}. These will be proven later in the paper. 

Recall the assumptions made in Section \ref{sec:intro}, i.e. that $\Gamma_{\mu}$ is non-elementary, that $\supp(\mu)<\infty$, that there is no $\Gamma_{\mu}$-invariant algebraic curve, and that $\lambda^++\lambda^-\geq 0$.

From Theorem \ref{4.1.5}, take our measurable family of connected algebraic subgroups $(x,\omega,k)\mapsto U^+(x,\omega,k)$. This family of subgroups satisfies Definition \ref{compat} but does not satisfy random QNI. Let $d_+ \defeq \dim_{\bbR}[U^+(x,\omega,k)]$. This $d_+$ is constant a.e. by property (1) of Definition \ref{compat} combined with ergodicity. There are three cases to address and these cases correspond to the three alternatives in Theorem \ref{bigthmeasy}.

Define $\calL^s(x,\omega,0)$ to be the Zariski closure of the support of the stable conditional, $m_{x,\omega}^{s}$, inside $W^s(x,\omega)\subset X$, and $\calL^u(x,\omega,0)$ to be the unstable analogue. Note that we have no concrete notion of Zariski closure ambiently, but under normal form coordinates (discussed later in Section \ref{sec:NFC}), you are now in $\bbC$ and you can take Zariski closure and then push it back.

Let $h(\mu)$ be the entropy of $\mu$. Let $\zeta$ be a measurable partition of $Z$ such that $F\zeta \leq \zeta$ and define 
$$h_{\hat{m}}(F,W^+_1) \defeq H_{\hat{m}}(F^{-1}\zeta|\zeta) = -\int \log \hat{m}_{\hat{x}}^{\zeta}(F^{-1}\zeta(\hat{x}))d\hat{m}(\hat{x}),$$
where $\hat{m}_{\hat{x}}^{\zeta}$ is the conditional measure at $\hat{x}\in Z$ relative to the partition $\zeta$. We now take $\zeta$ to be the partition of $Z$ such that $\zeta(\hat{x}) = W_1^+(\hat{x})$, then we define
\begin{align}\label{combent+}
h_{\hat{m}}(F,W^+_1) = h_{\hat{m}}(F,\zeta).
\end{align}
This quantity is the entropy of the 1-sided shift on the set of futures, the entropy we get from the combinatorial unstable $W^+_1$. There is a similar quantity, the entropy we get from the combinatorial stable $W^-_1$.

We will state the following Proposition from \cite[Section 4]{BEF} which is a version of the Ledrappier Young formula for our setting:
\begin{proposition}[\cite{BEF}, Section 4]\label{prop4.3.8}
    For $\hat{m}$-a.e. $(x,\omega,k)\in Z$ we have
    \begin{enumerate}
        \item $h_{\hat{m}}(F)\geq h_{\hat{m}}(F,W_1^+) + \lambda^+\dim (E^u(\hat{x}) \cap T_xU^+[\hat{x}])$,
        \item $h_{\hat{m}}(F,W_1^+) = h(\mu)$,
        \item $h_{\hat{m}}(F^{-1}) \leq h_{\hat{m}}(F^{-1},W_1^-) - \lambda^-\dim(E^s(\hat{x})\cap T_x\calL^s(\hat{x}))$,
        \item $h_{\hat{m}}(F^{-1},W_1^-) \leq h(\mu)$ with equality if and only if $\nu$ is $\mu$-invariant.    
    \end{enumerate}
\end{proposition}
Additionally, we have the following result that we take from \cite{BEF} which follows immediately from \cite{Led84}, Theorem 3.4.
\begin{proposition}[\cite{BEF}, Section 4]\label{Prop4.3.10}
    For $m$-a.e. $(x,\omega)\in Y$, if the unstable conditional measure $m^u_{x,\omega}$ is supported on an embedded submanifold $N_{x,\omega}\subset W^u(x,\omega)$, then 
    \begin{align}
        h_{\hat{m}}(F) \leq h(\mu) + \lambda^+\dim(E^u(x,\omega)\cap T_xN_{x,\omega})
    \end{align}
    with equality if and only if $m^u_{x,\omega}$ is equivalent to the Lebesgue measure on $N_{x,\omega}$ for $m$-a.e. $(x,\omega)\in Y$. A similar statement holds for stables.
\end{proposition}

\subsubsection{Alternative 1}
The first alternative of Theorem \ref{bigthmeasy}, that $\nu$ is finitely supported, corresponds to the case $d_+=0$. Here $U^+(x,\omega,k) = \{e\}$ a.s.

\begin{lemma}\label{0.1}
    Suppose $\Gamma_{\mu}$ is non-elementary, $U^+(x,\omega,k) = \{e\}$ a.e., and the conditional measure $m_{x,\omega}^{s}$ along $W^s(x,\omega,k)$ is non-trivial for $\hat{m}$-a.e. $(x,\omega,k)\in Z$. Then random QNI (Definition \ref{QNIdef}) holds.
\end{lemma}
We will prove this in Section \ref{sec:twolem} once we establish more notation and a few lemmas.

\begin{corollary}\label{0.2}
In our setting, if $d_+=0$, then $\nu$ is $\Gamma_{\mu}$ invariant, and the stable conditional measures $m_{x,\omega}^{s}$ on $W^s(x,\omega,k)$ and the unstable conditional measures $m_{x,\omega}^u$ on $W^u(x,\omega)$ are trivial for $\hat{m}$-a.e. $(x,\omega,k)$.
\end{corollary}
    \begin{proof}

        Since we are taking $U^+(x,\omega,k)$ as in Theorem \ref{4.1.5}, random QNI does not hold. Lemma \ref{0.1} gives us that we must have that the stable conditional measure $m_{x,\omega}^{s}$ on $W^s(x,\omega,k)$ is trivial for $\hat{m}$ a.e. $(x,\omega,k)$.

        It remains to discuss $\Gamma_{\mu}$-invariance of $\nu$. Since the stable conditional measures are trivial a.e., this follows from~\cite[Section 4]{BEF}. We will directly give their proof below:

Since $\mu$ is finitely supported, it can be written as $\sum_i p_i\delta_{f_i}$ for $f_i\in \supp(\mu)$ and for some $p_i$. Hence, its entropy, $h(\mu) = -\sum_{i}p_i\log(p_i)$ is clearly finite.

Recall $W^+_1$ from Equation \eqref{combo+}; 
 $h_{\hat{m}}(F, W_1^+)$ is the entropy of the 1-sided shift on the set of futures as in Equation \eqref{combent+}. Note that as $\nu$ is $\mu$-stationary, our measure $m$ from Lemma \ref{m} projects to the product $\nu \times \mu^{\bbN}$ and so $h_{\hat{m}}(F, W_1^+) = h(\mu)$ (see Proposition \ref{prop4.3.8}). Note that this is not apriori the case for $h_{\hat{m}}(F^{-1}, W_1^-)$, the analogue on the set of pasts. We work to show this in order to get invariance using (4) from Proposition \ref{prop4.3.8}.

We have that by Proposition \ref{prop4.3.8} (1), we have the following lower bound on the entropy of our map:
\begin{align}\label{entropy1}
    h_{\hat{m}}(F)\geq h(\mu) + \lambda^+\cdot \dim(T_xU^+). 
\end{align}

On the other hand, Proposition \ref{prop4.3.8} (3) gives us the lower bound:
\begin{align}\label{entropy2}
    h_{\hat{m}}(F^{-1}) &\leq h_{\hat{m}}(F^{-1},W_1^-)  - \lambda^-\dim( T_x\calL^s).
\end{align}

In our case Equations \eqref{entropy1} and $\eqref{entropy2}$ simplify because the conditional measures $m_{x,\omega}^{s}$ are trivial and so $\calL^s$ has dimension zero. Additionally we have assumed that $d_+=0$ and so $T_xU^+$ is also zero dimensional. So we have
\begin{align}
    h_{\hat{m}}(F)&\geq h(\mu),\\ 
    h_{\hat{m}}(F^{-1}) &\leq h_{\hat{m}}(F^{-1},W_1^-). 
\end{align}
 Also, by Proposition \ref{prop4.3.8}, we have that $h_{\hat{m}}(F^{-1},W_1^-)\leq h(\mu)$. So we get that
\begin{align}
    0 \leq h(\mu) - h_{\hat{m}}(F^{-1},W_1^-)  \leq -h_{\hat{m}}(F^{-1}) + h_{\hat{m}}(F).
\end{align}

But $h_{\hat{m}}(F^{-1}) = h_{\hat{m}}(F)$ and hence $h(\mu) = h_{\hat{m}}(F^{-1},W_1^-)$ which implies that $\nu$ is invariant by Proposition \ref{prop4.3.8}. 

Now, let $\check{\mu}$ be the measure defined by $\check{\mu}(g) = {\mu}(g^{-1})$, clearly $\nu$ is still $\check{\mu}$-invariant. Letting $\check{\omega} = (\dots f_2^{-1},f_1^{-1},f_0^{-1},f_{-1}^{-1},f_{-2}^{-1},\dots )$ we also have that
\begin{align}
W^u(x,\omega,k) = W^s(x,\check{\omega},k),& \ W^s(x,\omega,k) = W^u(x,\check{\omega},k)\\
m_{x,\omega}^{s} = m_{x,\check{\omega}}^{u},& \ \ m_{x,\omega}^{u} = m_{x,\check{\omega}}^{s},\\
\calL^u(x,\omega,0) =& \calL^s(x,\check{\omega},0).
\end{align}
This tells us that $m_{x,\omega}^u$ is also trivial a.e. which completes the proof.
\end{proof}

\begin{proposition}\label{finite}
    If $\Gamma_{\mu}$ is non-elementary, $\nu$ is $\Gamma_{\mu}$-invariant and is not finitely supported, then the conditional measure $m_{x,\omega}^{s}$ on $W^s(x,\omega,k)$ is non-trivial a.e.
\end{proposition}

Proposition \ref{finite} will take work and its proof is given in Section \ref{sec:alt1}.

\begin{proposition}\label{endfinite}
    In the case $d_+=0$, the measure $\nu$ is finitely supported.
\end{proposition}
\begin{proof}
    This follows immediately from Corollary \ref{0.2} and Proposition \ref{finite}.
\end{proof}

\subsubsection{Alternative 2}
\label{sec:alt22}
The second alternative of Theorem \ref{bigthmeasy} is the $d_+=1$ case.

\begin{lemma} \label{0.5}
    If $\dim(\calL^s)=2$ a.e. then Definition \ref{QNIdef}, random QNI, holds.
\end{lemma}

We will prove this in Section \ref{sec:twolem}.

\begin{lemma}\label{0.6}
If $d_+=1$, then $\nu$ is $\Gamma_{\mu}$-invariant and  $\calL^u(x,\omega,0) = U^+[x,\omega,k]$.
\end{lemma}
\begin{proof}
We use a proof from~\cite[Section 4]{BEF}  similarly to the way we did in our proof of Corollary \ref{0.2}.

We have that $h(\mu)<\infty$ and by Proposition \ref{prop4.3.8}, $h(\mu) = h_{\hat{m}}(F,W_1^+)$. We now consider the equations from (1) and (3) of Proposition \ref{prop4.3.8}. Lemma \ref{0.5} gives us that $\dim(T_x\calL^s)\leq 1$, also $d_+=1$ and so $\dim(T_xU^+)=1$. Now our equations simplify to be
\begin{align*}
    h_{\hat{m}}(F)&\geq h(\mu) + \lambda^+,\\
    h_{\hat{m}}(F^{-1}) &\leq h_{\hat{m}}(F^{-1},W_1^-) - \lambda^-\dim( T_x\calL^s) \leq h_{\hat{m}}(F^{-1},W_1^-) - \lambda^-.
\end{align*}
Since $h_{\hat{m}}(F^{-1},W_1^-) \leq h(\mu)$ (Proposition \ref{prop4.3.8}) and additionally, $\lambda^++\lambda^-\geq 0$ we have that
\begin{align}
    0\leq h(\mu)-h_{\hat{m}}(F^{-1},W_1^-) + \lambda^- + \lambda^+ \leq h_{\hat{m}}(F) - h_{\hat{m}}(F^{-1}).
\end{align}
The fact that $h_{\hat{m}}(F) = h_{\hat{m}}(F^{-1})$ gives us that $h(\mu)=h_{\hat{m}}(F^{-1},W_1^-)$ and so again,~$\nu$ is invariant by Proposition~\ref{prop4.3.8}.

Further, we see that
\begin{align}\label{this}
    h(\mu)  -\lambda^-\dim(T_x\calL^s) = h_{\hat{m}} (F^{-1}) = h(\mu) + \lambda^+,
\end{align}
This forces $\dim(\calL^s)=1$.

The equality in Equation \eqref{this} also allows us to use Proposition \ref{Prop4.3.10} to say that~$m^s_{x,\omega}$ is equivalent to Lebesgue on $\calL^s$. 

Let $\check{\mu}$ be the measure defined by $\check{\mu}(g) = {\mu}(g^{-1})$, clearly $\nu$ is still $\check{\mu}$-invariant. Letting $\check{\omega} = (\dots f_2^{-1},f_1^{-1},f_0^{-1},f_{-1}^{-1},f_{-2}^{-1},\dots )$ we also have that
\begin{align*}
W^u(x,\omega,k) = W^s(x,\check{\omega},k),& \ W^s(x,\omega,k) = W^u(x,\check{\omega},k),\\
m_{x,\omega}^{s} = m_{x,\check{\omega}}^{u},& \ \ m_{x,\omega}^{u} = m_{x,\check{\omega}}^{s},\\
\calL^u(x,\omega,0) =& \calL^s(x,\check{\omega},0).
\end{align*}
So our argument also gives us that for a.e. $(x,\omega,k)\in Z$, $m_{x,\omega}^{u}$ is absolutely continuous to Lebesgue on $\calL^u$ which is also 1-dimensional. Now note that our definition of subgroups compatible with the measure $\hat{m}$ item (2). of Definition \ref{compat} tells us that $U^+[x,\omega,k]\subset \calL^u(x,\omega,0)$, but both are 1-dimensional, thus $U^+[x,\omega,k] = \calL^u(x,\omega,0)$. In particular, $U^+[x,\check{\omega},k] = \calL^u(x,\check{\omega},0)= \calL^s(x,\omega,0)$. We therefore set $U^-(x,\omega,k) \defeq U^+(x,\check{\omega},k)$, and then clearly $\calL^s(x,\omega,0) = U^-[x,\omega,k]$. This completes the proof.
\end{proof}

In this alternative, we also have the statement of `joint integrability':
\begin{proposition}\label{full}
    For $m$-a.e. $(\hx,\ho)\in Y$, there exists a neighbourhood $U$ of $\hx$ in $X$ such that for $\nu$-a.e. $\hy\in U$ we have that $U_{loc}^+[\hx,\ho]\cap U^-_{loc}[\hy,\ho]\neq \emptyset$ and $U^-_{loc}[\hx,\ho]\cap U^+_{loc}[\hy,\ho]\neq \emptyset$. 
\end{proposition}
  Here the subscript $loc$ means it is of size $q(x,\omega)$ as defined in Section \ref{sec:stablemfld}. Proposition \ref{full} will be proven in Section \ref{sec:Sect3}.

\subsubsection{Alternative 3}

The third alternative in Theorem \ref{bigthmeasy} comes from the final case that $d_+=2$.

\begin{proposition}\label{3alt}
    The stable and unstable conditional measures are absolutely continuous with respect to Lebesgue.
\end{proposition}
\begin{proof}
    We use a proof from Section 4 of \cite{BEF} as we did in the proofs of Corollary~\ref{0.2} and Lemma \ref{0.6}. Note that now $\dim(T_xU^+)=2$ and that $\dim(\calL^s) \leq 2$. So our equations from Proposition \ref{prop4.3.8} simplify to 
    \begin{align*}
        h_{\hat{m}}(F) &\geq h(\mu) + 2\lambda^+\\
        h_{\hat{m}}(F^{-1}) &\leq  h_{\hat{m}}(F^{-1},W_1^{-})-2\dim( T_x\calL^s)\leq  h_{\hat{m}}(F^{-1},W_1^{-})-2\lambda^-.
    \end{align*}
Since we have that $h_{\hat{m}}(F^{-1},W_1^-)\leq h(\mu)$ (Proposition \ref{prop4.3.8}) and $\lambda^+ +\lambda^-\geq 0$ we have that
\begin{align}
    0\geq h(\mu) - h_{\hat{m}}(F^{-1},W_1^-) + 2\lambda^- + 2\lambda^+ \leq h_{\hat{m}}(F) - h_{\hat{m}}(F^{-1}).
\end{align}
Since $h_{\hat{m}}(F) = h_{\hat{m}}(F^{-1})$, we get that $h(\mu) = h_{\hat{m}}(F^{-1},W_1^-)$ and so $\nu$ is invariant by Proposition \ref{prop4.3.8}. Further
\begin{align}
    h(\mu) - \lambda^-\dim( T_x\calL^s) = h_{\hat{m}}(F^{-1}) = h(\mu) + 2\lambda^+,
\end{align}
    and so $\dim(\calL^s) = 2$. This equality allows us to use Proposition \ref{Prop4.3.10} to say that $m_{x,\omega}^s$ is equivalent to Lebesgue on $\calL^s$.

    Then, we run the same argument using $\check{\mu}$ as in Lemma \ref{0.6} and get the same homogeneity for the unstable conditional measures and that $U^+[x,\omega,k] = \calL^u(x,\omega,k)$ and $U^-[x,\omega,k] = \calL^s(x,\omega,k)$ as before.
\end{proof}
It follows immediately using \cite{LY2}, Corollary H, that $\nu$ is absolutely continuous with respect to Lebesgue.

\section{Tools required for the proof of Theorem \ref{bigthmeasy}}
\label{sec:sec3}

\subsection{Furstenberg theory}

This section presents material from Sections 2 and 5 of \cite{CD} in an effort to establish Corollary \ref{8.2} in Section \ref{sec:sect8}.

Recall from Section \ref{sec:prelim} that $H^{1,1}(X;\bbR)$, with the quadratic form it gets from the intersection form $\langle \cdot |\cdot \rangle$, is isometric to the standard quadratic form on Minkowski space $\bbR^{1,n}$ (where $n=h^{1,1}(X)-1$). We say that a subspace $W\subset \bbR^{1,n}$ is of Minkowski type if the quadratic form restricted to $W$ is non-degenerate and of signature $(1, \dim(W)-1)$. Then a subspace of $H^{1,1}(X;\bbR)$ is of Minkowski type if it is isometric to a Minkowski type subspace of $\bbR^{1,n}$. 

Let $\Gamma$ be a non-elementary subgroup of $O^+_{1,n}(\bbR)$.

\begin{lemma}[Lemma 2.6, Proposition 2.8, \cite{CD}]
One has a $\Gamma$-invariant decomposition
    \begin{align}
        H^{1,1}(X;\bbR) = H^{1,1}(X;\bbR)_+ \oplus H^{1,1}(X;\bbR)_0.
    \end{align}
    Here, $H^{1,1}(X;\bbR)_+$ is of Minkowski type, and the restriction of $\Gamma$ to this subspace is non-elementary. Further, $H^{1,1}(X;\bbR)_0$ is such that the restriction of the quadratic form to this subspace is negative definite and the restriction of $\Gamma$ to this subspace is contained in a compact subgroup of $GL(H^{1,1}(X;\bbR))$.
\end{lemma}
One can do a similar decomposition for $NS(X;\bbR)$ and one has $$ H^{1,1}(X;\bbR)_+=NS(X;\bbR)_+,$$ on which $\Gamma$ acts non-elementary (see \cite[Proposition 2.12]{CD}). We let 
$$\Pi_{\Gamma} \defeq H^{1,1}(X;\bbR)_+.$$

The limit set of $\Gamma$ (which is non-elementary), denoted $\lim(\Gamma)$, is the accumulation set of any $\Gamma$ orbit on $\bbP(v)$ for $v\not\in \Pi_{\Gamma}^{\perp}$. There are many equivalent definitions, see~\cite[2.3.6]{CD}. In our setting, $\lim(\Gamma) \subset \bbP(\Pi_{\Gamma})\cap \bbP(\nef(X))$ (see \cite[Lemmas 2.16, 2.17]{CD}).

Furstenberg theory gives us that there is a unique stationary measure $\nu_{\bbP(\Pi_{\Gamma})}$ on $\bbP(\Pi_{\Gamma}) \subset \bbP(H^{1,1}(X;\bbR))$ and it does not charge
any proper projective subspace of $\bbP(\Pi_{\Gamma})$. We will further describe this measure below in Proposition \ref{partial}.

Fix $\kappa_0$ is a Kahler form. Let the mass of a class $a$ be defined by 
\begin{align}\label{mass}
M(a) = \langle a | \kappa_o \rangle.
\end{align}

\begin{lemma}[Lemma 5.6, \cite{CD}]\label{ew}
    For $\mu^{\bbZ}$-a.e. $\omega\in \Omega$, there exists a unique nef class $e(\omega)$ with $M(e(\omega))=1$ and as $n\to \infty$ we have
    \begin{align}
        \frac{1}{M((f^n_{\omega})^*)a)}(f^n_{\omega})^*a \to e(\omega),
    \end{align}
    where $a$ is any pseudo class with $a^2>0$ (in particular any K\"ahler class). Also $e(\omega)$ is a.s. isotropic and $\bbP(e(\omega))$ is a point of $\lim(\Gamma)$.
\end{lemma}

\begin{proposition}[Theorem 5.8, \cite{CD}]
\label{partial}
    The probability measure $\nu_{\partial}$ is given by
    \begin{align}
        \nu_{\partial} = \int \delta_{\bbP(e(\omega))} d\nu^{\bbN}.
    \end{align}
    This measure is $\mu$-stationary, ergodic and has no atoms. It is supported on $\lim(\Gamma)$ and further, if it assigns positive measure to any subset $\Lambda \subset \lim(\Gamma)$, then $\Lambda$ is an uncountable set. It is also the unique stationary measure on $\bbP(H^{1,1}(X;\bbR))$ with $\nu_{\partial}(\bbP(\Pi_{\Gamma}^{\perp})) = 0$.
\end{proposition}

\subsection{Stable currents}
\label{stcur}
In this section we will deal with currents of type $(1,1)$. We will briefly interlude to some basics about currents following \cite[Sections 5, 6]{Csurvey}.

Let $\Lambda^{1,1}(X;\bbR)$ be the set of smooth real valued $(1,1)$ forms on $X$ with the Fr\'echet topology. Currents of type $(1,1)$ act like a dual to this space, they are continuous linear functionals on $\Lambda^{1,1}(X;\bbR)$. If $T$ is a current we write its action on some $\varphi \in \Lambda^{1,1}(X;\bbR)$ as $\langle T | \varphi \rangle$. We provide two important examples:
\begin{enumerate}
    \item Take $\alpha$ a continuous $(1,1)$ form. It defines a current: $$\langle \{\alpha\} | \varphi \rangle = \int_X \alpha \wedge \varphi.$$
    \item Take $C\subset X$ a curve. Then, $C$ (regardless of whether or not it is singular) defines a current: $$\langle \{C\} | \varphi \rangle = \int_C \varphi.$$ 
\end{enumerate}

We say a current is positive if it takes non-negative values on the convex cone of positive forms (i.e. (1,1) forms $\varphi$ such that $\varphi(u, \sqrt{-1}u)\geq 0$). We say that a current $T$ is closed if it vanishes on the subset of $\Lambda^{1,1}(X;\bbR)$ of exact forms. In Example~(1) above, the current defined by $\alpha$ is closed if and only if $\alpha$ is a closed form. It is positive if and only if $\alpha$ is a positive form. Example~(2) is closed and positive. For any closed current $T$, there is a unique cohomology class, $[T]$ such that
$$\langle T | \varphi \rangle  = \langle [T] | [\varphi]\rangle.$$

If $T$ a current is closed, then $\langle T| \varphi \rangle = \langle [T]| [\varphi] \rangle$ and we can define its mass analogously to how we did before in Equation \eqref{mass} as $M(T) = \langle [T]| \kappa_0 \rangle$.

\begin{theorem}[\cite{CD}, Corollary 6.13]\label{only}
    For a.e. $\omega\in \Omega$ we have that
    \begin{enumerate}
\item there exists unique closed positive current $T_{\omega}^s$ in the cohomology class of $e(\omega)$ (see Lemma \ref{ew}), 
\item for every K\"ahler form $\kappa$, as $n\to \infty$, we have
\begin{align*}
    \frac{1}{M((f^n_{\omega})^*\kappa)}(f^n_{\omega})^*\kappa \to T_{\omega}^s.
\end{align*}
    \end{enumerate}
\end{theorem}
We call $T_{\omega}^s$ the stable current associated to $\omega$. One can do similar constructions and get an unstable version of this result.

\subsection{Ahlfors Nevanlinna currents}
\label{sec:sect8}

Here we follow \cite[Section 8]{CD}. Recall that stable and unstable manifolds are injectively immersed entire curves $\psi: \bbC\to X$; we use this to construct a current like example (2) of Section \ref{stcur} where the curve $C$ is taken to be $\psi(B(0,t))$ where $B(0,t)$ is the ball in $\bbC$ centered at the origin of radius $t$. This current is denoted $\{\psi(B(0,t))\}$, i.e. the same notation as before. Let
\begin{align}
    A(R) &=\langle \{\psi(B(0,t))\} | \kappa_0  \rangle =  \int_{\psi(B(0,R))} \kappa_0, \\ T(R) &= \int_0^RA(t) \frac{dt}{t}.
\end{align}

Now, \cite{Brun} (see \cite[Proposition 8.1]{CD}) tells us that there are sequences $R_n\to \infty$ such that the sequence of currents given by
$$\frac{1}{T(R_n)} \int_0^{R_n} \{\psi(B(0,t))\} \frac{dt}{t},$$
converge to a closed positive current $T$. We call this closed positive current $T$ an Alhfors Nevanlinna current associated to the curve $\psi$. 

\begin{theorem}[\cite{CD}, Theorem 8.2/Corollary 8.3]\label{thm8.2}
    For $X$ a projective complex K\"ahler surface, $\mu$ a probability measure on $\aut(X)$, $\nu$ a hyperbolic $\mu$-stationary measure satisfying the moment condition in Equation \eqref{IC}, if there is no $\Gamma$-invariant algebraic curve then for $m$-a.e. $(x,\omega)$, the only normalized Ahlfors-Nevanlinna current associated to $W^s(x,\omega)$ is $T_{\omega}^s$. 
\end{theorem}

\begin{corollary}\label{8.2}
    In the setting of Theorem \ref{thm8.2}, for $m$-a.e. $(x,\omega)$, and $\mu^{\bbZ}$-a.e. $\omega'$, $W^s(x,\omega) \neq W^s(x,\omega')$. 
\end{corollary}
\begin{proof}
    Theorem \ref{thm8.2} tells us that in our setting we have that for a.e. $(x,\omega)$, the only normalized Alhfors-Nevanlinna current associated to $W^s(x,\omega)$ is $T^s_{\omega}$. Combining this with Proposition \ref{partial} and Theorem \ref{only} we get that $e(\omega)$ is distinct from $e(\omega')$ for a.e. $\omega'$ and that $T^s_{\omega}$ is the unique closed positive current in the cohomology class of $e(\omega)$. This tells us that $T_{\omega}^s$ is in its own distinct cohomology class for a.e. $\omega$. Relating this back to the stable manifolds, we see that two stable manifolds cannot be the same if they have different $T^s_{\omega}$ and so the result stands.
\end{proof}
\begin{remark}\label{8.2'}
    The analogue of Corollary \ref{8.2} for unstable manifolds also holds. 
\end{remark}
Let $\Lambda_s$ be the conull subset of $Y$ such that Corollary \ref{8.2} holds.

Theorem~\ref{thm8.2} and Corollary~\ref{8.2} apply generally in our setting, regardless of the value of $d_+$ in Section~\ref{sec:breakdown}. We will use this theorem for the case $d_+=0$ to prove Lemma \ref{0.5} in Section \ref{sec:twolem} and Proposition~\ref{finite} in Section~\ref{sec:alt1}. In the case that $d_+=1$, we use~\cite[Theorem C]{CD} which is stated below in the next section as Theorem~\ref{thm8.2}. This theorem give us that the Oseledets' directions are random in the setting $d_+=1$. This is because in this case we have positive fiber entropy and hence land in case c) of Theorem~\ref{thmC}. In the case $d_+=0$ we have zero fiber entropy and hence are in alternative a) of Theorem~\ref{thmC}.

 \subsection{Fibered entropy and Theorem C of \cite{CD}}

\label{sec:fiber}
We give the following result of~\cite{CD} which gives us randomness of Oseledets' directions in the case of positive fiber entropy:

\begin{theorem}[\cite{CD}, Theorem C, Theorem 9.1)]
\label{thmC}
    Let $X$ be a complex projective surface and $\mu$ a finitely supported probability measure on $\aut(X)$. If the group generated by $\supp(\mu)$, denoted $\Gamma_{\mu}$, is non-elementary, then any hyperbolic, ergodic $\mu$-stationary measure $\nu$ on $X$ satisfies (exactly) one of the following: 
    \begin{packedalph}
        \item $\nu$ is invariant and its fiber entropy, $h_{\nu}(X,\mu)$, vanishes,
        \item $\nu$ is supported on an algebraic curve that is $\Gamma_{\mu}$-invariant,
        \item the field of Oseledets' stable directions, i.e. the $E^s(x,\omega)$,  is not $\Gamma_{\mu}$-invariant. That is to say, it depends on the choice of future in $\omega$.
    \end{packedalph}
\end{theorem}

Case c) gives us that for almost every $(x,\omega)\in \Lambda$, a.e. $\omega' \in \Omega$ is such that $E^s(x,\omega) \neq E^s(x,\omega')$. It is similar for the unstables.

As noted previously, in the case that $d_+=0$, i.e. alternative 1 from Section \ref{sec:breakdown}, our fiber entropy is zero and we are in case a) of Theorem \ref{thmC}. Hence we only have that the stable manifolds are a.e. distinct for different $\omega\in \Omega$, see Section \ref{sec:sect8}, Corollary 8.2.

\subsubsection{Elementary lemma}

Fix $x\in X$ and consider the family of measures on $Gr(1,\bbC^2)$ given by $\rho_{(x,\omega)}^s \defeq \delta_{E^s(x,\omega)}$. By Theorem \ref{thmC} above, for a.e. $\omega^+\in \Omega^+$ these measures are distinct. Let $P^s_x \defeq \int_{\Omega^+} \rho_{(x,\omega)}^s d\mu^{\bbN}$, then this measure is non-atomic. Similarly, we define $\rho_{(x,\omega)}^u \defeq \delta_{E^u(x,\omega)}$ and let $P^u_x \defeq \int_{\Omega^-} \rho_{(x,\omega)}^u d\mu^{\bbN}$; this measure is also non-atomic.

\begin{lemma}\label{anglecontrol}
Consider $P$ a non-atomic probability measure on a compact metric space $M$. For every sequence $\epsilon_n\to 0$ there exists a sequence $\delta_n\to 0$ such that for all $x\in M$ and $n\in \bbN$ we have
\begin{align}
    P(B(x,\epsilon_n))<\delta_n.
\end{align}
\end{lemma}
This elementary lemma will be useful in Sections \ref{proof0.5} and \ref{sec:Sect3}.

\subsection{Non-uniform hyperbolicity and Pesin theory}
\label{sec:Sec22}

Many of the theorems used in non-uniformly hyperbolic settings like our own were originally only written for the case of single iterate maps (with restrictions on the measure $\nu$) and are due to Pesin (\cite{P1}, \cite{P2}) Many of these theorems have versions in our setting, the `random dynamics' setting, where instead of a single map $f$ acting as the dynamics, we have $f_{\omega}^n$ as in Equations \eqref{forwards} and \eqref{backwards}. The books written by \cite{K} and \cite{LQ} are standard references for this material. 

Let $\epsilon_0\ll \min\{1,\lambda^+/200, -\lambda^-/200\}.$

Since we have assumed that we have no zero Lyapunov exponents, we have ``non-uniform hyperbolicity'': upper and lower bounds on the action of $D_xf_{\omega}^n$ on the Oseledets' subspaces, as well as bounds on the angles between these subspaces. In the case of uniform hyperbolicity, these bounds are point independent, but this is not the case for non-uniform hyperbolicity. We instead get a statement of the following form:
\begin{lemma} [\cite{BRH}, Lemma 6.2, see also \cite{BP}, Lemma 3.5.7 and \cite{LQ}, Lemma 3.2] 
    Given the integrability condition, \eqref{IC}, \label{NUH}
    there is a measurable function $L: X \times \Omega\to [1,\infty[$ such that for $\mu$-a.e. $(x,\omega)$ (i.e. those in $\Lambda$) and $n\in \bbZ$ we have
    \begin{enumerate}
        \item For $v\in E^s(x,\omega)$,
        $$L(x,\omega)^{-1}e^{n\lambda^- - |n|\frac{1}{2}\epsilon_0}\Vert v\Vert \leq \Vert Df_{\omega}^nv\Vert \leq L(x,\omega)e^{n\lambda^-+|n|\frac{1}{2}\epsilon_0}\Vert v\Vert.$$
        \item For $v\in E^u(x,\omega)$,
        $$L(x,\omega)^{-1}e^{n\lambda^+ - |n|\frac{1}{2}\epsilon_0}\Vert v\Vert \leq \Vert Df_{\omega}^nv\Vert \leq L(x,\omega)e^{n\lambda^++|n|\frac{1}{2}\epsilon_0}\Vert v\Vert.$$
        \item It holds\footnote{
    Here, $\angle$ denotes Riemannian angle between the subspaces.} that
        $$\angle \Big( E^s[f_{\omega}^n(x),\sigma^n(\omega)],E^u[f_{\omega}^n(x),\sigma^n(\omega)] \Big) >\frac{1}{L(x,\omega)}e^{-|n|\epsilon_0}.$$
        \item It holds that 
        $$L(f_{\omega}^n(x),\sigma^n(\omega)) \leq L(x,\omega)e^{\epsilon_0|n|}.$$
    \end{enumerate}
\end{lemma}

\subsubsection{Lyapunov norms}

Lyapunov norms are an adapted norm meant to have the action of $D_xf$ reflect the Lyapunov exponents after one step in the random walk, see \cite[Section 9.2]{BRH} or \cite[chapter VI, Section 3]{LQ}. One can also see \cite{LY} for the single iterate case.

For $(x,\omega)\in Y$, $b \in  \{s,u\}$ and $v\in E^{b}(x,\omega)$ we define the Lyapunov norm (two-sided) as
$$|v|^{b}_{\epsilon_0,(x,\omega)} \defeq \left(\sum_{n\in \bbZ} \Vert Df_{\omega}^n v\Vert^2 e^{-2\lambda^{b}n-2\epsilon_0|n|} \right)^{1/2}.$$
These sums converge for $m$-a.e. $(x,\omega)\in Y$ by Lemma \ref{NUH}, also for $v\in E^{b}(x,\omega)$ we have:
$$|v|^{b}_{\epsilon_0,(x,\omega)}\geq \Vert v\Vert.$$

\begin{lemma}
    [Lemma 9.2 of \cite{BRH}] For $m$-a.e. $(x,\omega)\in Y$, $v\in E^{b}(x,\omega)$, $n\in \bbZ$ and $k\geq 0$ we have

 \begin{equation}
e^{n\lambda^{b}-|n|\epsilon_0} |v|^{b}_{\epsilon_0,(x,\omega)} \leq |Df_{\omega}^nv|^{b}_{\epsilon_0, (f_{\omega}^n(x),\sigma(\omega))} \leq e^{n\lambda^b+|n|\epsilon_0}|v|^b_{\epsilon_0,(x,\omega)} .
\end{equation}

\end{lemma}

From here, we pass to Lyapunov charts which rely on the construction of the Lyapunov norm. These charts takes this norm to the standard Euclidean norm and our Oseledets' directions to the standard orthogonal coordinate system in $\bbR^4$.

\subsubsection{Lyapunov charts}
\label{sec:Lyapcharts}

For every $0<\epsilon <\epsilon_0$, there is a measurable function $r : Y \to [1,\infty[$ and a full measure set $\Lambda \subset Y$ (same as the Oseledets' set) such that 
\begin{itemize}
    \item For $(x,\omega)\in \Lambda$ there is a neighbourhood $U(x,\omega) \subset X$ containing $x$ and a $C^{\infty}$ diffeomorphism $\varphi_{(x,\omega)} : U(x,\omega) \to B(r(x,\omega)^{-1})$ (where $B(r)$ is a ball of radius $r$ in $\bbR^4$) with
    \begin{itemize}
        \item $\varphi_{(x,\omega)}(x)=0$,
        \item $D\varphi_{(x,\omega)}(E^u(x,\omega)) = \bbR \times \bbR \times \{0\} \times \{0\} \defeq \bbR^u$,
        \item $D\varphi_{(x,\omega)}(E^s(x,\omega)) =  \{0\} \times \{0\} \times \bbR \times \bbR \defeq \bbR^s$.
    \end{itemize}
    \item Let $\tilde{f}_{(x,\omega)} = \varphi_{(f_{\omega}(x),\sigma(\omega))}\circ f_{\omega}^1 \circ \varphi_{(x,\omega)}^{-1}$ and\\ $\tilde{f}^{-1}_{(x,\omega)} = \varphi_{((f_{\omega}^{-1}(x),\sigma^{-1}(\omega))} \circ f_{\omega}^{-1}\circ \varphi_{(x,\omega)}^{-1}$ then,
    \begin{itemize}
        \item $\tilde{f}_{(x,\omega)}(0)=0$,
        \item $e^{\lambda^+-\epsilon}|v|\leq |D_0\tilde{f}_{(x,\omega)}v| \leq e^{\lambda^++\epsilon}|v|$, for $v\in \bbR \times \bbR \times \{0\} \times \{0\}$,\\ 
        $e^{\lambda^--\epsilon}|v|\leq |D_0\tilde{f}_{(x,\omega)}v| \leq e^{\lambda^-+\epsilon}|v|$ for $v\in \{0\} \times \{0\} \times \bbR \times \bbR$,
        \item Lip$(\tilde{f}_{(x,\omega)} - D_0\tilde{f}_{(x,\omega)})<\epsilon$,
        \item Lip$(D\tilde{f}_{(x,\omega)})<r(x,\omega)$.
    \end{itemize}
    \item Similar properties for $\tilde{f}_{(x,\omega)}^{-1}$.
    \item There exists a uniform constant $\psi$ with $\psi^{-1}\leq $Lip$(\varphi_{(x,\omega)}) \leq r(x,\omega)$ for all $(x,\omega)\in \Lambda$.
    \item $r((f_{\omega}^n(x),\sigma^n(\omega)) \leq r(x,\omega)e^{|n|\epsilon}$ for all $n\in \bbZ$.
\end{itemize}

Let us set notation for the action of $f$ in the Lyapunov chart. Given $\omega = (\dots f_{-1},f_0,f_1,\dots)$, define
\begin{align}\label{tildef}
\tilde{f}_{(x,\omega)} = \varphi_{(f_{\omega}(x),\sigma(\omega))}\circ f_{\omega}^1 \circ \varphi_{(x,\omega)}^{-1} = \varphi_{(f_{\omega}(x),\sigma(\omega))}\circ f_0 \circ \varphi_{(x,\omega)}^{-1}.
\end{align}
We will denote $\tilde{f}_{(x,\omega)}$ as $\tilde{f}_0$ in settings where $(x,\omega)$ has been made clear. Further, let
\begin{align}\label{tildefn}
    \tilde{f}_n\defeq \tilde{f}_{(f^{n}x,\sigma^n(\omega))} = \varphi_{(f_{\omega}^{n+1}(x),\sigma^{n+1}(\omega))}\circ f_n \circ \varphi_{(f_{\omega}^{n}(x),\sigma^n(\omega))}^{-1}.
\end{align}
Then we can write $\tilde{f}_{(x,\omega)}^n = \tilde{f}_{n-1}\circ \dots \circ \tilde{f}_0$.

\subsubsection{Into charts -- distortion estimates}
\label{sec:distortion}

 Working in Lyapunov charts gives us control over the first derivative, $D_0\tilde{f}_{(x,\omega)}$, by giving us uniform hyperbolicity. In this section we work to control the action of the second derivative, $D^2_xf_{\omega}$, obtaining a na\"ive bound to use in subsequent sections. Note that we only use the standard flow in this section.

Let $(x,\omega)\in \Lambda$, we get that there is a neighbourhood $U(x,\omega)\subset X$ of $x$ and a $C^{\infty}$ diffeomorphism $\varphi_{(x,\omega)} : U(x,\omega)\to B(r(x,\omega)^{-1})$ with the properties defined in Section~\ref{sec:Lyapcharts}. As in \ref{sec:Lyapcharts}, will use $|\cdot |$ to denote Euclidean norm in the chart. Let $k\in [0,1[$. We define 
\begin{align}\label{chartflow}
\tilde{F}_{(x,\omega),k}^t  :\bbR^4 \to \bbR^4, \ \ \tilde{F}_{(x,\omega),k}^t(v) = \tilde{f}^{\lfloor t+k\rfloor}_{(x,\omega)}(v), 
\end{align}
where $\tilde{f}_{(x,\omega)}$ is as defined in Equation \eqref{tildef} and $t$ is the amount we `flow' by (on $Z$). This is the action of `standard' flow on $X$ written in Lyapunov charts.

Taylor expanding our flow in charts, we get that there is a point $w$ on the line between $0$ and $v$ such that 
\begin{align}\label{tay}
    \tilde{F}_{(x,\omega),k}^t(v)- \tilde{F}_{(x,\omega),k}^t(0) = D\tilde{F}_{(x,\omega),k}^t(0)v + \frac{1}{2}v^TD^2 \tilde{F}_{(x,\omega),k}^t(0)(w)v.
\end{align}
Using this Equation \eqref{tay} we get the upper bound, 
\begin{align}\label{upper}
    |\tilde{F}_{(x,\omega),k}^t(v)- \tilde{F}_{(x,\omega),k}^t(0)|\leq |D\tilde{F}_{(x,\omega),k}^t(0)v|  +\frac{1}{2}\sup_{w\in [0,v]}|D^2 \tilde{F}_{(x,\omega),k}^t(0)(w)|\cdot |v|^2.
\end{align}

Additionally we will need a lower bound. Using Equation \eqref{tay} with the reverse triangle inequality we get
\begin{align} \label{lower}
    |\tilde{F}_{(x,\omega),k}^t(v)- \tilde{F}_{(x,\omega),k}^t(0)|\geq  
\left| |D\tilde{F}_{(x,\omega),k}^t(0)v| - \frac{1}{2}|v^TD^2\tilde{F}_{(x,\omega),k}^t(w)v| \right|.
\end{align}

Now, to bound the second derivative, we look at the action of a diffeomorphism on a subspace of the space of 2-jets for $\bbR^4$ based at the point 0 (it gives a linear map). In general, the space of k-jets based at a point $x_0$ is denoted $J_{x_0}^k(\bbR^n,\bbR^m)$ and is essentially the space of all $k$th Taylor polynomials of smooth functions $f:\bbR^n \to \bbR^m$ based at a point $x_0$. More precisely, it is the set of all equivalence classes of smooth functions $f:\bbR^n\to \bbR^m$ such that $f\sim g$ if $f$ and $g$, as well as all their first $k$ partials agree at the point $x_0$. We will denote $j_{x_0}^k(f)$ to be the $k$th Taylor polynomial of $f$ centered at $x_0$. For any $y_0\in \bbR^m$, the real vector space $J_{x_0}^k(\bbR^n,\bbR^m)$ has the subspace,
\begin{align}
    J_{x_0}^k(\bbR^n,\bbR^m)_{0} \defeq \{j_{x_0}^k(f) \in J_{x_0}^k(\bbR^n,\bbR^m) : f(x_0)=0   \}.
\end{align}

For our purposes we take $n=m=4$, $k=2$, and $x_0=0$; we denote this $J_0^2(\bbR^4)_0$. Similarly, as above, for smooth functions $f:\bbR^4\to \bbR^4$ we write $j_0^2(f)\in J^2_0(\bbR^4)$ for the 2nd Taylor polynomial of $f$ at 0. Note that if $f,g:\bbR^4\to \bbR^4$ are smooth functions mapping the origin to the origin, we have the property that 
\begin{align}\label{alg}
j_0^2(f \circ g) = j_{g(0)}^2(f)\circ j_0^2(g).
\end{align}
Note that if $t\in \bbR^+$ is such that $\lfloor t+k\rfloor = n+1$ and $\omega = (\dots \tilde{f}_{0},\dots \tilde{f}_n, \dots )$, then $\tilde{F}^t_{(x,\omega),k}$ is composed of the diffeomorphisms $\tilde{f}_{0},\dots \tilde{f}_n$. Since we are working in charts, these diffeomorphisms always take the origin to the origin.

Further, a smooth function $f:\bbR^4\to \bbR^4$ that maps the origin to the origin gives us a map on the space $J_0^2(\bbR^4)_0$ by pre-composition: 
\begin{align}
    J_0^2(\bbR^4)_0 &\to J_0^2(\bbR^4)_0,\\
    h(z)&\mapsto  h\circ j_0^2(f)(z) \mod z^{\otimes 3}. \label{precomp}
\end{align}
We endeavor to write out this map as a matrix. First one needs to write down a basis for $J_0^2(\bbR^4)_0$; recall that the $k$th Taylor polynomial for a smooth function $f:\bbR^n\to \bbR$ can be written as
\begin{align}
    j_{x_0}^k(f) = \sum_{0\leq |\beta| \leq k} \frac{D^{\beta}f(x_0)}{\beta!}(x-x_0)^{\beta},
\end{align}
where $\beta$ is a multi-index $\beta = (\beta_1,\dots \beta_n)$ with $\beta_i\in \{0,\dots k\}$, $|\beta | = \sum \beta_i$, $\beta! = \prod_{i=1}^n \beta_i!$, and $(x-x_0)^{\beta} = \prod_{i=1}^n(x_i-(x_0)_i)^{\beta_i}$. Hence we can write down a basis for $J_0^2(\bbR^4,\bbR)_0$ as $$\{x_1,x_2,x_3,x_4,x_1^2,\dots x_4^2,x_1x_2,\dots x_1x_4,x_2x_3,x_2x_4,x_3x_4\}.$$
We then write $f:\bbR^n\to \bbR^n$ as $f(x_1,\dots ,x_n) = (f_1(x_1,\dots , x_n), \dots f_n(x_1,\dots ,x_n))$, where $f_i : \bbR^n \to \bbR$. We see that we can write $J_0^2(\bbR^4) = J_0^2(\bbR^4,\bbR)\oplus \dots \oplus J_0^2(\bbR^4,\bbR)$. 

We see that the matrix representing the precomposition map (with respect to $f:\bbR^4\to \bbR^4$) in Equation \eqref{precomp} needs to be $60\times 60$ because it needs to be $15\times 15$ for the precomposition action of $f:\bbR^4\to \bbR^4$ on each component $J_0^2(\bbR^4,\bbR)$ (given the size of the basis of $J_0^2(\bbR^4,\bbR)$ is 15). To demonstrate what is going on, we will perform this computation in a simpler case for the map defined by $f:\bbR^2\to \bbR^2$ on $J_0^2(\bbR^2)_0$ where on each $J_0^2(\bbR^2,\bbR)$ component it is only a 5$\times 5$ matrix. 

Let $h\in J_0^2(\bbR^2)_0$, then it is of the form $ h (x,y) = (h_1(x,y),h_2(x,y))$ where
$$h_i(z) = h_i(x,y) = a_1x + a_2y + a_3x^2 + a_4y^2+a_5xy \in J_0^2(\bbR^2,\bbR).$$
Then
\begin{align*}
&j_0^2(f)(h_i(z)) = h_i(J_0^2(f)(z)) = a_1(\frac{\partial f_1}{\partial x}x + \frac{\partial f_1}{\partial y}y + \frac{\partial^2 f_1}{\partial x^2}x^2 + \frac{\partial^2 f_1}{\partial y^2}y^2 + \frac{\partial^2 f_1}{\partial x\partial y}xy)\\
&+ a_2(\frac{\partial f_2}{\partial x}x + \frac{\partial f_2}{\partial y}y + \frac{\partial^2 f_2}{\partial x^2}x^2 + \frac{\partial^2 f_2}{\partial y^2}y^2 + \frac{\partial^2 f_2}{\partial x\partial y}xy)\\
&+ a_3(\frac{\partial f_1}{\partial x}x + \frac{\partial f_1}{\partial y}y + \frac{\partial^2 f_1}{\partial x^2}x^2 + \frac{\partial^2 f_1}{\partial y^2}y^2 + \frac{\partial^2 f_1}{\partial x\partial y}xy)^2\\
&+ a_4(\frac{\partial f_2}{\partial x}x + \frac{\partial f_2}{\partial y}y + \frac{\partial^2 f_2}{\partial x^2}x^2 + \frac{\partial^2 f_2}{\partial y^2}y^2 + \frac{\partial^2 f_2}{\partial x\partial y}xy)^2\\
&+ a_5(\frac{\partial f_1}{\partial x}x + \frac{\partial f_1}{\partial y}y + \frac{\partial^2 f_1}{\partial x^2}x^2 + \frac{\partial^2 f_1}{\partial y^2}y^2 + \frac{\partial^2 f_1}{\partial x\partial y}xy)\\
&(\frac{\partial f_2}{\partial x}x + \frac{\partial f_2}{\partial y}y + \frac{\partial^2 f_2}{\partial x^2}x^2 + \frac{\partial^2 f_2}{\partial y^2}y^2 + \frac{\partial^2 f_2}{\partial x\partial y}xy), \\
\end{align*}
In the above, the variable $x$ in $h_i(x,y)$ gets replaced with the Taylor polynomial of $f_1$ and the variable $y$ in $h_i(x,y)$ gets replaced by the Taylor polynomial of $f_2$ (where $f(z)=f(x,y)=(f_1(x,y),f_2(x,y))$.\\

Rearranging we get
\begin{align*}
j_0^2(f)(h(z)) &= x(a_1\frac{\partial f_1}{\partial x} + a_2 \frac{\partial f_2}{\partial x})+ y(a_1\frac{\partial f_1}{\partial y} + a_2 \frac{\partial f_2}{\partial y})\\ 
&+ x^2(a_1\frac{\partial^2 f_1}{\partial x^2} + a_2 \frac{\partial^2f_2}{\partial x^2} + a_3\left(\frac{\partial f_1}{\partial x}\right)^2  +a_4 \left(\frac{\partial f_2}{\partial x}\right)^2 + a_5 \left(\frac{\partial f_1}{\partial x}\cdot \frac{\partial f_2}{\partial x}\right))\\
&+ y^2(a_1\frac{\partial^2 f_1}{\partial y^2} + a_2 \frac{\partial^2f_2}{\partial y^2} + a_3\left(\frac{\partial f_1}{\partial y}\right)^2  +a_4 \left(\frac{\partial f_2}{\partial y}\right)^2 + a_5 \left(\frac{\partial f_1}{\partial y}\cdot \frac{\partial f_2}{\partial y}\right))\\
&+ xy\left( a_1\frac{\partial^2f_1}{\partial x \partial y} + a_2 \frac{\partial^2 f_2}{\partial x \partial y} + a_3 \frac{\partial f_1 }{\partial x}\frac{\partial f_1 }{\partial y} + a_4  \frac{\partial f_2 }{\partial x}\frac{\partial f_2 }{\partial y} + a_5  \frac{\partial f_1 }{\partial y}\frac{\partial f_2 }{\partial x}  \right)
\end{align*}

So then the associated matrix (fixing the basis $\{x,y,x^2,y^2,xy\}$ in this order) is 
\begin{align} \label{Tf}
\tilde{T}_f=\begin{pmatrix}
\frac{\partial f_1}{\partial x} & \frac{\partial f_2}{\partial x} & 0 & 0 & 0\\
\frac{\partial f_1}{\partial y} & \frac{\partial f_2}{\partial y} & 0 & 0 & 0\\
\frac{\partial^2 f_1}{\partial x^2} & \frac{\partial^2 f_2}{\partial x^2} & \left(\frac{\partial f_1}{\partial x}\right)^2 & \left( \frac{\partial f_2}{\partial x}\right)^2 & \frac{\partial f_1}{\partial x}\cdot \frac{\partial f_2}{\partial x}\\
\frac{\partial^2 f_1}{\partial y^2} & \frac{\partial^2 f_2}{\partial y^2} & \left(\frac{\partial f_1}{\partial y}\right)^2 & \left( \frac{\partial f_2}{\partial y}\right)^2 & \frac{\partial f_1}{\partial y}\cdot \frac{\partial f_2}{\partial y}\\
\frac{\partial^2 f_1}{\partial x\partial y} & \frac{\partial^2 f_2}{\partial x\partial y} & \frac{\partial f_1}{\partial x}\frac{\partial f_1}{\partial y} & \frac{\partial f_2}{\partial x}\frac{\partial f_2}{\partial y} & \left(  \frac{\partial f_1}{\partial x}\frac{\partial f_2}{\partial y} +  \frac{\partial f_2}{\partial x}\frac{\partial f_1}{\partial y}  \right)
\end{pmatrix}
\end{align}
In the top left block (2x2) we get the Jacobian of $f$, in the bottom left block (3x2) we have second derivatives. In the right most bottom block (3x3) we get products of the first partials. Then overall, for the action of $f$ on $J^2_0(\bbR^2)_0$ we get a $10\times 10$ (2 copies of this matrix on the diagonal). We would call this $T_f$. A similar construction occurs for our setting with maps $f:\bbR^4\to \bbR^4$ acting on $J_0^2(\bbR^4)_0$, again, call this matrix $T_f$. 

Let us return to our setting now.
\begin{lemma}[\cite{FM}]
Fix $(x,\omega,k)\in Z$ and let $t>0$, consider $\tilde{F}^t_{(x,\omega),k}$, there exists $C$ such that $|D^2\tilde{F}^t_{(x,\omega),k}|\leq C_{t,(x,\omega,k)}^n$ where $n = \lfloor t+k\rfloor$. In particular, if $\supp(\mu)<\infty$, then there exists $C$ such that for any $(x,\omega,k)\in Z$ and $t>0$ we have that $|D^2\tilde{F}^t_{(x,\omega),k}|\leq C^n$ where $n = \lfloor t+k\rfloor$.   
\end{lemma}
\begin{proof}
Fix $(x,\omega,k)\in Z$.
Since $X$ is compact, every first and second partial derivative of $f\in \supp(\mu)$ has a na\"ive bound on it, call the maximum of these $C_f'$. Entries in our matrix $T_f$ then are bounded above by $C_f \defeq 2(C_f')^2$ as they may be products and/or sums of the first or second derivatives. Now if we take $t>0$ and flow by it under $\omega$, we have diffeomorphisms, $\tilde{f}_0,\dots \tilde{f}_{n} :\bbR^4\to \bbR^4$ where $n = \lfloor t+k \rfloor$. The map $T_{\tilde{f}_n\circ \dots \circ \tilde{f}_0}$ associated to the action of $\tilde{f}_n\circ \dots \circ \tilde{f}_0$ on $J_0^2(\bbR^4)_0$ splits into matrices $T_{\tilde{f}_n}\cdot \dots \cdot T_{\tilde{f}_0}$.
Each $T_{\tilde{f}_i}$ and $T_{\tilde{f}_n\circ \dots \circ \tilde{f}_0}$ is in a form similar to that described above in Equation \eqref{Tf}, but as explained, ours is a $60\times 60 $ matrix. In particular, the first and second partials of $\tilde{f}_i$, resp. $\tilde{f}_n\circ \dots \circ \tilde{f}_0$ in these matrices are arranged in an analogous way to Equation \eqref{Tf}. Since $T_{\tilde{f}_n}\cdot \dots \cdot T_{\tilde{f}_0} = T_{\tilde{f}_n\circ \dots \circ \tilde{f}_0}$, the first and second partials of $T_{\tilde{f}_n\circ \dots \circ \tilde{f}_0}$ are products and sums of the first and second partials of the $\tilde{f}_i$. 

Given that each entry in $T_{\tilde{f}_i}$ is bounded above by $C_{\tilde{f}_i}$ and that the size of the matrix is $60\times 60$, multiplying any two $T_{\tilde{f}_i}$ and $T_{\tilde{f}_j}$ gives a bound on the entries of the product given by $60
C_{\tilde{f}_i}\cdot C_{\tilde{f}_j}$. So for a given set of diffeomorphisms $\tilde{f}_1,\dots \tilde{f}_n$ we let
\begin{align}\label{Const}
C_{t,(x,\omega,k)} \defeq \sqrt{60}\max_{0\leq i\leq n}C_{\tilde{f}_i}. 
\end{align}
We conclude that all second derivatives of $\tilde{f}_n\circ \dots \circ \tilde{f}_0$ are bounded above by $C_{t,(x,\omega,k)}^n$, and so $|D^2\tilde{F}^t_{(x,\omega),k}|\leq C_{t,(x,\omega,k)}^n$. This $C_{t,(x,\omega,k)}$ can be taken to be independent of $\omega\in \Omega$ because $\supp(\mu)<\infty$ and $X$ is compact. Hence this na\"ive bound is only dependent on the number of diffeomorphisms we draw.
\end{proof}

\subsubsection{Stable and unstable manifolds}
\label{sec:stablemfld}

 Here we state a version of the local stable manifold theorem from \cite{BRH} that turns a local stable (or unstable) manifold into a holomorphic graph in the Lyapunov charts.

\begin{proposition}[\cite{BRH}, Theorem 6.4]\label{locstablemfld} 
    For $(x,\omega)\in \Lambda$ there exists holomorphic functions $$h^s_{(x,\omega)} : \bbR^s(r(x,\omega)^{-1}) \to \bbR^u(r(x,\omega)^{-1}),$$ $$h^u_{(x,\omega)} : \bbR^u(r(x,\omega)^{-1}) \to \bbR^s(r(x,\omega)^{-1}),$$ (where $r(x,\omega)$ is from Section \ref{sec:Lyapcharts}) such that,
    for any $*\in \{\text{s,u}\}$,
    \begin{enumerate}
        \item $h^*_{(x,\omega)}(0) = 0,$
        \item $D_0h^*_{(x,\omega)}=0,$
        \item $\Vert Dh^*_{(x,\omega)}\Vert\leq 1/3.$
    \end{enumerate}

Define
\begin{align}
W_{q(x,\omega)}^s(x,\omega) = \varphi^{-1}_{(x,\omega)}(\text{graph}(h^s_{(x,\omega)})),\quad 
W_{q(x,\omega)}^u(x,\omega) = \varphi^{-1}_{(x,\omega)}(\text{graph}(h^u_{(x,\omega)})),
\end{align}
where $q(x,\omega)$ is a measurable function dependent on $r(x,\omega)$. These are our local stable and unstable manifolds respectively. They have the following properties:
\begin{enumerate}
\item If $y\in W_{q(x,\omega)}^s(x,\omega)$ then for every $n\in \bbN$ we have
\begin{align*}
    d_X(F^n(x,\omega),F^n(y,\omega)) \leq k(x,\omega)\exp((\lambda^-+\delta)n)d_X(x,y).
\end{align*}
\item If $y\in W_{q(x,\omega)}^u(x,\omega)$ then for every $n\in \bbN$ we have
\begin{align*}
    d_X(F^{-n}(x,\omega),F^{-n}(y,\omega)) \leq k(x,\omega)\exp(-(\lambda^+-\delta)n)d_X(x,y).
\end{align*}
\item We have the following inclusions:
\begin{align*}
    F(W^s_{q(x,\omega)}(x,\omega)) \subset W^s_{q(F(x,\omega))}(F(x,\omega)), \\ F^{-1}(W^u_{q(F(x,\omega))}(F(x,\omega))) \subset W_{q(x,\omega)}^u(x,\omega).
\end{align*}    
\end{enumerate}
Where $k(x,\omega)$ is a measurable function dependent on $r(x,\omega)$.

\end{proposition}
For the cost of some $\epsilon_{loc}>0$ small, one can use Lusin's theorem to replace $k$ and $q$ with uniform constants in the Proposition \ref{locstablemfld} on some compact set $K_{loc}$ of measure $m(K_{loc})>1-\epsilon_{loc}$.

The global stable and unstable manifolds are defined by 
\begin{align*}
    W^s(x,\omega) &= \bigcup_{n\geq 0}F^{-n}\left( W^s_{q(x,\omega)}(F^n(x,\omega)) \right), \\
    W^u(x,\omega) &= \bigcup_{n\geq 0}F^{n}\left( W^u_{q(x,\omega)}(F^{-n}(x,\omega)) \right).
\end{align*}
As explained in Section \ref{sec:prelim}, $W^{s/u}(x,\omega)$ are injectively immersed holomorphic curves in $X_{\omega}$ that are biholomorphic to $\bbC$ for $m$-a.e. $(x,\omega)$ (see \cite[Proposition 7.8]{CD}) and are equivalent to what is given in Equations \eqref{stable} and \eqref{unstable} resp.

Let $K_{hol}$ be the conull set of $(x,\omega)\in Y$ for which our global stables and unstables are biholomorphic to $\bbC$. Let $\Lambda_{loc} \defeq K_{loc}\cap K_{hol}$ and note that it is almost conull.

\subsection{H\"older continuity of the Oseledets' subspaces}
\label{sec:Hcont}

For this section we largely follow the conventions and results of \cite[Chapter 3, Section 4]{LQ}. 

Given a metric space $X$ and a family of subspaces $\{E_x\}_{x\in X}$ in a Hilbert space $H$, this family is called ``$\alpha$-H\"older continuous in $x$'' for some $\alpha>0$ if there is some $L>0$ such that for any $x,y\in X$ we have 
\begin{align}
    d_{Gr}(E_x,E_y) \leq L d_X(x,y)^{\alpha},
\end{align}
where 
\begin{align}\label{grassmetric}
d_{Gr}(V,W) = \max \left\{\sup_{v\in V, \ |v|=1} \inf_{w\in W} |v-w|   , \ \sup_{w\in W, \ |w|=1} \inf_{v\in V}  |v-w|     \right\},
\end{align}
and $|v|$ is the Euclidean norm.

Now we pass to considering a family $\{E_x\}_{x\in X}$ where $E_x\subset T_xX_{\omega}$. In order to make sense of comparing subspaces in different tangent spaces, we use parallel transport. Let $\rho_0$ be such that $\exp_z$ is a $C^{\infty}$ diffeomorphism on $B(z,\rho_0)$ for all $z\in X$ and let $x,y\in X$ be such that $d_X(x,y)<\rho_0$. Denote $P(x,y)$ as the isometry from $T_xX$ to $T_yX$ that is parallel transport along the geodesic between $x$ and $y$, then we can define
\begin{align}\label{dist2}
    d(E_x,E_y) \defeq \begin{cases}
        1 & \text{if } d_X(x,y)\geq \rho_0/4,\\
        d_{Gr}(E_x,P(y,x)E_y) & \text{if }d_X(x,y)<\rho_0/4.
    \end{cases}
\end{align}

\begin{definition}
    Given $\Delta \subset X$, a family of subspaces $\{E_x\}_{x\in \Delta}$ such that $E_x\subset T_xX$ is $\alpha$-H\"older continuous on $\Delta$ if there is $L>0$ such that for any $x,y\in \Delta$ we have
    \begin{align*}
        d(E_x,E_y) \leq L (d_X(x,y)^{\alpha}).
    \end{align*}
\end{definition}

Next we define a form of ``Pesin set''; a set on which we have uniform hyperbolicity rather than non-uniform hyperbolicity. It is on a set of this form where we will be able to achieve H\"older continuity for our Oseledets' directions. Note that the sets in the definition may be empty depending on the choice of parameters. 
\begin{definition}\label{pesin}
    Let $\{f_i\}_{i=0}^{\infty} \in \supp(\mu)^{\bbN}$ and let $f^n = f_{n-1}\circ \dots \circ f_0$. Let $a<b$, $C\geq 1$ and define $\Delta_{a,b,C}\subset X$ to be the subset of points where there are splittings
    \begin{align*}
        T_xX = E_x \oplus E_x^{\perp},
    \end{align*}
    such that for $n\in \bbN$ we have that
    \begin{align*}
        \Vert D_xf^nv\Vert&\leq Ce^{an}\Vert v\Vert, \ v\in E_x,\\
        \Vert D_xf^nv\Vert&\geq C^{-1}e^{bn}\Vert v\Vert, \ v\in E_x^{\perp}.
    \end{align*}
\end{definition}
Since our sequence $\{f_i\}_{i=0}^{\infty}$ is given by either the non-negative indexed half of the vector $\omega$, denoted $\omega^+$, or the negative indexed half, denoted $\omega^-$, in application we will denote $\Delta_{a,b,C}^{\omega^+}$ for where the $\{f_i\}_{i=0}^{\infty}$ set comes from $\omega^+$, and similarly as $\Delta_{a,b,C}^{\omega^-}$ for where this sequence of functions is from $\omega^-$.

\begin{lemma}[\cite{LQ}, Chapter 3, Section 4, Lemma 4.2] \label{4.2}
    There exists $\Omega_0 \subset \Omega$ such that $\sigma(\Omega_0)\subset \Omega$ and $\mu^{\bbZ}(\Omega_0)=1$ and there exists measurable function $C: \Omega_0 \to ]0,\infty[$ such that for $\omega \in \Omega_0$ and $n\in \bbN$ we have 
    \begin{align}
        \prod_{i=0}^{n-1} \Vert f_i(\omega)\Vert_{C^{1,1}*} \leq C(\omega)e^{2c_0n},
    \end{align}
    where $f_i(\omega)$ is the function in the $i$th position in the vector $\omega$ and 
    \begin{align}\label{c_0}
    c_0 = \int \log \Vert f\Vert_{C^{1,1*}}d\mu(f) <\infty.\end{align}
    The norm $C^{1,1*}$ is equivalent to the usual $C^{1,1}$ norm, $C^{1,1*}$ is given by
    \begin{equation}
        \Vert f\Vert_{C^{1,1*}} = \Vert Df\Vert_{\infty} + \lip(Df).
    \end{equation}

\end{lemma}

Equation \eqref{c_0} is finite by the moment condition, see Remark \ref{higher}.

\begin{remark}\label{4.2-}
    In the same vain, a similar statement can be written for our backwards random walk where we get a $\sigma^{-1}$-invariant set $\Omega_0^-\subset \Omega$, $\mu^{\bbZ}(\Omega_0^-)=1$ and measurable function $C^-:\Omega_0^- \to ]0,\infty[$ such that for $\omega \in \Omega_0^-$ and $n\in \bbN$ we have 
    \begin{align}
        \prod_{i=1}^{n-1} \Vert f_i^-(\omega)\Vert_{C^{1,1}*} \leq C^-(\omega)e^{2c_0^-n},
    \end{align}
    where $f_i^-(\omega)$ is the inverse to the function in the $-i$th position in the vector $\omega$ and 
    \begin{align}\label{c_0-}
    c_0^- = \int \log \Vert f^{-1}\Vert_{C^{1,1*}}d\mu(f) <\infty.\end{align}
   \end{remark}

\begin{lemma}\label{cCunif}
    There exists an almost conull set $\Omega_{\epsilon} \subset \Omega$ on which $C(\omega)$ and $c_0$ from Lemma \ref{4.2} and $C^-(\omega)$ and $c_0^-$ from Remark \ref{4.2-} can be chosen uniformly.
\end{lemma}
\begin{proof}
Choosing $c_0$ and $c_0^-$ uniformly is trivial because of our moment condition (see Equation \eqref{IC}) we can take both constants to be the value of the integral in Equation \eqref{IC}.

Now, since $C(\omega)$ is a measurable function, the set $E_n = \{\omega \in \Omega_0: C(\omega) \leq n\}$ is a measurable set such that $E_n \to \Omega'$ as $n\to \infty$. Hence $\mu^{\bbZ}(E_n) \to \mu^{\bbZ}(\Omega') = 1$. Let $\epsilon_C>0$, we can find $n$ such that $\mu^{\bbZ}(E_n)>1-\epsilon_{C}$, call this set $\Omega_{\epsilon_C}\subset \Omega$ and take $C = \max\{1,n\}$. On this set $\Omega_{\epsilon_C}$ we have that 
\begin{align}
        \prod_{i=0}^{n-1} \Vert f_i(\omega)\Vert_{C^{1,1}} \leq Ce^{2c_0n}.
    \end{align}
    One can do similarly for the set $\Omega_0^-$ defined in Remark \ref{4.2-}, one gets a set $\Omega_{\epsilon_{C^-}}$ where
\begin{align}\label{needed1}
    \prod_{i=1}^{n-1} \Vert f_i^-(\omega)\Vert_{C^{1,1}*} \leq C^-e^{2c_0^-n}.
\end{align}

Let $\Omega_{\epsilon}\defeq \Omega_{\epsilon_C}\cap \Omega_{\epsilon_{C^-}}$. This completes the proof.
\end{proof}

\begin{proposition}[\cite{LQ}, Chapter 3, Section 4, Corollary 4.1]\label{holder}
    Let $\{f_i\}_{i=0}^{\infty}$ be a sequence of diffeomorphisms $f_i:X\to X$ such that 
    \begin{align}
        \prod_{i=0}^{n-1} \Vert f_i\Vert_{C^{1,1}*} \leq Ce^{cn},
    \end{align}
    for some $C\geq 1$, $c>0$. Fix $\hat{C}\geq 1$, and $a<b$ and consider the set $\Delta_{a,b,\hat{C}}\subset X$ associated to this sequence as in Definition \ref{pesin}. The family $\{E_x\}_{x\in \Delta_{a,b,\hat{C}}}$ is $\alpha$-H\"older continuous in $x$ on $\Delta_{a,b,\hat{C}}$ with constant $2\hat{C}^2e^{b-a}$ and $\alpha = \frac{a-b}{a-d}$ where $$d = \ln(2C^2) + 2c + |\ln(\rho_0/4)| + |a|.$$
\end{proposition}

Fix $\epsilon>0$ and $N\in \bbN$ and define $\Lambda_{\epsilon, N}^{s,\omega}$ to be the set of all $x\in X$ such that for any $n\geq N$, $v\in E^s(x,\omega)$ we have 
\begin{align} \label{f}
    e^{n(\lambda^- - \epsilon)}\Vert v\Vert \leq \Vert D_xf_{\omega}^nv\Vert \leq e^{n(\lambda^-+\epsilon)}\Vert v\Vert.
\end{align}
and for $v \in E^u(x,\omega)$,
\begin{align}\label{f'}
     e^{n(\lambda^+ - \epsilon)}\Vert v\Vert \leq \Vert D_xf_{\omega}^nv\Vert \leq e^{n(\lambda^+ +\epsilon)}\Vert v\Vert.
\end{align}
Similarly, define $\Lambda_{\epsilon,N}^{u,\omega}$ to be the set of $x\in X$ such that for all $n \geq N$ we have that for $v\in E^s(x,\omega)$, 
\begin{align}\label{b''}
    e^{-n(\lambda^-+\epsilon)}\Vert v\Vert \leq \Vert D_xf_{\omega}^{-n}v\Vert \leq e^{-n(\lambda^- - \epsilon)}\Vert v\Vert,
\end{align}
and for $v\in E^u(x,\omega)$,
\begin{align}\label{b'}
    e^{-n(\lambda^++\epsilon)}\Vert v\Vert \leq \Vert D_xf_{-\omega}^{-n}v\Vert \leq e^{-n(\lambda^+ - \epsilon)}\Vert v\Vert.
\end{align}
Note that if $N_1<N_2$, then $\Lambda_{\epsilon,N_1}^{s,\omega} \subset \Lambda_{\epsilon,N_2}^{s,\omega}$ and $\Lambda_{\epsilon,N_1}^{u,\omega} \subset \Lambda_{\epsilon,N_2}^{u,\omega}$.

\begin{lemma}
For any $(x,\omega)\in \Lambda$ and any $\epsilon>0$ there exists $N_s, N_u\in \bbN$ such that $x$ belongs to  $\Lambda_{\epsilon, N_s}^{s,\omega}$ and $\Lambda_{\epsilon,N_u}^{u,\omega}$.
\end{lemma}
\begin{proof}
    This follows immediately from Oseledets' theorem.
\end{proof}

Given a fixed $\epsilon>0$, we get a Lusin set $K_{\epsilon}$ such that if $(x,\omega)\in K_{\epsilon}$ then $N_u$ and $N_s$ are bounded above by $\tilde{N}$.
 
Note that given two points $(x,\omega),(y,\omega)\in \Lambda$ and $\epsilon>0$, each has a corresponding $N_s$, we call them $N_s^x$ and $N_s^y$ resp. then $x,y\in \Lambda_{\epsilon, N}^{s,\omega}$ for $N = \max\{N_s^x,N_s^y\}$. Similarly, if $N_u^x$ and $N_u^y$ were their corresponding $N_u$ values, then also $x,y\in \Lambda_{\epsilon, N}^{u,\omega}$.

\begin{lemma}
    There exists $L_u,L_s\in \bbR$ (independent of $\omega$) such that $$\Lambda^{s,\omega}_{\epsilon,N} \subset \Delta_{(\lambda^-+\epsilon), (\lambda^+-\epsilon),L_s}^{\omega^+}, \ \ \Lambda^{u,\omega}_{\epsilon,N} \subset \Delta_{-(\lambda^--\epsilon),-(\lambda^++\epsilon),L_u}^{\omega^-}.$$
\end{lemma}
\begin{proof}
 Let $L_s = (\max_{i=0,\dots N}\{1,\Vert Df_i\Vert\})^{N},$ (to make independent of $\omega$, simply let $L_s' = \max_{\supp(\mu)}\{1, \Vert Df_i\Vert \}$ and take $L_s = (L_s')^N$, we can do this since we assume $\supp(\mu)<\infty$ and $X$ is compact) then, for $y\in \Lambda_{\epsilon,N}^{s,\omega}$, all $n\in \bbN$ and $v\in E^s(y,\omega)$ we have 
\begin{align} \label{fn}
    \frac{1}{L_s}e^{n(\lambda^- - \epsilon)}\Vert v\Vert \leq \Vert D_xf_{\omega}^nv\Vert \leq L_se^{n(\lambda^-+\epsilon)}\Vert v\Vert,
\end{align}
and for $v \in E^u(y,\omega)$,
\begin{align}\label{fn'}
     \frac{1}{L_s}e^{n(\lambda^+ - \epsilon)}\Vert v\Vert \leq \Vert D_xf_{\omega}^nv\Vert \leq L_s e^{n(\lambda^+ +\epsilon)}\Vert v\Vert.
\end{align}
Hence $\Lambda_{\epsilon,N}^{s,\omega} \subset \Delta_{(\lambda^-+\epsilon),(\lambda^+-\epsilon),L_s}^{\omega^+}$.

Similarly, let $L_u = (\max_{i=0,-1\dots -N}\{1,\Vert Df_i\Vert\})^N,$ (we can make this independent of $\omega$ in the same way one does for $L_s$) then for $y\in \Lambda_{\epsilon,N}^{u,\omega}$ we have that for $v \in E^s(x,\omega),$
\begin{align}
    \frac{1}{L_u}e^{-n(\lambda^-+\epsilon)} \Vert v\Vert\leq \Vert D_xf_{\omega}^{-n}v\Vert \leq L_u e^{-n(\lambda^--\epsilon)}\Vert v\Vert,
\end{align}
and for $v\in E^u(x,\omega)$ we have 
\begin{align}
    \frac{1}{L_u}e^{-n(\lambda^++\epsilon)} \Vert v\Vert\leq \Vert D_xf_{\omega}^{-n}v\Vert \leq L_u e^{-n(\lambda^+-\epsilon)}\Vert v\Vert.
\end{align}
So we see that $\Lambda_{\epsilon,N}^{u,\omega} \subset \Delta_{-(\lambda^--\epsilon),-(\lambda^++\epsilon),L_u}$ for this $\omega$ (or really $\omega^-$).
\end{proof}

If further we have $(x,\omega) \in \tilde{\Lambda} \defeq \Lambda \cap (X\times \Omega_{\epsilon})$, then we can use Proposition \ref{holder} to say that if we pick $\epsilon>0$ and get the corresponding $N_u$ and $N_s$, then
on the subsets $\Lambda_{\epsilon,N_s}^{s,\omega} \subset \Delta_{(\lambda^-+\epsilon),(\lambda^+-\epsilon),L_s}$ and $\Lambda_{\epsilon,N_u}^{u,\omega} \subset \Delta_{-(\lambda^--\epsilon),-(\lambda^++\epsilon),L_u}$ we get the associated H\"older continuity of $E^s$ and $E^u$ respectively.

\subsection{Stable holonomies}
\label{stableholon}
Here we discuss stable holonomies and standard measurable connections. We follow the appendices of \cite{BEF}.

\begin{definition}
    We say a smooth cocycle $\calG : E\to E$ with vector bundle $E \to Y$ over $F$ has smooth stable holonomies if 
    for a.e. $(x,\omega)\in Y$ and for all $y\in W^s(x,\omega)$ we have  maps $$\calP^-_{\omega}(x,y) : E(x,\omega) \to E(y,\omega),$$ that vary smoothly in $y$ and such that if $y'\in W^s(x,\omega)$, then $$\calP^-_{\omega}(x,y') = \calP^-_{\omega}(y,y') \circ \calP^-_{\omega}(x,y),$$ (where defined) and such that they are $F$-equivariant: $$\calP^-_{\sigma(\omega)}(F(x,\omega),F(y,\omega))\circ \calG_x = \calG_y\circ \calP^-_{\omega}(x,y).$$ 
\end{definition}

The appendices (B, D and E) of \cite{BEF} explain that (using \cites{Ru,ASV}) we get the existence of holonomies for our cocycles $E^s$ and $E^u\oplus E^s/E^s$, which we denote $\calP^-_u$ and $\calP^-_s$ respectively. Combining these, we have that the stable associated graded $\text{gr}^- \defeq (E^s\oplus E^u/E^s) \oplus E^s$ also has stable holonomies $\calP^-$.

Now note that the $(E^s\oplus E^u/E^s)$ are (measurably) isomorphic to $E^u$ and so we get isomorphisms $E^u \oplus E^s \to (E^s\oplus E^u/E^s) \oplus E^s$. Combining these isomorphisms with the holonomies $\calP^-$, we get the standard measurable connection which we call $$P^- : E^u \oplus E^s \to E^u \oplus E^s.$$ Here $P^- = P^-_u\oplus P^-_s,$ where $$ P^-_u : E^u\to E^u, \ \ P^-_s : E^s \to E^s,$$ come from $\calP^-_u$ and $\calP^-_s$ respectively combined with their respective canonical isomorphisms. By the Ledrappier invariance principle (see \cites{Led,AV,EM,BEF}) these preserve measurable $F$-invariant subbundles like $U^+$. So we have maps  $$P^-_{\omega}(z,z') : TU^+[z,\omega] \to TU^+[z',\omega],$$ (where defined) with properties $$P^-_{\omega}(z,z'')= P^-_{\omega}(z'',z')\circ P^-_{\omega}(z,z'),$$ and $P^-_{\omega}(z,z)=Id$. Appendix E of \cite{BEF} explains how for any $\epsilon_{\beta}>0$ we can find a set $K_{\beta}$ (that we call a Pesin-Lusin set) of measure $1-\epsilon_{\beta}$ on which $P^{\pm}$ are H\"older continuous with some exponent $\beta$ and constant $C_{\beta}$ dependent on $\epsilon_{\beta}$. To be more precise, on $K_{\beta}$ we have
$$\Vert P_{\omega}^-(z,z') - Id\Vert \leq C_{\beta}d_X(z,z')^{\beta}.$$

One can have the above construction for the unstables too. We denote these counterparts as $\calP^+$ and $P^+$.

\begin{lemma}
    The tangent space of the unstable support of a point $(x,\omega)$, $TU^+[x,\omega]$, does not depend on $\omega^+$.
\end{lemma}
\begin{proof}
    Consider $\hat{x}=(x,\omega^-,\omega^+), \hat{x}'=(x,\omega^-,(\omega')^+)\in Y$. We note that $W^u(\hat{x}) = W^u(\hat{x}')$ and consider $P^+(\hat{x},\hat{x}')(TU^+[\hat{x}]) = TU^+[\hat{x}']$. Now we note that the backwards flag $E^u$ and $E^s\oplus E^u/E^u$ are only dependent on the past which $\hat{x}$ and $\hat{x}'$ share. This forces $P^+(\hat{x},\hat{x}')$ to be the identity. 
\end{proof}

\subsection{Normal form coordinates}
\label{sec:NFC}
Normal form coordinates (NFCs) are charts on either the stable or unstable manifold, $W^s(x,\omega)$ or $W^u(x,\omega)$, that (in our case) conjugate the action of the dynamics to a linear action. This construction was originally due to \cite{KS}. Our case is simple because our stable and unstable manifolds are complex 1-dimensional. One can also find this in \cite[Lemma 3.2 and Lemma 3.3]{KK2}, \cite[Proposition 6.5 ]{BRH} and \cite{BEF}.

\begin{proposition}[\cite{BRH}, Proposition 6.5] \label{NFC}
    For $m$-a.e. $(x,\omega) \in Y$ and for any $y\in W^u(x,\omega)$ (similar for the stable case), there exists a $C^{1}$ diffeomorphism
    \begin{align}
        H_{(y,\omega)}^u : W^u(x,\omega) \to T_yW^u(x,\omega) \cong \bbC,
    \end{align}
    that satisfies
    \begin{enumerate}
    \item 
        $D_yf_{\omega} \circ H_{(y,\omega)}^u = H_{F(y,\omega)}^u \circ f_{\omega}|_{W^u(x,\omega)},$ 
    \item
        $H^u_{(y,\omega)}(y)=0,$ $ D_yH_{(y,\omega)}^u = Id,$

    \item for $z\in W^u(x,\omega)$ the change of coordinates
    \begin{align}\label{change}
        H^u_{(y,\omega)} \circ \left( H^u_{(z,\omega)}\right)^{-1}: T_zW^u(x,\omega) \to T_yW^u(x,\omega),
    \end{align}
    is an affine map.
    \end{enumerate}
\end{proposition}

 We denote $\Lambda'$ to be the subset of $Y$ where we can use NFCs as in Proposition \ref{NFC}. 

\subsection{Time change suspension flow}
\label{sec:flow}

For the `time changed' flow we will construct a roof function that will allow us to grow the unstable manifold by exact amounts in normal form coordinates (after some slight modification to these coordinates, see Equation \eqref{newNFCform} below). More precisely, if we are in (modified) normal form coordinates (NFCs) for the unstable manifold we want flowing by time $\ell$ to expand the unstable manifold by $e^{\ell}$. We will also get the analogue for the stable manifolds, i.e. that flowing forwards by $\ell$ in (modified) NFCs under the `time changed' map will shrink distances by a factor of $e^{-\ell}$.

\subsubsection{Defining the time changed flow (forwards)}
As we see in Section \ref{sec:NFC}, the way normal form coordinates works is that the action of $f$ on the unstable manifold in these coordinates is entirely linear. Hence, this action is multiplication by $D_xf_{\omega}|_{E^u(x,\omega)}$ (or $D_xf_{\omega}|_{E^s(x,\omega)}$ if stable).
Let $\theta : Y \to \bbR$ be defined by
\begin{align}
    \theta(x,\omega) &\defeq \log |D_xf_0|_{E^u(x,\omega)}|.
\end{align}
We would use $\theta$ as our roof function for our flow, but it is not guaranteed to be positive. It is only positive on average; by Furstenburg's formula 
\begin{align}\label{furstenburg}
    \int_{X\times \Omega} \theta(x,\omega)d\nu d\mu^{\bbZ} = \lambda^+ >0.
\end{align}
We instead take a function cohomologous to $\theta$, a positive integrable function that which we call $\tau$; recall that to be cohomologous means that there is a positive and a.s. finite function $\phi$ such that
\begin{align} \label{cohom}
    \theta - \phi \circ F + \phi = \tau.
\end{align}

Since Equation \eqref{furstenburg} is also a finite quantity we could use Lemma 2.1 from \cite{BQ} get such a $\tau$. However, we prefer an explicit formula and so instead we give the following Lemma:
\begin{lemma}
    Let \begin{align}
        \tau(x,\omega) = \log|D_0\tilde{f}_{(x,\omega)}|_{\bbR^u}|,
    \end{align}
    where $\tilde{f}$ is as in Section \ref{sec:Lyapcharts}, then $\tau$ is positive and cohomologous to $\theta$.
\end{lemma}
\begin{proof}
    Note that we can write $\theta(x,\omega) = \log\left( \frac{\Vert D_xfv\Vert_{f(x)}}{\Vert v\Vert_x}\right)$ for any $v\in E^u(x,\omega)$. Similarly, we can write $\tau(x,\omega) = \log \left( \frac{|D_0\tilde{f}_{x,\omega}w |_{\tilde{f}_{(x,\omega)}(0)}}{|w|_0} \right)$ for any $w\in \bbR^u$. We want $\phi$ such that $\phi \circ F - \phi = \theta - \tau$, note that 
    \begin{align}
       \theta-\tau= \log \left( \frac{\Vert D_xfv\Vert_{f(x)}}{|D_0\tilde{f}_{(x,\omega)}w |_{\tilde{f}_{(x,\omega)}(0)}} \right) - \log \left( \frac{|w|_0}{\Vert v\Vert_x} \right).
    \end{align}
    Then let $$\phi(x,\omega) = \log \left( \frac{\Vert v\Vert_x}{|w|_0} \right),$$ where $w = (1,0,0,0)\in \bbR^u$ and $v = D_0\varphi_{(x,\omega)}^{-1}(w) \in E^u$. Then 
    \begin{align}
        \varphi\circ F(x,\omega) = \log\left( \frac{\Vert D_0\varphi_{F(x,\omega)}^{-1}(w)\Vert_{f(x)}}{|w|} \right).
    \end{align}
    Plugging in $\varphi^{-1}_{F(x,\omega)} = f_{\omega}^1\circ \varphi^{-1}_{(x,\omega)}\circ \tilde{f}^{-1}_{(x,\omega)}$, we get 
    \begin{align}
        \varphi\circ F(x,\omega) &= \log \left( \frac{\Vert D_{\varphi^{-1}_{(x,\omega)}\circ \tilde{f}^{-1}_{(x,\omega)}(0)}f_{\omega}^1 \circ D_{\tilde{f}_{(x,\omega)}^{-1}(0)}\varphi^{-1}_{(x,\omega)}\circ  D_0\tilde{f}^{-1}_{(x,\omega)} w  \Vert_{f(x)}}{|w|} \right),\\
        &= \log \left( \frac{\Vert D_{x}f_{\omega}^1 \circ D_{0}\varphi^{-1}_{(x,\omega)}\circ  D_0\tilde{f}^{-1}_{(x,\omega)} w  \Vert_{f(x)}}{|w|} \right),\\
        &= \log \left( \frac{\Vert D_{x}f_{\omega}^1 \circ D_{0}\varphi^{-1}_{(x,\omega)}   w  \Vert_{f(x)}}{|D_0\tilde{f}_{(x,\omega)}|_{\bbR^u}|\cdot |w|} \right),\\
        &= \log \left(\frac{\Vert D_xfv\Vert_{f(x)}}{|D_0\tilde{f}_{(x,\omega)}w|}  \right).
    \end{align}
Then it is clear that for our chosen $\phi$ that we have $\phi\circ F - \phi = \theta - \tau$. This completes the proof.
\end{proof}
Additionally, by the uniform hyperbolicity estimates in Lyapunov charts we have that 
\begin{align}\label{tau1}
        \lambda^+-\epsilon \leq \tau \leq \lambda^++\epsilon.
\end{align}

Throughout the rest of this subsection, we are flowing forwards by $\ell>0$. We define for $p\in \bbN, \ p\geq 1$,
\begin{equation}
    \tau_p(x,\omega) \defeq \tau(f_{\omega}^{p-1}x,\sigma^{p-1}\omega)+\dots + \tau(f_{\omega}^2x,\sigma^2\omega)+\tau(f_{\omega}^1x,\sigma(\omega)) + \tau(x,\omega).
    \end{equation}
Then the amount of time we flow has to surpass $\tau_p(x,\omega)$ in order to apply $p$ diffeomorphisms, $f_0,f_1,\dots f_{p-1}$ to $X$ at $x$.

Now we can define the `time changed' flow on $Z$ which we will denote $F_{tc}^{\ell}$. Note that \cite{BQ} simply defines the space over which they flow using the roof function, however as we have two different flows, we take the approach of \cite{EM} and simply change the speed of the flow rather than the $[0,1[$ component of $Z$.

For $\ell\in \bbR$, $\ell\geq 0$, we let $c=(x,\omega,k)\in Z$, and define $F^{\ell}_{tc}: Z \to Z$ by
$$F_{tc}^{\ell}(c) = ( f_{\omega}^{p_{\ell}(c)}x, \sigma^{p_{\ell}(c)}(\omega),k+\ell/\tau_{p_{\ell}(c)}(x,\omega) \mod 1),$$
where for $\ell>0$, 
\begin{align}\label{p+}
p_{\ell}(c) = \max\{p\in \bbN : k+\ell/\tau_p(x,\omega) \geq 1\}.\
\end{align}
Hence we move at the speed $1/\tau_{p_{\ell}(c)}$.

\subsubsection{The time changed flow in normal form coordinates}
\label{sec:new}
Now take the NFCs we defined in Section \ref{sec:NFC}, $H_{(y,\omega)}^u : W^u(x,\omega) \to T_yW^u(x,\omega)$, and define
\begin{align}\label{newNFCform}
    H_{(y,\omega),new}^u \defeq e^{-\phi(y,\omega)}H_{(y,\omega)}^u.
\end{align}

\begin{lemma}\label{newNFC}
    The map $ H_{(y,\omega),new}^u$ satisfies
    \begin{align}
    \Vert H_{(x,\omega),new}(F_{\omega}(z))\Vert = e^{\tau(x,\omega)} \Vert H_{(x,\omega),new}^u(z)\Vert.
    \end{align}
\end{lemma}

\begin{proof}
Note that $H_{(x,\omega)}(x) = 0$ so the initial distance between $x$ and $z$ in original NFCs is
\begin{align}
    \Vert H_{(x,\omega)}(x) - H_{(x,\omega)}(z)\Vert = \Vert H_{(x,\omega)}(z)\Vert.
\end{align}
Then applying $F$ we see that
\begin{align}
    \Vert H_{F(x,\omega)}(x) - H_{(x,\omega)}(F_{\omega}(z))\Vert = \Vert H_{(x,\omega)}(F_{\omega}(z))\Vert.
\end{align}
By condition 1 of the construction of NFCs in Proposition \ref{NFC} and the definition of $\theta$ we have that 
\begin{align}
    \Vert H_{(x,\omega)}(F_{\omega}(z))\Vert = |D_xf{\omega}|_{E^u}| \cdot \Vert H_{(x,\omega)}^u(z)\Vert = e^{\theta(x,\omega)} \cdot \Vert H_{(x,\omega)}^u(z)\Vert. 
\end{align}
Now using the definition of $ H_{(y,\omega),new}^u$ we have
\begin{equation}
    \begin{aligned}\label{new}
    \Vert H_{(x,\omega),new}(F_{\omega}(z))\Vert 
     &= e^{-\phi(F(x,\omega))} \Vert H_{(x,\omega)}(F_{\omega}(z))\Vert \\
     &= e^{-\phi(F(x,\omega))}|D_xf{\omega}|_{E^u}| \cdot \Vert H_{(x,\omega)}^u(z)\Vert \\
     &= e^{\theta(x,\omega) -\phi(F(x,\omega)) } \cdot \Vert H_{(x,\omega)}^u(z)\Vert. 
\end{aligned}
\end{equation}
Using Equation \eqref{cohom} $e^{\theta(x,\omega) - \phi(F(x,\omega)) } = e^{\tau(x,\omega) - \phi(x,\omega)}$ and so Equation \eqref{new} becomes
\begin{align}
    \Vert H_{(x,\omega),new}(F_{\omega}(z))\Vert = e^{\tau(x,\omega) - \phi(x,\omega)}\Vert H_{(x,\omega)}^u(z)\Vert = e^{\tau(x,\omega)} \Vert H_{(x,\omega),new}^u(z)\Vert.
\end{align}
This completes the proof.
\end{proof}

By construction if we put the unstable manifolds in these modified NFCs (as in Equation \eqref{newNFCform}) and flow forwards by $\ell$ under the time changed flow it expands distances by precisely $e^{\ell}$. If we put the stable manifold in these modified NFCs, flowing forwards under the time changed flow by $\ell$ shrinks distances by $e^{-\ell}$.

\subsubsection{Invertibility of the time changed flow}
\label{sec:invert}
Lastly in this section, we write down explicitly the inverse of the `time changed' flow.

We can take $\tau$ as above and define $\tau^-(x,\omega) = -\tau(f_{-1}^{-1}(x),\sigma^{-1}(\omega))$ (where $\omega =(\dots f_{-2},f_{-1},f_0,f_1,f_2,\dots) $), which is a negative function.

   Define for $p\geq 1$ 
\begin{equation}\label{tau-}
    \tau_p^-(x,\omega) = \tau^-(f_{\omega}^{-(p-1)}x,\sigma^{-(p-1)}\omega)+\dots \tau^-((f_{\omega}^{-1})x,\sigma^{-1}(\omega)) + \tau^-(x,\omega).
    \end{equation}
In this scenario we want that if $\ell<0$ surpasses $\tau_p^-(x,\omega)$ (in magnitude) we hit $x$ with $p$ diffeomorphisms from the past of $\omega$, i.e. hit $x$ with $(f_{-1})^{-1},\dots (f_{-p})^{-1}$.\\

Then we can define the `time changed flow' for $\ell<0$ as 
$$F_{tc}^{\ell}(c) = ( f_{\omega}^{-p_{\ell}^-(c)}x, \sigma^{-p_{\ell}^-(c)}(\omega),k-\ell/\tau^-_{p_{\ell}^-(c)}(\omega,x)),$$
where 
\begin{align}\label{p-}
p_{\ell}^-(c) = \max\{p\in \bbN : k-\ell/\tau_p^-(\omega,x) \leq -1\}.
\end{align}

\subsubsection{Invariant measures}

For our flow $F_{tc}^t$, the natural invariant measure $\hat{m}^{\tau}$ is given by
\begin{align}\label{timemeas}
    d\hat{m}^{\tau}(x,\omega,k) = C_{\tau}\tau(x,\omega)\cdot dm(x,\omega)dt(k),
\end{align}
where $C_{\tau}$ is a constant chosen such that $\hat{m}^{\tau}(Z) = 1$. This measure is clearly absolutely continuous with respect to $\hat{m}$.
One makes a similar observation to that of \cite[Lemma 15.1]{EM}:
\begin{lemma} [Lemma 15.1, \cite{EM}]\label{RN}
    For every $\epsilon_{\tau}>0$ there exists $\calK \subset Z$ a compact subset and $Q<\infty $ such that 
    \begin{align}
        \hat{m}^{\tau}(\calK) > 1-\epsilon_{\tau},
    \end{align}
    and for every $(x,\omega,k)\in \calK$
    \begin{align}
        \frac{d\hat{m}^{\tau}}{d\hat{m}}(x,\omega,k), \ \frac{d\hat{m}}{d\hat{m}^{\tau}}(x,\omega,k) \leq Q.
    \end{align}
\end{lemma}

\subsection{Some good sets}

\label{sec:lusin1}

First of all, recall from Sections \ref{sec:prelim} and \ref{sec:Lyapcharts} that $\Lambda$ was the conull set of points in $Y$ where Oseledets' theorem holds and where the Lyapunov charts are defined. There is also the set $\Lambda'$ where NFCs (and modified NFCs as in Equation \eqref{newNFCform}) can be used which is explained in Section \ref{sec:NFC}. Additionally, there is the set $\Lambda_{loc}$ which is almost conull; we fix this set with $\delta < \epsilon_0$, see Section \ref{sec:stablemfld}. Lastly, recall $\tilde{\Lambda}$ from Section \ref{sec:Hcont}.

Now recall that we denote Lyapunov charts at a point $(x,\omega)\in Y$ as $\varphi_{(x,\omega)}$, see Section \ref{sec:Lyapcharts}.
\begin{lemma}\label{comp1}
Let $\epsilon_r>0$ be given, there exists a set $K_r\subset Y$ of $m$-measure at least $(1-\epsilon_r)$ and a constant $M_r>0$ such that if $(x,\omega) \in K_r$, and two points $y,z\in X$ belong to the Lyapunov chart domain of $(x,\omega)$, $U(x,\omega)$, if $d = d_X(y,z)$ and $d' = |\varphi_{(x,\omega)}(y)-\varphi_{(x,\omega)}(z)|$, we have that $d = cd'$ for some $c \in [\frac{1}{M_r},M_r]$.
   \end{lemma}
\begin{proof}

Each $\varphi_{(x,\omega)}$ is a diffeomorphism of compact sets and hence is Bi-Lipschitz for some constant $L$ (dependent on $(x,\omega)$). This means that for any $y,z\in U(x,\omega),$
\begin{align}
    \frac{1}{L}d_X(y,z) \leq |\varphi_{(x,\omega)}(y)-\varphi_{(x,\omega)}(z))| \leq Ld_X(y,z).
\end{align}
 We see that if $d$ is a distance computed in a chart $\varphi_{(x,\omega)}$ and $d'$ is the distance in ambient coordinates, we have that
\begin{align}\label{d}
    \frac{1}{L}d\leq d'\leq Ld.
\end{align}
The inverse of $\varphi_{(x,\omega)}$ has the same Lipschitz constant as $\varphi_{(x,\omega)}$, hence swapping $d'$ and $d$ in Equation \eqref{d} still holds.

Of course each $\varphi_{(x,\omega)}$ will have a different Lipschitz constant. In Section \ref{sec:Lyapcharts}, we have bounds on the Lipschitz constants of every $\varphi_{(x,\omega)}$ for $(x,\omega)\in \Lambda$: there is a uniform constant $\psi$ such that
\begin{align}
    \psi^{-1} \leq \lip(\varphi_{(x,\omega)})  \leq r(x,\omega).
\end{align}
We only need the upper bound; since $r$ is measurable, for the cost of an $\epsilon_r>0$, we can ask that $r(x,\omega)\leq M_r$ for some constant $M_r>0$ on a set $K_r$ of $m$-measure $1-\epsilon_r$. We can use $M_r$ as $L$ in the above argument for every $\varphi_{(x,\omega)}$ such that $(x,\omega)\in K_r$ and then we get the result.
\end{proof}

\begin{lemma} \label{comp3}
    Let $\epsilon_{NFC}>0$ be given, there exists a set $K_{NFC} \subset \Lambda'\subset Y$ of $m$-measure at least $(1-\epsilon_{NFC})$ and a constant $M_{NFC}$ such that if $(x,\omega) \in K_{NFC}$, then the maps $$H_{(x,\omega),new}^{s/u}:W^u(x,\omega) \to T_xW^u(x,\omega),$$ are defined and if $y\in W^u(x,\omega)$, $z\in W^s(x,\omega)$, $d = d_X(x,y) $ and $$d' = \Vert H_{(x,\omega),new}^u(y) - H^u_{(x,\omega),new}(x)\Vert , \ \ d''= \Vert H_{(x,\omega),new}^s(z) - H^s_{(x,\omega),new}(x)\Vert,$$ then $d = c'd'$, $d=c''d''$ for $c',c''\in [\frac{1}{M_{NFC}},M_{NFC}]$. 
\end{lemma}
\begin{proof}
   Since $\phi$ is measurable, $e^{\phi}$ is measurable. Also, the Lipschitz constants for these maps $H_{(x,\omega)}^{s/u}$ are bounded above by a measurable function, hence one runs the same argument as in Lemma \ref{comp1}. 
\end{proof}

\begin{lemma}
    \label{comp4} 
    Let $\epsilon_{exp}>0$ be given, there exists a set $K_{exp}\subset X$ of $\nu$-measure at least $(1-\epsilon_{exp})$ and a constant $M_{exp}>1$ such that if $x\in K_{exp}$ and $y\in B(x,\rho_0/4)$ (i.e. the domain of the $\exp_x$ map) then if $v = \exp_x^{-1}(y)$, we have that $$\Vert v\Vert = Md_x(x,y),$$ where $M \in [\frac{1}{M_{exp}},M_{exp}]$.
\end{lemma}
\begin{proof}
    Proof is similar to those of Lemmas \ref{comp1} and \ref{comp3}.
\end{proof}

\begin{lemma}\label{Kss}
    Let $\epsilon_s>0$ be given, there exists a set $K_s\subset Y$ of $m$-measure at least $(1-\epsilon_s)$ and a constant $k_U>0$ such that if $(x,\omega)\in K_s$, then $\varphi_{(x,\omega)}$ is defined and the domain of this function, $U(x,\omega)$, contains a ball of radius $k_U$ in the Riemannian distance metric $d_X$.
\end{lemma}
\begin{proof}
We take steps similar to those of the proof of Lemma \ref{comp1}.

Again, the function $r$ defined in Section \ref{sec:Lyapcharts} is measurable. This function defines the co-domain of our diffeomorphism  $\varphi_{(x,\omega)} : U(x,\omega) \to B(1/r(x,\omega))$. Let $s(x,\omega) = \frac{1}{r(x,\omega)}$; this is also a measurable function. We can find a subset of $Y$ on which $s(x,\omega) \in ]0,1]$ is bounded from below by some constant $S$ on a set $K_{s}$ of measure $1-\epsilon_{s}$. One performs the following procedure: let $$E_n = \{(x,\omega) : s(x,\omega)>1/n\}\subset \Lambda,$$ this is a measurable set such that $E_n \to \Lambda$ as $n\to \infty$ where $\Lambda$ has full measure. Hence there is $N>0$ such that $m(E_N)>1-\epsilon_{s}$, take $S = \frac{1}{N}$. Additionally, if $s=\frac{1}{r}>S$, then $r<\frac{1}{S} = R$ on this set.

Now again, $\varphi_{(x,\omega)}^{-1}$ is a diffeomorphism of two compact sets and hence is Bi-Lipschitz for some Lipschitz constant $L$. To be more precise, for any $y,z\in B(s(x,\omega))$, we have that
\begin{align}
    \frac{1}{L}|y-z| \leq d_X(\varphi_{(x,\omega)}^{-1}(y),\varphi_{(x,\omega)}^{-1}(z)) \leq L|y-z|.
\end{align}
Then $\varphi_{(x,\omega)}^{-1}(B(S)) \subset U(x,\omega)$ contains the ball $B_X(x,S/L)$. 

In Section \ref{sec:Lyapcharts}, we have a uniform constant $\psi^{-1}$ such that for every $(x,\omega)\in \Lambda$
\begin{align}
    \psi^{-1} \leq \lip(\varphi_{(x,\omega)}) \leq r(x,\omega).
\end{align}
We only need the upper bound; on $K_s$, $r(x,\omega)<R = \frac{1}{S}$. So we replace $L$ with $R$ in the above argument for every $\varphi_{(x,\omega)}$ such that $(x,\omega)\in K_s$. Then on $K_s$ we have that $\varphi_{(x,\omega)}^{-1}(B(S)) \subset U(x,\omega)$ contains the ball $B_X(x,S/R)$. Let $k_U \defeq S/R = S^2<1$. This completes the proof. 
\end{proof}

\begin{lemma}\label{knbhd}
   Let $\epsilon_{nbhd}>0$ be given. There exists $r$ and a set $K_{nbhd}\subset Y$ of measure at least $1-\epsilon_{nbhd}$ such that for $(x,\omega)\in K_{nbhd}$, $\eta^{s/u}_{\omega}(x)$ contains a neighbourhood of size $r$ in $W^{s/u}(x,\omega)$.   
\end{lemma}
\begin{proof}
    Now note that the partition elements of $X_{\omega}$, $\eta^{s/u}_{\omega}(x)$, defined in Section \ref{sec:cond} contain a neighbourhood of $W^{s/u}(x,\omega)$ of size $r(x,\omega)$. This $r(x,\omega)$ is a measurable function. Similar to the proof of Lemma \ref{Kss}, for a cost of $\epsilon_{nbhd}>0$ in measure, we can choose $r(x,\omega)$ uniformly as some $r$ on a set which we call $K_{nbhd}$.
\end{proof}

\begin{lemma}\label{errorr}
    On the set $K_{\xi} \defeq K_r\cap K_s \cap \Lambda \cap \Lambda_{loc}$, given $\xi>0$, if $(x,\omega) \in K_{\xi}$ and $y\in W^s_q(x,\omega)$, the image of $E^s(y,\omega)$ under $\exp$ approximates $W^s(x,\omega)$ up to error $\xi$ on a ball of radius $\mathfrak{r} = C\sqrt{\xi}$ with respect to the Riemannian distance metric $d_X$ (where $C$ is dependent on $(x,\omega)$). Similar if $z\in W^u_q(x,\omega)$. 
\end{lemma}
\begin{proof}

Take a point $(x,\omega)\in K_r\cap \Lambda_{loc}\cap \Lambda \cap K_s$; we can define the Lyapunov chart $\varphi_{(x,\omega)}$ centered at this point. By Theorem \ref{locstablemfld}, we know that $W_q^u(x,\omega),\ W^s_q(x,\omega) \subset X$ are pullbacks of graphs of holomorphic functions $h^u_{(x,\omega)}$ and $h^s_{(x,\omega)}$ respectively in $\bbR^4$ under $\varphi_{(x,\omega)}$.
\begin{claim}\label{error}
    Let $\epsilon>0$ be given. Let $f: B\to B'$ be holomorphic where $B, B'\subset \bbR^2$ are compact subsets. Let $L = \lip(Df)$ and  let $\mathfrak{r}' = \sqrt{\frac{\epsilon}{L}}$. Take any point $x\in B$ such that $B_{\mathfrak{r}'}(x)\subset B$. On $B_{\mathfrak{r}'}(x),$ $f$ is approximated up to error $\epsilon$ by its tangent plane at $x$.
\end{claim}
\begin{proof}
    We are given $\epsilon>0$. 
    We let $L = \lip(Df)$ and take $y\in B_{\mathfrak{r}'}(x)\subset B$, let $z$ be such that 
    \begin{align*}
        f(y)-f(x) = Df(z)(y-x),
    \end{align*}
    then 
    \begin{align*}
        |f(y) - (f(x) +  &Df(x) \cdot (y-x))| \\ &\leq 0 + |Df(z)-Df(x)|\cdot |y-x| \\ &\leq L|z-x|\cdot |y-x|\leq L|y-x|^2 \leq L(\mathfrak{r}')^2 = \epsilon.
    \end{align*}
\end{proof}

Using Claim \ref{error}, take $\epsilon = \xi/M_r$, take $f$ in the collection $$\calC \defeq \{h^s_{(x,\omega)},h^u_{(x,\omega)}\}_{(x,\omega)\in \Lambda}.$$  
For $\mathfrak{r}' = \sqrt{\frac{\epsilon}{L}}$, we get that for $(x,\omega)\in K_r\cap \Lambda_{loc}\cap \Lambda$, $h^s_{(x,\omega)}$ and $h^u_{(x,\omega)}$ are approximated up to error $\epsilon$ by tangent plane of any point on the graph.

Now we return to ambient coordinates and examine how well $E^s$ and $E^u$ approximate the local stable and unstable manifolds for points $(x,\omega)\in K_r\cap \Lambda_{loc}\cap\Lambda\cap K_s$. We see that since we are in $K_r$ the largest the error could grow is by a factor of $M_r$, leaving us with an error of $\xi$. The ball in $X$ around $x$ for which this estimate is valid has radius $\mathfrak{r} = \mathfrak{r}'/M_r$. So the constant $C$ can be taken to be $\frac{1}{M_r\sqrt{L}}$, where $L$ depends on $(x,\omega)$ (because $L$ is dependent on the function we choose from the collection $\calC$). This completes the proof.
\end{proof}

\begin{lemma}\label{H}
Let $\epsilon_{\xi'}>0$ be given. There exists a constant $H>0$ and subset $K_{\xi}'\subset K_{\xi}$ such that $m(K_{\xi}')>(1-\epsilon_{\xi'})m(K_{\xi})$ and if $(x,\omega)\in K_{\xi}'$ then $\lip(Dh^{s/u}_{(x,\omega)})$ is uniformly bounded above by $H$. In particular, on $K_{\xi}'$ we have that the constant $C$ in Lemma \ref{errorr} is uniformly bounded from below by $\frac{1}{M_r\sqrt{H}}$.
\end{lemma}
\begin{proof}
   The functions $$L^s(x,\omega):(x,\omega) \mapsto \lip(Dh^{s}_{(x,\omega)}), \ \ L^u(x,\omega):(x,\omega) \mapsto \lip(Dh^{u}_{(x,\omega)})$$ are measurable. Similar to the argument in Lemma \ref{Kss}, for a price of $\epsilon_{\xi'}$ to the measure of $K_{\xi}$ we get a set $K_{\xi}'$ on which $L^s(x,\omega)$ and $L^u(x,\omega)$ are bounded above by $H$. 
\end{proof}

Now we must give a similar lemma for the orbits $U^+[x,\omega,k]$ from our family of $U^+(x,\omega,k)$ specified in Theorem \ref{4.1.5}. Note that the orbits $U^+[x,\omega,k]$ are smooth real 1-dimensional curves in the case $d_+=1$ (see Section \ref{sec:breakdown}). In this setting, define the function $G^+:Y \to Gr(1,\bbR^4)$ which takes a pair $(x,\omega)$ to the tangent line through $x$ of $U^+[x,\omega]$.

\begin{lemma}\label{errorr1}
    In the case $d_+=1$, given $\xi>0$, for $(x,\omega) \in K_{\xi}$ as in Lemma \ref{errorr}, if $y\in U^+[x,\omega]$, $G^+(y,\omega)$ approximates $U^+[x,\omega]$ up to error $\xi$ on a ball of radius $\mathfrak{r} = C\sqrt{\xi}$ with respect to the Riemannian distance metric $d_X$. This $C$ can be taken to be independent of the point $(x,\omega)$.
\end{lemma}
\begin{proof}
    The proof is similar to that of Lemmas \ref{errorr} and \ref{H}.
\end{proof}

\begin{lemma}\label{KG}
Let $\epsilon_G>0$ be given, in the case $d_+=1$, there exists a set $K_G\subset Y$ of measure at least $(1-\epsilon_G)$ such that the maps $G^s:(x,\omega) \to E^s(x,\omega)$ and $G^+$ are uniformly continuous on $K_G$. 
\end{lemma}
\begin{proof}
Let us note that the function $G^s : Y\to Gr(1,\bbC^2)$, $G^s: (x,\omega) \mapsto E^s(x,\omega)$ is measurable by Oseledets' theorem ($Gr(1,\bbC^2)$ has a natural metric, see Equation \eqref{grassmetric}). We have a similar notion for $G^+ : Y \to Gr(1,\bbR^4)$ which takes $(x,\omega)$ to $G^+(x,\omega)$ which we denote as the tangent line through $x$ of $U^+[x,\omega]$, this is measurable by \cite{BEF}. Then letting $\epsilon_G>0$ be small, by Lusin's theorem let $K_G\subset Y$ be the set of size $1-\epsilon_G$ such that $G^s$ and $G^+$ are uniformly continuous in $x$ on $K_G$.
\end{proof}

\begin{lemma}\label{Kdelta'}
Fix $\delta'>0$ small. In the setting where $d_+=1$ (i.e. where Theorem \ref{thmC} case c) applies), for given $\epsilon_v>0$, there exists a set $K_{\varepsilon}\subset X$ of $\nu$-measure at least $(1-\epsilon_v)$ and a constant $\varepsilon>0$ such that if $x\in K_{\varepsilon}$, then inside any ball of radius $\varepsilon$ inside $Gr(1,\bbC)$, the measures $P^u_x$ and $P^s_x$ have measure at most $\delta'$ (see Section \ref{sec:fiber}).

\end{lemma}
\begin{proof}
    We use Theorem \ref{thmC} and Lemma \ref{anglecontrol}.

    For each $x\in X$, we do not know that the measure $P^u_x$ (see Section \ref{sec:fiber}) is uniform on $Gr(1,\bbC^2)$. We only know that we can pick elements outside some small ball by Lemma \ref{anglecontrol}. For given $\delta'>0$, for each $x\in X$ let $\epsilon_{x,\delta'}$ be the size of the largest ball outside of which we have $P^u_x$-measure $(1-\delta')$ and $P^s_x$-measure $(1-\delta')$ in $Gr(1,\bbC)$. This is a measurable function for fixed $\delta'$. Let $K_{\varepsilon}$ be the set of measure $(1-\epsilon_{v})$ where $\epsilon_{x,\delta}$ is bounded above by some $\varepsilon.$ We do this the same way as we did in the proof of Lemma \ref{Kss}. 
\end{proof}

\begin{lemma}\label{epb}
    Let $\epsilon_b>0$ be given, there exists a set $K_{b}\subset Y$ of $m$-measure at least $(1-\epsilon_b)$ and a constant $\calC$ such that if $(x,\omega)\in K_b$ and $$\psi_W^{s} : \bbC \to \exp^{-1}_x(B(x,\rho_0/4)), \ \psi_W^s(z) = \vec{w}z + \vec{w}^{\perp}g(z),$$ is a parametrization of the holomorphic curve $\exp^{-1}_x(W^s(x,\omega))$ where $\vec{w}$ is the unit vector in direction $E^s(x,\omega)$, then $$\Vert D^2\psi_W^{s}(z)|_{v}\Vert\leq \calC,$$
    for $v\in \exp^{-1}_x(B(x,\rho_0/4))$.
\end{lemma}
\begin{proof}
    $W^{s/u}(x,\omega)$ are injectively immersed holomorphic curves a.e. and $X$ is compact. For each $(x,\omega)$ in a conull subset of $Y$, $\Vert D^2\psi^s(x,\omega)|_{v}\Vert $ is bounded from above by a constant $\calC(x,\omega)$ on the compact set $\exp^{-1}_x(B(x,\rho_0/4))$. This is a measurable function and so similar to the proof of Lemma \ref{Kss}, for the cost of $\epsilon_b$, we get a set $K_b$ of measure at least $1-\epsilon_b$ and a constant $\calC$ for which $\Vert D^2\psi_W^{s}(z)|_{v}\Vert\leq \calC$ on $\exp^{-1}_x(B(x,\rho_0/4))$.
\end{proof}

\begin{lemma}\label{ang}
    Let $\epsilon_{ang}>0$ be given, there exists a set $K_{ang}\subset Y$ of $m$-measure at least $(1-\epsilon_{ang})$ and a constant $\Theta>0$ such that if $(x,\omega)\in K_{ang}$ then $\angle(E^u(x,\omega),E^s(x,\omega) )>\Theta$.
\end{lemma}
\begin{proof}
    Use Oseledets' theorem and an argument similar to that of Lemma \ref{Kss}.
\end{proof}

The following lemma is standard and is a consequence of Birkhoff's theorem, one can see for example \cite{E}:
\begin{lemma}\label{birkhoff}
For a fixed set $K$, for all $\epsilon,\delta>0$ there is $S>0$, and $E\subset Z$ with $\hat{m}(E)<\epsilon$ and such that for $c \in E^{c}$ for all $T>S$,
$$|\frac{1}{T}\int_0^T \mathds{1}_{K}(F^t(c))dt - \int_{Z}\mathds{1}_{K} d\hat{m} |<\delta.$$
\end{lemma}
\begin{proof}
    This follows from Birkhoff's ergodic theorem. The skew product map is ergodic with respect to $\nu\times \mu^{\bbZ}$ (see Remark \ref{skewergodic}) and the `standard' flow preserves the measure $\hat{m} = m\times dt$.
\end{proof}

\section{Proofs of Lemmas \ref{0.1} and \ref{0.5}}
\label{sec:twolem}

Both lemmas in this section work to prove that under certain conditions, the QNI condition of Definition \ref{QNIdef} holds. Recall from Section \ref{sec:breakdown} that $\calL^s[x,\omega,0]$ is the Zariski closure of the support of the stable conditional measure $m_{x,\omega}^s$ defined in Section \ref{sec:cond}. One should also review Section \ref{sec:QNI+} to see the definitions of the combinatorial unstables and stables (Equations \eqref{Wu1} and \eqref{Ws1} resp.), and objects that appear in Definition \ref{QNIdef}, i.e. the measures defined in Equations \eqref{measu} and \eqref{meass}, and the sets defined in \eqref{etau1} and \eqref{etas1}. Throughout this section, let $B^s(x,\omega,R)$ (or just $B^s(x,R)$ if it is implicit) be the ball of radius $R$ around $x$ in the induced metric inside $W^s(x,\omega)$. Also, since $\supp(\mu)<\infty$, let $\calL$ be the upper bound on the Lipschitz constants of the $f\in \supp(\mu)$.

\subsection{Outline of proof of Lemma \ref{0.1}}

We are given that the stable conditional measures are non-trivial $m$-a.e. Also in this setting, $d_+=0$ a.e. Let $K_+$ be the set where both these conditions are true. We demonstrate that random QNI (Definition \ref{QNIdef}) holds.

\subsubsection{Overview}

We are given $\delta>0$ and we will construct the set $K\subset Z$ (dependent on $\delta$) of $\hat{m}$-measure $1-\delta$ and will determine the constants $\alpha_0,$ $\ell_0$ and $C$ as in Definition \ref{QNIdef}. The initial properties of $K$ will include being a subset of $K^*$ which is an intersection of many good sets from Section \ref{sec:lusin1}. Other properties of $K$ will allow $S$ and $S'$ (as in Definition \ref{QNIdef}) to be large enough in measure and have to their own good properties, such as belonging to $K^*$. Taking $\hat{x}_{1/2} = (x_{1/2},\omega_{1/2},k_{1/2})\in K\subset K^*$ we will have that $U^+[\hat{x}_{1/2}] = \{x_{1/2}\}$ and so $S$ is a subset of $$\{x_{1/2}\}\times \{\omega_{1/2}^-\} \times \Omega^+ \times \{k_{1/2}\},$$ i.e. we only have to pick a new future (in a way that keeps us in a good set). We will have to pick this new future so that $W^s(\hat{y}_{1/2})$ is sufficiently distinct from $W^s(\hat{x}_{1/2})$. The set $S'$ is slightly more complicated, and is a subset of $$X\times \Omega^-\times \{\omega_{1/2}^+\} \times \{k_{1/2}\},$$ or more precisely of $\eta^s_{\ell/2}[\hat{x}_{1/2}]$ (for some $\ell>\ell_0$). This means that we will be picking a new $X$-component in $W^s(\hat{x}_{1/2})$ close to $x_{1/2}$ (depending on $\ell$) but not equal to $x_{1/2}$, and a new past vector (in a way that keeps us in a good set). The set $S'$ is dependent on our choice of $\hat{y}_{1/2}\in S$, but regardless of the choice of $\hat{y}_{1/2}$, we will choose $S'\subset K^*$ which gives us that if $\hat{z}_{1/2} = (z_{1/2},(\omega'')^-,\omega_{1/2}^+,k_{1/2})\in S'$, then $U^+[\hat{z}_{1/2}]=\{z_{1/2}\}$. The $X$-components of $S'$ will have to be chosen in some annulus around $x_{1/2}$ to ensure that we can pick $z_{1/2}$ sufficiently bounded away from $W^s(\hat{y}_{1/2})$.

Then we must demonstrate that there are constants $C$ and $\ell_0$ (dependent only on $\delta$) such that for all $\ell>\ell_0$ if we pick $\hat{x}_{1/2}\in K$ (where $K$ is dependent only on $\delta$), $\hat{y}_{1/2}\in S$ (where $S$ depends on $\hat{x}_{1/2}$ and $\ell$) and $\hat{z}_{1/2}\in S'$ (where $S'$ is dependent on $\hat{y}_{1/2}$ and $\ell$) we have

\begin{align*}
        d_X(U^+[\hat{z}_{1/2}],W^{s}(\hat{y}_{1/2})) = d_X(z_{1/2},W^{s}(\hat{y}_{1/2}))\geq Ce^{-\alpha_0\ell}.
    \end{align*}
These constants $C$ and $\ell_0$ are technically chosen with $K$ but we have to compute what they need to be at the end. The constant $\alpha_0$ needs to be chosen even before $\delta$ is given, we give it at the beginning and use it in a computation at the end.

\subsubsection{Details}

As explained above, when constructing $K\subset Z$, we start out with it needing to belong to $K^*$ which is an intersection of many good sets from Section \ref{sec:lusin1}. Of the many good sets that $K^*$ belongs to, some had to be constructed directly in this section. For example $K_{\eta}$ from Lemma \ref{keta} whose construction relied in Lemma \ref{algcurve1}.

Now we need the set $K$ to have the property that any point  $\hat{x}\in K$ is such that there is a large subset $V_{\hat{x}}\subset \Omega^+$ of new futures that we can pick from and still belong to $K^*$. We demonstrate this in Lemma \ref{G!}. This allows us to keep the measure of $S$ sufficiently large once we pick a point $\hat{x}_{1/2}\in K$, $$S\subset \{\hat{x}_{1/2}\}\times \{\omega^-_{1/2}\}\times V_{\hat{x}_{1/2}}\times \{k_{1/2}\}.$$ It is only a subset because we will want to further reduce the choices of futures in $V_{\hat{x}_{1/2}}$ so that for $\hat{y}_{1/2}\in S$, $W^s(\hat{y}_{1/2})$ will be bounded away from $W^s(\hat{x}_{1/2})$. We do this in Lemma \ref{tildeS} and the resulting bound is independent of the choice of $\hat{x}_{1/2}$. We can do this without diminishing the measure of $\{\hat{x}_{1/2}\}\times \{\omega^-_{1/2}\}\times V_{\hat{x}_{1/2}}\times \{k_{1/2}\}$ by much. This gives us our set $S$ for a given $\hat{x}_{1/2}\in K$.

Things are slightly more complicated when dealing with $S'$ because we have to change the $X$-component and the $\Omega^-$-component and still land in the good set (e.g. one such that we belong to $K^*$). This must be incorporated into the definition of $K$. In changing the $X$-component we in particular must be in the set $\eta^s_{\omega_{1/2},\ell/2}(x_{1/2})\subset W^s(\hat{x}_{1/2})$. We then are concerned that restricting ourselves to choosing the $X$-component here will not allow us to be in a good set. This is where Lemma \ref{3.11} comes in, it tells us that for good $\omega\in \Omega$, if one has a large $\nu$-measure subset $X'$ of $X$, there is a large $\nu$-measure subset $X^*\subset X'$ such that for $x\in X^*$, a large $m_{x,\omega}^s$-measure of $\eta^s_{\omega_{1/2},\ell/2}(x_{1/2})$ belongs to $X'$. We define a set 
$X'''_{\omega,k}\subset X$ to have two properties: the first is that if $x\in X_{\omega,k}'''$ then $V_{x,\omega,k} = \{(\omega'')^-:(x,(\omega'')^-,\omega^+,k)\in K^*\}$ is large, and the second is that $(x,\omega,k)$ belongs to another good set $K'$. The set $K'$ is $K^*$ with the added property that we can still pick good futures as in the previous paragraph and belong to $K^*$. We apply Lemma \ref{3.11} to this set $X'''_{\omega,k}$ for good choices of $\omega$ and $k$ (after using Lemmas \ref{41} and \ref{42} to show that the set of good choices of $\omega$ and $k$ such that $X'''_{\omega,k}$ is large is also large) to get sets $X^*_{\omega,k}$. We use this to define the large set $K''$ where each element is of the form $(x,\omega,k)$ where $\omega$ and $k$ are `good' and $x\in X^*_{\omega,k}$. This set $K''$ is big by Lemma \ref{fubini2}.

We are not quite done but now we see that we can make $S$ and $S'$ have properties of `good sets'. We need to ensure that for $\hat{x}_{1/2}\in K$, $x_{1/2}$ is not the only element in $\supp(m_{{x}_{1/2},\omega_{1/2}}^s)\cap \eta^s_{\omega_{1/2},\ell/2}(x_{1/2})$, this is Lemma \ref{481}. This is important because when we pick $\hat{z}_{1/2} = (z_{1/2},(\omega'')^-,\omega_{1/2}^+,k_{1/2})\in S'$ we have $U^+[\hat{z}_{1/2}]=z_{1/2}$ and we want $W^s(\hat{y}_{1/2})$ bounded away from it. If $z_{1/2}=x_{1/2}$, we lose because the $X$-component of $\hat{y}_{1/2}$ is $x_{1/2}$. 
Proving Lemma \ref{481} involves Corollary \ref{independent} which allows us to keep a ball (of some radius dependent on $\ell$) inside $\eta^s_{\omega,\ell/2}(x)$. It is also important that $x$ is not an atom (so that the measure of $S'$ can be sufficiently large when we restrict the $X$-component to some annulus around $x_{1/2}$). This comes from Lemma \ref{algcurve} and allows us to prove Lemma \ref{annbd} using that $K_{\eta}\subset K^*$. Lemma \ref{annbd} tells us that we can pick $z_{1/2}$ in an annulus without significant loss to the measure of $S'$. We in fact know everything about this annulus in terms of $\ell$. Knowing that will construct $S'$ with $X$-components belonging to this annulus allows us some foresight in defining $S$. We use Lemma \ref{tildeS} to reduce the choice of futures in the definition of $S$. In particular, we reduce the choice of futures in a way that allows us to know how $W^s(\hat{y}_{1/2})$ will be bounded away from some point $z\in W^s(\hat{x}_{1/2})$ given we know how far $z$ is from $x_{1/2}$. 

Combining all of this, we make a final computation demonstrating that for sufficiently large $\ell_0$ there is an appropriate constant $C$ (that can be chosen alongside $K$) such that for $\ell>\ell_0$,
\begin{align*}
        d_X(U^+[\hat{z}_{1/2}],W^{s}(\hat{y}_{1/2})) = d_X(z_{1/2},W^{s}(\hat{y}_{1/2}))\geq Ce^{-\alpha_0\ell}.
    \end{align*}

\subsection{Proof of Lemma \ref{0.1}}

\begin{proof}(Lemma \ref{0.1})

    We are given that the stable conditional measures $m^s_{x,\omega}$ are non-trivial $m$-a.e. Further, we are given that $U^+(x,\omega,k)=\{e\}$ a.e. and so $U^+[x,\omega,k]=x$ a.e. As above, let $K_+\subset Z$ be the set on which both these conull conditions are true; this is some full measure set.

     On the set $K_+$, the stable conditional measures $m_{x,\omega}^s$ are non-trivial, so there is some $z\in \supp(m_{x,\omega}^s) \subset W^s(x,\omega)$ such that $z\neq x$. For each $(x,\omega)$ there exists $\mathfrak{r}(x,\omega)$ such that $z\in (B(x,\mathfrak{r}(x,\omega))\cap W^s(x,\omega))\setminus B^s(x,\frac{1}{\mathfrak{r}(x,\omega)}$), i.e. it lives in some annulus around $x$. For a small price to the measure of $K_+$, we can ask that $\mathfrak{r}(x,\omega)$ be bounded above by some $\mathfrak{r}>0$ for $(x,\omega)\in K_+$. 
    
    Write $\calL$ as the Lipschitz constant for the $f\in \supp(\mu)$ and let $\calK$ be such that $\calL=e^{\calK}$.

    Let $\delta>0$ be given. We will determine $\ell_0$ and $C$ as in Definition \ref{QNIdef} (the QNI definition). Take $i>j>2$ such that $$(1-2\delta^j)(1-\frac{\delta^i/2}{\delta^j-\delta^i/2})>1-\delta, \ \ (1-\sqrt{2}\delta^{j/2})(1-\delta)>1-\delta/4.$$

    Recall the Oseledets' set $\Lambda$ (conull), $\Lambda_{loc}$ from Section \ref{sec:stablemfld}, and the set $\Lambda'$ from \ref{sec:NFC}. Also recall $K_r$, $K_s$, $K_{nbhd}$, and $K_{ang}$ from Section \ref{sec:lusin} from Lemmas \ref{comp1}, \ref{Kss}, \ref{knbhd}, and \ref{ang} respectively.

    Recall $\Lambda_s$ from Section \ref{sec:sect8}. On this set Corollary \ref{8.2} holds. Take $(x,\omega^+)\in \Lambda_s$, then we get a conull subset $\Omega_{x,\omega^+}^+\subset \Omega$ such that if $(\omega')^+\in \Omega^+_{x,\omega^+}$ then $W^s(x,\omega^+)\neq W^s(x,(\omega')^+)$ which means there is some $R=R((\omega')^+)$ such that if $\psi_x$ and $\psi_y$  are parametrizations of the holomorphic curves $ \exp^{-1}_{x}(W^s(x,\omega^+))$ and $ \exp^{-1}_x(W^s(x,(\omega')^+))$ respectively, they differ for the first time at the $R$th derivative, i.e. $D^n\psi_x(0) = D^n\psi_y(0)$ for $n=0,1,\dots R-1$ but $D^R\psi_x(0) \neq D^R\psi_y(0)$. The function $R:\Omega_{x,\omega^+} \to \bbN$ is dynamically invariant and hence constant a.e. by ergodicity.

    For an additional price to the measure of the set $\Lambda_s$, bound the $R$th derivative of $ \exp^{-1}_x(W^s(x,(\omega')^+))$ above by $C'$.

    We state the following Lemma from \cite{BEF}. It will be important to the proof of Lemma \ref{0.5} in this form, however for the purposes of the proof of Lemma \ref{0.1} it simply tells us that there are no atoms (i.e. take $d=0$ in the below).
    \begin{lemma}\label{algcurve1}
    For $m$-a.e. $(x,\omega)\in Y$ and for all $R>0$, $m^s_{x,\omega}|_{B^s(x,\omega, R)}$ does not assign mass to any algebraic curve of degree $d$.
    \end{lemma}
    \begin{proof}  
    See \cite{BEF}, Section 4.       
    \end{proof}
    Let $K_{alg}$ be the conull subset of $Y$ we get from Lemma \ref{algcurve1}. Now define
    $$\tilde{K} = ((\Lambda\cap \Lambda_{loc} \cap \Lambda'\cap K_+\cap K_s\cap K_r\cap K_{nbhd} \cap K_{ang} \cap K_{NFC} \cap K_{alg} \cap (\Lambda_s\times \Omega^+))\times [0,1[).$$

For this set $\tilde{K}$, let $\epsilon$ be small, and $\delta' = \frac{1}{100}\hat{m}_{\tau}(\tilde{K})$ apply Lemma \ref{birkhoff} to $\epsilon$ and $\delta = \delta'$ and get a set $E^c\subset Z$ and a constant $s>0$ where its conclusion holds: For all $c\in E^c$ and $T>s$,
\begin{align}\label{prime3}
    \left|\frac{1}{T} \int_0^T \mathds{1}_{\tilde{K}}(F^t_{tc}(c))dt - \int_Z\mathds{1}_{\tilde{K}}d\hat{m}_{\tau}  \right|<\delta',
\end{align}
where $F^t_{tc}$ is the `time changed' flow.

Take $\tilde{K}^* = \tilde{K}\cap E^c$.

Note that the definition of QNI only need hold for $\ell$ such that $$F^{\ell/2}(\hat{x}_{1/2}), \ F^{-\ell/2}(\hat{x}_{1/2})\in K,$$ once we pick $\hat{x}_{1/2}\in K$.

\begin{lemma}\label{ball}
    For $(x,\omega, k)\in \tilde{K}^*$, given $\ell$ such that $$F^{-\ell/2}(x,\omega, k), F^{\ell/2}(x,\omega,k)\in \tilde{K}^*,$$ we have that $\eta_{\omega,\ell/2}^s(x)$ contains a ball of radius $r'$ in $W^s(x,\omega)$. This $r'$ can be taken to be
    $$r ' = \frac{r}{M_{NFC}^2} e^{-t},$$
     where $r$ is as in in Lemma \ref{knbhd}, $M_{NFC}$ is as in Lemma \ref{comp3}, and $t$ is such that $F_{tc}^t(F^{-\ell/2}(x,\omega,k)) = (x,\omega,k)$ and $F_{tc}$ is the `time changed' flow. 
\end{lemma}
\begin{proof}

    We assume that $\ell$ is such that $F^{-\ell/2}(x,\omega,k) \in \tilde{K}$. Then $\eta^s_{\sigma^{-\ell/2}\omega}(F_{\omega}^{-\ell/2}(x))$ contains a ball $B^s(F_{\omega}^{-\ell/2}(x),r)$ inside $W^s(F^{-\ell/2}(x,\omega,k))$ because $F^{-\ell/2}(x,\omega,k)$ belongs to $K_{nbhd}$. 

    Now we pass to (modified) NFCs (see Equation \eqref{newNFCform}, Lemma \ref{comp3}) since we are working in $W^s(F^{-\ell/2}(x,\omega,k))$ and $F^{-\ell/2}(x,\omega,k) \in \tilde{K}^*$. Since we have that $F^{-\ell/2}(x,\omega,k) \in K_{NFC}$, we have that the image of $B^s(F_{\omega}^{-\ell/2}(x),r)$ contains a ball of radius $$r_{new} = \frac{r}{M_{NFC}}.$$

    Now let $t$ be such that $F_{tc}^t(F^{-\ell/2}(x,\omega,k)) = (x,\omega,k)$. We know that in (modified) NFCs flowing the stable forwards by $t$ will shrink our ball of radius $r_{new}$ by a factor of $e^{-t}$ precisely. Hence, in (modified) NFCs for $W^s(x,\omega,k)$ we have that the image of our ball of radius $r_{new}$ is a ball of radius $\frac{r}{M_{NFC}}e^{-t}.$ When we return to ambient coordinates, the image of our ball of $\frac{r}{M_{NFC}}e^{-t}$ contains a ball of radius $r'= \frac{r}{M_{NFC}^2}e^{-t}$. This tells us that $$F^{\ell/2}(B^s(F_{\omega}^{-\ell/2}(x),r)) \subset F^{\ell/2}(\eta^s_{\sigma^{-\ell/2}\omega}(F_{\omega}^{-\ell/2}(x))) = \eta_{\omega,\ell/2}^s(x),$$
    contains a ball of radius $r'= \frac{r}{M_{NFC}^2}e^{-t}$. This completes the proof.
\end{proof}
\begin{corollary}\label{independent}
    For $(x,\omega, k)\in \tilde{K}^*$, given $\ell$ such that $$F^{-\ell/2}(x,\omega, k), F^{\ell/2}(x,\omega,k)\in \tilde{K}^*,$$ then $\eta_{\omega,\ell/2}^s(x)$ contains a ball of radius $r'$ in $W^s(x,\omega)$ where $$r' \geq\frac{r}{M_{NFC}^2}e^{-(\ell/2+2)(\lambda^++\epsilon)},$$
    and $M_{NFC}$ and $r$ are as in Lemmas \ref{comp3} and \ref{knbhd} respectively. 
   
\end{corollary}
\begin{proof}
In Lemma \ref{ball} we found $\eta_{\omega,\ell/2}^s(x)$ always contains a ball of radius $$r ' = \frac{r}{M_{NFC}^2} e^{-t}.$$ 

    We can compute the dependence of $t$ on $\ell$. 
    
    Note that the roof function $\tau(F^{-\ell/2}(x,\omega,k))$, though point dependent, is bounded: for any $(x,\omega)$, $\tau(x,\omega)\in [\lambda^+-\epsilon, \lambda^++\epsilon]$ (see Equation \eqref{tau1}). Flowing forwards by $\ell/2$ applies $n=\lfloor \ell/2 + k \rfloor$ diffeomorphisms by definition of the standard flow. Taking $k=0$, the smallest number of diffeomorphisms being applied is $n=\lfloor \ell/2\rfloor$, taking $k=1$ the largest number of diffeomorphisms being applied is $\lfloor \ell/2 + 1\rfloor = n+1$; this independent of the point $(x,\omega,k)\in \tilde{K}^*$. Hence, $$t\in [n(\lambda^+-\epsilon), (n+1)(\lambda^++\epsilon) ].$$ We then see that in terms of $\ell$ the smallest $r'$ can be is $$r'=\frac{r}{M_{NFC}^2}e^{-(n+1)(\lambda^++\epsilon)} = \frac{r}{M_{NFC}^2}e^{-(\lfloor \ell/2\rfloor+1)(\lambda^++\epsilon)} .$$ To remove the floor function we can go further and see that 
    $$r' \geq\frac{r}{M_{NFC}^2}e^{-(\ell/2+2)(\lambda^++\epsilon)}.$$
    This is independent of the point $(x,\omega,k)\in \tilde{K}^*$. This completes the proof.
\end{proof}

\begin{lemma} For $(x,\omega,k)\in \tilde{K}^*$ and for given $\ell$ such that $$F^{\ell/2}(x,\omega,k), F^{-\ell/2}(x,\omega,k)\in \tilde{K}^*,$$ the intersection $\eta_{\omega,\ell/2}^s({x})\cap \supp(m_{x,\omega}^s)$ is infinite. \label{481}
\end{lemma}
\begin{proof}
   
   Since $(x,\omega,k)$ is chosen in $E^c\subset \tilde{K}^*$, we can unpack Lemma \ref{birkhoff}, Equation \eqref{prime3}: for $T>s$,
    $$(-\delta' + \hat{m}_{\tau}(\tilde{K}))T < \leb(\{t\in [0,T] : F_{tc}^t(\hat{x}_{1/2}) \in \tilde{K} \}) < (\delta' + \hat{m}_{\tau}(\tilde{K}))T,$$
    which works out to be
    $$0.99\hat{m}_{\tau}(\tilde{K})T < \leb(\{t\in [0,T] : F_{tc}^t(\hat{x}_{1/2}) \in \tilde{K} \}) < 1.01\hat{m}_{\tau}(\tilde{K}). $$

    Note that by Corollary \ref{independent} the set $\eta_{\omega,\ell/2}^s(x)\cap W^s(x,\omega,k)$ contains a neighbourhood of $x$ in $W^s(x,\omega,k)$, some ball radius $r'$ (which depends only on $\ell$).

    Take $\mathfrak{r}$ as above. If the point $F_{tc}^{-T}(x,\omega,k)\in \tilde{K}^*$ then there is a point $$z\in (\supp(m_{F_{tc}^{-T}(x,\omega,k)}^s)\cap B^s(F^{-T}_{tc}(x,\omega,k),\mathfrak{r}))\setminus B^s(F^{-T}_{tc}(x,\omega,k), \frac{1}{\mathfrak{r}}).$$ 

    We will work in (modified) normal form coordinates (NFCs) as in Equation \eqref{newNFCform} and Lemma \ref{comp3}. Let $z'$ represent $z$ in (modified) NFCs. Then the annulus that $z'$ belongs to has outer diameter $M_{NFC}\mathfrak{r}$ and inner diameter $\frac{1}{M_{NFC}\mathfrak{r}}$.

    Firstly, distances in (stable modified) NFCs shrink by a factor of $e^{-T}$ when we flow by $T$ under the time change flow $F^T_{tc}$. So take $T_0$ large such that $$M_{NFC}\mathfrak{r}e^{-T_0}<\frac{r'}{M_{NFC}},$$ and such that $T_0>s$ (where $s$ is chosen such that Equation \eqref{prime3} holds). Let $x'$ represent $x$ in (modified) NFCs.

    From the application of Lemma \ref{birkhoff}, since $(x,\omega,k)\in \tilde{K}^* \subset E^c$, we get that about 99\% of choices of $T\in [0,2T_0]$ will land us in the set $\tilde{K}$. Hence there is a positive Lebesgue measure choices of $T\in [T_0,2T_0]$ such that $F_{tc}^{-T}(x,\omega,k)$ lands in $\tilde{K}$. For such $T$, we get (ambiently) a point
 $$z_0\in (\supp(m_{F_{tc}^{-T}(x,\omega,k)}^s)\cap B^s(F^{-T}_{tc}(x,\omega,k),\mathfrak{r}))\setminus B^s(F^{-T}_{tc}(x,\omega,k),\frac{1}{\mathfrak{r}}).$$ 
 Denote $z_0'$ as $z_0$ in (modified) NFCs, then let $F_{tc}^T(z_0')$ be the result of flowing forwards $z_0'$ by $F_{tc}^T$ under the $\Omega$-component of $F_{tc}^{-T}(x,\omega,k)$. We have that $F_{tc}^T(z_0')$ belongs to the ball of radius $r'/M_{NFC}$ around $x'$ by construction (which means ambiently that $F^T_{tc}(z_0) \in B^s(x,r')$). Also, we have that the closest $F^T(z_0')$ can be to $x'$ is $\frac{1}{M_{NFC}\mathfrak{r}}e^{-2T_0}$. 

    Now take $T_1>2T_0+\ln(M_{NFC}^2\mathfrak{r}^2)$. Then 
    \begin{align}\label{NFCeq}
    M_{NFC}\mathfrak{r}e^{-T_1}<\frac{1}{M_{NFC}\mathfrak{r}}e^{-2T_0}.
    \end{align}
Now we pick $T^{(1)} \in [T_1,2T_1]$ such that $F_{tc}^{-T^{(1)}}(x,\omega,k)$ lands in $\tilde{K}$ using Lemma \ref{birkhoff}. We get its corresponding point$$z_1\in(\supp(m_{F_{tc}^{-T^{(1)}}(x,\omega,k)}^s)\cap B^s(F^{-T^{(1)}}_{tc}(x,\omega,k),\mathfrak{r}))\setminus B^s(F^{-T^{(1)}}_{tc}(x,\omega,k), \frac{1}{\mathfrak{r}}).$$ 
Let $z_1'$ be $z_1$ in (modified) NFCs.
The farthest away from $x'$ that $F_{tc}^{T^{(1)}}(z_1')$ can be is $M_{NFC}\mathfrak{r}e^{-T_1}$, but that is still closer than the point $F_{tc}^T(z_0')$ could be to $x'$ by construction (i.e. by Equation \eqref{NFCeq}). The closest $F_{tc}^{T^{(1)}}(z_1')$ could be to $x'$ is $\frac{1}{M_{NFC}\mathfrak{r}}e^{-2T_1}$.

We continue in this way, taking $T_2> 2T_1+\ln(M_{NFC}^2\mathfrak{r}^2)$ and $T^{(2)}\in [T_2,2T_2]$, we get a point $z_2'$ such that $F^{T^{(2)}}_{tc}(z_2')$ is closer to $x'$ than $F_{tc}^{T^{(1)}}(z_1')$ is, in fact the farthest $F^{T^{(2)}}_{tc}(z_2')$ can be from $x'$ is $M_{NFC}\mathfrak{r}e^{-T_2}$ where $$ M_{NFC}\mathfrak{r}e^{-T_2}<\frac{1}{M_{NFC}\mathfrak{r}}e^{-2T_1}. $$

We get a sequence of choices of $\{T^{(i)}\}_{i\geq 0}$ such that $F^{T(i)}(x,\omega,k)\in \tilde{K}$ and there is a point 
$$z_i\in(\supp(m_{F_{tc}^{-T^{(i)}}(x,\omega,k)}^s)\cap B^s(F^{-T^{(i)}}_{tc}(x,\omega,k),\mathfrak{r}))\setminus B^s(F_{tc}^{-T^{(i)}}(x,\omega,k),\frac{1}{\mathfrak{r}}).$$ 
    For $z_i'$ corresponding to $z_i$ in (modified) NFCs, we have that $F_{tc}^{T^{(i)}}(z_i')$ is strictly closer to $x'$ than $F_{tc}^{T^{(j)}}(z_j')$ for all $j<i$, but is strictly farther away from $x'$ than $F_{tc}^{T^{(p)}}(z_p')$ for all $p>i$. For each $i$, $F_{tc}^{T^{(i)}}(z_i')$ belongs to the (modified) NFCs chart based at $(x,\omega,k)$ which is a diffeomorphism. When returning to ambient coordinates, we have a sequence of distinct points, approaching $x$ (but not equal to $x$) that belong to $\supp(m_{x,\omega}^s)\cap B^s(x,\omega,r')$. This completes the proof.
    \end{proof}

    \begin{lemma}\label{keta}
Fix $\eta>0$ and $\delta>0$, and let $\epsilon>0$ be given. There exists a constant $\epsilon_{\eta}>0$ and a set $K_{\eta}$ of measure at least $(1-\epsilon)$, such that if $(x,\omega)\in K_{\eta}$, then in a ball of radius $\eta$ inside $W^s(x,\omega)$ (denoted $B^s(x,\omega,\eta)$), if one removes a ball of radius $\epsilon_{\eta}$ the resulting annulus has measure at least $(1-\delta/2)m_{x,\omega}^s(B^s(x,\omega,\eta))$.
\end{lemma}
\begin{proof}
    It follows from Lemma \ref{algcurve1} that because $x$ is not an atom, for each $(x,\omega)$ there is a $\epsilon_{\eta}(x,\omega)$ such that by removing a ball of radius $\epsilon_{\eta}(x,\omega)$ the resulting annulus has measure at least $(1-\delta/2)m_{x,\omega}^s(B(x,\omega,\eta))$. This $\epsilon_{\eta}(x,\omega)$ is a measurable function because $\eta$ is fixed. We can then lower bound $\epsilon_{\eta}(x,\omega)$ by some constant $\epsilon_{\eta}$ on a set $K_{\eta}$ of measure at least $(1-\epsilon)$. This completes the proof.
\end{proof}

    Let $\alpha_0>0$ be defined by $\alpha_0 = \frac{1}{2}R(\lambda^++\epsilon)$. Note that this choice is independent of $\delta$ because $R$ is constant a.e.
    
    Lastly, let $\eta<\frac{r}{M_{NFC}^4}$ be fixed and consider the set $K_{\eta}$ from Lemma \ref{keta}.
    
     Take $\epsilon_{loc}$, $\epsilon_{nbhd}$, $\epsilon_{ang}$, and all other $\epsilon$-costs to be such that for $$K^* = \tilde{K}^*\cap (K_{\eta}\times [0,1[).$$
     we have that $\hat{m}(K^*)>1-\delta^{i+1}/8$. Note that belonging to $\Lambda_{loc}$ will ensure the existence of local stable manifolds of some size $q$.

    For each $(x,\omega^-,k)\in X\times \Omega^- \times [0,1[$, let $$V_{x,\omega^-,k} = \{(\omega')^+\in \Omega^+ : (x,\omega^-,(\omega')^+,k)\in K^*  \},$$ and define
    \begin{align}
        G \defeq \{ (x,\omega^-,k)\in X\times \Omega^-\times [0,1[ : \mu^{\bbN}(V_{x,\omega^-,k})>1-\delta/4 \}.
    \end{align}

    \begin{lemma}\label{G!}
        For the set $G$ defined above, $\nu\times \mu^{\bbN} \times dt(G)>1-\delta^{i}/2$.
    \end{lemma}
\begin{proof}

    Consider $G^c$ which is the set of $(x,\omega^-,k)$ such that the subset $W\subset \Omega^+$, defined by the property that if $(\omega')^+\in W$ then $(x,\omega^-,(\omega')^+,k)\in (K^*)^c$, has $\mu$-measure at least $\delta/2$.

    Now let $E\subset (K^*)^c$ be defined by
    \begin{align*}
    E \defeq \{(x,\omega^-,(\omega')^+,k)\in X\times \Omega^- \times \Omega^+ \times  [0,1[ : (x,\omega^-,k)\in G^c, \ (x,\omega^-, (\omega')^+, k) &\in (K^*)^c \} 
    \end{align*}

    Now using Fubini, note that
    \begin{align*}\hat{m}((K^*)^c) \geq \hat{m}(E) &= \int_{X \times \Omega^-  \times [0,1[} \int_{\Omega^+} d\mu^{\bbN} d(\nu\times \mu^{\bbN} \times  dt) = \mu^{\bbN}(W)\cdot (\nu\times \mu^{\bbN} \times  dt)(G^c)\\ &\geq  \theta \cdot (\nu\times \mu^{\bbN} \times dt)(G^c).
    \end{align*}

    Hence $(\nu\times \mu^{\bbN} \times dt)(G^c) \leq \frac{\hat{m}((K^*)^c)}{\delta/2}$. This gives us our result.
\end{proof}

    Note that the $\eta^{s/u}_{\omega}(x)$ (defined $m$-a.e.) are measurable partitions of $X$. We apply \cite[Lemma 3.11]{EM}, which we state below:
    \begin{lemma}[\cite{EM}, Lemma 3.11] \label{3.11}
    Let $\delta>0$, for a.e. $\omega \in \Omega$ given $X'_{\omega}\subset X_{\omega}$ with $\nu(X')>1-\delta^2$, there exists $X_{\omega}^*\subset X'$ with $\nu(X_{\omega}^*)>1-{\delta}$ and such that for $x\in X_{\omega}^*$ and any $\ell>0$ we have that
    \begin{align}\label{star}
        m_{x,\omega}^{s}(X_{\omega}'\cap \eta_{\omega,\ell}^{s}(x)) \geq (1-{\delta})m_{x,\omega}^{s}(\eta_{\omega,\ell}^{s}(x)).
    \end{align}
\end{lemma}
One can do the same for the unstable partitions. 

Let $\Omega'$ be the conull subset of $\Omega$ where Lemma \ref{3.11} holds.

Now let $(x,\omega^+,k)\in X\times \Omega^+ \times [0,1[$, let $$V_{x,\omega^+,k} = \{(\omega'')^-\in \Omega^- : (x,(\omega'')^-,\omega^+,k)\in K^*  \},$$ and define for a given $(\omega,k)$:
\begin{align}
    X''_{\omega,k} = \{x\in X : \mu^{\bbN}(V_{x,\omega^+,k})>1-\delta^j\}.
\end{align}

\begin{lemma}\label{41}
    There exists a subset $W' \subset \Omega\times [0,1[$ of measure at least $1-\delta^{i-2j}/2$ such that for $(\omega,k)\in W'$, $X''_{\omega,k}$ has measure at least $1-\delta^j$.
\end{lemma}
\begin{proof}
Using Lemma \ref{G!} we get that 
$$Y'' = \{(x,\omega^+,k) : \mu^{\bbN}(V_{x,\omega^+,k})>1-\delta^j  \},$$
has measure $\hat{m}(Y'') > 1-\delta^{i-j}/2$

Note that $X''_{\omega,k} = \{x\in X  :  (x,\omega^+,k) \in Y'' \}$; let $$W' = \{(\omega,k) : \nu(X''_{\omega,k})>1-\delta^j \},$$ by Lemma \ref{G!} the set $W'$ has measure $1-\delta^{i-2j}/2$.
\end{proof}

Let $W = W'\cap (\Omega'\times [0,1[)$.
Let $$K' = ((K_{\xi}\cap K_+ \cap K_{nbhd})\times [0,1[)\cap (G\times \Omega^+).$$ Note that $\hat{m}(K')>1-\delta^i/2$.

For $(\omega,k)\in W$ let 
\begin{align}
    X'_{\omega,k} = \{x\in X : (x,\omega,k)\in K'\}\subset X_{\omega}\\
    G' = \{(\omega,k)\in W : \nu(X'_{\omega,k})>1-\delta^j \}.
\end{align}

\begin{lemma}\label{42}
    The measure of $G'$ is $\mu^{\bbZ}\times dt(G')>1-\frac{\delta^i/2}{(\delta^j-\delta^i/2)}$.
\end{lemma}
\begin{proof}
    
    Consider $H\defeq W\setminus G'$. Define $U_{(\omega,k)}\subset X$ as the set of $x$ such that $(x,\omega,k)\in (K')^c$. The set $H$ is the set of $(\omega,k)\in W$ such that $U_{(\omega,k)}$ has measure at least $\delta^j$ inside $W$, (hence has measure $\delta^j(1-\delta^{i-2j}/2) = (\delta^j-\delta^{i-j}/2)$ where $i>j$).

    $$E = \{(x,\omega,k) : (\omega,k)\in H, (x,\omega,k)\in (K')^c  \} \subset (K')^c.$$

    \begin{align*}\hat{m}((K')^c) \geq \hat{m}(E) &= \int_{\Omega\times [0,1[} \int_{X} \mathds{1}_E  d\nu d(\mu^{\bbZ}\times dt)=\int_H \int_{U_{(\omega,k)}} d\nu d(\mu^{\bbZ}\times dt) \\
    &=\int_H \nu(U_{(\omega,k)}) \mu^{\bbZ}\times dt \geq (\delta^j-\delta^{i-j}/2) \cdot (\mu^{\bbZ}\times dt)(H)
    \end{align*}

    Note that $\hat{m}((K')^c) <\delta^i/2$, and so $(\mu^{\bbZ}\times dt)(H) \leq \frac{\delta^i/2}{(\delta^j-\delta^{i-j}/2)}$. Hence $$\mu^{\bbZ}\times dt(G') > 1-\frac{\delta^i/2}{(\delta^j-\delta^{i-j}/2)}.$$
    This completes the proof.
\end{proof}

Let $X'''_{\omega,k} = X'_{\omega,k}\cap X''_{\omega,k}$ and use this in Lemma \ref{3.11} to get a subset $X^*_{\omega,k}$ where Equation \eqref{star} holds. We know that $X'''_{\omega,k}$ has measure at least $1-2\delta^j$. Take

$$K = \left( \bigcup_{(\omega,k)\in G'} X^*_{\omega,k} \times \{\omega\} \times \{k\} \right) \subset K'.$$

\begin{lemma}
\label{fubini2}
    Let $(X_1,\eta_1)$ and $(X_2,\eta_2)$ be $\sigma$-finite measure spaces. Consider $V\subset X_1$, a set such that $\eta_1(V)>1-\theta$, and for each point $a\in V$ consider a set $W_a\subset X_2$ such that $\eta_2(W_a)>1-\theta'$. Then $K\defeq \bigcup_{a\in V}W_a\times \{a\}$ has measure at least $(1-\theta')(1-\theta)$ with respect to the $\eta_1\times \eta_2$ on $X_1\times X_2$.
\end{lemma}
\begin{proof}
    Writing $K = \{(w,a) : w\in W_a, \ a\in V\}$, we apply Fubini's theorem:
    \begin{align*}
        \eta_1\times \eta_2(K) = \int_V \int_{W_a}d\eta_2d\eta_1 = \int_V \eta_2(W_a)d\eta_1>(1-\theta')\eta_1(V)>(1-\theta')(1-\theta).
    \end{align*}
    This completes the proof.
\end{proof}

By Lemma \ref{fubini2}, the measure of $K$ is at least 
$$\hat{m}(K)>(1-2\delta^j)(1-\frac{\delta^i/2}{\delta^j-\delta^i/2})>1-\delta.$$

    Now take $\hat{x}_{1/2} = (x_{1/2},\omega_{1/2},k_{1/2})\in K$, then ${U}^+[\hat{x}_{1/2}] = \{{x}_{1/2}\}$, so $$\hat{U}^+[\hat{x}_{1/2}] = \{{x}_{1/2}\} \times \{\omega^-_{1/2}\} \times \Omega^+ \times \{k_{1/2}\}. $$

For any $\ell>0$, $$\eta_{\ell/2}^u[\hat{x}_{1/2}]=\eta^u_{\omega_{1/2},\ell/2}(x_{1/2}) \times \{\omega^-_{1/2}\} \times \Omega^+ \times \{k\}.$$ Hence, $$\hat{U}^+[\hat{x}_{1/2}]\cap \eta_{\ell/2}^u[\hat{x}_{1/2}] = \hat{U}^+[\hat{x}_{1/2}].$$

Now since $\{x_{1/2}\} \times \{\omega_{1/2}^-\} \times \{k_{1/2}\} \in G$ by construction, we have that the set $V_{x_{1/2},\omega^-_{1/2},k_{1/2}}\subset \Omega^+$ has measure greater than $(1-\delta/2)$.

Take $$\tilde{S} = \{x_{1/2}\} \times \{\omega_{1/2}^-\} \times V_{x_{1/2},\omega^-_{1/2},k_{1/2}}\times \{k_{1/2}\},$$ then $$\hat{m}^u_{\hat{x}_{1/2}}(\tilde{S})\geq (1-\delta/4)\hat{m}^u_{\hat{x}_{1/2}}(\hat{U}^+[\hat{x}_{1/2}]).$$ Note that by construction, $\tilde{S}\subset K$. We have one more restriction to add before we can define $S$ as in  Definition \ref{QNIdef}, this is seen in Lemma \ref{tildeS}.

In the lemma below, we use that our points belong to $$K^*\subset (\Lambda_s\times \Omega^+\times K_{ang})\times [0,1[.$$ We assume that $\ell_0$ is large enough in the following lemma:

\begin{lemma}\label{tildeS}

There exists a subset $S\subset \tilde{S}$ such that $$\hat{m}^u_{\hat{x}_{1/2}}(S)\geq (1-\delta/2)\hat{m}^u_{\hat{x}_{1/2}}(\hat{U}^+[\hat{x}_{1/2}]),$$ and for $\hat{y}_{1/2}\in S$ we have that if $z_{1/2}\in W^s(\hat{x}_{1/2})$ is such that $d_X(z_{1/2},x_{1/2})>\delta$ then $$d_X(z_{1/2},W^s(\hat{y}_{1/2}))>C'\delta^R,$$ (where $C'$ is as in the definition of $\Lambda_s$).

\end{lemma}
\begin{proof}
Remember that the $X$-component of $\hat{y}_{1/2}=x_{1/2}$. One examines the worst case scenario, when for given $(\omega')^+$, $W^s(x_{1/2},\omega_{1/2}^+)\neq W^s(x_{1/2},(\omega'))$ but $$E^s(x_{1/2},\omega_{1/2}^+) = E^s(x_{1/2},(\omega')).$$ Note also that $K^*\subset K_{ang}$.

Now the proof is nearly identical to that of Lemma \ref{omega+}.
\end{proof}

 Now take $\hat{y}_{1/2} = (x_{1/2},\omega^-_{1/2},(\omega')^+,k_{1/2})\in S$.

Since $\hat{x}_{1/2}$ was chosen in $K$, by construction, $x_{1/2}\in X^*_{\omega_{1/2},k_{1/2}} \subset X_{\omega_{1/2},k_{1/2}}'''$ where $(\omega_{1/2},k_{1/2})\in W$. Hence by Lemma \ref{3.11} we have that 
\begin{align}
    m^s_{x_{1/2},\omega_{1/2}}(X'''_{\omega_{1/2},k_{1/2}}\cap \eta_{\omega_{1/2},\ell/2}^s({x}_{1/2}) )\geq (1- \sqrt{2}\delta^{j/2})m_{x_{1/2},\omega_{1/2}}^s(\eta_{\omega_{1/2},\ell/2}^s(x_{1/2})).
\end{align}

Let 
$$\tilde{G} = \bigcup_{z\in X'''_{\omega_{1/2},k_{1/2}}} \{z\} \times V_{z,\omega^+_{1/2},k_{1/2}} \times \{\omega_{1/2}^+\} \times \{k_{1/2}\}. $$

Take $\tilde{S'} = \eta_{\ell/2}^s[\hat{x}_{1/2}]\cap \tilde{G}$. By Lemma \ref{fubini2} and the conditions on $i$ and $j$ it has measure at least $(1-\delta/2)$ relative to the measure of $\eta_{\ell/2}^s[\hat{x}_{1/2}]$.

\begin{lemma}\label{annbd}
    There exists a constant $\tilde{\epsilon}>0$ small such that for all $\hat{x}_{1/2}\in K$ and $\ell$ large enough, if we assume that $F^{\ell/2}(\hat{x}_{1/2}),F^{-\ell/2}(\hat{x}_{1/2})\in K$, then removing a ball $B^s(x_{1/2},\tilde{\epsilon}r')$ from $\eta^s_{\ell/2,\omega_{1/2}}(x_{1/2})$ only reduces the measure of $\eta^s_{\ell/2}[\hat{x}_{1/2}]\cap \tilde{G}$ by at most $\delta/2$.

\end{lemma}
\begin{proof}

   We know by Lemma \ref{ball} that $\eta_{\ell/2}^s(x_{1/2})$ contains a ball of radius $r'$. We got an explicit formula for this $r'$ as
    $$r ' = \frac{r}{M_{NFC}^2} e^{-t}.$$
     Here, $r$ is as in in Lemma \ref{knbhd}, $M_{NFC}$ is as in Lemma \ref{comp3}, and $t$ is such that $F_{tc}^t(F^{-\ell/2}(\hat{x}_{1/2})) = \hat{x}_{1/2}$ and $F_{tc}$ is the `time changed' flow.  

    Since $F^{-\ell/2}(\hat{x}_{1/2})\in K\subset K_{\eta}$, in a ball of radius $\eta$ around $F^{-\ell/2}(x_{1/2})$, if we remove a ball of radius $\epsilon_{\eta}$, then the measure of the resulting annulus is $$(1-\delta/2)m_{F^{-\ell/2}(\hat{x}_{1/2})}^s(B^s(F^{-\ell/2}(x_{1/2}),\eta)).$$

    Since $m_{x,\omega}^s$ is dynamically invariant, the proportion (in measure) between the images of annulus and the ball of radius $\eta$ stays the same under the dynamics. We now wish to flow forwards by $\ell/2$ back to the point $\hat{x}_{1/2}$ and shrink this ball of radius $\eta$ to be inside our ball of radius $r'$ around $x_{1/2}$. 

    When we pass to (modified) NFCs at $\hat{x}_{1/2}$, the smallest the image of $B^s(x_{1/2},r')$ could be is a ball of radius $\frac{r'}{M_{NFC}}$. Similarly, the largest the ball of radius $\eta$ in (modified) NFCs at $F^{-\ell/2}(x_{1/2})$ could be is a ball of radius $\frac{\eta}{M_{NFC}}$. Then to guarantee that after flowing by $\ell$ (or $t$ under `time changed' flow) that $$F^{\ell/2}(B^s(F^{-\ell/2}(x_{1/2}),\eta)) \subset B^s(\hat{x}_{1/2},r'),$$ we need that 
    $$\eta M_{NFC}e^{-t} < \frac{r'}{M_{NFC}} = \frac{r}{M_{NFC}^3}e^{-t}.$$
This holds true because $\eta$ was fixed such that $\eta <\frac{r}{M_{NFC}^4}$ at the beginning of the argument.

Now the smallest that the image our ball of radius $\epsilon_{\eta}$ could be in (modified) NFCs is a ball of radius $\frac{1}{M_{NFC}}\epsilon_{\eta}$. We flow forwards back to the point $x_{1/2}$ by going forwards by $\ell$, this ball shrinks to size $\frac{1}{M_{NFC}}\epsilon_{\eta}e^{-t}$.

Returning to ambient coordinates, we have that $B^s(\hat{x}_{1/2},r')$ which is of radius $r'=\frac{r}{M_{NFC}^2}e^{-t}$ contains a ball of radius $\epsilon_{\eta}'=\frac{1}{M_{NFC}^2}\epsilon_{\eta}e^{-t}$. If we remove this ball of radius $\epsilon_{\eta}'$, by construction this reduces the measure of $\eta^s_{\ell/s}[\hat{x}_{1/2}]\cap \tilde{G}$ by at most $\delta/2$. Let $\tilde{\epsilon} = \frac{\epsilon_{\eta}}{r}$, then $\tilde{\epsilon}r' = \epsilon_{\eta}'$ and this constant $\tilde{\epsilon}$ gives us what we wanted. This completes the proof.
\end{proof}
Let $S'$ be the subset of $\eta^s_{\ell/2}[\hat{x}_{1/2}]\cap \tilde{G}$ that we get from Lemma \ref{annbd}. It has the appropriate measure by construction.

Take a point $\hat{z}_{1/2} = (z_{1/2}, (\omega'')^-,\omega_{1/2}^+,k)\in S'$. Since $S'\subset K$, $U^+[\hat{z}_{1/2}] = z_{1/2}$ which lives on $W^s(\hat{x}_{1/2})$. Since we removed an annulus of radius $\tilde{\epsilon}r'$ around $x_{1/2}$, we have that $d_X(z_{1/2},x_{1/2})>\tilde{\epsilon}r'$. By Lemma \ref{tildeS}, we have that

\begin{align}
d_X(z_{1/2},W^s(\hat{y}_{1/2}))\geq C'(\tilde{\epsilon})^R(r')^R
\end{align}

We know by Corollary \ref{independent} that $r' \geq \frac{r}{M_{NFC}^2}e^{-(\ell/2+2)(\lambda^++\epsilon)}$. Plugging that in we get
\begin{align}
d_X(z_{1/2},W^s(\hat{y}_{1/2})) &\geq C'(\tilde{\epsilon})^R \left(\frac{r}{M_{NFC}^2}e^{-(\ell/2+2)(\lambda^++\epsilon)}\right)^R\\ \nonumber
    &= \left(C'(\tilde{\epsilon})^R\left(\frac{r}{M_{NFC}^2}\right)^Re^{-2(\lambda^++\epsilon)R}\right)e^{-\ell/2(\lambda^++\epsilon)R}\\ \nonumber
    &= \left(C'(\tilde{\epsilon})^R\left(\frac{r}{M_{NFC}^2}\right)^Re^{-2(\lambda^++\epsilon)R}\right)e^{-\ell \alpha_0},
\end{align}
where the last inequality is because we took $\alpha_0 = \frac{1}{2}R(\lambda^++\epsilon)$.

Take $C = C'(\tilde{\epsilon})^R\left(\frac{r}{M_{NFC}^2}\right)^Re^{-2(\lambda^++\epsilon)R}$, this is independent of $\delta$. This completes the proof of QNI and Lemma \ref{0.1}. 
\end{proof}

\subsection{Outline of the proof of Lemma \ref{0.5}}

\subsubsection{Overview}

In this subsection we are in the setting of Alternative 2, (see Section \ref{sec:alt22}). Here, $d_+=1$ a.e. We assume further that $\dim(\calL^s)=2$.

We let $\delta>0$ and we demonstrate that the QNI definition, i.e. Definition \ref{QNIdef}, holds in our setting.

We open with defining a good set $K'$ that we want the set $K$ to be a subset of. This set $K$ is where we will choose $\hat{x}_{1/2} = (x_{1/2},\omega_{1/2},k_{1/2})$ from as in Definition \ref{QNIdef}. Note that we need to pick $S \subset \eta^u[\hat{x}_{1/2}]\cap (U^+[\hat{x}_{1/2}]\times W^+_1)$ which means that we are picking a new $X$-component that lives along $U^+[x,\omega,k]$ and a new future but leaving the $\Omega^-\times [0,1[$-component the same. Note that in the last subsection, the setting of Lemma \ref{0.1} assumed that $U^+[x,\omega,k]=\{x_{1/2}\}$ a.e., this is no longer the case, here $\dim(U^+[x,\omega,k])=1$. The set $S'$ (dependent on the choice of $\hat{y}_{1/2}\in S$) is a subset of $\eta^s_{\omega_{1/2},\ell/2}(x_{1/2})$ meaning that we pick a new $X$-component belonging to $W^s(\hat{x}_{1/2})$, a new past component, and keep $\Omega^+\times [0,1[$-component the same. We work to ensure that $S,S'\subset K'$ so they have good properties. 

We then prove Lemma \ref{elllarge} which demonstrates the QNI condition, with Lemma \ref{firstap} as an intermediate linear approximation step. To get the idea, first think of $W^s(\hat{y}_{1/2})$ and $W^s(\hat{x}_{1/2})$ as a fixed planes. Then think of $U^+[z_{1/2},\omega^-]$ for $(z_{1/2},\omega,k)\in S'$ as a straight line. Vary the $X$-component $z_{1/2}$ throughout some ball $B$ centered around $x_{1/2}$ in $W^s(\hat{x}_{1/2})$ and imagine that for any $z'_{1/2}$ in this ball, $U^+[z_{1/2}',\omega^-]$ is a straight line with the same direction vector as $U^+[z_{1/2},\omega^-]$. The set of $z_{1/2}'\in B$ such that $U^+[z_{1/2}',\omega^-]$ will intersect $W^s(\hat{y}_{1/2})$ is a 1-dimensional set. However $\dim(\calL^s)=2$ and our measure $m_{x,\omega}^s$ cannot be supported on an algebraic curve by Lemma \ref{algcurve}. This allows us to get rid of this 1-dimensional set (for a choice of $\omega^-$) in the definition of $S'$. We then add back in the error terms we get from making these approximations and use Lemma \ref{algcor} to say we can eliminate a tubular neighbourhood of an algebraic curve and maintain the necessary measure required for $S'$. In this way we have that for $\hat{y}_{1/2}\in S$ and $\hat{z}_{1/2}\in S'$ (where $S'$ is dependent on $\hat{y}_{1/2}$) $U^+[\hat{x}_{1/2}]$ is sufficiently bounded away from $W^s(\hat{y}_{1/2})$ and hence have the QNI condition.

\subsubsection{Details}

Again, we are in the setting of Alternative 2, i.e. Section \ref{sec:alt22}. Here, $\dim(\calL^s)=2$ and $d_+=1$ a.e. We let $\delta>0$ be given and define a good set $K'$ with good properties from good sets defined in Sections \ref{stableholon} and \ref{sec:lusin1}. We add to $K'$ another good property that for fixed degree there is some lower bound on the size of tubular neighbourhood around an algebraic curve outside of which we still have large measure. This comes from Lemma \ref{algcor}.

In order to make sure that $S$ and $S'$ belong to $K'$ as well we have to do some work to construct $K$. We first construct $K_1$ and if $K\subset K_1$ then $S'$ can be constructed to belong to $K'$. We employ a Lemma of \cite{EM}, Lemma \ref{3.112} that allows us to say for a.e. $\omega$ and for a large measure subset $X'''$ of $X$, there is a large measure subset $X^*$ such that if $x\in X^*$ then the measure of $X'''$ inside $\eta^s_{\omega,\ell}(x)$ compared to the measure of $\eta^s_{\omega,\ell}(x)$ is large. We define a set $X''_{\omega,k}$ that is the set of $X$ such that there is a large set of $\Omega^-$-components to choose from such that if wee keep $k$ and $\omega^+$ the same we belong to $K'$. We then demonstrate in Lemma \ref{w'large} that the set of $(\omega,k)$ such that $X''_{\omega,k}$ is large in measure is also large in measure.  We take these $(\omega,k)$ and restrict to the conull subset of $\Omega$ such that Lemma \ref{3.112} holds and call this $W$. This construction allows us to change the past and land in a good set. Now we focus on changing the $X$-component and landing in a good set. We then let $X'_{\omega,k}\subset X$ be the set of $X$-components such that along with $\omega $ and $k$ we belong to $K'$. We let $G$ be the set of $(\omega,k)\in W$ such that $X'_{\omega,k}$ is large in measure. We then demonstrate in Lemma \ref{Glarge} that the size of $G$ is big. We let $X'''_{\omega,k} = X'_{\omega,k}\cap X''_{\omega,k}$ which is large for $(\omega,k)\in G$ and apply Lemma \ref{3.112} to get a large measure subset $X^*_{\omega,k}$. We define $K_1$ such that the $X$-component belongs to this set $X'''_{\omega,k}$ for $(\omega,k)\in G$.

We repeat this process to ensure that $S$ can be constructed to belong to $K'$. Here we use a modified version of Lemma \ref{3.112} which also comes from \cite[Lemma 6.3]{EM}. We state this as our Lemma \ref{6.3}. This allows us to say that if we have a large subset $X'''\subset X$, there is a subset $X^*$ such that restricting to the unstable supports this subset still has relatively large measure (relative to the conditional measure denoted $|\cdot |$ on the unstable support). This is used to ensure that we can change $X$-components from $\hat{x}_{1/2}$ to some element on $U^+[\hat{x}_{1/2}]$ and land in a good set. We also have to change futures and we go about that in the same way as we did above in the definition of $(X''_{\omega,k})^+$. We construct the set $K_2$ which is analogous to $K_1$ and take $K=K_1\cap K_2$. Now when we construct $S$ and $S'$ they will belong to $K'$.

Then Lemma \ref{elllarge} works to demonstrate the QNI condition directly. First we do a linear approximation step. For a fixed $(\omega')^-\in \Omega^-$, the tangent lines to the $U^+[z ,(\omega')^-]$ as we move $z$ around a small ball around $x_{1/2}$ inside $W^s(\hat{x}_{1/2})$ vary H\"older continuously. If we fix a direction vector $v$ representing such a tangent line and drag it around a small ball inside $E^s(\hat{x}_{1/2})$ we demonstrate the the set of points inside this small ball where the line in direction $v$ hits the tangent plane through $\exp^{-1}(y_{1/2})(W^s(\hat{y}_{1/2}))$ is 1-dimensional, i.e. is some curve in the ball. This is Lemma \ref{firstap}. Given that $\dim(\calL^s)=2$ and that $m_{\hat{x}_{1/2}}^s$ cannot be supported on some algebraic curve (Lemma \ref{algcurve}), we are almost done up to some error terms, but we must make sure that this curve, and some error around it, does not support too much measure. This allows us to ensure that when we define $S'$, it is sufficiently large. This is accomplished by using Lemma \ref{algcurve} and \ref{algcor}.

\subsection{Proof of Lemma \ref{0.5}}
\begin{proof}(Lemma \ref{0.5})\\
\label{proof0.5}

    Fix $\alpha_0>0$ and let $\delta>0$, we will determine $\ell_0$ and $C$ as in Definition \ref{QNIdef}.

    We are given that $\dim(U^+[x,\omega,k])=1$ a.e. and that $\dim(\calL^s)=2$ a.e. We restate Lemma \ref{algcurve1} from \cite[Section 4]{BEF}.
    
    \begin{lemma}\label{algcurve}
    For $m$-a.e. $(x,\omega)\in Y$ and for all $R>0$, $m^s_{x,\omega}|_{B^s(x,\omega, R)}$ does not assign mass to any algebraic curve of degree $d$.
    \end{lemma}
    \begin{proof}  
    See \cite{BEF}, Section 4.       
    \end{proof}
    Let the conull set where Lemma \ref{algcurve} holds be called $K_C\subset Y$. 

    \begin{lemma}\label{algcor}
        Fix $d\geq 1$ an integer and $R>0$. For any $\delta\in ]0,1[$, $(x,\omega)\in K_C$, there exists $\epsilon>0$ such that given any algebraic curve of degree $d$ in $B^s(x,\omega, R)$, outside of a tubular neighbourhood of diameter $\epsilon$ we have $m^s_{x,\omega}$-measure at least $(1-\delta) \cdot m^s_{x,\omega}(B^s(x,\omega,R))$.
    \end{lemma}
    \begin{proof}

    Suppose not, then there exists $\delta$ and $(x,\omega)\in K_C$ such that for all $\epsilon>0$, there exists an algebraic curve $C$ in $B^s(x,\omega,R)$ such that outside a tubular neighbourhood of radius $\epsilon$ around $C$ we have less than $(1-\delta) \cdot m^s_{x,\omega}(B^s(x,\omega,R))$.

    Take a sequence of $\epsilon_n\to 0$, each has an algebraic curve $C_n$ such that inside of a $\epsilon_n$ tubular neighbourhood of $C_n$ we have measure at greater than $(1-\delta)m^s_{x,\omega}(B^s(x,\omega,R))$. The sequence of $C_n$ has a convergent subsequence by compactness, and so we have some limiting curve $C$. This limiting curve $C$ supports mass which is a contradiction to Lemma \ref{algcurve}. This completes the proof.        
    \end{proof}

    Taking $\delta=\delta/4$ and $R=1$ in Lemma \ref{algcor}, one can get a Lusin set $K_{C,1}$ on which the $\epsilon$ from Lemma \ref{algcor} is bounded below by $\epsilon^*$.

    Fix $i>j>2$ such that $i\gg2j$, and 
    \begin{align*}(1-2\delta^j)(1-\frac{\delta^i/2}{(\delta^j-\delta^i/2)})>1-\delta/4,\\ (1-\delta^j)(1-\sqrt{2}\delta^{j/2})>(1-\delta^k),\end{align*} where $k$ is such that $(1-\delta^k)(1-\delta^{j/2})>1-\delta$. 

    Recall the Oseledets' set $\Lambda$ (conull) and $\Lambda_{loc}$ from Section \ref{sec:stablemfld}; also recall $K_r$ and $K_s$ from Section \ref{sec:lusin} from Lemmas \ref{comp1} and \ref{Kss} and $K_{\xi} = \Lambda\cap \Lambda_{loc} \cap K_r\cap K_s$ from Lemma \ref{errorr} and Lemma \ref{errorr1}. Belonging to $K_{\xi}$, Lemma \ref{errorr} tells us the size of ball on which we can approximate the stable and unstable manifolds by their tangent spaces up to error $\xi$ ambiently in $X$ is $R=\sqrt{\xi/L}\cdot \frac{1}{M_r^{3/2}}$ where $M_r$ is chosen before $\xi$. Also similar for the tangent lines of $U^+[x,\omega]$ using Lemma \ref{errorr1}. Belonging to $\Lambda_{loc}$ will ensure the existence of local stable manifolds of some size $q$. Also recall the set $K_{\beta}$ from Section \ref{stableholon}.

    Let $$K' = \Lambda\cap \Lambda_{loc}\cap K_r\cap K_s\cap K_{\xi}\cap K_{\beta}\cap K_{C,1}$$ and be such that it has measure $\hat{m}(K')>1-\delta^i/2$.

    Now we apply Lemma \ref{3.11} as we did in the proof of \ref{0.1}, we recall the statement below:
    \begin{lemma}[\cite{EM}, Lemma 3.11] \label{3.112}
    Let $\delta>0$, for a.e. $\omega \in \Omega$ given $X_{\omega}'\subset X_{\omega}$ with $\nu(X')>1-\delta^2$, there exists $X_{\omega}^*\subset X'$ with $\nu(X_{\omega}^*)>1-{\delta}$ and such that for $x\in X_{\omega}^*$ and any $\ell>0$ we have that
    \begin{align}
        m_{x,\omega}^{s}(X_{\omega}'\cap \eta_{\omega,\ell}^{s}(x)) \geq (1-{\delta})m_{x,\omega}^{s}(\eta_{\omega,\ell}^{s}(x)).
    \end{align}
\end{lemma}

Let $\Omega'$ be the conull subset of $\Omega$ where Lemma \ref{3.112} holds.

Let 
$$V_{x,\omega,k} = \{(\omega')^- \in \Omega^- : (x,(\omega^-)',\omega^+,k) \in K'\},$$
and $$X''_{\omega,k} = \{x\in X:\mu^{\bbN}(V_{x,\omega,k})>1-\delta^j\}.$$

\begin{lemma}\label{w'large}
    There exists a subset $W' \subset \Omega\times [0,1[$ of measure at least $1-\delta^{i-2j}/2$ such that for $(\omega,k)\in W'$, $X''_{\omega,k}$ has measure at least $1-\delta^j$.
\end{lemma}
\begin{proof}
Same proof as in Lemma \ref{41}.
\end{proof}
Let $W = W'\cap (\Omega'\times [0,1[)$.

For $(\omega,k)\in W$ let
$$X_{\omega,k}' = \{x\in X : (x,\omega,k)\in K'\} \subset X_{\omega},$$
\begin{align}G = \{(\omega,k) \in W : \nu(X_{\omega,k}')>1-\delta^j \}.
\end{align}

\begin{lemma}\label{Glarge}
    The measure of $G$ is $\mu^{\bbZ}\times dt(G)>1-\frac{\delta^i/2}{(\delta^j-\delta^i/2)}$.
\end{lemma}
\begin{proof}
  This is the same as the proof of Lemma \ref{42}.   
\end{proof}

Let $X'''_{\omega,k} = X'_{\omega,k} \cap X''_{\omega,k}$ and use this in Lemma \ref{3.112}, get a subset $X^*_{\omega,k}$ where Equation \eqref{star} holds. For $(\omega,k)\in G$ we know that $X'_{\omega,k}$ has measure at least $1-\delta^j$, and that $X''_{\omega,k}$ has measure at least $1-\delta^j$. This implies that $X'''_{\omega,k}$ has measure at least $1-2\delta^j$.

Then take 
$$K_1 = \left(\bigcup_{(\omega,k)\in G} X_{\omega,k}^* \times \{\omega\} \times \{k\} \right) \subset K'.$$
By Lemma \ref{fubini2}, we have that this has measure $$\hat{m}(K_1)> (1-2\delta^j)(1-\frac{\delta^i/2}{(\delta^j-\delta^i/2)})>1-\delta/4.$$

 We can do something very similar for the unstable supports. Let $$\calB_{\ell}[x] \defeq U^+[x,\omega]\cap \eta_{\omega,\ell}^{u}(x),$$ as in Section 6 of \cite{EM}. Additionally, define $|\cdot |$ as the conditional measure on $\calB_0[x]$.
 \begin{lemma}[\cite{EM}, Lemma 6.3]\label{6.3}
     Let $\delta>0$, for a.e. $\omega\in \Omega$, given $(X_{\omega}')^+\subset X_{\omega}$ with $\nu(K)>1-\delta^2$, there exists a subset $(X_{\omega}^*)^+\subset (X_{\omega}')^+$ with $\nu(K^*)>1-\delta$ such that for any $x\in (X_{\omega}^*)^+$ and $\ell>0$ we have
     \begin{align}\label{u+holds}
         |(X_{\omega}^*)^+\cap \calB_{\ell}[x]| \geq (1-\delta)|B_{\ell}[x]|.
     \end{align}
 \end{lemma}
We do as before, let $(\Omega')^+$ be the conull subset of $\Omega$ where Lemma \ref{6.3} holds. 

Let $$V_{x,\omega,k}^+ = \{(\omega')^+ \in \Omega^+ : (x,\omega^-,(\omega^+)',k)\in K'  \}, $$
then set
$$(X''_{\omega,k})^+ = \{x\in X : \mu^{\bbN}(V_{x,\omega,k}^+)>1-\delta^j\}.$$

Lemma \ref{41} tells us that there is a subset $(W')^+\subset \Omega\times [0,1[$ of measure at least $1-\delta^{i-2j}/2$ such that for $(\omega,k)\in W'$, $(X_{\omega,k}'')^+$ has measure at least $1-\delta^j$.

Let $W^+ = (W')^+\cap ((\Omega')^+\times [0,1[)$

For $(\omega,k)\in W^+$ let
$$(X'_{\omega,k})^+ = \{x\in X : (x,\omega,k)\in K'\} \subset X_{\omega}$$
$$G^+ = \{(\omega,k)\in W : \nu((X'_{\omega,k})^+) >1-\delta^j\}.$$
Lemma \ref{42} gives us that $\mu^{\bbZ}\times dt(G^+)>1-\frac{\delta^i/2}{(\delta^j -\delta^i/2)}$.

Now let $$(X''')_{\omega,k}^+ = (X'_{\omega,k})^+\cap (X''_{\omega,k})^+,$$ use this in Lemma \ref{6.3} to get a subset $(X_{\omega,k}^*)^+$ where Equation \eqref{u+holds} holds. For $(\omega,k)\in G^+$ we know that $(X'_{\omega,k})^+ $ has measure at least $1-\delta^j$ and that $(X''_{\omega,k})^+$ has measure at least $1-\delta^j$ so $(X'''_{\omega,k})^+$ has measure at least $1-2\delta^j$.

Take 
$$K_2 = \left(\bigcup_{(\omega,k)\in G^+} (X^*_{\omega,k})^+ \times \{\omega\} \times \{k\} \right)\subset K'.$$

By Lemma \ref{fubini2} we know that the measure of $K_2$ is $$\hat{m}(K_2) > (1-2\delta^j)(1-\frac{\delta^i/2}{(\delta^j - \delta^i/2)})>1-\delta/4.$$

Let $K = K_1\cap K_2$, $\nu(K)>1-\delta$.

Now take $\hat{x}_{1/2} = (x_{1/2},\omega_{1/2},k_{1/2})\in K$, here 
\begin{align*}\hat{U}^+[\hat{x}_{1/2}] = U^+[x,\omega,k] \times \{\omega^-_{1/2}\} \times \Omega^+ \times \{k_{1/2}\}, \\ \eta^u_{\ell/2}[\hat{x}_{1/2}] = \eta^u_{\omega^-_{1/2},\ell/2}(x_{1/2}) \times \{\omega^-_{1/2}\}\times \Omega^+ \times \{k_{1/2}\}.
\end{align*}
Let

$$Y' \defeq \bigcup_{y\in (X'''_{\omega_{1/2},k_{1/2}})^+} \{y\} \times V_{y,\omega_{1/2},k_{1/2}}^+.$$
By Lemma \ref{fubini2}, the measure of $Y'$ is at least $(1-\delta^j)(1-\sqrt{2\delta^j})>(1-\delta^k)$

Let
$$S \defeq (Y'\cap ((\eta^u_{\omega^-_{1/2},\ell/2}(x_{1/2}) \cap U^+[x_{1/2},\omega_{1/2},k_{1/2}]) \times \Omega^+)))\times \{\omega^-_{1/2}\} \times \{k_{1/2}\}.$$

Note again that by Lemma \ref{6.3}, since $x_{1/2}\in (X^*_{\omega_{1/2},k_{1/2}})^+$,
\begin{align*}
        m_{x_{1/2},\omega_{1/2}}^{u}((X_{\omega_{1/2},k_{1/2}}''')^+\cap \eta_{\omega_{1/2},\ell/2}^{u}(x_{1/2})\cap U^+[x_{1/2},\omega_{1/2},k_{1/2}])  &\geq \\ (1-{\delta^{j/2}})m_{x_{1/2},\omega_{1/2}}^{u}(\eta_{\omega_{1/2},\ell/2}^{u}(x_{1/2})\cap U^+[x_{1/2},\omega_{1/2},k_{1/2}]).
    \end{align*}

And so, altogether,
$$\hat{m}_{\hat{x}_{1/2}}^u(S)\geq (1-\delta)\hat{m}_{\hat{x}_{1/2}}^u 
(\eta^u_{\ell/2}[\hat{x}_{1/2}]\cap \hat{U}^+[x_{1/2},\omega_{1/2},k_{1/2}])$$

Additionally, $S\subset K'$ by construction because of how $V_{y,\omega_{1/2},k_{1/2}}^+$ is defined.

Now let $\hat{y}_{1/2}\in S\subset K'$, it is of the form $\hat{y}_{1/2}=(y_{1/2}, \omega_{1/2}^-,(\omega')^+,k_{1/2})$.

Let $$Y'' = \bigcup_{y\in X'''_{\omega_{1/2},k_{1/2}}} \{y\} \times V_{y,\omega_{1/2},k_{1/2}}.$$

Again, by Lemma \ref{fubini2} the measure of $Y''$ is at least $(1-\delta^j)(1-\sqrt{2\delta^j})>(1-\delta^k)$.

Let 

$$\tilde{S}' = (Y'' \cap (\eta^s_{\omega^-_{1/2},k_{1/2}}(x_{1/2})\times \Omega^-)) \times \{\omega^+_{1/2}\} \times \{k_{1/2}\}.$$

Then as before 
$$\hat{m}_{\hat{x}_{1/2}}^s(\tilde{S}') \geq (1-\delta)\hat{m}^s_{\hat{x}_{1/2}}(\eta^s_{\ell/2}[\hat{x}_{1/2}]).$$

\begin{lemma} \label{elllarge}
    For $\ell$ large enough, we can find $S'\subset \tilde{S}'$ such that $\hat{m}_{\hat{x}_{1/2}}^s(S') \geq (1-\delta)\hat{m}^s_{\hat{x}_{1/2}}(\eta^s_{\ell/2}[\hat{x}_{1/2}])$ and that for $\hat{z}_{1/2}=(z_{1/2},(\omega')^-,\omega^+_{1/2},k_{1/2}) \in S'$, $U^+[\hat{z}_{1/2}]$ does not intersect $W^s(\hat{y}_{1/2}).$ 
\end{lemma}
We will prove this in a series of approximation steps. First, the linear approximation.

We take $\ell_0$ large enough so that for all $\ell\geq \ell_0$, $\eta^{s/u}_{\omega_{1/2},\ell/2}({x}_{1/2})$ is in the ball of radius $\rho_0/4$ around $x_{1/2}$ and around $y_{1/2}$. Then it lives in the domain of $\exp^{-1}_{x_{1/2}}$ and $\exp^{-1}_{y_{1/2}}$. Let $P = E^s(\hat{x}_{1/2})$, $Q$ be the tangent plane at $\exp^{-1}_{x_{1/2}}(y_{1/2})$ of $\exp^{-1}_{x_{1/2}}(W^s(y_{1/2},(\omega')^+))$. Generically (in the choice of $(\omega')^+$), $P\cap Q = \{p\}$, we take $\ell$ large enough so that $p$ is guaranteed to be in $P\cap \exp_{x_{1/2}}^{-1}(B(x_{1/2},\rho_0/4))$. Fix a direction vector $v\in \bbR^4$, or really in $\exp_{x_{1/2}}^{-1}(B(x_{1/2},\rho_0/4))$. For any $z\in P\cap \exp_{x_{1/2}}^{-1}(B(x_{1/2},\rho_0/4))$, draw a line through $z$ in direction $v$ denoted $L_{z,v}$. These lines $L_{z,v}$ are to be thought of as $U^+[z,(\omega'')^-]$, which are 1-dimensional for most points. We ask what subset of $z\in P$ are such that $L_{z,v}$ does not intersect $Q$.

\begin{lemma}\label{firstap}

For generically chosen $v$, the set of $z\in P\cap \exp_{x_{1/2}}^{-1}(B(x_{1/2},\rho_0/4))$ such that $L_{z,v}\cap Q\neq \emptyset$ is a line in $P\cap \exp_{x_{1/2}}^{-1}(B(x_{1/2},\rho_0/4))$.

\end{lemma}

\begin{proof}
     Note that $P\cap \exp_{x_{1/2}}^{-1}(B(x_{1/2},\rho_0/4))$ is a 2-dimensional ball and that here $\dim(\calL^s)=2$ a.e. including at the chosen point $\hat{x}_{1/2}$ by construction (which is sent to the origin under $\exp^{-1}_{x_{1/2}}$). We will first change coordinates.

    WLOG let $Q = \text{span}\{(0,0,1,0), (0,0,0,1)  \}$, and assume $P\cap Q=(0,0,0,0)$, and the direction vector $v =( \ell,m,n,p)$. Any $z\in P\cap \exp_{x_{1/2}}^{-1}(B(x_{1/2},\rho_0/4))$ is of the form $z = (s,t,u,v)$, if we take the line $L_{z,v}$ in direction $v$ thru $z$, we ask which values of $z$ will this line $L_{z,v}$ will intersect $Q$. Clearly, these $z$ will be such that both equations $s+\ell T = 0$ and $t+mT=0$ hold for some $T$. This amount to the condition $t=\frac{m}{\ell}s$ on the plane $P$ inside the ball. This is a line through the origin in the ball. 

    When we shift back to the original coordinates, we see that our solution set is a line through $p$ inside $P\cap \exp_{x_{1/2}}^{-1}(B(x_{1/2},\rho_0/4))$.
\end{proof}

\begin{proof}(of Lemma \ref{elllarge})

We were given a $\delta$ at the beginning of this subsection and by construction $K_{\epsilon} \supset \pi_X(\tilde{S}')$ (see Lemma \ref{Kdelta'} for the definition of $K_{\epsilon}$). 

We have that for our chosen $X$-component from an element of $\tilde{S}'$, outside of any ball of size we will call $M_{\delta}$ in $Gr(1,\bbC)$ (where the stable and unstable Oseledets' directions live),  we have measure at least $1-\delta/8$. Call the corresponding angle defining this ball $\alpha_{M_{\delta}}$.

Fix $\eta>0$ small. Take $r$ small such that $\sin(\alpha_{M_{\delta}})>r^{1-\eta}$.

We work in $T_{x_{1/2}}X$. Let $P$ and $Q$ be as above, we work inside a ball of radius $r$ based at the origin (which is $\exp_{x_{1/2}}^{-1}$). This radius $r$ is the largest ball around the origin inside $B(x_{1/2},\rho_0)$ such that $$B^s(\hat{x}_{1/2},r)\subset \eta^s_{\omega_{1/2},\ell/2}(x_{1/2}),  \ \ B^u(\hat{x}_{1/2},r)\subset \eta^u_{\omega_{1/2},\ell/2}(x_{1/2}).$$

Inside this ball of radius $r$, we know by Claim \ref{error} that $P$ and $Q$ are within of $\epsilon$ of $\exp^{-1}_{x_{1/2}}(W^s(\hat{x}_{1/2}))$ and $\exp^{-1}_{x_{1/2}}(W^s(y_{1/2},(\omega')^+))$ respectively. Here $\epsilon = r^2/L$ and $L$ is the maximum between the Lipschitz constants of $Dh_1$ and $Dh_2$, where $h_1$ is the function tracing out $\exp^{-1}_{x_{1/2}}(W^s(\hat{x}_{1/2}))$ and $h_2$ is the function tracing out $\exp^{-1}_{x_{1/2}}(W^s(y_{1/2},(\omega')^+))$.

Fix $\hat{z} = (z, (\omega'')^-,\omega_{1/2},k_{1/2})\in  \tilde{S}' $, consider $\exp^{-1}_{x_{1/2}}(U^+[\hat{z}])$. Let $v$ be the direction vector of the tangent line through $\exp_{x_{1/2}}^{-1}(z)$ of $\exp^{-1}_{x_{1/2}}(U^+[\hat{z}])$ for this fixed $\hat{z}$. Let $L_{z,v}$ be the line through $z$ in the direction $v$. Keep $v$ fixed and let $z$ vary in  $\exp^{-1}_{x_{1/2}}(W^s(\hat{x}_{1/2}))$, i.e. points $z'$ such that $(z',(\omega'')^-,\omega_{1/2},k_{1/2})$ is still in $S'$. We draw lines $L_{z',v}$ through these $z'$ in the direction $v$.

Let $\alpha_v$ be the angle $v$ makes with $P$, if $L_{z',v}$ hits an element of $P$, it is within $\epsilon /\sin(\alpha_v)$ of $\exp^{-1}_{x_{1/2}}(W^s(\hat{x}_{1/2}))$. Let $\ell_v$ be the line one gets applying Lemma \ref{firstap} to $P$, $Q$ and the collection of lines $L_{z',v}$. Then the set of bad points in $\exp^{-1}_{x_{1/2}}(W^s(\hat{x}_{1/2}))$ contains those points within $\epsilon /\sin(\alpha)$ of $\ell_v$. It suffices to have the size of the neighbourhood of potential bad points in $\exp^{-1}_{x_{1/2}}(W^s(\hat{x}_{1/2}))$ be at worse on the order of $r^{1+\eta}$. Thankfully, we can ask that $v$ always be such that $\sin(\alpha_v)>r^{1-\eta}$ (then $\epsilon/\sin(\alpha_v) \sim r^2/r^{1-\eta} = r^{1+\eta}$) because we pick $v$ outside of the angle $\alpha_M$ and we decided at the start that $\sin(\alpha_M)>r^{1-\eta}$. So we get a tubular neighbourhood inside $\exp^{-1}_{x_{1/2}}(W^s(\hat{x}_{1/2}))$ of error on the order of $r^{1+\eta}$ that we must get rid of. In throwing out this neighbourhood, we throw out choices of $(\omega'')^-$, but we do not lose more than some $\delta^2/2$ $\mu^{\bbN}$-measure worth. We can take $r$ small enough so that for these $v$ and for every $z'\in \exp^{-1}_{x_{1/2}}(W^s(\hat{x}_{1/2}))$, $L_{z',v}$ hits only one element of $P$.

Now $Q$ is within $\epsilon$ of $\exp^{-1}_{x_{1/2}}(W^s(y_{1/2},(\omega')^+))$, we deal with the cases where it could hit $Q$ but not $\exp^{-1}_{x_{1/2}}(W^s(y_{1/2},(\omega')^+))$. To answer that question for a given $v$ direction vector, we ask the following question: if we base the line at a point in $Q$ in direction $v$, will it hit $\exp^{-1}_{x_{1/2}}(W^s(y_{1/2},(\omega')^+))$? This is again a matter of taking $r$ small enough so that this is true for those $v$ we specified.

One now has to take into account that $v$ should vary as we change $z'$ because when we switch the $X$-component to $z'$ we should use the tangent vector $v'$ for $U^+[z',(\omega'')^-,\omega_{1/2},k_{1/2}]$. Then we examine $L_{z',v'}$ at the point $z'$ instead of $L_{z',v}$.

Here we use that $\hat{x}_{1/2}\in K_{\beta}$ (see Section \ref{stableholon}) so that along `good' points in $W^s(\hat{x}_{1/2})$ we have that $P^-_{\omega'}(x_{1/2},\cdot )$ is H\"older in the sense that for good $\omega'$ and for $z,z'\in\tilde{S}'$ good, we have that $$\Vert P^-_{\omega'}(z,z')-Id\Vert\leq C_{\beta} d_X(z,z')^{\beta}.$$ Here $P^-_{\omega'}(z,z')$ is such that $TU^+[z',\omega'] = P^-_{\omega'}(z,z') TU^+[z,\omega']$. Since all points we consider are within $r$, we can say that for good points $z,z'$
$$\Vert P^-(z,z')-Id\Vert\leq C_{\beta} r^{\beta}.$$
Note that $r$ are can be taken to be $r^{\beta}\ll 1$ by making $\ell$ large.
Then when one compares $v$ (a unit direction vector) to $v' = P^-(z,z')v$ one sees that 
\begin{align*}|\langle P^-(z,z')v,v\rangle - \langle v,v\rangle| &= |\langle (P^-(z,z') - Id) v ,v \rangle | = |v^T(P^-(z,z')-Id)v| \\ &\leq \Vert P^-(z,z')-Id\Vert \cdot \Vert v\Vert^2 \leq C_{\beta}r^{\beta}\end{align*}
hence,
$$\langle P^-(z,z')v,v\rangle \leq C_{\beta}r^{\beta} + 1.$$
Further,
$$\Vert (P^-(z,z')-Id)v\Vert \leq C_{\beta} r^{\beta}$$
and hence since $\cos(\angle (P^-(z,z')v,v)) = \frac{\langle P^-(z,z')v,v\rangle}{\Vert P^-(z,z')v\Vert}$  we get that 
$$1\geq \cos(\angle (P^-(z,z')v,v)) \geq \frac{1-C_{\beta}r^{\beta}}{1+C_{\beta}r^{\beta}} \geq 1-2C_{\beta}r^{\beta}.$$

For each $v'$ we apply Lemma \ref{firstap} to $P$, $Q$ and the collection of $L_{z,v'}$ and get a (possibly degenerate) line $\ell_{v'}$ in $P$. Now we want to argue that in $P$, $\ell_{v'}$ is close to $\ell_v$.

In the proof of Lemma \ref{firstap} we picked coordinates and let $$Q = \Span\{(0,0,1,0), (0,0,0,1) \}.$$ We saw that if $v = (\ell,m,n,p)$ is the unit direction vector, then $$\ell_v = \{(s,t,u,w)\in P : t = \frac{m}{\ell}s\}.$$ Now if $\gamma = \angle(v,v')$, we know that $\cos(\gamma) \geq 1-2C_{\beta}r^{\beta}$ so we have a upper bound on the angle $\gamma$. 

\begin{claim}\label{cochange}
    If $v = (\ell,m,n,p)$ is the initial unit direction vector, and we rotate it by $\gamma$, the most any coordinate (i.e. $\ell,$ $m$, $n$ or $p$) can change is $2\sqrt{C_{\beta}}r^{\beta/2}$.
\end{claim}
\begin{proof}
If one changes coordinates and WLOG lets $v = (1,0,0,0)$ and we rotate by angle $\gamma$ in the $xy$-plane, we get that the first coordinate is $\cos(\gamma)$, and hence changes by at most $2C_{\beta}r^{\beta}$ and the second coordinate is $\sin(\gamma) \leq 2\sqrt{C_{\beta}}r^{\beta/2}$.
\end{proof}

We then get the following claim:
\begin{claim}
    The line $\ell_v$ can deviate from $\ell_{v'}$ by at most $Cr^{1+\alpha}$ in the ball of radius $r$.
\end{claim}

So we have that we must eliminate choices of $z\in W^s(\hat{x}_{1/2})$ in the tubular neighbourhood around a curve of diameter $Cr^{1+\alpha}$ in the ball of radius $r$ inside $W^s(\hat{x}_{1/2})$ and then the remaining $S'$ must have $\hat{m}^s_{x,\omega}$-measure $1-\delta$. We call this tubular neighbourhood $T$.

Now note that Definition \ref{QNIdef} of QNI requires $\ell>\ell_0$ to be such that 
$$F^{\ell/2}(\hat{x}_{1/2}), F^{-\ell/2}(\hat{x}_{1/2})\in K.$$

\begin{claim}\label{lastclaim}
For $\ell$ large enough, the tubular neighbourhood $T$ inside $B^s(\hat{x}_{1/2})$ has $m^s_{x,\omega}$-measure at most $\delta/4$. 
  
\end{claim}
\begin{proof}
The proof is similar to that of Lemma \ref{annbd}.
\end{proof}
Claim \ref{lastclaim} concludes the proof of Lemma \ref{elllarge}.
\end{proof}

This concludes the proof of Lemma \ref{0.5}.
\end{proof}

\section{Proof of Proposition \ref{full}}
\label{sec:Sect3}
In order to finish the case where $d_+=1$ and get alternative 2 of Theorem \ref{bigthmeasy}, we must now turn our attention to the proof of Proposition \ref{full}.

\subsection{Outline of the proof of Proposition \ref{full}}
\label{sec:outline1}
Recall from Section \ref{sec:breakdown} that the subscript $loc$ in $U^{\pm}_{loc}(x,\omega)$, $W^{s/u}_{loc}(x,\omega)$ means that they are of size $q(x,\omega)$ as defined in Section \ref{sec:stablemfld}.

In order to prove Proposition \ref{full}, we start by reducing it to something slightly weaker, the following proposition:
\begin{proposition}
    For $m$-a.e. $(\hx,\ho)\in Y$, there exists a neighbourhood $U$ of $\hx$ in $X$ such that for $\nu$-a.e. $\hy\in U$ we have that $U^+_{loc}[\hy,\ho]\cap W^s_{loc}(\hx,\ho)\neq \emptyset$. \label{U-}
\end{proposition}

Now we work by contradiction to prove Proposition \ref{U-}. We assume there is a $m$-positive measure subset of $Y$, which we will call $A$ where for any $(\hx,\ho)\in A$ and for any neighbourhood of $\hx$ in $X$, there is a $\nu$-positive measure subset of points $\hy$ such that $U^+_{loc}[\hy,\ho]\cap W^s_{loc}(\hx,\ho)= \emptyset$. We say that the pairs $(\hx,\ho)$ and $(\hy,\ho)$ satisfy the non-jointly integrable (NJI) condition.

The proof of Proposition \ref{U-} is based on the work of \cite{BQ} and \cite{EM}, and uses results and techniques from \cite{BEF}. By assuming that we have this set of positive measure such that NJI holds we will contradict what we demonstrated in Lemma \ref{0.6}, that each a.e. stable and unstable leaf has exactly 1 (smooth) real curve as the support of the conditional measure. 

We will start by finding a pair of points in $X\times \Omega$ that satisfy the NJI condition and that have a $[0,1[$-component that allows both points to live in some `good' subset of $Z$. This good subset will have properties that allow us to complete other aspects of the argument. The properties of all the `good' sets will be described in Section \ref{sec:lusin}.

Next, we will adjust these starting points by flowing by some very large time $T$. The purpose will be to get to points where the unstable planes are close, but note that the unstable supports are distinct (because of the NJI condition). We will follow the image of these unstable supports throughout the argument. The closeness of these unstable planes will be used in a future part of the argument. An argument that uses Birkhoff's theorem, Lemma \ref{birkhoff}, will be used to ensure these points are also in a `good' set that will allow other parts of the argument to go through. Once we have made these adjustments, we choose a new $\Omega$-component for these points. Note that when we choose this new $\Omega$-component, we will do so in a way that ensures we are still in a (different) `good' set. 

We will then define a `V' shape for each choice of a parameter $m \in \bbR$, see Figure \ref{fig:points}; the legs of the `V' are defined by flowing (both forwards and backwards) from our first pair of points (after flowing by $T$) to other pairs of points; there are 8 points (4 pairs of points) in total in our `V'. From our first pair of points (after flowing by $T$), we go backwards by $m$ under the `standard' flow and (after picking a new future) forwards by $\ell(m)$ under the `time change' flow. We also flow forwards from the first pair of points (after flowing by $T$) by a quantity that allows the legs of the `V' to be the same length w.r.t. the `time changed' flow. Let `V1' and `V2' denote the two components of our `V' (one for each point in our pair of fixed points); both are `V' shapes in their own right.
Note that if one switches $\ell(m)$ under the `time changed' flow to a quantity in the standard flow, this amount will differ between the two components because of the point dependency of the roof function $\tau$ of the time change flow (see Section \ref{sec:flow}). This is why within one component, the quantity that allows the legs of the `V' to be the same ($\ell(m)$ under the time change flow) will have to differ between `V1' and `V2'. We call this quantity $d_1(m)$ for `V1' and $d_2(m)$ for `V2'.

We will have to argue that $d_1$ and $d_2$ are Bi-Lipschitz in order to use a Birkhoff's theorem argument to choose $m$ such that flowing backwards from our fixed pair of points by $m$ we land in a good set and flowing forwards from our fixed points by $d_1(m)$ and $d_2(m)$ resp. lands us in a different good set. The good set that we land in flowing forwards by $d_1(m)$ and $d_2(m)$ will allow us to use a lemma, denoted by \cite{BQ} as the `Law of Last Jump' whereby we can use the Martingale convergence theorem to land our last pair of points, the ones achieved from flowing by $\ell(m)$ under the time change flow, in a good set.

Meanwhile, the good set we have landed in after flowing backwards by $m$ allows us to keep track of the distance between the image of the supports that we will compute in Lyapunov charts; this allows us to setup a situation where we pass to normal form coordinates. Normal form coordinates give us exacting control over this when flowing by $\ell$ once we project onto the unstable manifold of one of the points. The parameter we flow forwards by in normal form coordinates, $\ell(m)$, is computed as the distance we need to flow in order to separate the image of the supports to a distance greater than a fixed $\epsilon_f>0$. In the end, we will use a limiting argument to get the supports in the same plane; this plane is achieved by sending $\ell$ to infinity, shrinking the stable components of the distance between the image of the supports. This results in having two distinct support lines in one unstable manifold. It turns out that this suffices, we will argue this using the sets $K_{00}$ and $\calW$ defined in Section \ref{sec:lusin}.  

  The way we send $\ell$ to infinity is by shrinking the distance between the initially chosen points (this will be explained in Section \ref{sec:original}). The initial points will always be chosen in a ball of radius $\rho$ and we will send $\rho$ to 0. This will also force $T$ and $m$ to grow along with $\ell$.

    The parameter $T$ will be chosen such that the unstable supports are very close to being $\epsilon_f$-apart. Also the parameter $T$ will be very big compared to $m$ and $\ell$. The parameter $m$ will small compared to $\ell$. Flowing by $m$ ultimately undoes some of the good positioning we have achieved when flowing by $T$, so $T$ has to grow faster than $m$. Still, $m$ needs to grow because we need to be able to use the Birkhoff theorem argument to land the right endpoints of the `V' in the good set. This will be explained later.

\begin{figure}[H]
    \centering
    \includegraphics[width=0.95\textwidth]{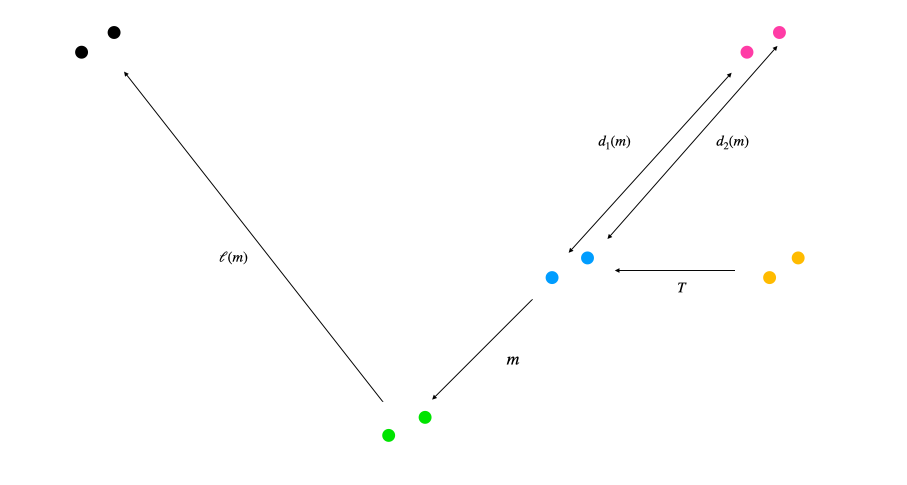}
    \caption{This figure shows the basic configuration of the points used in the proof of Proposition \ref{U-}. The magnitude of time that we flow between the points is labelled as well.}
    \label{fig:points}
\end{figure}

\subsection{Reduction of Proposition \ref{full} to Proposition \ref{U-}}

Assuming one can prove Proposition \ref{U-}, the same argument (run backwards) proves the following:

\begin{proposition} \label{other}
For $m$-a.e. $(\hx,\ho)\in Y$, there exists a neighbourhood $U$ of $\hx$ in $X$ such that for $\nu$-a.e. $\hy\in U$ we have that $U^-_{loc}[\hy,\ho]\cap W^u_{loc}(\hx,\ho)\neq \emptyset$.   
\end{proposition}

\begin{remark}\label{another}
    Proposition \ref{U-} can be `flipped', i.e. the statement may also say $U^+_{loc}[\hx,\ho]\cap W^s_{loc}(\hy,\ho) \neq \emptyset$ with everything else unchanged. Remarks \ref{flip1} and \ref{flip2} in Section \ref{sec:blue} where NJI is used explain why. This is the only place where the NJI condition is used. We can also `flip' Proposition \ref{other} to say $U^-_{loc}[\hx,\ho]\cap W^u_{loc}(\hy,\ho)\neq \emptyset$.
\end{remark}

With Propositions \ref{U-}, \ref{other} and Remark \ref{another} in hand, we get Proposition~\ref{full}:

\begin{proof}(Proposition \ref{full})

    Proposition \ref{U-} gives us that we have a full $m$-measure subset $A_1\subset Y$ such that for $(\hx,\ho)\in A_1$ there exists a neighbourhood $U_1$ of $\hx$ in $X$ such that for $\nu$-a.e. $\hy\in U_1$ we have $$U^+_{loc}[\hy,\ho]\cap W^s_{loc}(\hx,\ho) = \{p\} \neq \emptyset.$$ Also (using Proposition \ref{other} and Remark \ref{another}) there is a full measure subset $A_2\subset Y$ where for $(\hx,\ho)\in A_2$ there exists a neighbourhood $U_2$ of $\hx$ in $X$ such that for $\nu$-a.e. $\hy\in U_2$ we have $$U^-_{loc}[\hx,\ho]\cap W^u_{loc}(\hy,\ho) = \{q\} \neq \emptyset.$$

    Note that these intersections must be only a point, otherwise the points would have two different behaviours, being in the unstable forces the points apart and being in the stable forces them together. Then on $A_1\cap A_2$ (which is still full measure) letting $U = U_1\cap U_2$ we get both results.This implies that $p=q$ because $U^-_{loc}[\hx,\ho] \subset W^s_{loc}(\hx,\ho)$ and $U^+_{loc}[\hy,\ho] \subset W^u_{loc}(\hy,\ho)$. Then we have that $$U^-_{loc}[\hx,\ho] \cap U^+_{loc}[\hy,\ho] = \{p\} \neq \emptyset.$$

     The complementary statement, that $$U^-_{loc}[\hy,\ho] \cap U^+_{loc}[\hx,\ho] \neq \emptyset,$$ is proven similarly.
\end{proof}

\subsection{Proof of Proposition \ref{U-}}

\label{proofstart}
As explained above, we work by contradiction, so we assume there exists a set of positive $m$-measure $A\subset Y$, $m(A)>\epsilon_0$, such that if $(\hx,\ho)\in A$ then any neighbourhood $U$ of $\hx$ has a positive measure subset $ U$ of points $\hy$ such that $U^+_{loc}[\hy,\ho]\cap W^s_{loc}(\hx,\ho)=\emptyset$. If two points $(\hx,\ho), \ (\hy,\ho) \in Y$ are such that $U^+_{loc}[\hy,\ho]\cap W^s_{loc}(\hx,\ho)=\emptyset$ we say that they satisfy the non-jointly integrable condition, abbreviated as the NJI condition.

\subsubsection{The `good' sets for the proof of Proposition \ref{U-}}\label{sec:lusin}

Now let us list the sets we will need for the proof of Proposition \ref{U-} and their properties. We also discuss their size and how they will be used. First let $L \defeq \max\{ 2\ln(e^{-(\lambda^--\epsilon)}+\frac{1}{2}C) +\lambda^++\epsilon, \frac{1}{\lambda^++\epsilon} \}.$ This quantity will be derived in Section \ref{sec:pink} in the proof of Lemma \ref{BiLip}, see Equation \eqref{Lipconst}. Let $\calL$ be upper bound on the Lipschitz constants of the automorphisms in $\supp(\mu)$ and let $\calK>0$ be such that $\frac{1}{\calL} = e^{-\calK}$. Fix $\tilde{\epsilon}>0$ small. Let $K_{3\tilde{\epsilon}}$ be the set defined in Section \ref{sec:Hcont} for $\epsilon = 3\tilde{\epsilon}$. This Lusin set has size $1-\epsilon_{hol}$ for some $\epsilon_{hol}$ small. We can ask $\epsilon_{hol}$ to be as small as needed and get a corresponding $\tilde{N}$. 

Also, let $C_b = \frac{(\lambda^+-\delta)}{2\ln(\calL)}$ and fix $i$ such that $C_b>\frac{1}{10^i}$, $\frac{C_b}{10^{i+2}}\ll(\lambda^++\delta)$ so that
\begin{align}
    \frac{\lambda^++\delta}{\frac{C_b}{10^{i+2}}}\gg \ln(k\epsilon_fM_r/2) + (\lambda^++\epsilon),
\end{align}
and $\frac{\calL^{\frac{C_b}{10^{i+2}}}}{e^{(\lambda^++\delta)}}<1$ (see Equations \eqref{lm}, \eqref{206} and \eqref{186} respectively).

Let $J$ be chosen such that $|\frac{C_b}{10^{i+J}}(\lambda^--\epsilon)| < |\frac{C_b}{10^{i+2}}(\lambda^-+\epsilon)| $ (needed for Equation \eqref{eta11}). Let $\alpha_1 = \frac{C_b}{10^{i+J}}$ and $\alpha_2 = \frac{C_b}{10^{i+2}}$, these choices will be used in the definition of the parameter $m$ but we can define them here which is important to the argument of Lemma \ref{dd}. Fix some $\varepsilon'>0$ small such that the expression
$$\calT = \varepsilon' - \frac{(\lambda^--\epsilon)\alpha_1}{-(\lambda^-+\epsilon)\alpha_2-\kappa} + \frac{(\lambda^-+\epsilon)\alpha_2}{-(\lambda^-+\epsilon)\alpha_2-\kappa},$$
is positive. Then fix $\eta<\calT$ and let $\calS = \left(1-\frac{1}{1+\eta}\right)$. We use these choices to define sets we need to use Lemma \ref{birkhoff} in Lemma \ref{dd}, one can see this choices made in the definitions of $\tilde{K}_2$, $E^c_2$, $K_3$ and $E^c_3$. This must be fixed from the beginning.
 
All the $\epsilon$ terms used in the definitions below will be assumed sufficiently small for their purpose unless otherwise stated more explicitly. Sometimes the size of these sets, stated here, will be proven later.

\begin{itemize}   
    \item $\calW \subset Z$, $$\calW = \{ (x,\omega,k)\in Z : \calL^u(x,\omega,k)\supsetneq U^+[x,\omega,k] \},$$ by Lemma \ref{0.6} we know that this set has measure zero.
    
    \item $K_{00} \subset Z$, arbitrary compact subset of size $\hat{m}(K_{00})> 1-\epsilon_{00}.$ We assume that $K_{00}\subset \calW^c$, we can do this since $\hat{m}(\calW^c)=1$.

    \item $K_0' \subset Y$, $K_0' \subset \Lambda \cap \Lambda_{loc} \cap K_s\cap K_r \cap K_{NFC} \cap K_{\xi}$, $m(K_0')>1-\frac{\epsilon_0}{2}$.

    \item $K_0 \subset Y,$ $K_0 \subset \Lambda \cap \Lambda_{loc} \cap K_s\cap K_r \cap K_{NFC} \cap K_{\xi}$, $m(K_0)\geq m(K_0')$.

    \item $K_1'\subset Y$, $K_1' \subset \Lambda \cap \Lambda_{loc} \cap K_s\cap K_r \cap K_{NFC} \cap K_{\xi}\cap K_{\xi}'\cap K_{exp} \cap\tilde{\Lambda}\cap K_{3\tilde{\epsilon}}\cap K_G\cap K_{\varepsilon}$, $m(K_1') > 1-\epsilon_1'$.

    \item $K_1 \subset Y$, $K_1 \subset \Lambda \cap \Lambda_{loc} \cap K_s\cap K_r\cap (K_{exp}\times \Omega)\cap K_{\varepsilon}$, $m(K_1) > 1-\epsilon_1$.

    \item $K_2' \subset Y$, $K_2' \subset \Lambda \cap \Lambda_{loc} \cap K_s\cap K_{\xi}\cap K_b\cap  K_{ang}\cap K_{exp} \cap (K_{\varepsilon}\times \Omega)$, $m(K_2')>1-\epsilon_2'$.

    \item $K_2 \subset Y$, $K_2 \subset \Lambda \cap \Lambda_{loc} \cap \Lambda' \cap K_s\cap K_r\cap K_{\xi}\cap K_{NFC}\cap K_{new} \cap K_G\cap K_b\cap K_{exp} \cap(K_{\varepsilon}\times \Omega)$, $m(K_2)>1-\epsilon_2$.

    \item $B' \subset X\times \Omega^-$,
    \begin{align*}
        B' \defeq &\{(x,k) : \exists V_{(x,\omega^-)}\subset \Omega^+, \mu^{\bbN}(V_{(x,\omega^-)})>1-\sqrt{\epsilon},\\ &\text{ such that if }\omega^+\in V_{(x,\omega^-)},
        \text{ then }(x,\omega^-,\omega^+) \in K_2 \},
    \end{align*}
    $\nu \times \mu^{\bbN} (B') > 1-\sqrt{\epsilon_2}$.

    \item $\tilde{K_2} \subset Y$, $\tilde{K}_2 \defeq (B' \times \Omega^+)\cap K_2'$, $\epsilon_2$ and $\epsilon_2'$ small such that the corresponding measure of $\tilde{K}_2$ which we call $\gamma_2=1-\tilde{\epsilon}_2$ satisfies 
    \begin{align*}
        \gamma_2> 1-\frac{1}{200L^4} >3/4.
    \end{align*}

    \item $E_2^c \subset Z$, take $\delta =\alpha_2(\tilde{K}_2)$ such that $\beta_2\defeq 2\gamma_2(1-\alpha_2)-1>\frac{1}{200L^4}$, $\epsilon \ll\epsilon_1$, $K=\tilde{K}_2 \times [0,1[$ and apply Lemma \ref{birkhoff}, we get $S>0$ such that the set 
    \begin{align*}
        E_2^c \defeq \{(x,\omega,k) : \ \forall T>S, \  |\frac{1}{T}\int_0^T \mathds{1}_{K}(F^t(x,\omega,k))dt - \int_{Y}\mathds{1}_{K} d\hat{m} |<\delta \},
    \end{align*}
    has $\hat{m}(E_2^c)>1-\epsilon$.

    \item $K_3 \subset Z$, $K_3 \subset (\Lambda \cap \Lambda_{loc})\times [0,1[$, additionally, $K_3$ satisfies the conditions of Lemma \ref{RN} and belongs to the $\hat{m}$-conull set from Lemma \ref{LLJ}. It further satisfies the properties of the set $L'$ in Lemma \ref{48}. The size of this set is $\gamma_3\defeq 1-\epsilon_3$ where
    \begin{align*}
       \gamma_3> (1-\frac{1}{100L^2}) + \frac{1}{100L^2}\left( \frac{1}{1+\eta} \right)>3/4.
    \end{align*}

    \item $E_3^c \subset Z$, take $\delta = \alpha_3\gamma_3$ such that $\beta_3 = 2\gamma_3(1-\alpha_3)-1> (1-\frac{1}{100L^2})+\frac{1}{100L^2}\left(\frac{1}{1+\eta}\right)$, $\epsilon \ll\epsilon_1$, $K=K_3$ and apply Lemma \ref{birkhoff}, we get $S>0$ such that the set 
    \begin{align*}
        E_3^c \defeq \{(x,\omega,k) : \ \forall T>S, \  |\frac{1}{T}\int_0^T \mathds{1}_{K}(F^t(x,\omega,k))dt - \int_{Y}\mathds{1}_{K} d\hat{m} |<\delta \},
    \end{align*}
    has $\hat{m}(E_3^c)>1-\epsilon$.

    \item $M \defeq (K_1\times [0,1[) \cap E_2^c \cap E_3^c$, $\hat{m}(M)>1-\epsilon_M$.

    \item $B\subset X\times [0,1[$, fix $\theta>0$ small
    \begin{align*}
        B \defeq \{(x,k) : \exists V_{(x,k)}\subset \Omega, \mu^{\bbZ}(V_{(x,k)})>1-\theta, \text{ such that if }\omega\in V_{(x,k)}, \text{ then }(x,\omega,k) \in M \},
    \end{align*}
    $\nu \times dt (B) > 1-\frac{\epsilon_M}{\theta}$.

    \item $\tilde{K}_1\subset Z$, $\tilde{K}_1 \defeq (B\times \Omega) \cap (K_1' \times [0,1[)$, $\epsilon_M, \epsilon_1'$ are small so that $\hat{m}(\tilde{K}_1)>3/4$.

    \item $K_4 \subset Y$, $K_4\subset \Lambda\cap \Lambda_{loc}\cap K_{new} \cap K_{NFC}$.

    \item $K_4' \subset Z$, $K_4' = (K_4\times [0,1[) \cap K_{00}$.

    \item $E^c \subset Z$, take $\delta = \frac{1}{100}\hat{m}((\tilde{K}_1)\cap E_2^c\cap E_3^c)$, $\epsilon \ll m(B)$, $K = \tilde{K}_1$ and apply Lemma \ref{birkhoff}, we get $S>0$ such that the set 
    \begin{align*}
        E^c \defeq \{(x,\omega,k) : \ \forall T>S, \  |\frac{1}{T}\int_0^T \mathds{1}_{K}(F^t(x,\omega,k))dt - \int_{Y}\mathds{1}_{K} d\hat{m} |<\delta \}.
    \end{align*}
    \end{itemize}

These sets are not in the order they appear in the proof because the definition of the earlier sets often requires foresight from later in the proof.

    \subsubsection{A brief sketch on how good sets will be used}\label{sketch}
    \begin{itemize}
    \item $\calW$ - we know this set has zero measure by Lemma \ref{0.6} but we will show this is not the case when we assume that Proposition \ref{U-} is false. This will prove Proposition \ref{U-}.

    \item $K_{00}$ - We will land in this set at the end and use it to show that $\calW$ must have positive measure when we assume that Proposition \ref{U-} is false. This will prove Proposition \ref{U-}.
        
    \item $K_0'\subseteq \ K_0\subset Y$ - Our starting points, $(\hx,\ho)$ and $(\hy,\ho)$ will be chosen from here. A placeholder set, $K_0'$ will be used at the beginning of the argument in Section \ref{sec:original}, but ultimately the starting points will live in $K_0$.
    
    \item $K_1'\subset Y$ - After flowing the starting points $(\hx,\ho,k_0)$ and $(\hy,\ho,k_0)$ by some time $T$ forward under the `standard' flow, we want the points $F^T(\hx,\ho)$ and $F^T(\hy,\ho)$ to land in this set. 

    \item $K_1\subset Y$ - After flowing the starting points $(\hx,\ho,k_0)$ and $(\hy,\ho,k_0)$ by $T$ under the `standard' flow we get the points $F^T(\hx,\ho)$ and $F^T(\hy,\ho)$. We denote the $X$-components as $\tilde{x} \defeq f_{\omega}^n(\hx)$ and $\tilde{y} \defeq f_{\omega}^n(\hy)$ where $n = \lfloor T+k_0\rfloor$. We will change the $\Omega$-component of $F^T(\hx,\ho)$ and $F^T(\hy,\ho)$ to some $\tilde{\omega} \in \Omega$ so that $(\tilde{x},\tilde{\omega})$ and $(\tilde{y},\tilde{\omega})$ belong to this set $K_1$.
    
    \item $E^c_2$, $E^c_3 \subset Z$ - When we take the points $F^T(\hx,\ho,k_0)$ and $F^T(\hy,\ho,k_0)$ and change the $\Omega$-component to $\tilde{\omega}$, we need the resulting points to belong to $E^c_2$ and $E^c_3$. This allows us to run arguments in Sections \ref{sec:green} and $\ref{sec:pink}$, related to Lemma \ref{birkhoff}, to get appropriate $m$ such that $m$, $d_1(m)$ and $d_2(m)$ are good times to flow by from the points $(\tilde{x},\tilde{\omega})$ and $(\tilde{y},\tilde{\omega})$ to land in sets $\tilde{K}_2$ and $K_3$ resp. 
    
     \item $B\subset X \times [0,1[$ - After flowing  $(\hx,\ho,k_0)$ and $(\hy,\ho,k_0)$ by some time $T$ forward under the `standard' flow, we want the $X\times [0,1[$-component of the resulting points to land in this set so that we will be able to pick a new $\omega$ so that $(x,\omega,k)\in M$.

    \item $\tilde{K}_1 \subset Z$ - We land in this set after flowing by $T$ from the starting points. Note that $K_1'\times [0,1[$ is part of the definition of this set. 

    \item $K_2'\subset Y$ - After flowing $(\tilde{x},\tilde{\omega})$ and $(\tilde{y},\tilde{\omega})$ backwards by $m$ under the `standard' flow we get the points $F^{-m}(\tilde{x},\tilde{\omega})$ and $F^{-m}(\tilde{y},\tilde{\omega})$ which need to be in this set $K_2'$. 

    \item $K_2 \subset Y$ - After flowing by $-m$ under the `standard' flow we get the points $F^{-m}(\tilde{x},\tilde{\omega},k_1)$ and $F^{-m}(\tilde{y},\tilde{\omega},k_1)$. We denote the $X$-components by ${x} \defeq f_{\omega}^{n'}(\hx)$ and ${y} \defeq f_{\omega}^{n'}(\hy)$ where $n' = \lfloor k_1-m\rfloor$. We will change the $\Omega^+$-component of $F^{-m}(\tilde{x},\tilde{\omega},k_1)$ and $F^{-m}(\tilde{y},\tilde{\omega},k_1)$ to get some ${\omega'} \in \Omega$ so that $({x},{\omega'})$ and $({y},{\omega'})$ belong to the set $K_2$. 

    \item $B'\subset X \times \Omega^-$ - After flowing  $(\tilde{x},\tilde{\omega},k_1)$ and $(\tilde{y},\tilde{\omega},k_1)$ by some time $m$ backwards under the `standard' flow, we want the points $f_{\omega}^{n'}(x) = {x}$ and $f_{\omega}^{-n}(y) = {y}$ to be such that $(x,\sigma^{n'}(\tilde{\omega}^-))$ and $(y,\sigma^{n'}(\tilde{\omega}^-))$ land in this set, where $n' = \lfloor k_1- m\rfloor$ and $\omega^-$ represents the part of the vector $\omega$ that contains positions 0, -1, -2... etc. This way we can pick a new future $\omega^+\in \Omega^+$ so that $(x,(\omega^-,\omega^+))\in K_2.$

    \item $\tilde{K}_2 \subset Y$ - when we flow backwards by $m$ starting from $(\tilde{x},\tilde{\omega},k_1)$ and $(\tilde{y},\tilde{\omega},k_1) \in E_2^c$ we take $K = \tilde{K}_2\times[0,1[$ as the set in Lemma \ref{birkhoff}. The bound on the measure of this set is for an argument in Section \ref{sec:pink}, Lemma \ref{landing}.

     \item $K_3\subset Z$ - Starting from the points $(\tilde{x},\tilde{\omega},k_1)$ and $(\tilde{y},\tilde{\omega},k_1)$, we will flow forwards by $d(m,(\tilde{x},\tilde{\omega},k_1))$ and $d(m,(\tilde{y},\tilde{\omega},k_1))$ resp. under the `time change' flow. We want the resulting points to land in the set $K_3$. The lower bound on the measure of this set is for the purposes of Lemma \ref{landing} in Section \ref{sec:pink}. 

    \item $K_{4}\subset Y, K_4'\subset Z$ - From the points $(x,\omega')$ and $(y,\omega')$ we will flow forwards by $\ell$ under the `time changed' flow. We want the points $F^{\ell}(x,\omega')$ and $F^{\ell}(y,\omega')$ to land in the set $K_4$. The set $K_4'$ just extends the properties of this set to $Z$.

    \item $E^c \subset Z$ - The original points $(\hx,\ho)$ and $(\hy,\ho)$ must belong to $E^c$ for some choice of $k\in [0,1[$ and it is demonstrated that we can do this in Section~\ref{sec:original}. 
\end{itemize}

\subsubsection{Notation for Section \ref{sec:Sect3}}
\label{sec:not5}
Here we recall/define a few constants that are needed throughout the rest of this section. It is important to note that all of these constants can be defined before beginning the argument. 

Recall that we are making the assumption that $\supp(\mu)$ is a finite set. We let $\calL$ be the upper bound on the Lipschitz constant and we can assume $\calL$ is at least 3. We also let $C$ be the upper bound on the second derivative of any automorphism in $\supp(\mu)$ as in Equation \eqref{Const} (given by $C_{\tilde{f}_i}$).

Recall $\tilde{\epsilon}>0$ from Section \ref{sec:lusin} and define $\alpha = \frac{a-b}{a-d}$ where $a = -(\lambda^--3\tilde{\epsilon})$, $b = -(\lambda^++3\tilde{\epsilon})$ and $d = \ln(2(C^-)^2) + 2c_0^- + |\ln(\rho_0/4)| + |a|$ where $C^-$ and $c_0^-$ are chosen uniformly on $\Omega_{\epsilon}$ as in Lemma \ref{cCunif} and $\rho_0$ is as in Section \ref{sec:Hcont}.

Lastly, let $\epsilon_f$ be such that
\begin{align}\epsilon_f<\frac{1}{2M_r}\max
\left\{q/2, k_U/2,\rho_0/4, \frac{1}{2M_r}, \frac{2k_1e^{-(\lambda^-+\epsilon)}}{CM_r}, \frac{2k_2e^{-(\lambda^++\epsilon)}}{CM_r}\right\},\end{align} 
where $k_1,k_2<1$ are fixed such that $(1-k_1)e^{-(\lambda^-+\epsilon)}>1$, $k$ and $q$ are from the definition of the set $K_{loc}$ in Section \ref{sec:stablemfld}, and $M_r$ and $k_U$ are from Section \ref{sec:lusin1}.

\subsubsection{The starting points}\label{sec:original}

We need to pick points $(\hx,\ho), \ (\hy,\ho) \in Y$ such that $(\hx,\ho) \in A$, and both $(\hx,\ho)$ and $(\hy,\ho)$ belong to a set with the features of $K_0$ as defined in Section \ref{sec:lusin}. We will need to use an intermediate set $K_0'$ and enlarge it at the end to get the points in $K_0$; the points also need to satisfy the NJI condition with respect to each other. Our approach relies on the fact that $m$ is the product measure and Fubini's theorem; we will use the following lemma:
\begin{lemma}
\label{fubini}
    Let $(X_1,\eta_1)$ and $(X_2,\eta_2)$ be $\sigma$-finite measure spaces. Consider $V\subset X_1$, a set such that $\eta_1(V)>0$, and for each point $a\in V$ consider a set $W_a\subset X_2$ such that $\eta_2(W_a)>0$. Then $S \defeq \bigcup_{a\in V}W_a\times \{a\}$ has positive measure with respect to the $\eta_1\times \eta_2$ on $X_1\times X_2$.
\end{lemma}
\begin{proof}
    Writing $S = \{(w,a) : w\in W_a, \ a\in V\}$, we apply Fubini's theorem:
    \begin{equation*}
        \eta_1\times \eta_2(S) = \int_S d\eta_1d\eta_2 = \int_V \int_{W_a}d\eta_2d\eta_1 = \int_V \eta_2(W_a)d\eta_1.
    \end{equation*}
    Let $V_n = \{a\in V : \eta_2(W_a)>1/n\} \subset V$, then $V_n \to V$ as $n\to \infty $. $V$ has positive measure, so there exists $N>0$ such that $V_N$ has positive measure as well. Then,
    \begin{align*}
        \eta_1\times \eta_2(S) = \int_V \eta_2(W_a)d\eta_1 \geq \int_{V_N} \eta_2(W_a) d\eta_1 = \int_{V_N} \frac{1}{N}d\eta_1 = \eta_1(V_N)\cdot \frac{1}{N}>0.
    \end{align*}
    This completes the proof.
\end{proof}

Let $K_0'$ be defined as in Section \ref{sec:lusin}. Since $m(K_0')>1-\frac{\epsilon_0}{2}$ and $m(A)>\epsilon_0$, we have that $B_1\defeq A\cap K_0'\subset Y$ has positive $m$-measure. Recall the natural projection map $\pi_{\Omega} : Y\to \Omega$. Since $B_1$ has positive $m$-measure and $m$ is the product measure, we must have that $\pi_{\Omega}(B_1)$ has positive $\mu^{\bbZ}$-measure in $\Omega$; if not, and $\pi_{\Omega}(B_1)$ has zero measure, then the set $X\times \pi_{\Omega}(B_1)$ has measure zero and contains $B_1$. This is a contradiction.

For each $(z,\omega) \in B_1\subset A$, every neighbourhood of $z$ in $X$ has a positive measure subset of points $y$ in $X$ such that NJI holds. Fix some $\rho>0$, and for each $(z,\omega)\in B_1$ take a ball of radius $\rho$ around $z$ and let $C_{(z,\omega)}$ be the positive measure subset of this ball of radius $\rho$ where NJI holds. Define $\tilde{B}_1 \defeq (B_1\times [0,1[)\cap E^c\subset Z$.

\begin{lemma}\label{start}
    There exists a set $D\subset Z$ is such that if we make $K_0$ large enough, $D\cap (K_0\times [0,1[)$ has positive measure. Then if $(\hy,\ho,k_0)\in D\cap (K_0\times [0,1[)$, there is a corresponding $(\hx,\ho,k_0)\in (B_1\times [0,1[)\cap E^c$, with $(\hx,\ho)\in A$ such that $(\hx,\ho)$ has the NJI condition with $(\hy,\ho)$.
\end{lemma}
\begin{proof}
For each $(z,\omega,k) \in \tilde{B}_1$ (this set has positive $\hat{m}$-measure by the choices we made), we take $C_{(z,\omega)}\subset X$ and define $D_{(z,\omega,k)} = C_{(z,\omega)}\times \{(\omega,k)\}\subset Z$. Note that because $\tilde{B}_1$ has positive $\hat{m}$-measure, the projection of $\tilde{B}_1$ onto $\Omega \times [0,1[$ has positive measure as well. One defines $D = \bigcup_{(z,\omega,k) \in \tilde{B}_1} D_{(z,\omega,k)}$ and by Lemma \ref{fubini}, we have that $D$ has positive measure. Take $K_0$ to have all the specified properties including that it contains $K_0'$, and large enough such that $D\cap (K_0\times [0,1[)$ has positive measure. Take $(\hy,\ho,k_0)$ from $D\cap (K_0\times [0,1[)$, it comes with the corresponding $(\hx,\ho,k_0)\in (B_1\times [0,1[)\cap E^c $ (where $E^c$ is as defined in Section \ref{sec:lusin}) and naturally they have the NJI condition as needed by construction.  
\end{proof}
We use Lemma \ref{start} to get our starting points.

\begin{remark}
    The NJI condition between $(\hx,\ho)$ and $(\hy,\ho)$ gives us that $\hy$ does not belong to $W^s(\hx,\ho)$.
\end{remark}

\begin{remark}\label{close}
    It is worth noting that in the construction above, we can shrink $\rho$ as small as needed and always find points $(\hx,\ho)$ and $(\hy,\ho)$ that satisfy our conditions. This allows us to find points of arbitrary closeness, or more precisely, find $\hy$ arbitrarily close to $\hx$.
\end{remark}

\subsubsection{Flowing by $T$}
\label{sec:blue}

Using Lemma \ref{start}, we choose our starting points, $(\hx,\ho,k_0)$ and $(\hy,\ho,k_0)$ such that $d_X(\hx,\hy)\ll \rho$ where $$\rho \ll \frac{\epsilon_f}{10000}, $$ 
(see Remark \ref{close}).

These starting points belong to $K_{loc}$ by definition of the set $K_0$ (see Section \ref{sec:lusin}), this means we can define local stable and unstable manifolds for these points of size $q$, (see the definition of $q$ in Section \ref{locstablemfld}) note that $\epsilon_f< q/4$ so the scale we are working in is much smaller.

In this section we will take these points $(\hx,\ho,k_0)$ and $(\hy,\ho,k_0)$ and flow them by some very large time $T$ under this $\ho$ (using the `standard' flow) to put them in a `good' position. What we mean by this is that after flowing by $T$ we want the  image of $W^u_{q}(\hx,\ho)$ and $W^u_{q}(\hy,\ho)$ (i.e. $W^u_{q}(F^T(\hx,\ho))$ and $W^u_{q}(F^T(\hy,\ho))$ to be close (in some ball), but still want the image of the unstable supports, i.e. the image under $F^T$ of $U^+_q[\hx,\ho]$ and $U^+_q[\hy,\ho]$ (i.e. $U^+_q[F^T(\hx,\ho)]$ and $U^+_q[F^T(\hy,\ho)]$) to be bounded apart. The point of this is to have $U^+_q[F^T(\hx,\ho)]$ and $U^+_q[F^T(\hy,\ho)]$ roughly in the same complex plane. This will be used later in Section \ref{sec:green}.

 Consider $\hz \in W_q^s(\hx,\ho)\cap W^u_q(\hy,\ho)$. Let $n = \lfloor T+k_0\rfloor$ and denote $\tilde{x} = F^n_{\omega}(\hx)$, $\tilde{y} = F^n_{\omega}(\hy)$, and $\tilde{z} = F^n_{\omega}(\hz)$. We will see later in this Section that $T$ can be chosen such that $(\tilde{x},\sigma^n(\ho))$ and $(\tilde{y},\sigma^n(\ho))$ belong to the set $\tilde{\Lambda}\subset K_1'$ as defined in Section \ref{sec:Hcont}. For now we make the following standing assumption:

\begin{assumption}\label{ass1}
    The parameter $T$ can be chosen such that $(\tilde{x},\sigma^n(\ho))$ and $(\tilde{y},\sigma^n(\ho))$ belong to the set $K_1'$ as defined in Section \ref{sec:lusin}. We demonstrate that we can choose the parameter $T$ in this way in Lemma \ref{bigint}.
\end{assumption}

Take $\tilde{\epsilon}>0$ from the beginning of Section \ref{sec:blue}, we get $N_u\in \bbN$ such that $\tilde{x}, \tilde{y} \in \Lambda_{\tilde{\epsilon},N_u}^{u,\sigma^n(\ho)}$ (see Section \ref{sec:Hcont}).

\begin{lemma} \label{goodland}
   The point $\tilde{z}$ belongs to $\Lambda_{3\tilde{\epsilon},3N_u}^{u,\sigma^n(\ho)}$. 
\end{lemma}
\begin{proof}
     Since $\tilde{y}$ belongs to $\Lambda_{\tilde{\epsilon},N_u}^{u,\sigma^n(\ho)}$ and $\tilde{z} \in W^u(\tilde{y},\sigma^n(\ho))$, this follows from \cite{Ru}, Theorem 4.1; see Theorem A.1.6 of \cite{BEF}.
\end{proof}

The choices of Lusin sets we made in Section \ref{sec:lusin} will give a H\"older continuity between $E^u(\tilde{x}, \sigma^n(\ho))$ and $E^u(\tilde{z},\sigma^n(\ho))$ in the next lemma. The H\"older exponent and constant will be the same for any choice of starting points belonging to $K_0$ and any choice of $T$ that lands us in $\tilde{\Lambda}\subset K_1'$. This uniformity is required later in the argument.

\begin{lemma}\label{hun1}
    The points $\tilde{x}$ and $\tilde{z}$ belong to a set on which we have $\alpha$-H\"older continuity of $E^u(\cdot ,\sigma^n(\ho))$, where $\alpha$ is as defined in Section \ref{sec:lusin}. Additionally, the Holder constant can be taken to be $$C_{\alpha} \defeq 2 (\max_{\supp(\mu)}\{1,\Vert Df_i\Vert\})^{2\tilde{N}}e^{(\lambda^--\lambda^+-6\tilde{\epsilon})},$$ where $\tilde{N}$ is defined in Section \ref{sec:lusin}.
\end{lemma}
\begin{proof}

Since $\Lambda_{\tilde{\epsilon},N_u}^{u,\sigma^n(\ho)} \subset \Lambda_{3\tilde{\epsilon},3N_u}^{u,\sigma^n(\ho)}$ we have that $\tilde{x},\tilde{z} \in \Lambda_{3\tilde{\epsilon},3N_u}^{u,\sigma^n(\ho)}$.

We showed in Section \ref{sec:Hcont} that for $$L_u = (\max_{\supp(\mu)}\{1,\Vert Df_i\Vert\})^{N_u(3\tilde{\epsilon})},$$ we have that $$\Lambda_{3\tilde{\epsilon},3N_u}^{u,\sigma^n(\ho)}\subset \Delta_{-(\lambda^- - 3\tilde{\epsilon}),-(\lambda^++3\tilde{\epsilon}),L_u}.$$ We know that $\sigma^n(\ho) \in \Omega_0^-$ (because $(\tilde{x},\sigma^n(\ho))$ and $(\tilde{y},\sigma^n(\ho))$ belong to $\tilde{\Lambda}$). Hence, by Proposition \ref{holder} we have $\alpha$-Holder continuity of $E^u(\cdot, \sigma^n(\ho))$ (with constant $ C_{N_u(3(\tilde{\epsilon}))} = 2 L_u^2e^{(\lambda^--\lambda^+-6\tilde{\epsilon})}$) on the set $\Delta_{-(\lambda^- - 3\tilde{\epsilon}),-(\lambda^++3\tilde{\epsilon}),L_u}$ to which $\tilde{x}$ and $\tilde{z}$ belong (because $ \tilde{x} \in \Lambda_{\tilde{\epsilon},N_u}^{u,\sigma^n(\ho)} \subset \Lambda_{3\tilde{\epsilon},3N_u}^{u,\sigma^n(\ho)}$). We get that
\begin{align} \label{156}
    d_{Gr}(E^u(\tilde{x},\sigma^n(\ho)),E^u(\tilde{z},\sigma^n(\ho)))\leq C_{N_u(3(\tilde{\epsilon}))}d_X(\tilde{x},\tilde{z})^{\alpha}.
\end{align}

Since $(\tilde{x},\sigma^n(\ho)),(\tilde{y},\sigma^n(\ho))\in K_1'$ they belong to $K_{3\tilde{\epsilon}}$ and hence $N_u$ is bounded above by $\tilde{N}$. This means that regardless of our starting points, the H\"older constant in Equation \eqref{156}, $C_{N_u(3(\tilde{\epsilon}))}$ is bounded above by 
$$C_{\alpha} \defeq 2 (\max_{\supp(\mu)}\{1,\Vert Df_i\Vert\})^{2\tilde{N}}e^{(\lambda^--\lambda^+-6\tilde{\epsilon})}.$$ 
Also regardless of our starting points, we will always have $\alpha$-H\"olderness where $\alpha$ is defined at the beginning of this section.
\end{proof}

\begin{lemma}\label{hun}
    The unstable manifold $W^u(\tilde{y},\sigma^n(\ho))$ is approximated by $\exp_z(E^u(\tilde{z},\sigma^{n}(\ho)))$ up to error $\xi= \mathfrak{r}^2\cdot M_r^3\cdot H$ on the ball of radius $\mathfrak{r}$ where $H$ is the uniform upper bound on $\lip(Dh^u_{(x,\omega)})$ for $(x,\omega)\in K_{\xi}'$ (see Lemma \ref{H}). 
\end{lemma}
\begin{proof}
Since $(\tilde{y},\sigma^n(\ho))\in K_{\xi}$, we know that the local unstable, $W^u(\tilde{y},\sigma^n(\ho))$, is well-approximated by (the image under exp of) $E^u(\tilde{z},\sigma^{n}(\ho))$ up to error $\xi$ on the ball of radius $\mathfrak{r}$. For given $\mathfrak{r}$, the error is given by $\xi = \mathfrak{r}^2\cdot M_r^3\cdot \lip(Dh^u_{(\tilde{y},\sigma^n(\ho))})$ (where $M_r$ is fixed from Section \ref{sec:lusin1} and $h_{(\tilde{y},\sigma^n(\ho))}$ is as in Proposition \ref{locstablemfld}). Note that $\lip(Dh^u_{(\tilde{y},\sigma^n(\ho))})$ is bounded above by $H$ as in Lemma \ref{H} because $(\tilde{y},\sigma^n(\ho))\in K_{\xi}'$. 
\end{proof}

Lemmas \ref{hun1} and \ref{hun} motivate us to contract the stable distance between $\hx$ and $\hy$ which is $d_X(\hx,\hz)$. We will also be interested in increasing the unstable distance between $\hx$ and $\hy$ which is $d_X(\hy,\hz)$. We will take $\mathfrak{r}$ in the above lemma to be twice the resulting unstable distance. This is the scale we want to view things on.

\begin{remark}\label{flip1}
    If the NJI condition was instead written as $W^s(\hy,\ho)\cap U^+[\hx,\ho] =\emptyset$ we would instead take $\hz\in W^s_q(\hy,\ho)\cap W^u_q(\hx,\ho)$ and work to compare $E^u(\tilde{y},\sigma^n(\ho))$ and $E^u(\tilde{z},\sigma^n(\ho))$ in Lemma \ref{hun1} (where $n=\lfloor T+k_0 \rfloor$). Then we would care about $\tilde{y}$ and $\tilde{z}$ being close. Also note that the stable distance between $\tilde{x}$ and $\tilde{y}$ would now be the distance between $\tilde{y}$ and $\tilde{z}$ and the unstable distance would now be the distance between $\tilde{x}$ and $\tilde{z}$.
\end{remark}

By Theorem \ref{locstablemfld} we know exactly how the distance between $\hat{z}$ and $\hx$ contracts and $\hat{z}$ and $\hy$ expands.

Let $T_{big}\in \bbR$ be such that $d_X(\tilde{y},\tilde{z}) = \epsilon_f$, let $n_B = \lfloor T_{big}+k_0\rfloor$. We know we can reach $\epsilon_f$ in the unstable distance because by Theorem \ref{locstablemfld} we have 
\begin{align}\label{thm8bd}
    d_X(\tilde{y},\tilde{z}) \geq \frac{1}{k}d_X(\hat{y},\hz)e^{n(\lambda^+-\delta)}.
\end{align}
Note also that since the original points are in $\Lambda_{loc}$ we have replaced $k(\hx,\ho)$ with the constant $k$ and $\delta$ is as in Section \ref{sec:stablemfld}.

\begin{lemma}\label{Tbig}
The parameter $T_{big}$ increases with decreasing $d_X(\hy,\hz)$.
\end{lemma}
\begin{proof}
Since the Lipschitz constant of all the diffeomorphisms in $\supp(\mu)$ is bounded above by $\calL$, we have
\begin{align}
    d_X(\tilde{y},\tilde{z})\leq \calL^nd_X(\hy,\hz).
\end{align}
Rearranging, we find that
\begin{align}\label{lipcal}
  \frac{1}{\ln(\calL)}  \ln\left( \frac{d_X(\tilde{y},\tilde{z})}{d_X(\hy,\hz)}\right) \leq n.
\end{align}
Setting $d_X(\tilde{y},\tilde{z}) = \epsilon_f$, it is clear that the closer the original points (in unstable distance) are to each other, the larger $n_B$ has to be, i.e. the longer we have to flow for to reach distance $\epsilon_f$.
\end{proof}

\begin{choice}\label{choice1}
We will pick $T\in [T_{big}(1-\theta_1) , T_{big}(1-\theta_2)] =I $ for some $\theta_1,\theta_2>0$ that will be determined later.
\end{choice}
We will pick the $\theta_i$'s such that flowing by $T$ does not take us all the way to distance $\epsilon_f$. It will become clear why we have to pick $T$ in some interval shortly. 

Let $\alpha^s = d_X(\tilde{x},\tilde{z})$, $\alpha^u = d_X(\tilde{z},\tilde{y})$ for a chosen $T$. Denote the $X$-components of $F^{T_{big}}(\hy,\ho,k_0)$ and $F^{T_{big}}(\hz,\ho,k_0)$ are $y_B$ and $z_B$ respectively.

\begin{lemma}\label{Tshort}
    For fixed $\theta_1, \theta_2>0$ as $T_{big}\to \infty$, $\alpha^u\to 0$.
\end{lemma}
\begin{proof}

 Since $\theta_1$ and $\theta_2$ are fixed and $\theta \in [\theta_2,\theta_1]$ is such that $T = T_{big}(1-\theta)$, we see that $T_{big}-T = T_{big}-T_{big}(1-\theta) = T_{big}\theta$ grows with $T_{big}$. We apply Theorem \ref{locstablemfld} to the points $y_B$ and $z_B$ and flow backwards to the points $\tilde{y}$ and $\tilde{z}$. We get that 

\begin{align}
    \alpha^u= d(\tilde{y},\tilde{z}) \leq ke^{-(\lambda^++\delta)T_{big}\theta}\epsilon_f.
\end{align}
Since $\lambda^+>0$ we have that $\alpha^u\to 0$ as $T_{big}\to \infty$.
\end{proof} 

 Let $n = \lfloor T+k_0 \rfloor$.
 
\begin{lemma}\label{sshort}
    For fixed $\theta_1,\theta_2>0$, as $T_{big}\to \infty$, $\alpha^s \to 0.$
\end{lemma}
 \begin{proof}
 By Theorem \ref{locstablemfld} we get
\begin{align}
    \alpha^s = d_X(\tilde{x},\tilde{z})\leq k e^{n (\lambda^- + \delta)} d_X(\hx,\hat{z}),
\end{align}
As $T_{big}$ grows, $T$ will also grow, so $\alpha^s\to 0$.
\end{proof}

\begin{remark}
Given that $\lambda^-<0$, we have that
\begin{align}\label{stbldist}
    \alpha^s= d_X(\tilde{x},\tilde{z}) \leq d_X(\hx,\hat{z}) k \exp \left((\lambda^-+\delta)\cdot \left( \ln \left(\frac{d_X(\tilde{y},\tilde{z})}{d_X(\hy,\hat{z})} \right) \frac{1}{\ln(\calL)} \right)\right).
\end{align}
So we see that all the distances are related. 
\end{remark}

Recall that $M_{exp}$ is as in Lemma \ref{comp4} and that our points belong to $K_{exp}$, see Section \ref{sec:lusin}.

Combining Lemmas \ref{hun1}--\ref{sshort} we get the following corollary:

\begin{corollary}\label{finals}
    The distance between $W^u(\tilde{x},\sigma^n(\ho))$ and $W^u(\tilde{y},\sigma^n(\ho))$ goes to zero as $T_{big} \to \infty$, i.e. as the initial points get closer.
\end{corollary}
\begin{proof}
Since $(\tilde{x},\sigma^n(\ho))\in K_{\xi}$ and $(\tilde{y},\sigma^n(\ho))\in K_{\xi}' \subset K_{\xi}$, the distance between $W^u(\tilde{x},\sigma^n(\ho))$ and $W^u(\tilde{y},\sigma^n(\ho))$ is bounded above by that quantity
$$\alpha^s + 2\xi + M_{exp}d_{Gr}(E^u(\tilde{x},\sigma^n(\ho)),E^u(\tilde{z},\sigma^n(\ho)))$$
where $\xi$ is the error in approximating $W^u$ by the image of $E^u$ under $\exp$ (see Lemma \ref{errorr} and Lemma \ref{hun}). 

This error term $\xi = (2\alpha^u)^2M_r^3 \lip(Dh_{\tilde{y},\sigma^n(\ho)}) \leq (2\alpha^u)^2M_r^3 H$ (because $(\tilde{y},\sigma^n(\ho))\in K_{\xi}'$, see Lemma \ref{H}) then since $\alpha^u \to 0$ by Lemma \ref{Tshort}, this term goes to zero. 

Additionally, $\alpha^s \to 0$ with $T_{big} \to \infty$ by Lemma \ref{sshort} and by Lemma \ref{hun1} we have 
$$d_{Gr}(E^u(\tilde{x},\sigma^n(\ho)),E^u(\tilde{x},\sigma^n(\ho))) \leq C_{\alpha} d_X(\tilde{x},\tilde{z})^{\alpha} = C_{\alpha}(\alpha^s)^{\alpha}.$$
We explained how $\alpha$ and $C_{\alpha}$ do not change with our choice of starting points and hence do not change as $T_{big}\to \infty$. Hence $d_{Gr}(E^u(\tilde{x},\sigma^n(\ho)),E^u(\tilde{x},\sigma^n(\ho))) \to 0$ as $T_{big}\to \infty$.
\end{proof}

As explained in the outline of the proof, we also want to make sure that $F^T(U^+[\hx,\ho])$ and $F^T(U^+[\hy,\ho])$ are bounded away from each other. 

\begin{claim}\label{doo}
     To analyse the distance between $F^T(U^+[\hx,\ho])$ and $F^T(U^+[\hy,\ho])$, 
     it suffices to analyse the distance $d_X(\tilde{y},\tilde{z})$. 
\end{claim}
\begin{proof}
    Lemma \ref{KG} gives us that on the set $K_G$ we have uniform continuity of the map $G^+$ which takes a point $(x,\omega)$ and returns the unit tangent vector of $U^+[x,\omega]$ through $x$. Since $\tilde{x},\tilde{y}\in K_G$ we have this property. Combining this with Lemmas \ref{Tshort} and \ref{sshort} this gives us that for $T_{big}$ large, $G^+(F^T(\hx,\ho))$ and $G^+(F^T(\hy,\ho))$ are almost parallel, hence it suffices to show that $\tilde{x}$ maintains a bounded distance from $F^T(U^+[\hy,\ho])$. 

Now note that the triangle inequality gives us that for $\tilde{y}'\in F^T(U^+[\hy,\ho])$
\begin{align}\label{tri}
    d_X(\tilde{x},\tilde{y}')\geq d_X(\tilde{y}',\tilde{z}) - d_X(\tilde{x},\tilde{z}).
\end{align}

Now, by the NJI condition, $W^s_{loc}(\hx,\ho)\cap U^+_{loc}[\hy,\ho] = \emptyset$. That is to say, the point $\hz$, which is the intersection of $W^s_q(\hx,\ho)$ with $W^u_q(\hy,\ho)$ is not an element of $U^+[\hy,\ho]$. There is positive distance from $\hz$ to any point on $U^+_q[\hy,\ho]$ along $W^u_q(\hy,\ho)$, this is unstable distance between $\hx$ and $U^+_q[\hy,\ho].$ Using Theorem \ref{locstablemfld} one can choose $T$ large enough such that this unstable distance grows and $d_X(\tilde{y}',\tilde{z})$ is sufficiently large. It suffices to examine $d_X(\tilde{y},\tilde{z})$ as any other point $\tilde{y}'$ would amount to the same computation.
\end{proof}

\begin{remark}\label{flip2}
Like in Remark \ref{flip1}, if instead the NJI condition was written as $W^s(\hy,\ho)\cap U^+[\hx,\ho] = \emptyset$, we would change claim \ref{doo} slightly too. We would say that $\hz$ is not an element of $U^+[\hx,\ho]$ and we would analyse the distance between $\tilde{z}$ and $\tilde{x}$ (same computation).
\end{remark}

By Claim \ref{doo} we will analyse $d_X(\tilde{x},\tilde{y})$ and require that $T$ be chosen such that we always exceed the stable distance with the unstable distance. So we use this to compute a restriction on $\theta$.

\begin{lemma}\label{tgap}
    Let $x_t, y_t, z_t$ be the $X$-components resp. of $F^t(\hx,\ho)$, $F^t(\hy,\ho)$, $F^t(\hz,\ho)$, if $t$ is the time it takes such that $d_X(y_t,z_t)>2d_x(x_t,z_t)$, then
    \begin{align}\label{tt}     
    t \leq \frac{\ln(k)}{\ln(\calL)-(\lambda^-+\delta)} + \frac{\ln(\calL)}{\ln(\calL)-(\lambda^-+\delta)}T_{big}.
    \end{align}
\end{lemma}
\begin{proof}

In the worse case scenario, the unstable distance starts out small and the stable distance starts out large. As much as we have asked that $\hx$ and $\hy$ are chosen such that $d_X(\hx,\hy)<\rho\ll \epsilon_f$, let us assume that the stable distance is close to $\epsilon_f$ (or say it is $\epsilon_f$) and let us assume that the unstable distance is close to 0 (we know that it cannot be zero by the NJI assumption).

We will compute the time $t$ such that $d_X(y_t,z_t)=2d_X(x_t,z_t)$.

If $d_X(\hx,\hz) = \epsilon_f$, then after time $t$ we would have $d_X(\tilde{x},\tilde{z})\leq k\epsilon_f e^{(\lambda^-+\delta)t}$ and hence we have $2d_X(\tilde{x},\tilde{z})\leq 2k\epsilon_f e^{(\lambda^-+\delta)t}$. Additionally, $d_X(\tilde{y},\tilde{z}) \geq \left(\frac{1}{\calL}\right)^{T_{big}-t}\epsilon_f$. So we set equal 
\begin{align}
    \left( \frac{1}{\calL} \right)^{T_{big}-t}\epsilon_f = \epsilon_f 2k e^{(\lambda^-+\delta)t},
\end{align}
and solve for $t$. We get
\begin{align}\label{tcomp}
    t = \frac{\ln(2k)}{\ln(\calL)-(\lambda^-+\delta)} + \frac{\ln(\calL)}{\ln(\calL)-(\lambda^-+\delta)}T_{big}.
\end{align}
Note that both terms in Equation \eqref{tcomp} are positive because $\lambda^-<0$, but also that $\frac{\ln(\calL)}{\ln(\calL)-(\lambda^-+\delta)}<1$. 
\end{proof}

This gives us information on how to pick $\theta_1$ and $\theta_2$ so that $T$ is such that $d_X(\tilde{y},\tilde{z})>2d_X(\tilde{x},\tilde{z})$.

\begin{lemma}\label{Q22}
    We can take $T_{big}$ large enough (i.e. the initial points close enough together) and find a constant $Q_0<1$ such that fixing $\theta_1\ll \frac{1-Q_0}{1000}$, we always have that $T\gg t$. 
\end{lemma}
\begin{proof} 
In the equation for $t$ given by Equation \eqref{tt} the term $\frac{\ln(2k)}{\ln(\calL)-(\lambda^-+\delta)}$ is a constant and does not grow with $T_{big}$, we can start by taking $T_{big}$ large enough such that 
$$\frac{\ln(2k)}{\ln(\calL)-(\lambda^-+\delta)} + \frac{\ln(\calL)}{\ln(\calL)-(\lambda^-+\delta)}T_{big} < \left(\frac{\ln(\calL)}{\ln(\calL)-(\lambda^-+\delta)} + \epsilon_{s}\right)T_{big}.$$
where $\epsilon_s$ is chosen small such that $Q_0\defeq \left(\frac{\ln(\calL)}{\ln(\calL)-(\lambda^-+\delta)} + \epsilon_{s}\right)<1$ (remember $\lambda^-<0$).

We want to take $\theta_1$ such that $Q_0\ll(1-\theta_1)$, that way $T\gg t$. We fix arbitrarily a $\theta_1$ such that 
\begin{align}
    \theta_1 \ll \frac{1-Q_0}{1000}.
\end{align}
\end{proof}

\begin{choice}\label{thetabds}
    We will also require that $\theta_1$ and $\theta_2$ be small enough such that \begin{align}
        \frac{e^{(\lambda^-+\delta)(1-\theta_1)}}{e^{-(\lambda^+-\delta)\theta_1}}<1, \left( \frac{\calL^{\frac{C_b}{10^{i+2}}\theta_2}}{e^{-(\lambda^-+\delta)(1-\theta_2)\alpha}} \right) <1, \ \ \ 
    \end{align}
where $\alpha$ is defined at the beginning of Section \ref{sec:blue} and $C_b$ and $i$ are as defined in Section \ref{sec:lusin}, i.e. $C_b = \frac{(\lambda^+-\delta)}{2\ln(\calL)}$ and $i$ was fixed such that $C_b>\frac{1}{10^i}$ and $\frac{C_b}{10^{i+2}}\ll(\lambda^++\delta)$ so that
\begin{align}
    \frac{\lambda^++\delta}{\frac{C_b}{10^{i+2}}}\gg \ln(k\epsilon_fM_r/2) + (\lambda^++\epsilon).
\end{align}

This will be used in Lemma \ref{bdnded1} (see Claim \ref{claim1} and Equation \eqref{theta1bd}), Lemma \ref{bdnded2} (see Equations \eqref{theta2bd}, \eqref{188}) and Lemma \ref{bdnded3} (See Equation \eqref{lm} and \eqref{206}) in Section \ref{sec:green}. 
\end{choice}

Recall the definitions of $K_{1}$ and $K_1'$ from Section \ref{sec:lusin}; They are two sets of measure $m(K_1)>1-\epsilon_1, \ m(K_1')>1-\epsilon_1'$ resp. We will have to initially land in $K_1'$ after flowing by $T$, i.e. $X\times \Omega$-components of the points $F^T(\hx,\ho,k_0)$ and $F^T(\hy,\ho,k_0)$ which we denote $ (\tilde{x},\sigma^n(\ho))$ and $(\tilde{y},\sigma^n(\ho))$ (where $n = \lfloor T+k_0\rfloor$) need to be in $K_1'$. Then we argue that we can change the $\Omega$-component of $(\tilde{x},\sigma^n(\ho))$ and $(\tilde{y},\sigma^n(\ho))$ to be some $\tilde{\omega}$ such that $(\tilde{x},\tilde{\omega})$ and $(\tilde{y},\tilde{\omega})$ that are in $K_1$.  We need to change the $\Omega$-component because in the next part of the argument we need to flow backwards by a new past vector (we change the future vector as well).

To do any of this, we actually have to work in $Z$ because we can only define the flow on $Z$. Additionally, we also have to land in a subset of $Z$ such that when we change the $\Omega$-component our points will land in $E^c_2 \subset Z$ and $E^c_3\subset Z$ (fixed from the beginning in Section \ref{sec:lusin}). This will allow us to run the same argument as in Lemma \ref{bigint} (which uses Lemma \ref{birkhoff}, a standard corollary of Birkhoff's theorem) in Sections \ref{sec:green} and \ref{sec:pink}.

Recall the definition of $M$ from Section \ref{sec:lusin}. For a point $(x,k)\in X\times [0,1[$, define $$V_{(x,k)} = \{\omega\in \Omega : (x,\omega,k)\in M\}.$$ 
Now also recall from Section \ref{sec:lusin} the set 
$$B \defeq \{ (x,k)\in [0,1[ : \mu(V_{(x,k)})>1-\theta \}\subset X\times [0,1[ .$$
Then we need to land the points $F^T(\hx,\ho,k_0)$ and $F^T(\hy,\ho,k_0)$ so that their $X\times [0,1[$-components live in this set $B$. Note that the set $M$ has $\hat{m}$-measure $1-\epsilon_M$, where $\epsilon_M$ is small.

\begin{lemma}\label{B}
    The set $B$ defined above has $\nu \times dt$-measure at least $1-\frac{\epsilon_M}{\theta}$ in $X \times [0,1[$.
\end{lemma}
\begin{proof}
    Let $W_{(x,k)}=\{\omega : (x,\omega,k)\in M^c \}.$

    Note that $B^c$ can be written as the set $$B^c = \{(x,k)\in X\times [0,1[ : \mu(W_{(x,k)})>\theta \}.$$
    
    Now let 
    $$E \defeq \{(x,\omega,k)\in X\times \Omega\times  [0,1[ : (x,k)\in B^c, \ (x,\omega, k) \in M^c \}\subset M^c.$$

    Now using Fubini, note that
    $$\hat{m}(M^c) \geq \hat{m}(E) = \int_{X\times [0,1[} \int_{\Omega} d\mu^{\bbZ} d(\nu\times dt) = \mu^{\bbZ}(W)\cdot (\nu\times dt)(B^c) \geq  \theta \cdot (\nu\times dt)(B^c).$$

    Hence $(\nu\times dt)(B^c) \leq \frac{\hat{m}(M^c)}{\theta}$. This completes the proof.
\end{proof}

We take $\theta = \sqrt{\epsilon_M}$, then $(\nu\times dt)(B)>1-\sqrt{\epsilon_M}$.

Knowing that $B$ has $\nu \times dt$-measure $(1-\sqrt{\epsilon_M})$, we have that $B\times \Omega$ has $\hat{m}$-measure $(1-\sqrt{\epsilon_M})$. The size of $K_1'$ is large enough ($\epsilon_1'>0$ small enough) such that $\tilde{K}_1 = (B\times\Omega) \cap (K_1' \times [0,1[)$ has $\hat{m}$-measure like $(1-\tilde{\epsilon}_1)$ for some $\tilde{\epsilon}_1$ small. All we really need is that $\hat{m}(\tilde{K}_1)>1/2$ (we asked that $\hat{m}(\tilde{K}_1)>3/4$ just for convenience in Section \ref{sec:lusin}) which can clearly achieve. We now use Lemma \ref{birkhoff} to land $F^T(\hx,\ho,k_0)$ and $F^T(\hy,\ho,k_0)$ in the set $\tilde{K}_1$.

\begin{lemma}\label{bigint}
  For $T_{big}$ large enough, given the interval $[T_{big}(1-\theta_1),T_{big}(1-\theta_2)]$, there exists a subset of positive measure $\calG \subset [T_{big}(1-\theta_1),T_{big}(1-\theta_2)]$ such that, if $T\in \calG$, then flowing forwards by $T$ from the points $(\hx,\ho,k_0)$ and $(\hy,\ho,k_0)$ we land in the set $\tilde{K}_1$.
\end{lemma}

\begin{proof}

We take $T_{big}$ such that $T_{big}(1-\theta_1)$ is greater $S$ where $S$ is as in the definition of $E^c$ in Section \ref{sec:lusin} (a quantity fixed at the beginning of the argument, i.e. before beginning Section \ref{sec:blue}). This allows us to use Lemma \ref{birkhoff}.

Take $K=\tilde{K}_1$ and apply Lemma \ref{birkhoff}. Our starting points $(\hx,\ho,k_0)$ and $(\hy,\ho,k_0)$ belong to $E^c$ (see Section \ref{sec:original}). Remember that $E^c$ was fixed (in Section \ref{sec:lusin}) by the choice of $\delta = \frac{1}{100}\gamma$ where $\gamma = \hat{m}(K)>3/4$. We will denote these starting points $\alpha$ and $\beta$ resp.  Let $T_1 = T_{big}(1-\theta_1) $  $T_0> S$. Let $T_2 = T_{big}(1-\theta_2)$ and note that $T_2(1.98\gamma -1)\gg T_1$. Then let
$$\int_0^{T_2} \mathds{1}_{K}(F^{t}(\alpha))dt = Leb\{t\in [0,T_2] : F^t(\alpha)\in K\} = : \eta_{T_2,\alpha}.$$
Then,
\begin{align}
   (-\delta+\gamma)T_2<\eta_{T_2,\alpha}<(\delta + \gamma)T_2,\label{gamma} \\
    0.99 \gamma T_2 < \eta_{T_2,\alpha} < 1.01\gamma T_2. \label{eta}
\end{align}

Run the same process for $\beta$. Then, using Equation \eqref{eta} we have
\begin{align}\label{inter}
Leb(\eta_{T_2,\alpha}\cap \eta_{T_2,\beta}) &= Leb(\eta_{T_2,\alpha}) + Leb(\eta_{T_2,\beta}) - Leb(\eta_{T_2,\alpha}\cup \eta_{T_2,\beta})\\ \nonumber
&\geq T_2(1.98\gamma -1)>T_1 = Leb([0,T_1]). 
\end{align}
So then there is a positive measure subset of $[T_1,T_2]$ where both $F^t(\alpha)$ and $F^t(\beta)$ belong to ${K}$. 
 \end{proof}

  \begin{remark}\label{half}
    If one takes Equation \eqref{gamma} and uses it to compute the inequality in Equation \eqref{inter}, one finds that 
    \begin{align*}
        Leb(\eta_{T,\alpha}\cap \eta_{T,\beta}) = Leb(\eta_{T,\alpha}) + Leb(\eta_{T,\beta}) - Leb(\eta_{T,\alpha}\cup \eta_{T,\beta}) \geq 2(\gamma - \delta) T -T.
    \end{align*}
    Hence one needs $2(\gamma - \delta)-1 >0$. This happens if and only if $\gamma -\frac{1}{2}>\delta>0$. This is why we needed the condition $\gamma = (\nu\times \mu^{\bbZ})({K})>3/4$, though really $>1/2$ suffices as we could find an appropriate $\delta$. 
 \end{remark}

\subsubsection{Flowing by $m$}
\label{sec:green}

\begin{convention}
    We pick the convention that $m$ is given as a positive real number and we will flow backwards by it and hence write $F^{-m}$.
\end{convention}

From the points $\tilde{x}$ and $\tilde{y}$ we will go in two directions. First, we will choose a new $ \Omega$-component and flow backwards by some $m$ (under the `standard' flow) from $\tilde{x}$ and $\tilde{y}$ to the points which we will denote $x$ and $y$. Let $\tilde{\omega}\in \Omega$ be such that $(\tilde{x},\tilde{\omega})$ and $(\tilde{y},\tilde{\omega}) \in K_1$; we can choose such $\tilde{\omega}$ because we landed the points $F^T(\hx,\ho,k_0) = (\tilde{x},\sigma^{n}(\ho), k_1)$ and $F^T(\hy,\ho,k_0)=(\tilde{y},\sigma^{n}(\ho), k_1)$ (where $n = \lfloor T + k_0 \rfloor$) in the set $\tilde{K}_1$ in Section \ref{sec:blue}. This gives us that we have a (large) positive measure subset of $\Omega$ (the intersection of $V_{(\tilde{x},k_1)}$ and $V_{(\tilde{y},k_1)}$ which has large positive measure because they are both large) from which to choose $\tilde{\omega}$ from. Still, we ask more from our choice of $\tilde{\omega}$ in the next Lemma.

\begin{lemma}\label{apartm}
    We can choose $\tilde{\omega}$ such that $W^s(\tilde{x},\tilde{\omega})$ is bounded away from $W^u(\tilde{x},\sigma^n(\ho))$. Similarly we ask that the stables and unstables of $\sigma^n(\ho)$ and $\tilde{\omega}$ all distinct.
\end{lemma}
\begin{proof}
    The unstable plane $W^u(\tilde{x},\sigma^n(\ho))$ (where $n = \lfloor T + k_0 \rfloor$) is fixed at this point for some $T_{big}$ choice large enough. Only $W^s(\tilde{x},\tilde{\omega})$ changes with $\tilde{\omega}$.

    The tangent planes $E^s(\tilde{x},\omega)$ and $E^u(\tilde{x},\sigma^n(\ho))$ approximate $W^s(\tilde{x},\tilde{\omega})$ and $W^u(\tilde{x},\sigma^n(\ho))$ respectively up to error $\xi = \mathfrak{r}^2M_r^3H$ on a ball of radius $\mathfrak{r}$ using Lemma \ref{hun}. We need to take $\mathfrak{r} = 2\alpha^u$ as this is the scale we view things on as explained in Section \ref{sec:blue}. We can assume we have taken $T_{big}$ large enough such that $4\xi\ll \varepsilon$ (see the definition of $K_{\varepsilon}$ in Lemma \ref{Kdelta'} and note that $(\tilde{x},\sigma^n(\ho))\in K_1' \subset K_{\varepsilon}$ and $(\tilde{x},\tilde{\omega})\in K_1\subset K_{\varepsilon}$). 
    
    Then we use Lemma \ref{anglecontrol} to say we can pick $\tilde{\omega}$ such that $E^s(\tilde{x},\omega)$ and $E^u(\tilde{x},\sigma^n(\ho))$ are at least $4\xi$ apart and so $W^s(\tilde{x},\tilde{\omega})$ is bounded away from $W^u(\tilde{x},\sigma^n(\ho))$ on the appropriate scale, i.e. a ball of size $\mathfrak{r}$.
\end{proof}

Let $K_2$ and $K_2'$ be as defined in Section \ref{sec:lusin} (they have measure $1-\epsilon_2$ and $1-\epsilon_2'$ resp.). As before in Section \ref{sec:blue}, when we flow from the points $(\tilde{x},\tilde{\omega},k_1)$ and $(\tilde{y},\tilde{\omega},k_1)$ to the points $(x,\sigma^{n'}\tilde{\omega},k_2)$ and $(y,\sigma^{n'}\tilde{\omega},k_2)$ (where $n' = \lfloor  k_1-m \rfloor$, $k_2 = k_1-m \mod 1$), we need to land in $K_2' \times [0,1[ \subset Y$. In this case we additionally need that $(x,\sigma^{n'}\tilde{\omega}^-)$ and $(y,\sigma^{n'}\tilde{\omega}^-)$ are a good set we called $B' \subset X\times \Omega^-$ (see Section \ref{sec:lusin}). 

 Note that the measures on $\Omega^+$ and $\Omega^-$ are both $\mu^{\bbN}$ in the obvious way. 

\begin{lemma}\label{B'}
    The set $B'$ defined above has $\nu\times \mu^{\bbN}$-measure $1-\sqrt{\epsilon_2}$ in $X \times \Omega^-$.
\end{lemma}
\begin{proof}
    The proof is identical to the argument of Lemma \ref{B}.
\end{proof}

Let $\tilde{K}_2 = (B'\times\Omega^+) \cap K_2'$ be as in Section \ref{sec:lusin}.

\begin{lemma}\label{m2}
    We can choose $m$ such that $F^{-m}(\tilde{x},\tilde{\omega},k_1)$ and $F^{-m}(\tilde{y},\tilde{\omega},k_1)$ belong to the set $\tilde{K}_2 \times [0,1[$.
\end{lemma}
\begin{proof}
Knowing that $B'$ has $\nu \times \mu^{\bbN}$-measure at least $(1-\sqrt{\epsilon_2})$, we have that $B'\times \Omega^+$ has $\nu\times \mu^{\bbZ}$-measure at least $(1-\sqrt{\epsilon_2})$. We asked in Section \ref{sec:lusin} that $K_2'$ be large enough ($\epsilon_2'>0$ small enough) such that $\tilde{K}_2 = (B'\times\Omega^+) \cap K_2'$ has $\nu\times \mu^{\bbZ}$-measure like $(1-\tilde{\epsilon}_2)$ for some $\tilde{\epsilon}_2$ small such that 
\begin{align}
    1-\tilde{\epsilon_2}>3/4.
\end{align}

In Section \ref{sec:blue} we showed that we could land $F^T(\hx,\ho,k_0)$ and $F^T(\hy,\ho,k_0)$ in a set $\tilde{K}_1$ where $\tilde{\omega}$ could be picked such that the resulting points $(\tilde{x},\tilde{\omega}, k_1)$ and $(\tilde{y},\tilde{\omega},k_1)$ (where $k_1 = F^T(k_0)$) belong to the set $M = (K_1\times [0,1[)\cap E_2^c\cap E_3^c \subset Z$. The set $E_2^c$ as defined in Section \ref{sec:lusin} is where we need to start from in order to use Lemma \ref{birkhoff} to choose $m$ such that we land in $\tilde{K}_2\times [0,1[$. We run the same argument as in Lemma \ref{bigint} on an interval $[m_1,m_2]$ (for $m_1$ sufficiently large), starting from the points $(\tilde{x},\tilde{\omega}, k_1)$ and $(\tilde{x},\tilde{\omega},k_1)$ in $E_2^c$ and taking $K = \tilde{K}_2 \times [0,1[$ (Note that $\epsilon$ and $\delta$ were fixed along with $E_2^c$ in Section \ref{sec:lusin}). Then we get a positive measure subset of $m\in [m_1,m_2]$ such that we land the $X\times \Omega$ -components of $F^{-m}(\tilde{x},\tilde{\omega},k_1)$ and $F^{-m}(\tilde{y},\tilde{\omega},k_1)$ in the set $\tilde{K}_2$. 
\end{proof}

We will be more explicit about the proportion of good choices of $m$ we can choose inside some interval $[m_1,m_2]$ in Lemma \ref{dd}.

We have some conditions on $m$, firstly, the minimum value of $m$ in the interval we need, which we call $m_1$, must be bigger than the $S$ required from Lemma \ref{birkhoff} corresponding to the set $E^c_2$. Additionally, the maximum value of $m$ for the interval, called $m_2$, is such that we do not surpass a certain amount of the remaining distance to $\epsilon_f$. Additionally, of course the error between the unstable planes that we worked for should not be significantly changed by flowing by $m$. We will discuss this more later, but in Lyapunov charts.

Lastly, we want the relationship between $m_1$ and $m_2$ to be $m_2>2m_1$. Ensuring that $[m_1,m_2]$ is at least $50\%$ of the interval $[0,m_2]$ allows us to ensure that the `bad' set is still small;  We will use this in Section \ref{sec:pink}. We will also demonstrate another property in our choice of $m$ in Section \ref{sec:pink}, i.e. that $d_1(m_1)\geq (1+\eta)d_1(m_2)$ and $d_2(m_1)\geq (1+\eta)d_2(m_2)$; see Lemma \ref{dd}.

Let $C_b$ and $i$ be as in Section \ref{sec:lusin}/Choice \ref{thetabds}, and let 
\begin{align}\label{mchoice}
m_1 \defeq \frac{C_b}{10^{i+J}}T_{big}\theta_2, \ \ \ m_2 \defeq \frac{C_b}{10^{i+2}}T_{big}\theta_2,
\end{align}
where $J$ is chosen such that $m_2>2m_1$ and  $|\frac{C_b}{10^{i+J}}(\lambda^--\epsilon)| < |\frac{C_b}{10^{i+2}}(\lambda^-+\epsilon)| $. Also start take $T_{big}$ large enough so that $m_1>S$.
\begin{lemma}\label{bdnded1}
    For $m\in [m_1,m_2]$, the distance between $\tilde{x}$ and $\tilde{y}$ does not surpass $\epsilon_f$ after flowing backwards by $m$ under $\tilde{\omega}$.
\end{lemma}
\begin{proof}
    
Lemma \ref{apartm} bounded $E^s(\tilde{x},\omega)$ away from $E^u(\tilde{x},\sigma^n(\ho))$, so we see that the stable distance (which is the one growing because we are flowing backwards by $m$) is strictly less than the unstable distance with respect to $\sigma^n(\ho)$. After flowing by $T$, $d_X(\tilde{x},\tilde{y})$ is mostly in the unstable direction, but now with respect to this new $\omega$ it has some stable and unstable component by Lemma \ref{apartm}. Let $m_{big}$ be the amount of time we flow backwards with respect to $\omega$ until this distance is now $\epsilon_f$. We have
\begin{align}\label{tric2}
    \left(\frac{1}{\calL}\right)^{m_{big}}\epsilon_f \leq d_X(\tilde{x},\tilde{y}).
\end{align}

We can also say that 
\begin{align}\label{Cseq}
    d_X(\tilde{x},\tilde{y}) \leq d_X(\tilde{y},\tilde{z})+d_X(\tilde{x},\tilde{z}) \leq \epsilon_f e^{-(\lambda^+-\delta)(T_{big}-T)} + d_X(\hx,\hz)e^{(\lambda^-+\delta)T}.
\end{align}

We simplify Equation \eqref{Cseq} and note that we can find $C_s>0$ such that \begin{align}\label{tric}
    \epsilon_f e^{-(\lambda^+-\delta)(T_{big}-T)} + d_X(\hx,\hz)e^{(\lambda^-+\delta)T} \leq C_s\epsilon_fe^{-(\lambda^+-\delta)(T_{big}-T)}.
\end{align}

\begin{claim}\label{claim1}
    We can fix $C_s$ for all $T_{big}$ because of our choice of $\theta_1$ in Choice \ref{thetabds} in Section \ref{sec:blue}
\end{claim}
\begin{proof}
    We see that for $C_s$ to satisfy Equation \eqref{tric}, we require it to satisfy 
\begin{align} \label{theta1bd}
    C_s\geq 1+\frac{d_X(\hx,\hz)}{\epsilon_f}\left(\frac{e^{(\lambda^-+\delta)(1-\theta)}}{e^{-(\lambda^+-\delta)\theta}} \right)^{T_{big}}.
\end{align}
The right-hand most term of Equation \eqref{theta1bd} goes to 0 as $T_{big}\to \infty$, so we can fix $C_s$ for all $T_{big}$ and $\rho$.
\end{proof}

Combining the inequalities \eqref{tric2} and \eqref{tric}, we get:
\begin{align}
    \left(\frac{1}{\calL} \right)^{m_{big}}\epsilon_f \leq C_s\epsilon_f e^{-(\lambda^+-\delta)(T_{big}-T)}
\end{align}
We find that then
\begin{align}
    m_{big}&\geq \frac{(\lambda^+-\delta)}{\ln(\calL)}(T_{big}-T) - \frac{\ln(C_s)}{\ln(\calL)}\\
    &=\frac{(\lambda^+-\delta)}{\ln(\calL)}(T_{big}\theta) - \frac{\ln(C_s)}{\ln(\calL)}.
\end{align}
Since $\frac{\ln(C_s)}{\ln(\calL)}$ is fixed and $$m_2 = \frac{C_b}{10^{i+2}}T_{big}\theta_2 = \frac{(\lambda^+-\delta)}{2\ln(\calL)} \frac{1}{10^{i+2}}T_{big}\theta_2,$$ where $\theta_2<\theta$ we have that for $T_{big}$ large enough, $m_2<m_{big}$. This implies that the distance between $\tilde{x}$ and $\tilde{y}$ does not surpass $\epsilon_f$ after flowing backwards by $m<\leq m_2$. This completes the proof.
\end{proof}

Let $d_m$ be the distance between $F^{-m}_{\tilde{\omega}}(W^u(F^n(\hx,\ho)))$ and $F^{-m}_{\tilde{\omega}}(W^u(F^n(\hy,\ho)))$ ($n= \lfloor T+k_0\rfloor$).  Let $\delta(m) = d_X(F^{-m}_{\tilde{\omega}}(\tilde{x}),F^{-m}_{\tilde{\omega}}(\tilde{y}))$ (i.e. represents the distance between $F^{-m}_{\tilde{\omega}}(U^+(F^n(\hx,\ho)))$ and $F^{-m}_{\tilde{\omega}}(U^+(F^n(\hy,\ho)))$ by Claim \ref{doo}).

\begin{lemma}\label{bdnded2}
    The distance $d_m$ goes to zero as $T_{big},m\to \infty$. 
    
\end{lemma}
\begin{proof}
 
We computed that the distance between $W^u(F^n(\hx,\ho))$ and $W^u(F^n(\hy,\ho))$ as $$\calD\defeq (\alpha^s+2\xi + M_{exp}d_{Gr}(E^u(\tilde{x},\sigma^n(\ho)),E^u(\tilde{z},\sigma^n(\ho))),$$ and showed it goes to zero as $T_{big}\to \infty$ in Lemma \ref{finals}.

The largest $m$ could be is $\frac{C_b}{10^{i+2}}T_{big}\theta_2$. Originally this error between the unstable planes that we computed was mostly in the stable direction for $\sigma^n(\ho)$. We have asked that $\tilde{\omega}$ be chosen such that the distance between has now stable and unstable components (see Lemma \ref{anglecontrol}).

Say in the worst case scenario that the distance $\calD$ between unstable planes that we measured is entirely in the stable direction for $\tilde{\omega}$. Let $\calD_m$ be the distance between the images of $W^u(F^n(\hx,\ho))$ and $W^u(F^n(\hy,\ho))$ after flowing backwards by $m$, then
\begin{align*}
    \calD_m &\leq \calL^{\frac{C_b}{10^{i+2}}T_{big}\theta_2}(\alpha^s+2\xi + M_{exp}d_{Gr}(E^u(\tilde{x},\sigma^n(\ho)),E^u(\tilde{z},\sigma^n(\ho)))    
\end{align*}

We deal with each term in $\calD$ one at a time.
\begin{align}
     \calL^{\frac{C_b}{10^{i+2}}T_{big}\theta_2}\alpha^s\leq \frac{\calL^{\frac{C_b}{10^{i+2}}T_{big}\theta_2}}{e^{-(\lambda^-+\delta)T_{big}(1-\theta)}} d_X(\hx,\hz) = \frac{\calL^{\frac{C_b}{10^{i+2}}T_{big}\theta_2}}{e^{(-(\lambda^-+\delta)(1-\theta))T_{big}}} d_X(\hx,\hz)
\end{align}

 where $\theta \in [\theta_2,\theta_1]$ because we used Theorem \ref{locstablemfld} for $T$. Recall from Choice \ref{thetabds} that
 \begin{align}\label{theta2bd}
     \frac{\calL^{\frac{C_b}{10^{i+2}\theta_2}}}{e^{-(\lambda^-+\delta)(1-\theta_2)}}<\frac{\calL^{\frac{C_b}{10^{i+2}\theta_2}}}{e^{-(\lambda^-+\delta)(1-\theta_2)\alpha}}<1,
 \end{align}
 So as $T_{big}\to \infty$, this quantity goes to zero.

\begin{align}\label{186}
    \calL^{\frac{C_b}{10^{i+2}}T_{big}\theta_2}2\xi \leq \calL^{\frac{C_b}{10^{i+2}}T_{big}\theta_2} k\epsilon_fe^{-(\lambda^++\delta)T_{big}\theta}
\end{align}
but note that by choices made in Section \ref{sec:lusin} that $\frac{\calL^{\frac{C_b}{10^{i+2}}}}{e^{\lambda^++\delta}}<1$ and so this goes to zero as $T_{big},m\to \infty$.

Lastly, 
\begin{align}
    \calL^{\frac{C_b}{10^{i+2}\theta_2}}M_{exp}d_{Gr}(E^u(\tilde{x},\sigma^n(\ho)),E^u(\tilde{z},\sigma^n(\ho))) &\leq \calL^{\frac{C_b}{10^{i+2}\theta_2}}M_{exp} (\alpha^s)^{\alpha}\\
    &\leq M_{exp}\frac{\calL^{\frac{C_b}{10^{i+2}}T_{big}\theta_2}}{e^{(-(\lambda^-+\delta)(1-\theta)\alpha)T_{big}}} d_X(\hx,\hz)^{\alpha}. \label{188}
\end{align}
Equation \eqref{theta2bd} then gives us that this goes to zero as $T_{big},m\to \infty$. Then $\calD_m$ goes to zero as $m\to \infty.$ This completes the proof.
 \end{proof}

Let $\tilde{D}$ be the distance (in ambient coordinates) between $\tilde{x}$ and $\tilde{y}$. Note that by construction, this distance is still less than $q$ and that the set $M$ is contained in $\Lambda_{loc}$ as defined in \ref{sec:lusin}. The points $(\tilde{x}, \tilde{\omega},k_1)$ and $(\tilde{y},\tilde{\omega},k_1)$ belong to $M$. Additionally, the distance $\tilde{D}$ between $\tilde{x}$ and $\tilde{y}$ is less than $k_U$ and hence they live in each other's Lyapunov chart domain.

Originally, the distance $\tilde{D}$ was largely in the unstable direction for that choice of $\Omega$-component, i.e. $\sigma^n(\ho)$. Now that we have switched $\Omega$-components to $\tilde{\omega}$, generically, the distance $\tilde{D}$ has stable and unstable distance components with respect to $\tilde{\omega}$. Also we picked $\tilde{\omega}$ such that $\tilde{x}$ and $\tilde{y}$ are not on each other's stable or unstable manifolds. Also, since we are going backwards by $m$, the stables grow and the unstables contract. 

In Lyapunov charts, the images of the subspaces $E^u(\tilde{x},\tilde{\omega})$ and $E^s(\tilde{x},\tilde{\omega})$ form a coordinate axis, and $\tilde{y}$ is sent to some $v\in \bbR^4$ in these charts. We can decompose $v$ into a component in the $E^s$-direction and the $E^u$-direction denoted respectively $v_s$ and $v_u$, i.e. $v=v_s+v_u$. In these charts, we know exactly $v_s$ and $v_u$ grow and shrink resp. under flowing by $-m$. We denote $\alpha_s = |v_s|,$ $\alpha_u = |v_u|$.

Then, the distance between 0 and $v$ after flowing by $m$, is $D_m\defeq |\tilde{F}^{-m}_{(\tilde{x},\tilde{\omega}),k_1}(v)-\tilde{F}^{-m}_{(\tilde{x},\tilde{\omega}),k_1}(0)|$. Let $\alpha_u^m \defeq|\tilde{F}^{-m}_{(\tilde{x},\tilde{\omega})}(v_u)-\tilde{F}^{-m}_{(\tilde{x},\tilde{\omega})}(0)|$ and $\alpha_s^m \defeq |\tilde{F}^{-m}_{(\tilde{x},\tilde{\omega}),k_1}(v_s)-\tilde{F}^{-m}_{(\tilde{x},\tilde{\omega}),k_1}(0)|$.

Now we perform a calculation that we will need in Section \ref{sec:pink}.

\begin{lemma}\label{bdnded3}
    For $T_{big}$ large enough (hence $m$ large enough), we have $$D_{m+1} \leq \sqrt{1+k^2} (e^{-(\lambda^--\epsilon)}+\frac{1}{2}C)\alpha_s^m = (e^{-(\lambda^--\epsilon)}+\frac{1}{2}C)D_m.$$
    Further, we have $D_{m}<D_{m+1}$.
\end{lemma}
\begin{proof}
    
Write $\alpha_u = k\alpha_s$ for some $k>0$. Note that this is dependent on the points but it will not matter for the computation we will do.  In fact, by Pythagoras, $D_m =  \sqrt{1+k^2}\alpha_s^m$. 

We examine how this compares to $D_{m+1} =|\tilde{F}^{-(m+1)}_{(\tilde{x},\tilde{\omega}),k_1}(v)-\tilde{F}^{-(m+1)}_{(\tilde{x},\tilde{\omega}),k_1}(0)|$. Again by Pythagoras, $D_{m+1} = \sqrt{(\alpha_u^{m+1})^2 + (\alpha_s^{m+1})^2}$. Then using distortion estimates from Section \ref{sec:distortion} as before, we additionally compute that 
\begin{align}
    \alpha_s^{m+1} &\leq e^{-(\lambda^--\epsilon)}\alpha_s^m  +\frac{1}{2}C(\alpha_s^m)^2\\
    &\leq (e^{-(\lambda^--\epsilon)}+\frac{1}{2}C)\alpha_s^m,\label{108}
\end{align}
\begin{align}
    \alpha^{m+1}_u &\leq e^{-(\lambda^+-\epsilon)}\alpha_u^m + \frac{1}{2}C(\alpha_u^m)^2\\
    &\leq k(e^{-(\lambda^+-\epsilon)} +\frac{1}{2}C)\alpha_s^m\label{110}\\
    &\leq k(e^{-(\lambda^--\epsilon)}+\frac{1}{2}C)\alpha_s^m.\label{alpu}
\end{align}
Where Equations \eqref{108} and \eqref{110} hold because $\alpha_s^m>(\alpha_s^m)^2$ because $\epsilon_f<\frac{1}{M_r}$ and $\alpha_u^m= k\alpha_s^m$.

Then
\begin{align}\label{Dmrel}
    D_{m+1} \leq \sqrt{1+k^2} (e^{-(\lambda^--\epsilon)}+\frac{1}{2}C)\alpha_s^m = (e^{-(\lambda^--\epsilon)}+\frac{1}{2}C)D_m.
\end{align}

 The above bound is necessary for the proof of Lemma \ref{BiLip}, however we also need to demonstrate the fact that $D_m<D_{m+1}$. 

Again, from Section \ref{sec:distortion} we have that
\begin{align}
    \alpha_s^{m+1} \geq \alpha_s^m(e^{-(\lambda^-+\epsilon) }-\frac{1}{2}C\alpha_s^m),\\
    \alpha_u^{m+1} \geq \alpha_u^m(e^{-(\lambda^++\epsilon) }-\frac{1}{2}C\alpha_u^m).
\end{align}

At the beginning of Section \ref{sec:blue} we fixed $k_1<1$ and $k_2<1$ small but be such that $(1-k_1)e^{-(\lambda^-+\epsilon)}>1$ and added a condition to $\epsilon_f$, we use it now to get that
\begin{align}
    \alpha_s^{m+1} \geq (1-k_1)e^{-(\lambda^- +\epsilon)}\alpha_s^m,\\
    \alpha_u^{m+1} \geq (1-k_2) e^{-(\lambda^++\epsilon)}\alpha_u^m.
\end{align}
Note that everything was chosen such that $(1-k_1)e^{-(\lambda^- +\epsilon)}>1$ and also that naturally $(1-k_2) e^{-(\lambda^++\epsilon)}<1$.

Then to have that $D_m<D_{m+1}$ (or the square which we use below) we can show that 
\begin{align}\label{Dm}
    (\alpha_s^m)^2(1-k_1)^2e^{-2(\lambda^- +\epsilon)} + (\alpha_u^m)^2(1-k_2)^2e^{-2(\lambda^+ +\epsilon)} > (\alpha_s^m)^2 + (\alpha_u^m)^2
\end{align}
Solving we see that if we call $\psi_s = (1-k_1)^2e^{-2(\lambda^- +\epsilon)} $ and $\psi_u = (1-k_2)^2e^{-2(\lambda^+ +\epsilon)}$ we need
\begin{align}
  \left(\frac{\alpha_s^m}{\alpha_u^m} \right)^2 > \left(\frac{1-\psi_u}{\psi_s-1}\right)\defeq \Psi.
\end{align}
Again, note here that $\psi_s>1$ and $\psi_u<1$. We can assume that $\Psi>1$ and argue that $m$ can be taken large enough that this ratio holds true. That will complete the proof.

So we need $\alpha_u^m < \frac{1}{\sqrt{\Psi}}\alpha_s^m$. We will demonstrate that $\frac{1}{2}C^{|n'|}\alpha_s^2 < e^{-|n'|(\lambda^-+\epsilon)}\alpha_s$ (where again $n' = \lfloor k_1-m\rfloor$), then using distortion estimates, it will suffice to show the following bound:
\begin{align}\label{ineq}
    e^{-(\lambda^++\epsilon)|n'|}\alpha_u + \frac{1}{2}C^{|n'|}\alpha_u^2 < \frac{1}{\sqrt{\Psi}}\left( e^{-|n'|(\lambda^-+\epsilon)}\alpha_s - \frac{1}{2}C^{|n'|}\alpha_s^2 \right).
\end{align}
Note $\Psi$ is completely independent of any of the other parameters, in fact could be computed at the very start of this argument. Note also that $C $, our bound on the second derivative is also fixed from the beginning and for ease of notation we will write $C= e^j$ and work with this. So to begin with, we want
\begin{align}\label{jm}
    e^{jm}\alpha_s \leq 2e^{-(\lambda^-+\epsilon)m}.
\end{align}
Note that $$\alpha_s, \alpha_u \leq M_r d_X(\tilde{y},\tilde{x}) \leq M_r k \epsilon_f e^{-(\lambda^++\delta)T_{big}\theta} 
,$$ for $\theta \in [\theta_2,\theta_1]$ and that at its smallest, $m =\frac{C_b}{10^{i+J}}T_{big}\theta_2$. Then instead of Equation \eqref{jm} we look at
\begin{align}
    e^{jm}\alpha_s \leq M_r k \epsilon_f e^{j\frac{C_b}{10^{i+J}}T_{big}\theta_2 -(\lambda^++\delta)T_{big}\theta}\leq 2e^{-(\lambda^-+\epsilon)\frac{C_b}{10^{i+J}}T_{big}\theta_2}.
\end{align}
We solve this and see that we need
\begin{align} \label{lm}
    j\leq \left( \frac{(\lambda^++\delta)}{\frac{C_b}{10^{i+J}}}- (\lambda^-+\epsilon)  \right) - \frac{\ln(k\epsilon_fM_r)}{T_{big}\theta \frac{C_b}{10^{i+J}}}.
\end{align}
Note that $i$ was chosen such that $\frac{C_b}{10^{i+J}} < \frac{C_b}{10^{i+2}}<<(\lambda^++\delta)$ (see Choice \ref{thetabds}) and we can start with $T_{big}$ large enough so that term $\frac{\ln(k\epsilon_fM_r)}{T_{big}\theta \frac{C_b}{10^{i+J}}}$ is small enough so that Equation \eqref{lm} is true for all $T_{big}$ afterwards. Hence we establish dominance of the linear term in the Taylor expansion of $\alpha_s$ because we have 
\begin{align}
    \frac{1}{2}|v_s^TD^2\tilde{F}_{(x,\omega),k}^t(w)v_s| \leq \frac{1}{2}C^m \alpha_s^2 \leq e^{-(\lambda^-+\epsilon)m}\alpha_s \leq |D\tilde{F}^t_{(x,\omega),k}(0)v_s|.
\end{align}

So now we have to show Equation \eqref{ineq}. The same argument as before shows that $\frac{1}{2}C^m\alpha_u^2 \leq e^{-(\lambda^++\epsilon)}\alpha_u$, we get that
\begin{align}\label{206}
    j \leq \left( \frac{(\lambda^++\delta)}{\frac{C_b}{10^{i+J}}}- (\lambda^++\epsilon)  \right) - \frac{\ln(k\epsilon_fM_r)/2}{T_{big}\theta \frac{C_b}{10^{i+J}}}.
\end{align}
Again, in Choice \ref{thetabds}, we chose everything so that this would be true because $\lambda^-,\lambda^+,M_r,$ $\epsilon_f,k$ are all chosen at the very start.

Then Equation \eqref{ineq} becomes
\begin{align}\label{endeq}
    2e^{-(\lambda^++\epsilon)m}\alpha_u < \frac{1}{2\sqrt{\Psi}} e^{-(\lambda^-+\epsilon)m}\alpha_s,
\end{align}
which is satisfied for $m$ large, i.e. $T_{big}$ large.

Then Equation \eqref{ineq} gives us $\alpha_u^m < \frac{1}{\sqrt{\Psi}}\alpha_s^m$. This establishes that $D_m<D_{m+1}$.
\end{proof}

For the purposes of Lemma \ref{dd}, we need the following Lemma:
\begin{lemma} \label{bb}
The constant $0<\eta<1$ as fixed in Section \ref{sec:lusin} is such that for $T_{big}$ large enough (hence $m$ large enough) $D_{m_2}^{(1+\eta)}>D_{m_1}$. 
\end{lemma}
 \begin{proof}
As in the previous argument, we can replace the usual distortion bounds with 
\begin{align}
    2e^{-(\lambda^+-\epsilon)m}\alpha_s \leq \alpha_s^m \leq \frac{1}{2}e^{-(\lambda^--\epsilon)m}\alpha_s,\\
    2e^{-(\lambda^+-\epsilon)m}\alpha_u \leq \alpha_u^m \leq \frac{1}{2}e^{(\lambda^+-\epsilon)m}\alpha_u.
\end{align}
Again, $T_{big}$ has to be large enough for this to hold. We assume this throughout the remainder of the computation wherever $m$ appears, $T_{big}$ is large. Let $$D = \sqrt{(\alpha_s)^2+(\alpha_u)^2},$$ i.e. the original distance we compute in Lyapunov charts, then we compute
\begin{align}
    D_m \leq 2\sqrt{e^{-2(\lambda^--\epsilon)m}\alpha_s^2 + e^{-2(\lambda^+-\epsilon)m}\alpha_u^2} &\leq 2e^{-(\lambda^--\epsilon)m} \sqrt{ \alpha_s^2 + (e^{-(\lambda^+-\epsilon)m+(\lambda^--\epsilon)m}\alpha_u)^2}\\ &\leq 2e^{-(\lambda^--\epsilon)m} D.
\end{align}

Further we have using Lemma \ref{anglecontrol} we can ask that $\alpha_s$ always be at least some fixed fraction of $D$, let us say that $\alpha_s>\Delta D$ uniformly.

\begin{align}
    D_m\geq  \frac{1}{2} \sqrt{e^{-2(\lambda^-+\epsilon)m}\alpha_s + e^{-2(\lambda^+-\epsilon)m}\alpha_u} \geq \frac{1}{2}e^{-(\lambda^-+\epsilon)m}\alpha_s \geq \frac{\Delta}{2}e^{-(\lambda^--\epsilon)m}D.
\end{align}

Further since $D$ is the distance $d_X(\tilde{x},\tilde{y})$ but in Lyapunov charts, we can account for the change between ambient coordinates and Lyapunov charts and get
\begin{align}\label{needed}
    \frac{1}{M_r}\epsilon_f\left(\frac{1}{\calL}  \right)^{T_{big}\theta_2} \leq D.
\end{align}

Now in order to find $\eta$ such that $D_{m_1}<D_{m_2}^{(1+\eta)}$ we can show the middle inequality in 
\begin{align}
    D_{m_1}\leq 2e^{-(\lambda^--\epsilon)m_1}D < \left(\frac{\Delta}{2} \right)^{(1+\eta)} e^{-(1+\eta)(\lambda^--\epsilon)m_2} D^{(1+\eta)} \leq D_{m_2}^{(1+\eta)}.
\end{align}
Further, plugging in Equation \eqref{needed} it suffices to demonstrate that
\begin{align}
     2e^{-(\lambda^--\epsilon)m_1} <\left(\frac{\Delta}{2}\right)^{(1+\eta)}\left( \frac{\epsilon_f}{M_r} \right)^{\eta} \left(\frac{1}{\calL}  \right)^{T_{big}\theta_2\eta} e^{-(1+\eta)(\lambda^--\epsilon)m_2}.
\end{align} 
Note that $\Delta/2<1$, $\epsilon_f/M_r<1$ and $\eta<1$ so we can solve
\begin{align}
     2e^{-(\lambda^--\epsilon)m_1} <\left( \frac{\Delta}{2} \right) \left(\frac{1}{\calL}  \right)^{T_{big}\theta_2\eta} e^{-(1+\eta)(\lambda^--\epsilon)m_2}.
\end{align} 

Writing $m_1 = \alpha_1 T_{big}\theta_2 $, $m_2 = \alpha_2 T_{big}\theta_2$, $\calC \defeq \frac{4}{\Delta}$ and $\frac{1}{\calL} = e^{-\calK}$ (where $\calK >0)$ we solve
\begin{align}
    \calC e^{-(\lambda^--\epsilon)\alpha_1T_{big}\theta_2} < e^{-(1+\eta)(\lambda^-+\epsilon)\alpha_2T_{big}\theta_2 - \calK \eta T_{big}\theta_2},
\end{align}
which gives us that we need
\begin{align}\label{eta11}
   \eta < \frac{\ln(C)}{T_{big}\theta_2} - \frac{(\lambda^--\epsilon)\alpha_1}{-(\lambda^-+\epsilon)\alpha_2-\calK} + \frac{(\lambda^-+\epsilon)\alpha_2}{-(\lambda^-+\epsilon)\alpha_2-\calK}.
\end{align}
We can ignore the constant term $\frac{\ln(C)}{T_{big}\theta_2}$ as $T_{big}$ can be taken large so that this term is appropriately small. Then we note that the right hand side of Equation \eqref{eta11} is positive because of choices made in Section \ref{sec:lusin} that give us that $(\lambda^-+\epsilon)\alpha_2 > (\lambda^--\epsilon)\alpha_1$. The parameter $\eta$ defined in Section \ref{sec:lusin} fulfills this requirement.
\end{proof}

Now we take a (good) choice of $m\in [m_1,m_2]$ as in Lemma \ref{m2}, and denote $x \defeq F^{-m}_{\tilde{\omega}}(\tilde{x})$ and $y \defeq F^{-m}_{\tilde{\omega}}(\tilde{y})$. We get a positive measure subset $V\subset \Omega^+$ of futures $\omega'^+$ to choose from that allow the points $(x,(\sigma^{-\lfloor k_1- m \rfloor}(\tilde{\omega})^-, \omega'^+ ))$ and $(y,(\sigma^{-\lfloor k_1-m\rfloor}(\tilde{\omega})^-, \omega'^+ ))$ where $\tilde{\omega}^- \in \Omega^-$ is the past part of $\tilde{\omega}$. Let $\omega'^+\in V$ be arbitrary, let $\omega' = (\sigma^{-\lfloor k_1- m \rfloor}(\tilde{\omega})^-, \omega'^+ ) \in \Omega$. We will have additional restrictions on choosing $(\omega')^+$, but certainly, $(\omega')^+$ will belong to $V$.

At the points $(x,\omega')$ and $(y,\omega')$ we will project $(y,\omega')$ along $W^s(y,\omega')$, onto $W^u(x,\omega')$ and use (modified) NFCs (see Equation \eqref{newNFCform}) to control the distance between the final points when flowing by $\omega'$ by $\ell$. We want this distance to be some fixed $\epsilon_f$.

Note that our good sets make the changes between ambient coordinates, Lyapunov charts, and (modified) NFCs, a bounded change as described in Section \ref{sec:lusin}. Also note that $W^u(x,\omega')$ is the same as $W^u(x,\sigma^{-\lfloor k_1- m \rfloor}\tilde{\omega})$ because the past vectors of $\omega'$ and $\sigma^{-\lfloor k_1-m \rfloor}\tilde{\omega}$ are the same. The stable manifold changes passing from $\sigma^{-\lfloor k_1-m \rfloor}\tilde{\omega}$ to $\omega'$.

Recall $\tilde{D}$ represents the distance between $U^+[F^T(\hx,\ho)]$ and $U^+[F^T(\hy,\ho)]$  
in ambient coordinates because $\tilde{D}$ is the distance between $\tilde{x}$ and $\tilde{y}$. Denote $L_x(m)$ and $L_y(m)$ as the image of $U^+[F^T(\hx,\ho)]$ and $U^+[F^T(\hy,\ho)]$ after flowing backwards by $m$. Recall that $\delta(m)$ is the distance between $x\defeq f^{-m}_{\tilde{\omega}}(\tilde{x})$ and $y\defeq f^{-m}_{\tilde{\omega}}(\tilde{y})$, it represents the distance between $L_x(m)$ and $L_y(m)$. Then, we let $\gamma$ be the distance between $x$ and $y'$ where $y'$ is the projection of $y$ onto $W^u(x,\omega')$ along $W^s(y,\omega')$.

Note that $x,y\in K_{\varepsilon}$ (See Lemma \ref{Kdelta'} and Section \ref{sec:lusin} for the definition of $K_2'$ and $K_2$). Recall from Lemma \ref{bdnded2} that $d_m$ is the error between $F^{-m}_{\tilde{\omega}}(W^u(F^n(\hx,\ho)))$ and $F^{-m}_{\tilde{\omega}}(W^u(F^n(\hy,\ho)))$ and that $d_m\to 0$ as $T_{big},m\to \infty$.

\begin{lemma}\label{dvg}
    There exists $m'$ and a constant $M_a>1$ such that for all $m>m'$ one can choose $(\omega')^+$ such that $\gamma(m) \in [\frac{1}{M_a}\delta(m),M_a\delta(m)]$. 
\end{lemma}
\begin{proof}

First let us take $m$ large enough ($T_{big}$ large enough) such that $$M_{exp}^3d_m/100<\varepsilon,$$ where $\varepsilon$ is from Lemma \ref{Kdelta'}.

Now $\delta(m)=d_X(x,y)$ and $\gamma = d_X(x,y')$. We will work in $\exp_y^{-1}(B(y,\rho_0))$, note that our scale of interest is smaller than $\rho_0/4$ so this is appropriate. Let $\vec{v}=\exp_y^{-1}(x)$, note that there exists $M\in [\frac{1}{M_{exp}},M_{exp}]$ such that $\Vert v\Vert = \delta(m)M$.

Let $\vec{w}$ be the vector in $T_yX$ representing $E^s(y,\omega')$ of length $\Vert w\Vert =\eta$ corresponding to the intersection with $\exp^{-1}_y(W^u(x,\omega))$. Using Lemma \ref{anglecontrol} and Lemma \ref{Kdelta'} one can choose $\omega'$ such that the angle $\theta_{1}$ between $\vec{v}$ and $\vec{w}$ is bounded from below by $\varepsilon$ (so this is much bigger than the error between the tangent plane through $F^{-m}(W^u(F^n(\hy,\ho))$ at $y$ and the tangent plane through $F^{-m}(W^u(F^n(\hx,\ho))$ at $x$).

\begin{claim}
    There exists $m'$, $c'$ and $d'$ such that for all $m>m'$,
    $$\Vert w\Vert = \eta = c\Vert v\Vert = c\delta(m)M, \ \ \ \Vert v-w\Vert = d\Vert v\Vert = d\delta(m)M$$
    for some $c\in [\frac{1}{c'},c']$, $d\in [\frac{1}{d'},d']$.
\end{claim}
\begin{proof}
Fixing $\delta'>0$ small we choose $m'$ large enough so that $F^{-m}(W^u(F^n(\hx,\ho))$ is within $\delta'/M_{exp}$ of $W^s(x,\sigma^{-n'}(\tilde{\omega}))$. Then we note that $(x,\sigma^{-n'}(\tilde{\omega})),(y,\sigma^{-n'}(\tilde{\omega})) \in K_2'\subset K_{ang}$ and so $$\angle(E^s(x,\sigma^{-n'}(\tilde{\omega})),E^u(x,\sigma^{-n'}(\tilde{\omega})) ) = \angle(E^s(x,\sigma^{-n'}(\tilde{\omega})),E^u(x,\omega') )>\Theta.$$

The law of sines says that $\Vert v-w\Vert = \Vert v\Vert \frac{\sin(\theta_1)}{\sin(\beta)} = \delta(m)M\frac{\sin(\theta_1)}{\sin(\beta)} $ where $$\beta = \pi-\theta_1-\theta_2,$$ $\theta_2 \in [\Theta,\pi-\Theta]$, and $\theta_1\in [\varepsilon ,\pi-\varepsilon]$. We let $d = \frac{\sin(\theta_1)}{\sin(\beta)}$ which is clearly bounded. 

The law of sines similarly gives the required relation between $\Vert w\Vert $ and $\Vert v\Vert $.
\end{proof}

Let $\psi_W(z) = \vec{w}z + \vec{w}^{\perp}g(z)$ be the parametrization of $\exp^{-1}_y(W^s(y,\omega'))$. The function $g$ is an error term representing the Taylor remainder and  $|g(z)|\leq C''|z|^2$ by Lemma \ref{epb}. We will take $m$ large enough so that $C''(M\delta(m)c')^3\ll \frac{\delta(m)M}{2d'}$

Then, using Taylor's theorem we get for some $z'$,
\begin{align*}
    \gamma \cdot M = \Vert\vec{v}-\vec{w}-\vec{w}^{\perp}g(z')\Vert &\geq \left| \Vert\vec{v}-\vec{w}\Vert - \Vert\vec{w}^{\perp}g(z')\Vert \right|\\
    &\geq \frac{1}{2d'}\delta(m)M .
\end{align*}

Now for the upper bound we see that 

 \begin{align*}
 \gamma \cdot M = \Vert\vec{v}-\vec{w}-\vec{w}^{\perp}f(z')\Vert   
 &\leq \Vert\vec{v}\Vert + \Vert\vec{w}\Vert + \Vert\vec{w}^{\perp}g(z')\Vert \\
 &\leq M\delta(m) + Mc'\delta(m) + C''(Mc'\delta(m))^3 \\
&\leq M\delta(m) + Mc'\delta(m) + \frac{\delta(m)M}{2d'} \\
&= (M+Mc' +\frac{M}{2d'}) \delta(m).
\end{align*}
We take $M_a = \max\{(M+Mc' +\frac{M}{2d'}) , \frac{2d'}{M}\}$. This completes the proof.
\end{proof}

\begin{lemma}\label{omega+1}
    For the choice of $(\omega')^+$ made in Lemma \ref{dvg}, the projection of $L_y$ along $W^s(y,\omega')$ onto $W^u(x,\omega')$ does not intersect $L_x$. 
\end{lemma}

We will exploit the fact that $L_x$ and $L_y$ are almost in the same plane and the following fact about complex subspaces:

 \begin{fact} \label{C}
    If two subspaces of complex dimension 1 intersect, they are the same (excluding their intersection at the origin).
\end{fact}

First I will explain things heuristically, thinking of all our stable and unstable manifolds as complex 1 dimensional subspaces:

\begin{sketch}[of Lemma \ref{omega+1}]
    
Let $P' \defeq F^{-m}(W^u(F^n(\hx,\ho)))$ (where $n = \lfloor T+k_0\rfloor$). Since $U^+[F^n(\hx,\ho)]\subset W^u(F^n(\hx,\ho))$, we have that $L_x \subset P'$. At the points $F^n(\hx,\ho)$ and $F^n(\hy,\ho)$, the corresponding $U^+[F^n(\hx,\ho)]$ and $U^+[F^n(\hy,\ho)]$ are almost in the same plane, $W^u(F^n(\hx,\ho))$ (i.e. $W^u(F^n(\hx,\ho))$ and $W^u(F^n(\hy,\ho))$ where $U^+[F^n(\hx,\ho)]$ and $U^+[F^n(\hy,\ho)]$ live are very close). We calculated an upper bound on the error and showed it goes to zero as $T_{big} \to \infty$ in Corollary \ref{finals}. After flowing by $m$ we showed that $W^u(F^n(\hx,\ho))$ and $W^u(F^n(\hy,\ho))$ do not move too far apart because of the size of $m$, in fact we showed that this error still goes to zero in Lemma \ref{bdnded2}. Hence $L_y$ is almost in $P'$. Since we have a positive measure subset $V$ from which to choose $\omega'$ after flowing backwards by $m$, by Lemma \ref{dvg} we can pick $\omega'$ such that $W^s(x,\omega')\neq P'$. We know that $P'$ and $W^s(x,\omega')$ both contain $x$ and by our Fact \ref{C} we know that then they must not share any other points; in particular, $L_y$ belongs (roughly) to $P'$ and hence not to $W^s(x,\omega')$. Further, $W^s(x,\omega')\cap W^u(x,\omega') = \{x\}$ and $W^s(x,\omega')$ and $W^s(y,\omega')$ are roughly parallel, hence $W^s(y,\omega')\cap W^u(x,\omega') = \{y'\}$ where $y'\neq x$ (so $y$ does not project to $x$). Similarly, any other point of $L_y$ lives (roughly) in $P'$ and hence not in $W^s(x,\omega')$ and so does not project to $\{x\}$ because $W^s(x,\omega')$ and $W^s(y,\omega')$ are parallel. This suffices because $\{x\}$ is the only point of $L_x$ on $W^u(x,\omega')$ by Fact \ref{C}.

\end{sketch}

\begin{proof}(of Lemma \ref{omega+1})
We will first approximate all the stable/unstable manifolds by their image under exp of $E^s/E^u$. We do the same for the unstable supports. Now write $P'$ as the tangent plane of $F^{-m}(W^u(F^n(\hx,\ho))$ at $x$, let $Q'$ be the tangent plane of $F^{-m}(W^u(F^n(\hy,\ho))$ at $y$, and $L_x$ and $L_y$ now are the tangent lines through $x$ and $y$ resp. of $F^{-m}(U^+[F^n(\tilde{x},\tilde{\omega})])$ and $F^{-m}(U^+[F^n(\tilde{y},\tilde{\omega})]$ respectively. 

Now we will need that (the images of) $E^s(x,\omega')$ and $E^s(y,\omega')$ are roughly parallel. Note that $G^s$ is uniformly continuous on $K_2$ and that $d_X(x,y) \to 0$ by Lemma \ref{bdnded2} and so we can take this angle to be as small as necessary. 

Now we begin the argument; the distance between $P'$ and $Q'$ was calculated to be upper bounded by $d_m$ (which goes to zero as $m\to \infty$ by Lemma \ref{bdnded2}) hence $L_y$ is in $P'$ up to error $d_mM_{exp}$ which can be made arbitrarily small. Now, $\omega'$ in Lemma \ref{dvg} is chosen such that $E^s(x,\omega')$ is bounded away from $P'$ by much more than this error between $L_y$ and $P'$. We also see that $E^s(x,\omega')\neq P'$, both are complex planes. Note that since $E^s(x,\omega')$ and $E^s(y,\omega')$ are roughly parallel (again we can make $\theta_s$ appropriately small) and distinct and $E^s(x,\omega')\cap E^u(x,\omega') = \{x\}$ so we must have that the projection of any point on $L_y$ along $E^s(y,\omega')$ does not intersect $x$. Going back to the original problem of projecting along $W^s(y,\omega')$, the error with $E^s(y,\omega')$ approximating $W^s(y,\omega)$ goes to zero for $T_{big}$ large and so we can also say that the projection of any point on $L_y$ along $E^s(y,\omega')$ does not intersect $x$. 
\end{proof}

Now $y'$ and $y$ both belong to $W^s(y,\omega')$ so as we flow forwards under $\omega'$ this distance contracts and so $\gamma$ is a good proxy for $\delta$.

Recall that we were working towards having our end points be $\epsilon_f$ apart in ambient coordinates (up to some bounded constant).

Now that we know that $\gamma \neq 0$ and is comparable to $\delta$, we want to work with this instead of $\delta$ because this is a distance that lives in $W^u(x,\omega')$ and we can use (modified) normal form coordinates (NFCs) (see Equation \eqref{newNFCform}). When we pass to (modified) NFCs we can solve for $\ell(m)$, i.e. the amount we want to flow by from the points $(x,\omega')$ and $(y,\omega')$ regardless of the choice of $\omega'$. Ultimately we will have to compute $\ell$ slightly differently, but the idea is that $\gamma$ gets changed by (modified) NFCs to some $C_1\gamma$ where $C_1 \in [\frac{1}{M_{NFC}},M_{NFC}]$ and we would like to grow this distance to be $\epsilon_f\cdot M_{NFC}$ (so that when we go back to ambient coordinates it is bounded from below by $\epsilon_f$). Then we solve $C_1 \gamma e^{\ell} = M_{NFC}\epsilon_f$ for $\ell$ and get
\begin{align}
    \ell(m) = \ln\left(\frac{M_{NFC}\epsilon_f}{C_1\gamma(m)}  \right).
\end{align}
Since $C_1$ changes with $m$, we instead replace it with $\frac{1}{M_{NFC}}$ (i.e. its lower bound). Solving $\ell = \ln\left(\frac{M_{NFC}\epsilon_f}{\frac{1}{M_{NFC}}\gamma}  \right)$ means we would most certainly reach $\epsilon_f$ or go past $\epsilon_f$ by a (uniformly) bounded amount since $C_1$ is bounded from above as well. This fact is important to Section \ref{sec:conc1}.

To make the computation in Lemma \ref{BiLip} simpler, we use $D_m$ instead of $\gamma$ which would make the formula look more like this: 
\begin{align} \label{olddef1}
    \ell(m) = \ln\left(\frac{M_{NFC}\epsilon_f}{ \frac{1}{M_{NFC}M_rM_a} D_m}  \right).
\end{align}
Note that the constants in front of $D_m$ appear because we are thinking of $D_m$ in relation to $\gamma$. $D_m$ is related to $\delta(m)$ by a constant bounded from below by $\frac{1}{M_r}$ and $\delta(m)$ is related to $\gamma$ by $C'$, the lower bound being $\frac{1}{M_a}$. In this we clearly surpass $\epsilon_f$ still the distance is (uniformly) bounded above since $C'$ and the change between $\delta(m)$ and $D_m$ is bounded from above as well by $M_a$ and $M_r$ respectively. This fact is again important to Section \ref{sec:conc1}.

However, this definition for $\ell(m)$ is not continuous in $m$ because $D_m$ changes only every integer (up to taking into account $k_1$). We need continuity for Section \ref{sec:pink}. Note that we showed that $D_m$ is increasing in $m$, so we define $\ell$ to be the piecewise linear interpolation between these points. We can write this explicitly as 
\begin{equation} \label{betterdef1}
\begin{aligned}
    \ell(m) &= \ln\left(\frac{M_{NFC}\epsilon_f}{ \frac{1}{M_{NFC}M_rM_a} D_m }  \right) \\
&~~+ (m-\lfloor m \rfloor ) \left(\ln\left(\frac{M_{NFC}\epsilon_f}{ \frac{1}{M_{NFC}M_rM_a} D_{m+1} }  \right) - \left(\ln\left(\frac{M_{NFC}\epsilon_f}{ \frac{1}{M_{NFC}M_rM_a} D_{m} } \right) \right) \right)
\end{aligned}
\end{equation}
which simplifies to 
\begin{align}
    \label{bettererdef1}
    \ell(m) = \ln\left(\frac{M_{NFC}\epsilon_f}{ \frac{1}{M_{NFC}M_rM_a} D_m }  \right) + (m-\lfloor m \rfloor ) \ln\left( \frac{D_m}{D_{m+1}}\right).
\end{align}

We make one last modification to this definition of $\ell$, to correct concerns that in the interpolation, if $m$ is between two integers, technically $\ell(m)$ computed using Equation \eqref{olddef1} is smaller than that of $\ell(m)$ computed using Equation \eqref{bettererdef1}, we tack on a last constant, let $$M_{D} \defeq \left( e^{-(\lambda^-  +\delta)} + \frac{1}{2}C \right),$$ then let
\begin{align} \label{ellrel1}
    \ell(m) = \ln\left(\frac{M_{NFC}\epsilon_f}{ \frac{1}{M_{NFC}M_rM_aM_D} D_m }  \right) + (m-\lfloor m \rfloor ) \ln\left( \frac{D_m}{D_{m+1}}\right).
\end{align}
This expression, \eqref{ellrel1}, is how we define $\ell$. Note that the second term is negative because $\ell$ decreases in $m$.

Additionally, take note of the fact that $D_m$ increases with $m$ because $m$ is large enough; see Lemma \ref{bdnded3}.

\subsubsection{Flowing by $\ell$}
\label{sec:pink}

 Once we have made our initial adjustments (from the starting points to the points $(\tilde{x},\tilde{\omega}), \ (\tilde{y},\tilde{\omega})$) we have a `V' shape where we go backwards by $m$ (under the standard flow) and forwards by $\ell$ (under the time change flow) where $\ell$ is determined by $m$ and the distance $\epsilon_f$ that we reach (within some constant) at the end points using (modified) NFCs. But we need the `legs' of the `V', i.e. the forward and backward directions to be of the same length to run the argument in Section \ref{sec:MCT}. This is why we will introduce extra points in this section.

We fix $T$ and the points $(\tilde{x},\tilde{\omega},k_1)$ and $(\tilde{y},\tilde{\omega},k_1)$ from Section \ref{sec:blue}, and for each $m$ we pick in Section \ref{sec:green} (from the positive measure subset of `good choices' in some interval, see Lemma \ref{m2}) we get the points $(x,\omega',k_2)$ and $(y,\omega',k_2)$ where $\omega'\in \Omega$ is any vector such that $\omega'^- = \sigma^{-\lfloor k_1-m \rfloor}(\tilde{\omega})^-$ and $\omega'$ belongs to the positive measure subset $V$ corresponding to the points $(x,\sigma^{-\lfloor k_1-m \rfloor}(\tilde{\omega}))$ and $(y,\sigma^{-\lfloor k_1-m \rfloor}(\tilde{\omega}))$ which live in $\tilde{K}_2$.

Additionally, for each choice of $m$ we can compute $\ell(m)$ as we did in Section \ref{sec:green}. This is independent of the choice of future $\omega'$ that we pick from some good set. Starting from the points $(\tilde{x},\tilde{\omega})$ and $(\tilde{y},\tilde{\omega})$ (i.e. we are in the set $K_1$), we want to flow forwards by an amount that will allow us to complete our `V' shape.

Let us reiterate our `V' shape more definitively. Since the points $(\tilde{x},\tilde{\omega},k_1)$ and $(\tilde{y},\tilde{\omega},k_1)$ are fixed at this point, our `V' is determined entirely by a choice of $m$ (which is in a fixed interval at this point). The length of the legs is $\ell(m)$ under time change flow, with the middle point of the `V' being the points $(x,\omega',k_2)$ and $(y,\omega',k_2)$. Note that because $\tau$ in the  `time change' flow depends on which point you start at, the resulting pair of points after flowing by $\ell(m)$ will have different $\Omega$ and $[0,1[$-components compared to one another, which was never the case with the `standard' flow. We have to convert $m$, that we flowed by under the standard flow, into the quantity we would flow by from $(x,\sigma^{n'}(\tilde{\omega}),k_2)$ and $(y,\sigma^{n'}(\tilde{\omega}),k_2)$ to $(\tilde{x},\tilde{\omega},k_1)$ and $(\tilde{y},\tilde{\omega},k_1)$ under the time change flow. This quantity will differ between points. We define $m_{\tau, (\tilde{x},\tilde{\omega},k_1)}$ and $m_{\tau, (\tilde{y},\tilde{\omega},k_1)}$ to be these quantities resp. Now to complete our `V' we flow backwards from $(\tilde{x},\tilde{\omega},k_1)$ and $(\tilde{y},\tilde{\omega},k_1)$ by $d(m,(\tilde{x},\tilde{\omega},k_1)) \defeq \ell(m) + (-m)_{\tau, (\tilde{x},\tilde{\omega},k_1)} $ and $d(m, (\tilde{y},\tilde{\omega},k_1)) \defeq \ell(m) + (-m)_{\tau, (\tilde{y},\tilde{\omega},k_1)} $ (note that both functions $d$ are positive). We denote the resulting points $(\overline{x}, \overline{\omega}_x, k_x)$ and $(\overline{y},\overline{\omega}_y,k_y)$ resp. Note that $k_x$ and $k_y$ are not necessarily the same; neither are $\overline{\omega}_x$ and $\overline{\omega}_y$.

\begin{remark}
    The minus sign was included next to the $m$ before doing the time change because we want to remember to compute it using $\tau^-$, i.e. we are going backwards and we need to remember that as we compute the time changed value.
\end{remark}

For simplicity, we will take our overall `V' and divide it into its two components, the `V' we have at the fixed point $(\tilde{x},\tilde{\omega},k_1)$ will be called `V1' and the `V' we have at the point  $(\tilde{y},\tilde{\omega},k_1)$ will be called `V2'. This allows us to drop some notation regarding time changed $m$ and the function $d$. Let us just focus on `V1' as everything we have to say will hold for `V2' as well. Rename $m_{\tau, (\tilde{x},\tilde{\omega},k_1)}$ as $m_1^{\tau}$ and $d(m,(\tilde{x},\tilde{\omega},k_0))$ as $d_1(m)$. We will have to show that $d_1$ is Bi-Lipschitz in $m$ (on a good set) in order to pick $m$ such that flowing backwards by $m$ from $(\tilde{x},\tilde{\omega},k_1)$ we land in the set $\tilde{K}_2$, and flowing forwards by $d_1(m)$ we land in the set $K_3$. We can use Lemma \ref{birkhoff} to do one or the other, but we have to run a different argument to ensure we can do both at the same time. A similar argument will work for `V2' where we use the notation $m_2^{\tau}$ and $d_2(m)$ for $m_{\tau, (\tilde{y},\tilde{\omega},k_1)}$ and $d(m,(\tilde{y},\tilde{\omega},k_0))$ respectively.

\begin{remark}\label{bi}
    It is worthwhile to recall how flowing backwards works. When one passes zero in the $[0,1[$-component, one applies the inverse of $f_{-1}$(i.e. the automorphism in the -1 spot of the $\Omega$-component). So if one starts at $k=0.22$ (the $[0,1[$-component) and flows backwards by $0.23$ one has applied a single diffeomorphism.

    The time changed flow going backwards was described in Section \ref{sec:invert}. Being closer to 1 in the $[0,1[$-component is being further away from zero and hence being further away from reaching $\tau^-$ of the point (where $\tau^-$ is always negative). 

\end{remark}

\begin{lemma}\label{BiLip}
    Take the interval $[m_1,m_2]$ from Section \ref{sec:green}. The map $$d_1 : [m_1,m_2] \to ]-,\infty,0],$$ is Bi-Lipschitz, (similar for $d_2$ and we get the same Lipschitz constant works in this case).
\end{lemma}
\begin{proof}
    First we want to show that there is $L>0$ such that $|d_1(m) - d_1(m')|\leq L|m-m'|$ for arbitrary $m$ and $m'$. 
    Recall the definition of $d_1$ from Equation \eqref{d} and plug in:
\begin{align*}
    |d_1(m) - d_1(m')| &= |(-m)_1^{\tau}-(-(m'))_1^{\tau} + \ell(m) -\ell(m')|\\
    &\leq |(-(m'))_1^{\tau}-(-m)_1^{\tau}| + |\ell(m) - \ell(m')|.
\end{align*}
By Equation \eqref{ellrel1} we have that this is equal to 
\begin{align*}
    |(-m)_1^{\tau}-(-(m'))_1^{\tau}| + |\ln\left(\frac{M_{NFC}\epsilon_f}{\frac{1}{M_{NFC}M_rM_aM_D}D_{m'} }\right) + (m'-\lfloor m' \rfloor)\ln\left(\frac{D_{m'}}{D_{m'+1}} \right) \\ - \ln\left(\frac{M_{NFC}\epsilon_f}{\frac{1}{M_{NFC}M_rM_aM_D}D_m }\right) - (m-\lfloor m \rfloor)\ln\left(\frac{D_m}{D_{m+1}}\right) |
\end{align*}

Let us control the first term in the sum above. Recall Remark \ref{bi} and Section \ref{sec:invert}. Note also that $k_1$ plays a role here, we had defined $d_1$ once we fixed a point that had this $[0,1[$-component. It will show up in our computations. Making sure to keep the convention that $k_1$ is written as a positive number between 0 and 1, we can write
\begin{align}
    (-m)^{\tau}_1 &= k_1\tau^-(\tilde{x},\tilde{\omega}) + \tau^-(F^{-1}(\tilde{x},\tilde{\omega})) +\dots \tau^-(F^{\lfloor k_1-m\rfloor + 1}(\tilde{x},\tilde{\omega})) + a, \\
    (-m')^{\tau}_1 &= k_1\tau^-(\tilde{x},\tilde{\omega}) + \tau^-(F^{-1}(\tilde{x},\tilde{\omega})) +\dots \tau^-(F^{\lfloor k_1-m' \rfloor +1}(\tilde{x},\tilde{\omega})) + b.
\end{align}
Note that the first term $k_1\tau^-(\tilde{x},\tilde{\omega})$ appears because once we pass 0 we hit it with the automorphism $f_{-1}^{-1}$. For example if $k_1=0.22$, then we are already $78\%$ on our way to hitting $\tau^-(\tilde{x},\tilde{\omega})$, the remainder that we need is that $22\%$. The $a$ and $b$ terms show up because 
we have not quite reached the next $\tau^-$ value, i.e. $a\in [0, \tau^-(F^{\lfloor k_1-m\rfloor }(\tilde{x},\tilde{\omega}))[$ and $b\in [0, \tau^-(F^{\lfloor k_1-m'\rfloor }(\tilde{x},\tilde{\omega}))[$. We can easily compute $a$ and $b$ from what we have as clearly the following relations holds:
\begin{align}
    \psi_m \defeq \lfloor k_1-m\rfloor - (k_1-m) +1 &= \frac{a}{\tau^-(F^{\lfloor k_1-m\rfloor }(\tilde{x},\tilde{\omega}))}>0,\\
       \psi_{m'} \defeq \lfloor k_1-m'\rfloor -(k_1-m') +1 &= \frac{b}{\tau^-(F^{\lfloor k_1-m'\rfloor }(\tilde{x},\tilde{\omega}))}>0.
\end{align}
WLOG assume $m'>m$ (remember these are positive numbers), then $$(-(m'))^{\tau}_1> (-m)^{\tau}_1.$$ We have
\begin{align*}
    |(-(m'))^{\tau}_1 - (-m)^{\tau}_1| = \tau^-(F^{\lfloor k_1-m' \rfloor +1}(\tilde{x},\tilde{\omega})) + \dots  \tau^-(F^{\lfloor k_1-m\rfloor }(\tilde{x},\tilde{\omega})) + b-a
    \end{align*}
    \begin{align*}
    =| \tau^-(F^{\lfloor k_1-m' \rfloor +1}&(\tilde{x},\tilde{\omega})) + \dots  \tau^-(F^{\lfloor k_1-m\rfloor }(\tilde{x},\tilde{\omega})) \\ \nonumber &+ \psi_{m'}\cdot \tau^-(F^{\lfloor k_1-m'\rfloor }(\tilde{x},\tilde{\omega})) - \psi_m\cdot \tau^-(F^{\lfloor k_1-m\rfloor }(\tilde{x},\tilde{\omega}))|,
\end{align*}
\begin{align}\label{expanded1}
    =| \tau^-(F^{\lfloor k_1-m' \rfloor +1}&(\tilde{x},\tilde{\omega})) + \dots + \tau^-(F^{\lfloor k_1-m\rfloor -1 }(\tilde{x},\tilde{\omega}))\\ \nonumber &+  \tau^-(F^{\lfloor k_1-m\rfloor }(\tilde{x},\tilde{\omega}))(1-\psi_m) + \psi_{m'}\cdot \tau^-(F^{\lfloor k_1-m'\rfloor }(\tilde{x},\tilde{\omega}))|, 
\end{align}
\begin{align*}
    \leq | \tau^-(F^{\lfloor k_1-m' \rfloor +1}(\tilde{x},\tilde{\omega}))| &+ \dots + |\tau^-(F^{\lfloor k_1-m\rfloor -1 }(\tilde{x},\tilde{\omega}))|\\ &+  |\tau^-(F^{\lfloor k_1-m\rfloor }(\tilde{x},\tilde{\omega}))|(1-\psi_m) + \psi_{m'}|\tau^-(F^{\lfloor k_1-m'\rfloor }(\tilde{x},\tilde{\omega}))|.
\end{align*}

Using the bounds on $\tau$ we get that this is bounded above by 

\begin{align*}
   \left[\left( \left( \lfloor ( k_1-m \rfloor -1) - (\lfloor k_1-m' \rfloor +1 \right)+1\right) + (1-\psi_m) + \psi_{m'}\right] (\lambda^++\epsilon) =(\lambda^++\epsilon)(m'-m).
\end{align*}

On the other hand, we have to lower bound this quantity too. So we consider again $m'>m$ and take Equation \eqref{expanded1}. Again, the $\tau^-$-terms are lower bounded (in absolute value) by $\lambda^+-\epsilon$, so we get that Equation \eqref{expanded1} is bounded below by $(\lambda^+-\epsilon)(m'-m)$.

    Now we control the second term. It reduces to 
    \begin{align*}
       |\ln\left(\frac{M_{NFC}\epsilon_f}{\frac{1}{M_{NFC}M_rM_aM_D}D_{m'} }\right) + (m'-\lfloor m' \rfloor)\ln\left(\frac{D_{m'}}{D_{m'+1}} \right) \\ - \ln\left(\frac{M_{NFC}\epsilon_f}{\frac{1}{M_{NFC}M_rM_aM_D}D_m }\right) - (m-\lfloor m \rfloor)\ln\left(\frac{D_m}{D_{m+1}}\right) |
    \end{align*}
    
    We learned in Section \ref{sec:green} how $D_{m+1}$ and $D_m$ are related, see Lemma \ref{bdnded3} or Equation \eqref{Dmrel}. When we apply an automorphism, we never grow bigger than a factor of $\ln(e^{-(\lambda^--\epsilon)}+\frac{1}{2}C)$. Remember throughout that $D_m$ is constant on an interval of length 1. 

We deal with cases:

The first case is that $m'>m$ but $(m'-m)< 1$. There are two subcases, one is that $D_m=D_{m'}$, the other is that $D_{m'}<D_m$.

Let us deal with the first subcase: $D_m = D_{m'}$, then 
\begin{align*}
    \ell(m)-\ell(m') = (m'-m')\ln\left( \frac{D_{m+1}}{D_m} \right) \leq (m'-m) \ln(e^{-(\lambda^--\epsilon)}+\frac{1}{2}C).
\end{align*}
    So we are done.

    Now we deal with the second subcase: $D_{m'}<D_m$. Then really $D_{m'}=D_{m+1}$.
    \begin{align*}
        \ell(m)-\ell(m') &= \ln\left(\frac{D_{m+1}}{D_m} \right) - (m'-\lfloor m'\rfloor )\ln\left(\frac{D_{m+1}}{D_{m+2}} \right) - (m - \lfloor m \rfloor )\ln\left( \frac{D_{m+1}}{D_m} \right)\\ \nonumber
        &= (1-m+\lfloor m \rfloor)\ln\left( \frac{D_{m+1}}{D_m}\right) + (m'-\lfloor m' \rfloor) \ln\left( \frac{D_{m+2}}{D_{m+1}} \right)\\ \nonumber
        &\leq (1-m+\lfloor m \rfloor + m' - \lfloor m' \rfloor) \ln(e^{-(\lambda^--\epsilon)}+\frac{1}{2}C)\\ \nonumber &= (m'-m)\ln(e^{-(\lambda^--\epsilon)}+\frac{1}{2}C).
    \end{align*}
    The last equality holds because the condition of the subcase is really that $\lfloor m \rfloor + 1 = \lfloor m' \rfloor$.

The second case is $m'>m$ but $(m'-m)\geq 1$. We could have stopped with the above and just said that, since $d$ is continuous on a compact interval,  $|d(m)-d(m')|$ is uniformly bounded from above and hence we are done. However, we actually want that this Lipschitz constant computed in this proof is independent of the interval $[m_1,m_2]$. The only restriction we have on $[m_1,m_2]$ is that we need $m$ large enough so that that Lemma \ref{bdnded3} holds. Hence we do more work.

Note first that $\ell(m) = \ell(m') \leq \ell(\lfloor m \rfloor) - \ell(m')$. We will work with this quantity.
\begin{align}
    \frac{\ell(\lfloor m \rfloor) - \ell(m')}{m'-m} = \frac{\ln \left(\frac{D_{m'}}{D_m} \right) - (m'-\lfloor m' \rfloor)\ln \left( \frac{D_{m'}}{D_{m'+1}}  \right)}{m'-m}.
\end{align}
Since $(m'-m)\geq 1$, we have that $(m'-m)>m'-\lfloor m' \rfloor$. Also note that $D_{m'}<D_{m'+1}$ and so $\ln \left( \frac{D_{m'}}{D_{m'+1}}  \right)$ is negative.
\begin{align*}
    &\leq \frac{\ln\left( \frac{D_{m'}}{D_m}\right)}{m'-m} + \ln \left( \frac{D_{m'+1}}{D_{m'}} \right)\\
    &\leq \frac{\lfloor m'-m \rfloor}{m'-m}\ln(e^{-(\lambda^--\epsilon)}+\frac{1}{2}C) + \ln(e^{-(\lambda^--\epsilon)}+\frac{1}{2}C) \\
    &\leq 2\ln(e^{-(\lambda^--\epsilon)}+\frac{1}{2}C).
\end{align*}

Now we do the other side of the computation. Note that $d$ is decreasing because as $m$ increases, $(-m)_1^{\tau}$ decreases and $\ell$ decreases (in fact both do so strictly and hence so does $d$). Take $m,m'\in [m_1,m_2]$ without loss of generality, $m'>m$, then
\begin{align}
    |d_1(m)-d_1(m')| = d_1(m)-d_1(m') &\geq (-m)_1^{\tau}-(-m')_1^{\tau} + \ell(m)-\ell(m')\\
    &= (\lambda^+-\epsilon)(m'-m) - \ell(m+1) + \ell(m) \geq \lambda^+-\epsilon.
\end{align}

We get Lipschitz constant to be 
\begin{align} \label{Lipconst}
L \defeq \max\{ 2\ln(e^{-(\lambda^--\epsilon)}+\frac{1}{2}C) +\lambda^++\epsilon, \frac{1}{\lambda^++\epsilon} \}.
\end{align}
\end{proof}

Now that we have established that $d_1$ is $L$-BiLipschitz, a similar argument works for $d_2$, we get the same Lipschitz constant; now we can use the following lemma.

\begin{lemma}\label{Ld}
    Consider $d:[a,b] \to [0, \infty[$ where $0\leq a,b <\infty$. If $E\subset [a,b]$ is such that $Leb(E)<\epsilon (b-a)$ and $d$ is $L$-Lipschitz then $Leb(d(E))<L\cdot \epsilon (b-a)$. Further, if additionally $d$ is BiLipschitz, then for any subset $F \subset [a,b]$ such that $Leb(F)>\eta (b-a)$, we have $Leb(d(F))>\frac{1}{L}\cdot \eta (b-a)$.
\end{lemma}
\begin{proof}
    The Lebesgue measure is regular and hence $$Leb(E) = \inf\{Leb(A):A\supset E,\text{ open and measurable}\}.$$ Since any open subset $A$ can be written as a countable disjoint union of open intervals, if we can prove that for $A = ]a',b'[\subset [a,b]$, $Leb(d(A))\leq L \cdot Leb(A)$ we will be done.

    Since $d$ is Lipschitz on $[a,b]$, it is continuous on $[a,b]$. If we close the interval $]a',b'[=A$ by adding the endpoints it is now compact and $d$ attains a max and min on $[a',b']$; call the maximizer and minimizers $c_1,c_2$. Then $d(]a',b'[)\subset [d(c_1),d(c_2)]$ and so the measure of $d(A) = d(]a',b'[)$ is bounded above by $|d(c_1)-d(c_2)|$. Since $d$ is Lipschitz, we have that $$|d(c_1)-d(c_2)|\leq L|c_1-c_2|\leq L|b'-a'|,$$ (since $c_1,c_2\in [a',b']$). Hence $Leb(d(A))\leq L\cdot Leb(A).$

    We conclude this is true for any open set $A$. Then by regularity, for any $A\supset E$, $Leb(d(A))\leq L \cdot Leb(A)$, taking infinimum of both sides over open sets $A\supset E$ we get $\inf_{A\supset E}Leb(d(A)) \leq L \cdot Leb(E)$, but since $d(E)\subset d(A)$ for every $A\supset E$, by monotonicity $Leb(d(E))\leq \inf_{A\supset E}Leb(d(A))$ so we have the result.

    The other argument is similar.
\end{proof}

Note that since $d$ is injective (as it is strictly decreasing), it is invertible on its image. Note also that the inverse of a Bilipschitz function is Bilipschitz with the same Bilipschitz constant.

\begin{lemma}\label{dd}
    For our choice of $m_1$ and $m_2$ as in Section \ref{sec:green} Equation \eqref{mchoice}, and $T_{big}$ large enough, we have that the interval $[d_1(m_2),d_1(m_1)]$ inside $[0,d_1(m_1)]$ has proportion at least $(1-\frac{1}{1+\eta})$ in measure (where $\eta$ is as in Section \ref{sec:lusin} and Lemma \ref{bb}). A similar result holds for $d_2$.
\end{lemma}
\begin{proof}
For $T_{big}$ large enough we can use Lemma \ref{bb}. Let $\eta$ be as in this lemma.
\begin{claim}
   We have that $d_1(m_1)\geq (1+\eta)d_1(m_2)$.
\end{claim}

 \begin{proof}
 Let $M = M_{NFC}^2M_rM_aM_D$.

We can drop the interpolation terms for a small price and simply check that
\begin{align*}
    \ln\left(\frac{\epsilon_fM}{D_{\lfloor m_1+1\rfloor}}\right) + \left(-\frac{C_b}{10^{i+J}}T_{big}\theta_2\right)^{\tau} > (1+\eta)\left( \ln\left( \frac{\epsilon_fM}{D_{m_2}} \right) + \left( -\frac{C_b}{10^{i+2}} T_{big}\theta_2 \right)^{\tau}  \right).
\end{align*}

Note that $\tau$ is bounded (between $\lambda^+-\epsilon$ and $\lambda^++\epsilon)$. Also as $T_{big} \to \infty$, the length of $[m_1,m_2]$ goes to infinity, so we have that eventually $$\left(-\frac{C_b}{10^{i+3}}T_{big}\theta_2\right)^{\tau} \geq (1+\eta) \left( -\frac{C_b}{10^{i+2}} T_{big}\theta_2 \right)^{\tau},$$ (remember both terms are negative), so it suffices to compare the remaining terms (coming from the definition of $\ell$).

So we want to show $\ln\left(\frac{\epsilon_fM}{D_{\lfloor m_1+1\rfloor}}\right) \geq (1+\eta) \ln\left( \frac{\epsilon_fM}{D_{m_2}} \right)$ which reduces to showing
\begin{align}\label{ddd}
   (1+\eta) \ln(D_{m_2})> \ln(D_{\lfloor m_1+1 \rfloor}) + \ln(\epsilon_f M).
\end{align}
Remember $\epsilon_f$ is small and $\epsilon_fM<1$. Then this is equivalent to 
$D_{m_2}^{(1+\eta)} > D_{m_1}\epsilon_f M$.

This can be done for $T_{big}$ large by Lemma \ref{bb} because it gives us that $D_{m_2}^{(1+\eta)}>D_{m_1}$.
\end{proof}
Then, note that the Lebesgue measure of $[0,d_1(m_1)]$ is $d_1(m_1)$ while the Lebesgue measure of $[d_1(m_2),d_1(m_1)]$ is $d_1(m_1)-d_1(m_2)$ which by our claim is at least $d_1(m_1)(1-\frac{1}{1+\eta})$.    
\end{proof}

\begin{lemma}\label{goodlanding}
     Let $K\subset Z$ be such that $\hat{m}(K) = 1-\epsilon' = \gamma$ for some $\epsilon'>0$ small. Let $\delta = \alpha \gamma$ for some $\alpha>0$ small, and fix $\epsilon>0$ small. Let $E^c$ be the corresponding Birkhoff set (for this $\delta$, $\epsilon$ and $K$) from applying Lemma \ref{birkhoff} and $S$ the corresponding parameter from this lemma. Take $b>S$ and $\eta\in ]0,1]$, and consider $a$ such that $b>(1+\eta)a$. Let $\xi>0$ be small and $x,y\in E^c$. If $$\beta = 2\gamma(1-\alpha)-1>(1-\xi) + \xi\left( \frac{1}{1+\eta}\right),$$ and if $G\subset [a,b]$ is such that $G = \{t\in [a,b] : F^t(x),F^t(y)\in K\}$, then $G$ is at least $(1-\xi)\times 100 \%$ of the Lebesgue measure of the interval $[a,b]$. That is to say, $G$ has Lebesgue measure at least $(b-a)(1-\xi)$.
\end{lemma}
\begin{proof}
    Similar to the argument in Lemma \ref{bigint}, we have that given $x,y\in E^c$ Lemma \ref{birkhoff} gives us that the set of `good points' in $[0,b]$,
    $$G' = \{t\in [0,b] : F^t(x),F^t(y)\in K \} \subset [0,b],$$
    has proportion $\beta = 2\gamma(1-\alpha)-1$ inside $[0,m_2]$ in (Lebesgue) measure, i.e. $G'$ is $\beta \times 100\%$ of $[0,m_2]$ in measure.

    Now we are interested in the the proportion of `good points' inside $[a,b]$ where $b>(1+\eta)a$. This means that $a<\frac{b}{(1+\eta)}$ and so $[\frac{b}{1+\eta},b] \subset [a,b]$. We can analyse the proportion of good points in $[\frac{b}{(1+\eta)},b]$ to get a lower bound. 

    We know that inside $[0,b]$, a proportion of $(1-\beta)$ is `bad', i.e. belongs to $B'=[0,b]\setminus G'$. This is $(1-\beta)\times 100\%$ of the interval $[0,b]$. Let us assume (the worst case) that all of the bad points in $[0,b]$ belong to $[\frac{b}{(1+\eta)},b]$. Without loss of generality we can assume then that inside the interval $[\frac{b}{(1+\eta)},b]$, the interval $[\frac{b}{(1+\eta)},b-(1-\beta)b]$ is all good points and the interval $[b-(1-\beta)b,b]$ is all bad points. Then the proportion of good points in $[\frac{b}{(1+\eta)},b]$, i.e. the proportion of $G$ inside $[\frac{b}{(1+\eta)},b]$ is
    $$P\defeq \frac{\text{Leb}(G)}{\text{Leb}([\frac{b}{1+\eta},b])}=\frac{b-(1-\beta)b-\frac{b}{(1+\eta)}}{b-\frac{b}{(1+\eta)}}.$$

    Now we want to show that $P \geq (1-\xi)$, then we are done. Solving
    $$\frac{b-(1-\beta)b-\frac{b}{(1+\eta)}}{b-\frac{b}{(1+\eta)}}\geq (1-\xi),$$
    for $\beta$ gives us exactly that we require $\beta>(1-\xi) + \xi\left( \frac{1}{1+\eta}\right)$. This completes the proof.
\end{proof}

Recall the definitions of $m_1$ and $m_2$ as given in Equation \eqref{mchoice} with $T_{big}$ is as appropriately large as specified in that section. With this, we have the following lemma.

\begin{lemma}\label{landing}
    Given the interval $[m_1,m_2]$ (where $m_1,\ m_2$ as specified above), there exists a positive measure subset $G$ such that,
    \begin{enumerate}
        \item if $m\in G$ then flowing backwards by $m$ from the points $(\tilde{x},\tilde{\omega},k_1)$ and $(\tilde{y},\tilde{\omega},k_1)$ lands us in the set $\tilde{K}_2\times [0,1[$,
        \item if $m\in G$ then flowing backwards by $d_1(m)$ and $d_2(m)$ from the points $(\tilde{x},\tilde{\omega},k_1)$ and $(\tilde{y},\tilde{\omega},k_1)$ resp. lands both points in the set $K_3$.
    \end{enumerate}
\end{lemma}
\begin{proof}
First let $\gamma_2 = m(\tilde{K}_2)$ and recall $\delta_2=\alpha_2\gamma_2$ in the definition of $E^c_2$ from Section \ref{sec:lusin}. We showed in Lemma \ref{m2} that there is some subset of positive measure of $[m_1,m_2]$, call it $G_1$, where we could choose $m$ from in Section \ref{sec:green} to get the points $(x,\sigma^{n'}\tilde{\omega})$ and 
$(y,\sigma^{n'}\tilde{\omega})$ (where $n' = \lfloor k_1-m \rfloor$) in the good set $\tilde{K}_2$ (i.e these are `good' choices of $m$). Now we will be explicit with the size of $G_1$.

Since $m_2>2m_1$ we can apply Lemma \ref{goodlanding} with $\eta=1$ and $\xi = \frac{1}{100L^4}$ since by construction $\beta > \left(1 - \frac{1}{200L^4} \right) = (1-\frac{1}{100L^4}) + \frac{1}{100L^4}\frac{1}{1+\eta}$ (see Section \ref{sec:lusin}). We get that $G_1$ as proportion $(1-\frac{1}{100L^4})$ inside $[m_1,m_2]$ and so $B_1=[m_1,m_2]$ has proportion $\frac{1}{100L^4}$ inside $[m_1,m_2]$.

We take this interval $[m_1,m_2]$ and apply $d_1$ to it, since it is a decreasing function we get the interval $[d_1(m_2),d_1(m_1)]$ and by Lemma \ref{Ld} we get that $$Leb([d_1(m_2),d_1(m_1)])>\frac{1}{L}(m_2-m_1).$$ We run the same Birkhoff argument for the points $(\overline{x},\overline{\omega}_x,k_x)$ and $(\overline{y},\overline{\omega}_y,k_y)$ using Lemma \ref{birkhoff} with $K=K_3$ (see Section \ref{sec:lusin}) using this interval $[d_1(m_2),d_1(m_1)]$. We get a set $G_2$ which is the set of `good' choices in $[d_1(m_2),d_1(m_1)]$ to flow by from the points $(\tilde{x},\tilde{\omega})$ and $(\tilde{y},\tilde{\omega})$ to land in the set $K_3$. We can do so since $|d_1(m_2)|>S_p$ where $S_p$ is the Birkhoff bound one gets with $E^c_3$ when using Lemma \ref{birkhoff}. Here, we are using $\hat{m}(K_3)=\gamma_2$ and $\delta=\alpha_3\gamma_3$ as in Section \ref{sec:lusin}; we let $$\beta_3 = 2\gamma_3(1-\alpha_3)-1.$$ By Lemma \ref{bb}, we have that $d_1(m_1)>(1+\eta)d_1(m_2)$ where $\eta$ is fixed in Section \ref{sec:lusin}. We take this to be the $\eta$ in Lemma \ref{goodlanding} with $\xi = \frac{1}{100L^2}$. Since $$\beta_3>(1-\frac{1}{100L^2}) + \frac{1}{100L^2}(1-\frac{1}{1+\eta}),$$ we get that $G_2$ has proportion $(1-\frac{1}{100L^2})$ inside $[d_1(m_2),d_1(m_1)].$ Let $$B_2=[d_1(m_2),d_1(m_1)]\setminus G_2,$$ it has proportion at most $\frac{1}{100L^2}$ inside $[d_1(m_2),d_1(m_1)]$.

Our ultimate goal will be to choose $m$ in the set $$S_1 \defeq G_1\cap d_1^{-1}(G_2\cap d_1(G_1)) \subset [m_1,m_2].$$ If we can do this, then being in $G_1$ ensures that we land in $\tilde{K}_2$ when we flow by $m$, but also that such an $m$ in $G_1$ will give a $d_1(m)$ that allows us to land in $K_3$ after flowing by $d_1(m)$. So we show that $S_1$ has large positive measure.

The measure of $B_1$ is $\frac{1}{100L^4}(m_2-m_1)$, when we apply $d_1$ to this set, its measure (by Lemma \ref{Ld}) is at most $(m_2-m_1)\frac{1}{100L^3}$ whereas the size of $[d_1(m_2),d_1(m_1)]$ is at least $\frac{1}{L}(m_2-m_1)$ (also by Lemma \ref{Ld}). So the proportion of $d_1(B_1)$ inside $[d_1(m_2),d_1(m_1)]$ is at most $\frac{1}{100L^2}$. This means that $d_1(G_1)$ has proportion at least $1-\frac{1}{100L^2}$ inside $[d_1(m_2),d_1(m_1)]$. Meanwhile $G_2$ has proportion $1-\frac{1}{100L^2}$ inside $[d_1(m_2),d_1(m_1)]$ and so $d_1(G_1)\cap G_2$ has proportion at least $1-\frac{2}{100L^2}$ inside $[d_1(m_2),d_1(m_1)]$. Let $$B_3 \defeq [d_1(m_2),d_1(m_1)]\setminus (d_1(G_1)\cap G_2),$$ it has proportion at most $\frac{2}{100L^2}$ inside $[d_1(m_2),d_1(m_1)]$.

Now we apply $d_1^{-1}$ to $[d_1(m_2),d_1(m_1)]$ to recover $[m_1,m_2]$. By Lemma \ref{Ld} we have that $$Leb([m_1,m_2])\geq\frac{1}{L}Leb([d_1(m_2),d_1(m_1)]),$$ and that $$Leb(d_1^{-1}(B_3))\leq L \cdot \frac{2}{100L^2} Leb([d_1(m_2),d_1(m_1)]) = \frac{2}{100L}Leb(d_1([m_1,m_2])).$$ Hence we have shown that $S_1\subset [m_1,m_2]$ is at least $98\%$ of the interval $[m_1,m_2]$.

Run the same argument for $d_2$ and get the corresponding analogous set $S_2$. Since they are both about 98\% of the interval, we take their intersection and get a subset of positive measure. We call this $G=S_1\cap S_2$. This completes the proof.
\end{proof}

\subsubsection{The end points}
\label{sec:MCT}
From the points $x, y\in X$, we flow forwards by the `time change' flow for time $\ell$ under a new future $\omega'$ to our final pair of points denoted $x',y'\in X$. First we will discuss how we pick this new future. We will need to do so in a way that ensures that we land in an appropriate set called $K_{4}'$ (see Section \ref{sec:lusin}). Throughout this section we let $K\defeq K_4'.$

\begin{claim}\label{Q1}
    We have that $ \hat{m}^{\tau}(K^c) < Q\epsilon_4.$
\end{claim}
\begin{proof}
The set $K_3$ belongs to the set where Lemma \ref{RN} holds and so by Lemma \ref{RN}, and the fact that $\hat{m}(K^c)<\epsilon_4$, we have that 
\begin{align}
    \hat{m}^{\tau}(K^c) = \int_{K} \frac{d\hat{m}^{\tau}}{d\hat{m}}d\hat{m} \leq Q m(K^c) < Q\epsilon_4.
\end{align}
This gives us that $\hat{m}^{\tau}(K) >1-Q\epsilon_4$.
\end{proof}

Let $\hat{\calF}_0^{\tau}$ be the sigma algebra of $\hat{m}^{\tau}$-measurable functions, i.e. the $\hat{m}^{\tau}$ completion of the natural sigma algebra on $Z$. Since $\hat{m}^{\tau}$ is absolutely continuous with respect to $\hat{m}$, they have the same measure zero sets and so this is the same as the $\hat{m}$-completion. Let $Q_{\ell}^{\tau} \defeq (F^{-\ell}_{tc})^{-1}(\hat{\calF}^{\tau}_0)$; for a function to be measurable with respect to this $\sigma$-algebra means that it only depends on events after $\ell$ time in the past under the time changed flow with roof function $\tau$. 

We start our discussion at the points $(\overline{x},\overline{\omega}_x,k_x)$ and $(\overline{y},\overline{\omega}_y,k_y)$ which live in $K_3$. By our definition of $K_3$ in Section \ref{sec:lusin}, this means that they satisfy the following lemma from \cite[Corollary 3.8]{BQ}:

\begin{lemma}[Law of last jump, \cite{BQ}, Corollary 3.8]\label{LLJ}
     For any $\hat{m}$-measurable function $\varphi:Z \to \bbR$ (e.g. $\mathds{1}_K)$, for every $\ell\geq 0$, for $\hat{m}$-a.e. $(\overline{x},\overline{\omega},\overline{k})\in Z$ we have
    \begin{align}
    \bbE_{\tau}[\varphi | Q_{\ell}^{\tau}](\overline{x},\overline{\omega},\overline{k}) = \int_{\Omega^+} \varphi(F_{tc}^{\ell} (x,(\omega^-,\omega^+),k))d\mu^{\bbN}(\omega^+),
\end{align}
where $\bbE_{\tau}$ is expectation with respect to $\hat{m}^{\tau}$.
\end{lemma}

Consider the tail $\sigma$-algebra $Q_{\infty}^{\tau} = \bigcap Q_{\ell}^{\tau}$.

\begin{claim} \label{Q2} We have that 
    \begin{align}
\int \bbE_{\tau}[\mathds{1}_K|Q_{\infty}^{\tau}](x,\omega,k)d\hat{m}^{\tau}(x,\omega,k) >1-Q\epsilon_4.
\end{align}
\end{claim}
\begin{proof} 
Using Claim \ref{Q1}, we have that
\begin{align}
\int \bbE_{\tau}[\mathds{1}_K|Q_{\infty}^{\tau}](x,\omega,k)d\hat{m}^{\tau}(x,\omega,k) = \hat{m}^{\tau}(K)>1-Q\epsilon_4.
\end{align}
This completes the proof.
\end{proof}

\begin{lemma}\label{48}
We have that 
    \begin{align}\label{V'}
\mu^{\bbN}(\{ \omega^+ \in \Omega^+ : F_{tc}^{\ell} (x,(\omega^-,\omega^+),k)\in K \}) \geq 1-\sqrt{Q\epsilon_4}.
\end{align}
\end{lemma}
\begin{proof} 

Using Lemma \ref{Q2}, there is a compact subset $L'$ of the conull set where Lemma \ref{LLJ} holds with $\varphi = \mathds{1}_K$, with $\hat{m}^{\tau}(L')>1-\sqrt{Q\epsilon_4}$, such that for $(x,\omega,k)\in L'$ we have
$$\bbE_{\tau}[\mathds{1}_K|Q_{\infty}^{\tau}](x,\omega,k)>1-\sqrt{Q\epsilon_4}.$$

On the other hand, we apply the Martingale convergence theorem to the Lemma~\ref{LLJ} with $\varphi = \mathds{1}_K$:
\begin{align*}\bbE_{\tau}[\mathds{1}_K|Q_{\infty}^{\tau}](x,\omega,k) &= \lim_{\ell\to \infty} \bbE_{\tau}[\mathds{1}_K|Q_{\ell}^{\tau}](x,\omega,k)\\ &= \lim_{\ell\to \infty}  \int_{\Omega^+} \mathds{1}_K(F_{tc}^{\ell} (x,(\omega^-,\omega^+),k))d\mu^{\bbN}(\omega^+).\end{align*}

Egorov's theorem ensures that (with some arbitrarily small loss to $L'$) that this convergence is uniform, hence there is $\ell_0$ such that for all $\ell\geq \ell_0$ we have for every $(x,\omega,k)\in L'$
$$\bbE_{\tau}[\mathds{1}_K|Q_{\ell}^{\tau}](x,\omega,k) =  \int_{\Omega^+} \mathds{1}_K(F_{tc}^{\ell} (x,(\omega^-,\omega^+),k))d\mu^{\bbN}(\omega^+)\geq 1-\sqrt{Q\epsilon_4}.$$
This translates to the statement that
\begin{align*}
\mu^{\bbN}(\{ \omega^+ \in \Omega^+ : F_{tc}^{\ell} (x,(\omega^-,\omega^+),k)\in K \}) \geq 1-\sqrt{Q\epsilon_4}.
\end{align*}
This completes the proof.
\end{proof}

So there is a large measure set of appropriate new futures to choose from. Run this argument for each of the points, $(\overline{x},\overline{\omega}_x,k_x)$ and $(\overline{y},\overline{\omega}_y,k_y)$ (separately).

\subsubsection{Conclusion}
\label{sec:conc1}

From Section \ref{sec:MCT}, Lemma \ref{48}, for each choice of $T$ and $m$ fixed (where $m$ is chosen in an interval determined by $T_{big}$) we are given $V_{T,m}$ ($\ell$ is determined and we assume it is greater than $\ell_0$ by taking $T$ large), a positive measure subset of $\Omega^+$ such that we can choose $(\omega')^+$ from.

 Additionally, we got in Section \ref{sec:green}, i.e. when we used Lemma \ref{birkhoff} we landed in a set $\tilde{K}_2$ where the points have a large measure set of $\omega'$ to choose from (just the future part, the past stays the same), i.e. these sets $V_{(x,\sigma^{n'}(\tilde{\omega}^-))}, V_{(y,\sigma^{n'}(\tilde{\omega}^-))} \subset \Omega^+$ such that picking $\omega'$ from there lands us in the set $K_2$. This set $K_2$ had properties we needed to make arguments in Section \ref{sec:green}. We defined $V = V_{(x,\sigma^{n'}(\tilde{\omega}^-))} \cap V_{(y,\sigma^{n'}(\tilde{\omega}^-))}$. Now let $V' = V_{T,m}\cap V $. The parameter $Q$ is fixed at the beginning and $\epsilon_4$, and $\epsilon_2$ were chosen (at the beginning) such that $V'$ has positive measure. We have one last restriction on $\omega'$ that we used in Lemma \ref{dvg}. This comes from the use of Lemma \ref{anglecontrol}, the required $\epsilon$ in this lemma gets smaller as $T_{big}$ gets bigger. So we can take $T_{big}$ large enough so that the corresponding is $\delta$ is very small and such that this restriction added to $V\cap V'$ still has (large) positive measure.

Take $\ell\to \infty$, more precisely, I want the legs of the V to go to infinity. To achieve this, we take $\rho\to 0$, in turn, $T_{big}$ gets bigger, and since $\theta_1$ and $\theta_2$ are fixed, even though $T$ increases, we see that the proportion between $T$ and $T_{big}$ decreases. This implies that the amount $T$ separates the unstable component (i.e. the distance between $\tilde{y}$ and $\tilde{z}$ decreases even though we are flowing for longer) and the stable distance still goes to zero. Together this constitutes the distance between the unstable planes going to zero. The parameter $m$ grows as $\rho \to 0$ but the error between the unstable planes still goes to zero as in Lemma \ref{bdnded2}. We also have that $d_1(m_1)>(1+\eta)d_1(m_2)$ and $d_2(m_1)>(1+\eta)d_2(m_2)$ as in Lemma \ref{dd} so that we can make an argument for our points along the way to stay in a good set as in Lemma \ref{landing}. Being in a good set after flowing by $d_1$ or $d_2$ allows us to use Lemmas \ref{LLJ} and \ref{48} to get the set $V_{T,m}$ in the previous paragraph. As $\ell$ goes to infinity, the image of the supports land in the same plane as the stable components for $\omega'$ shrink. We end up with two support lines in one unstable manifold separated by at least distance $\epsilon_f$ (but not more than some bounded constant multiple of $\epsilon_f$). However, these limit points by construction belong to the compact set $K_{00}\subset \calW^c$. This is a contradiction to the definition of $\calW$. This means $\calW$ must have positive measure. This is a contradiction to Lemma \ref{0.6}. So we have proven Proposition \ref{U-} and hence also proven Proposition \ref{full}.

\section{Proof of Proposition \ref{finite}}
\label{sec:alt1}

\subsection{Outline of the proof of Proposition \ref{finite}}
\label{sec:outline2}

In this setting, our fiber entropy is zero and unlike in the proof of Proposition \ref{U-}, we cannot use Theorem \ref{thmC} to say that the Oseledets' directions are random. We only have Corollary \ref{8.2} that tells us that the $W^s(x,\omega)$ are random. 

We will work by contradiction and run a `floating Benoist and Quint' argument. We will assume that $\nu$ is not finitely supported and derive a contradiction to the fact that $m^u_{x,\omega}$ is trivial a.e. in a very similar vain to the proof of Proposition \ref{U-}. To be more precise, we will construct an unstable manifold with two distinct support points. This will suffice in the same way it did for the proof of Proposition \ref{U-}; we will argue using the sets $K_{00}$ and $\calW$ as defined in Section \ref{sec:lusin2}.

In assuming that $\nu$ is not finitely supported, we will show in Section \ref{sec:starting}  that we can pick points $(\hx,\ho,k_0)$ and $(\hy,\ho,k_0)$ such that both belong to a good set $K_0$, $\hy\not \in W^s(\hx,\ho)$ and $\hx$ and $\hy$ can be chosen arbitrarily close together. This setup allows us to follow aspects of the process of the proof of Proposition \ref{U-}. 

We then follow the same steps as in Section \ref{sec:Sect3} to choose parameters $T$ and $m$. This is so we can use all the work we did in Section \ref{sec:Sect3}, we do not actually require the step of flowing by a large $T$ to ``get the unstable planes of the points close" which we needed for the proof of Lemma \ref{dvg} and \ref{omega+1}. It also allows us to keep almost all the sets (their definitions and purpose) the same, even if some of it is unnecessary. We then compute $\ell$ similarly, up to a minor caveat seen in Lemma \ref{omega+}. Then the argument that $d$ (which is defined similarly as in Section \ref{sec:Sect3}) is Bi-Lipschitz is nearly identical to that of Lemma \ref{BiLip}. This allows us the run the same Martingale convergence theorem argument as in Section \ref{sec:MCT} and draw the conclusion that assuming that $\nu$ is not finitely supported contradicts the fact that $m^u_{x,\omega}$ is trivial a.e. in a very similar way to that of Section \ref{sec:conc1}.

\subsection{Proof of Proposition \ref{finite}}

\subsubsection{The `good' sets for the proof of Proposition \ref{finite}}
\label{sec:lusin2}

We keep the definitions of the sets as in Section \ref{sec:lusin} with the exception of the few listed below. We keep all the constants such as $C_b$, $J$, $\eta$, etc. except we let $L \defeq \max\{ 2r\ln(e^{-(\lambda^--\epsilon)}+\frac{1}{2}C) +\lambda^++\epsilon, \frac{1}{\lambda^++\epsilon} \} $.

As before, all the $\epsilon$ terms used in the definitions below will be assumed sufficiently small for their purpose unless otherwise stated more explicitly.

\begin{itemize}

\item $\calW \subset Z$, $\calW = \{(x,\omega,k) : m_{x,\omega}^u \text{ is non-trivial} \} $, by the argument at the start of Section \ref{sec:alt1} we know this set has measure zero.

\item $K_{00} \subset Z$, an arbitrary compact subset of size $\hat{m}(K_{00})>1-\epsilon_{00}$. We assume $K_{00}\subset \calW^c$, we can do this since $\hat{m}(\calW^c)=1$.

\item $K_+\subset Y$, defined as the conull subset where $m_{x,\omega}^u$ and $m_{x,\omega}^s$ are trivial.

\item $K_0\subset Z$, $K_0 \subset ((\Lambda\cap \Lambda_{loc}\cap K_+\cap K_s\cap K_r\cap K_{\xi}K_{NFC})\times [0,1[) \cap E^c$, has size $\hat{m}(K_0)>1-\epsilon_0.$

   \item $\tilde{K_2} \subset Y$, $\tilde{K}_2 \defeq (B' \times \Omega^+)\cap K_2'$, $\epsilon_2$ and $\epsilon_2'$ small such that the corresponding measure of $\tilde{K}_2$ which we call $\gamma_2=1-\tilde{\epsilon}_2$ satisfies 
    \begin{align}
        \gamma_2> 1-\frac{1}{200L^4} >3/4.
    \end{align}

    \item $E_2^c \subset Z$, take $\delta =\alpha_2(\tilde{K}_2)$ such that $\beta_2\defeq 2\gamma_2(1-\alpha_2)-1>\frac{1}{200L^4}$, $\epsilon \ll\epsilon_1$, $K=\tilde{K}_2 \times [0,1[$ and apply Lemma \ref{birkhoff}, we get $S>0$ such that the set 
    \begin{align}
        E_2^c \defeq \{(x,\omega,k) : \ \forall T>S, \  |\frac{1}{T}\int_0^T \mathds{1}_{K}(F^t(x,\omega,k))dt - \int_{Y}\mathds{1}_{K} d\hat{m} |<\delta \},
    \end{align}
    has $\hat{m}(E_2^c)>1-\epsilon$.

   \item $K_3 \subset Z$, $K_3 \subset (\Lambda \cap \Lambda_{loc})\times [0,1[$, additionally, $K_3$ satisfies the conditions of Lemma \ref{RN} and belong to the $\hat{m}$-conull set from Lemma \ref{LLJ}. The size of this set is $\gamma_3\defeq 1-\epsilon_3$ where
    \begin{align}
       \gamma_3> (1-\frac{1}{100L^2}) + \frac{1}{100L^2}\left( \frac{1}{1+\eta} \right)>3/4.
    \end{align}

     \item $E_3^c \subset Z$, take $\delta = \alpha_3\gamma_3$ such that $\beta_3 = 2\gamma_3(1-\alpha_3)-1> (1-\frac{1}{100L^2})+\frac{1}{100L^2}\left(\frac{1}{1+\eta}\right)$, $\epsilon \ll\epsilon_1$, $K=K_3$ and apply Lemma \ref{birkhoff}, we get $S>0$ such that the set 
    \begin{align}
        E_3^c \defeq \{(x,\omega,k) : \ \forall T>S, \  |\frac{1}{T}\int_0^T \mathds{1}_{K}(F^t(x,\omega,k))dt - \int_{Y}\mathds{1}_{K} d\hat{m} |<\delta \},
    \end{align}
    has $\hat{m}(E_3^c)>1-\epsilon$.
\end{itemize}

The sets are not necessarily in the order they appear in the proof because the definition of earlier sets is often dependent on what is needed later.

All sets are used as in Section \ref{sketch} with the exception of some of those listed above. The sets $E^c_2$, $E^c_3$, $\tilde{K}_2$ and $K_3$ have the same descriptions.

\begin{itemize}

    \item $\calW$ - we know this set has full measure by Lemma \ref{0.2} but we will show this is not the case when we assume that Proposition \ref{finite} is false. This will prove Proposition \ref{finite}.

    \item $K_{00}$ - We will land in this set at the end and use it to show that $\calW$ must have positive measure when we assume that Proposition \ref{finite} is false. This will prove Proposition \ref{finite}.

    \item $K_+$ - This set is local to Section \ref{sec:alt1} (i.e. is only conull in the zero entropy case as given by Lemma \ref{0.2}) and is needed to define the set $K_0$ where we choose our starting points.
        
    \item $ K_0\subset Z$ - Our starting points, $(\hx,\ho)$ and $(\hy,\ho)$ will be chosen from here. 
\end{itemize}

\subsubsection{Original points}
\label{sec:starting}
From now onwards, we assume that $\nu$ is not finitely supported.

In this section we use a proof strategy from \cite{E}, Section 3, Lemma 3.4.

\begin{lemma}\label{firstpts}
    For every $\rho>0$, we can find points $(\hx,\ho,k_0), \ (\hy,\ho,k_0) \in K_0$ such that $\hy \not \in W^s(\hx,\ho)$ and $d_X(\hx,\hy)<\rho$.
    
\end{lemma}
\begin{proof}

 There are two cases when it comes to the support of $\nu$. Either it is countable or uncountable. We follow the argument of \cite{E} Lemma 3.4:

 \begin{claim}\label{uncount}
     The support of $\nu$ is uncountable.
 \end{claim}
 \begin{proof}

     If the support of $\nu$ is countable, then $\nu$ is supported on a sequence of points $x_n$ each with some mass $\delta_n$.  Note that the dynamics does not change the mass of a point, so a set of these points with the same mass is an invariant set. Using that $\nu$ is ergodic we get that all points must have the same mass. This forces $\nu$ to have finite support which is a contradiction.
 \end{proof}

Let $W_{\ho,k_0} = \{\hx \in X : (\hx,\ho,k_0)\in K_0\}$ then let $$V = \{(\ho,k_0) : \mu^{\bbZ}\times dt(W_{\ho,k_0}) >1-\sqrt{\epsilon_0} \}. $$
\begin{claim}
    The set $V$ has measure at least $1-\sqrt{\epsilon_0}$.
\end{claim}
\begin{proof}
    The proof is identical to that of Lemma \ref{G!}.
\end{proof}
Now take $(\ho,k_0)\in V$ and consider $S'=W_{\ho,k_0}\cap \supp(\nu)$. This set $S'$ is also uncountable and it contains uncountably many accumulation points. One takes $\hx\in S'\subset X$ an accumulation point, one can find a sequence $\hy_n\in S'$ converging to $\hx$. Now we can pick $\hy$ as one of the $\hy_n$ such that $d_X(\hy_n,\hx)<\rho$ for the given $\rho$. Then since $\hx$ and $\hy$ belong to $S'$, $(\hx,\ho,k_0)$ and $(\hy,\ho,k_0)$ both belong to $K_0$.

Lastly, we must argue that $\hy$ can be chosen such that $\hy\not\in W^s(\hx,\ho)$. 

If not, then $\supp(\nu)$ is contained in a finite union of stable manifolds;, $$\supp(\nu) \subset \bigcup_{i=1}^n W^s(x_i,\omega_i).$$ We note that the $W^s(x_i,\omega_i)$ may not be measurable sets, but if we let $B^s(x,\omega,R)$ be the ball of radius $R$ inside $W^s(x,\omega)$ in the induced metric, this is a measurable set.

\begin{claim}\label{switch}
    The leaves $W^s(x_i,\omega_i)$ in the union above that carry positive measure have the same measure.
\end{claim}
\begin{proof}
     Let $f(x,\omega) = \lim_{R\to \infty} \nu(B^s(x,\omega,R))$; this limit exists along some subsequence because $f$ is upper bounded by 1. Since the support of the probability measure is contained in this finite union of leaves, for some $i$, $f(x_i,\omega_i)$ is positive. 

    If we apply the skew product map, we know that $F(W^s(x_i,\omega_i)) = W^s(F(x_i,\omega_i))$. Further since $\nu$-measure is preserved under the dynamics $F$, we have that $f$ is invariant under $F$. This implies that it is constant a.e. by ergodicity of $m$.

    This tells us that amongst the finitely many leaves containing the support of $\nu$, those that contain positive measure have the same amount of measure.
\end{proof}

\begin{claim}\label{contr}
    Each leaf $W^s(x_i,\omega_i)$ that contains positive measure has finitely many points in the support of $\nu$. 
\end{claim}
\begin{proof}
  
   In Claim \ref{switch} we learned that under $F$, the leaves of positive measure must map to one another.
   
    Take a leaf $W^s(x_i,\omega_i)$ in the finite union that contains positive measure, and let $y\in \supp(\nu)$ belong to this leaf. Let $\epsilon>0$, and consider $B^s(y,\epsilon)\subset W^s(x_i,\omega_i)$, i.e. the ball of radius $\epsilon$ around $y$ in the induced metric.
    
    Running the dynamics backwards, Theorem \ref{locstablemfld} tells us that the radius of this ball grows and it also stays within the finite union of leaves. In a finite number of steps, one can grow the radius of this ball to the point which it contains most of the measure of leaf (there are finitely many leafs). Since $\nu$ is invariant and Claim \ref{switch} tells us that the measure of all the leaves is the same, this tells us that $B^s(y,\epsilon)$ contains the entire measure of the leaf. This is true for any $\epsilon>0$ and hence $y$ is an atom containing the entire measure of $W^s(x_i,\omega_i)$. One can do this for each positive measure containing leaf. 
\end{proof}
Claim \ref{contr} is a contradiction to our assumption that $\nu$ is not finitely supported.
\end{proof}

\subsubsection{Flowing by $T$ and $m$}
\label{sec:green2}
Fix $\epsilon_f>0$ as in Section \ref{sec:not5}.

To start, take any $\rho>0$ such that
$$\rho<\epsilon_f/10000,$$
and use Lemma \ref{firstpts} to pick a point $(\hx,\ho,k_0)\in K_0$ with its corresponding $(\hy,\ho,k_0)\in K_0$ such that $U^+[\hx,\ho] = \hx$, $U^+[\hy,\ho] = \hy$, $\hy\not \in W^s(\hx,\ho)$, and $d_X(\hx,\hy)<\rho$. From here on out, we acknowledge that $U^+[\hx,\ho,k_0]$ and $U^+[\hy,\ho,k_0]$ are just the points $\hx$ and $\hy$ themselves.

 We first note that for our choice of $\hx$ and $\hy$, $\hy\not \in W^s(\hx,\ho)$ and so there exists a point $\hz\in W^s_q(\hy,\ho)\cap W^u_q(\hx,\ho)$. We call $d_X(\hz,\hy)$ the `unstable distance'. As before in Section \ref{sec:Sect3}, the local stable manifold theorem, Theorem \ref{locstablemfld}, tells us that there is number $T_{big}$ such that if $y_B = F^{T_{big}}_{\ho}(\hy)$, $z_B = F^{T_{big}}_{\ho}(\hz)$, $d_X(y_B,z_B)=\epsilon_f$.

Let $T$ be chosen as in Choice \ref{choice1}, i.e. $T\in [T_{big}(1-\theta_1),T_{big}(1-\theta_2)]$ where $\theta_1$ and $\theta_2$ satisfy the requirements of Choice \ref{thetabds}.

We run the same Birkhoff arguments using Lemma \ref{birkhoff} , i.e. Lemmas \ref{B} and \ref{bigint} which allow us to pick $T$ in the given interval such that $F^T(\hx,\ho,k_0)$ and $F^T(\hy,\ho,k_0)$ land in the set $\tilde{K}_1$. This set $\tilde{K}_1$ has the properties of the good set $K_1'$ and also is a subset of $(B\times \Omega)$ so we can pick a good new $\tilde{\omega}$ such that our points $(\tilde{x},\tilde{\omega},k_1)$ and $(\tilde{y},\tilde{\omega},k_1)$ (where $k_1 = \lfloor T + k_0\rfloor$) belong to $K_1$. The properties of these sets allow us to use the same arguments in Section \ref{sec:green} when we choose $m$.

When we choose $\tilde{\omega}$, we do so as in Lemma \ref{apartm}.

\begin{convention}
    We pick the convention that $m$ is given as a positive real number and we will flow backwards by it and hence write $F^{-m}$.
\end{convention}

From the points $\hx$ and $\hy$ we will go in two directions. Firstly we flow backwards by some $m$ (under the `standard' flow) from $\hx$ and $\hy$ to the points which we will denote $x$ and $y$ respectively.

Let $m_1$ and $m_2$ be the same as in Equation \eqref{mchoice} with $T_{big}$ large enough so that $m_2>2m_1$ and $m_1>S$ (we will put further restrictions on how large $T_{big}$ should be).

We can demonstrate, similarly to Section \ref{sec:green}, that we can choose $m\in [m_1,m_2]$ such that the $X\times \Omega$-component of $F^{-m}(\tilde{x},\tilde{\omega},k_1)$ and $F^{-m}(\tilde{y},\tilde{\omega},k_1)$ land in the set $\tilde{K}_2 \subset K_2'$. These are good choices of $m$. We get that there is a large positive measure set of futures $(\omega')^+\in \Omega^+$ we can pick such that if $x = F^{-m}_{\tilde{\omega}}(\tilde{x})$ and $y = F^{-m}_{\tilde{\omega}}(\tilde{y})$, then $(x,(\sigma^{-n'}(\tilde{\omega}))^-,(\omega')^+)$ and $(y,(\sigma^{-n'}(\tilde{\omega}))^-,(\omega')^+)$ (where $n'=\lfloor k_1-m\rfloor$) belong to $K_2$. 

We still have that flowing by $m$ does not take our points all the way to distance $\epsilon_f$ (see Lemma \ref{bdnded1}), we also have the following statement (made obvious by Theorem \ref{locstablemfld}):

\begin{lemma}\label{tend}
As $T_{big}\to \infty$ (and hence $m\to \infty$), the distance between $y$ and $W^s(x,\sigma^{-n'}(\tilde{\omega}))$ goes to zero.
\end{lemma}
Lemma \ref{tend} is required in the proof of Lemma \ref{omega+}.

Let $\tilde{D}$ be the distance (in ambient coordinates) between $\tilde{x}$ and $\tilde{y}$. Note that by construction, this distance is still less than $q$ and that the set $M$ is contained in $\Lambda_{loc}$ as defined in \ref{sec:lusin}. The points $(\tilde{x}, \tilde{\omega},k_1)$ and $(\tilde{y},\tilde{\omega},k_1)$ belong to $M$. Additionally, the distance $\tilde{D}$ between $\tilde{x}$ and $\tilde{y}$ is less than $k_U$ and hence they live in each other's Lyapunov chart domain.

Originally, the distance $\tilde{D}$ was largely in the unstable direction for that choice of $\Omega$-component, i.e. $\sigma^n(\ho)$. Now that we have switched $\Omega$-components to $\tilde{\omega}$, generically, the distance $\tilde{D}$ has stable and unstable distance components with respect to $\tilde{\omega}$. Also we picked $\tilde{\omega}$ such that $\tilde{x}$ and $\tilde{y}$ are not on each other's stable or unstable manifolds. Also, since we are going backwards by $m$, the stables grow and the unstables contract.

In Lyapunov charts, the images of the subspaces $E^u(\tilde{x},\tilde{\omega})$ and $E^s(\tilde{x},\tilde{\omega})$ form a coordinate axis, and $\tilde{y}$ is sent to some $v\in \bbR^4$ in these charts. We can decompose $v$ into a component in the $E^s$-direction and a component in the $E^u$-direction denoted respectively $v_s$ and $v_u$, i.e. $v=v_s+v_u$. In these charts, we know exactly $v_s$ and $v_u$ grow and shrink resp. under flowing by $-m$. We denote $\alpha_s = |v_s|,$ $\alpha_u = |v_u|$.

Then, the distance between 0 and $v$ after flowing by $m$, is $$D_m\defeq |\tilde{F}^{-m}_{(\tilde{x},\tilde{\omega}),k_1}(v)-\tilde{F}^{-m}_{(\tilde{x},\tilde{\omega}),k_1}(0)|.$$ Let $\alpha_u^m \defeq|\tilde{F}^{-m}_{(\tilde{x},\tilde{\omega})}(v_u)-\tilde{F}^{-m}_{(\tilde{x},\tilde{\omega})}(0)|$ and $\alpha_s^m \defeq |\tilde{F}^{-m}_{(\tilde{x},\tilde{\omega}),k_1}(v_s)-\tilde{F}^{-m}_{(\tilde{x},\tilde{\omega}),k_1}(0)|$.

\begin{lemma}\label{bdnded32}
    For $T_{big}$ large enough (hence $m$ large enough), we have $$D_{m+1} \leq \sqrt{1+k^2} (e^{-(\lambda^--\epsilon)}+\frac{1}{2}C)\alpha_s^m = (e^{-(\lambda^--\epsilon)}+\frac{1}{2}C)D_m.$$
    Further, we have $D_{m}<D_{m+1}$.
\end{lemma}
\begin{proof}
    The proof is the same as in Lemma \ref{bdnded3}.
\end{proof}
This is used in argument that $d$ is bi-Lipschitz, see Lemma \ref{BiLip2}.

For the purposes of Lemma \ref{dd2}, we need the following Lemma:
\begin{lemma} \label{bb2}
The constant $0<\eta<1$ as fixed in Section \ref{sec:lusin} is such that for $T_{big}$ large enough (hence $m$ large enough) $D_{m_2}^{(1+\eta)}>D_{m_1}$. 
\end{lemma}
\begin{proof}
    The proof is the same as in Lemma \ref{bb}.
\end{proof}

For a choice of $m$, define $x \defeq F_{\tilde{\omega}}^{-m}(\tilde{x})$ and $y \defeq F_{\tilde{\omega}}^{-m}(\tilde{y})$.

For any (good) $m$ chosen from the interval $[m_1,m_2]$, we get a positive measure subset $V\subset \Omega^+$ of futures $\omega'^+$ to choose from that allow the points $(x,(\sigma^{-\lfloor k_1- m \rfloor}(\tilde{\omega})^-, \omega'^+ ))$ and $(y,(\sigma^{-\lfloor k_1-m\rfloor}(\tilde{\omega})^-, \omega'^+ ))$ to belong to the set $K_2$ (where $\tilde{\omega}^- \in \Omega^-$ is the past part of $\tilde{\omega}$). Let $\omega'^+\in V$ be arbitrary, let $\omega' = (\sigma^{-\lfloor k_1- m \rfloor}(\tilde{\omega})^-, \omega'^+ ) \in \Omega$. We will have additional restrictions on choosing $(\omega')^+$, but certainly, $(\omega')^+$ will belong to $V$.

At the points $(x,\omega')$ and $(y,\omega')$ we will project $(y,\omega')$ along $W^s(y,\omega')$, onto $W^u(x,\omega')$ and use (modified) normal form coordinates (NFCs) (see Equation \eqref{newNFCform}) to control the distance between the final points when flowing by $\omega'$ by $\ell$. We want this distance to be some fixed $\epsilon_f$.

Note that our good sets make the changes between ambient coordinates, Lyapunov charts and (modified) NFCs a bounded change as described in Section \ref{sec:lusin2}. Also note that $W^u(x,\omega')$ is the same as $W^u(x,\sigma^{-\lfloor k_1- m \rfloor}\tilde{\omega})$ because the past vectors of $\omega'$ and $\sigma^{-\lfloor k_1-m \rfloor}\tilde{\omega}$ are the same. It is the stable manifold that changes when passing from $\sigma^{-\lfloor k_1-m \rfloor}\tilde{\omega}$ to $\omega'$.

Let $\delta(m)$ be the distance between $x=f^{-m}_{\tilde{\omega}}(\tilde{x})$ and $y=f^{-m}_{\tilde{\omega}}(\tilde{y})$ in ambient coordinates. Let $\gamma$ be the distance between $x$ and the projection of $y$ along $W^s(y,\omega')$ onto $W^u(x,\omega')$ which we denote as $y'$.

We will work in $T_yX$, i.e. since we are working on a scale less than $\rho_0/4$, we have that $\exp_y^{-1}$ is the origin represented by $\vec{0}$, $\exp^{-1}_y(x) = \vec{v}$ and the distance between them is some $M\cdot \delta(m)$ where $M\in [\frac{1}{M_{exp}},M_{exp}]$. Let $\vec{u}$ be the unit length vector in the direction $v$. Consider the function $g(z) = z\vec{u}$. For each choice of $(\omega')^+$ there exists a natural number $r$ such that in any chart of $T_yX$, $\exp^{-1}_y(W^s(y,\omega')$ and $g(z)$ will differ for the first time at the $r$th derivative at the origin. Note that this natural number $r$ is $F$-equivariant and so is constant a.e. Call this conull set $U\subset \Omega^+$. We can further restrict this conull set to a set $U'\subset U$ of size $1-\epsilon_r$ on which the norm of this $r$th derivative is bounded above by $C_r$.

\begin{lemma}\label{omega+}
    There exists $m'$ and $M_a>1$ such that for all $m>m'$ we can choose $(\omega')^+$ such that $\gamma\neq 0$ and further $\gamma(m) = C'\delta(m)$ where $C'\in [\frac{1}{M_a},M_a]$.
    \end{lemma}
    \begin{proof}
     
 The points $x$ and $y$ are separated by distance $\delta(m)$. As explained above, since this distance is smaller than $\rho_0/4$ we have that there is some $v\in T_yX$ such that $\exp_y(v)=x$. The distance between the origin (which represents $y$) and $v$, given by $\Vert v\Vert $ is equal to $\delta(m)\cdot M$ where $M\in [\frac{1}{M_{exp}},M_{exp}]$.

 Let $\vec{u}$ be the unit vector in the direction of $\vec{v}$. We can choose $\omega'\in U$ such that the holomorphic function $\exp_y^{-1}(W^s(y,\omega'))$ differs the function $g:\bbC\to T_yX$, $g(z)=z\vec{u}$ for the first time at the $r$th derivative. 

 \underline{Case 1: $r=1$}
 
The argument is the same as in Lemma \ref{dvg} as we are in a setting much like the consequence of Theorem \ref{thmC} c) and our setup (i.e. parameter choice) is the same as in Section \ref{sec:Sect3}.

 \underline{Case 2: $r>1$}

We assume that the first derivative of $\exp_y^{-1}(W^s(y,\omega'))$ is $\vec{u}$ and so we parametrize this holomorphic function as $\psi_S(z) = z\vec{u} + f(z)\vec{u}^{\perp}$.

    Fixing $\delta'>0$ small, we can choose $m'$ large such that $y$ is within $\delta'/M_{exp}$ of $W^s(x,\sigma^{-n'}(\tilde{\omega}))$. We can do this by Lemma \ref{tend}.
    
    Then, since $(x,\sigma^{-n'}(\tilde{\omega})),(y,\sigma^{-n'}(\tilde{\omega}))\in K_2'\subset K_{ang}$ we have that $$\angle(E^u(x,\sigma^{-n'}(\tilde{\omega})),E^s(x,\sigma^{-n'}(\tilde{\omega})))=\angle(E^u(x,\omega'),E^s(x,\sigma^{-n'}(\tilde{\omega})))>\Theta.$$ This means that $E^u(x,\sigma^{-n'}(\tilde{\omega}))$ lives in some cone centered around $E^s(x,\sigma^{-n'}(\tilde{\omega}))^{\perp}$, i.e. $E^u(x,\sigma^{-n'}(\tilde{\omega}))$ is bounded away from $E^s(x,\sigma^{-n'}(\tilde{\omega}))$ by at least some angle $\theta$.  Further, since $E^u(x,\omega')$ approximates $\exp^{-1}_y(W^u(x,\omega')$ up to some error $\delta(m)^2$ (times a bounded constant) we can say that $\exp^{-1}_y(W^u(x,\omega'))$ is in some fixed cone around $E^s(x,\sigma^{-n'}(\tilde{\omega}))^{\perp}$. Additionally, we know that the $r$th derivative of $\exp_y^{-1}(W^s(y,\omega'))$ is bounded. Together we get the following claim:
    
    \begin{claim}
There exists $m'$ and $c'$ such that for all $m>m'$ we have that $\frac{1}{c'}|z_0|\leq |z_1|\leq c'|z_0|$.   
\end{claim}

This tells us that $\gamma = |z_1|^r/\sin(\theta) = c^r(C_r)^r\delta(m)^r/\sin(\theta)$ where $c\in [\frac{1}{c'},c']$ and $\theta \in [\Theta,\pi/2]$. This completes the proof.
    \end{proof}

Now $y'$ and $y$ both belong to $W^s(y,\omega')$ so as we flow forwards under $\omega'$ this distance contracts and so $\gamma$ is a good proxy for $\delta$.

Recall that we were working towards having our end points be $\epsilon_f$ apart in ambient coordinates, in reality it suffices to land within some bounded constant of $\epsilon_f$. Ensuring this final distance is bounded from above will allow us to take limits in Section \ref{sec:conc2} and bounding it from below ensure our goal of distinct unstables is achieved (see the heuristic argument in Section \ref{sec:Sect3}). 

Knowing the relationship between $\gamma$ and $\delta$ is important in defining $\ell$. As explained before we can use (modified) normal form coordinates (NFCs) to control how $\gamma$ grows because this is a distance that lives in $W^u(x,\omega')$. We use this for a preliminary computation for what $\ell(m)$ should be, i.e. the amount we want to flow by from the points $(x,\omega')$ and $(y,\omega')$ regardless of the choice of $\omega'$. Ultimately we will have to compute $\ell$ slightly differently but the first consideration that we make is that $\gamma$ gets changed by (modified) NFCs to some $C_1\gamma$ where $C_1 \in [\frac{1}{M_{NFC}},M_{NFC}]$ and we would like to grow this distance to be $\epsilon_f\cdot M_{NFC}$ (so that when we go back to ambient coordinates it is bounded from below by $\epsilon_f$). Then we solve $C_1 \gamma e^{\ell} = M_{NFC}\epsilon_f$ for $\ell$ and get

\begin{align}
    \ell(m) = \ln\left(\frac{M_{NFC}\epsilon_f}{C_1\gamma(m)}  \right).
\end{align}

Since $C_1$ changes with $m$, we instead replace it with $\frac{1}{M_{NFC}}$ (i.e. its lower bound). Solving $\ell = \ln\left(\frac{M_{NFC}\epsilon_f}{\frac{1}{M_{NFC}}\gamma}  \right)$ is a good first step. If one returns to ambient coordinates, $F^{\ell}_{\omega'}(x)$ and $F^{\ell}_{\omega'}(y')$ differ by at least $\epsilon_f$ and at most $\epsilon_f\cdot M_{NFC}^2$.

To make the computation in Lemma \ref{BiLip} simpler, we use $D_m$ instead of $\gamma$ which would make the formula look more like this: 
\begin{align} \label{olddef}
    \ell(m) = \ln\left(\frac{M_{NFC}\epsilon_f}{ \frac{1}{M_{NFC}M_r^rM_a} D_m^r}  \right).
\end{align}
Note that the constants in front of $D_m$ appear because we are thinking of $D_m$ in relation to $\gamma$. $D_m$ is related to $\delta(m)$ by a constant bounded from below by $\frac{1}{M_r}$ and $\delta(m)^r$ is related to $\gamma$ by $C'$ (where $r$ is as in Lemma \ref{omega+}), the lower bound being $\frac{1}{M_a}$. Note however that $C'$ is also bounded from above by $M_a$ and the relation between $D_m$ and $\delta(m)$ is bounded above by $M_r$. This means that this definition pushes us to at least distance $\epsilon_f$ ambiently, but does not take us further than $M_aM_r^rM_{NFC}^2\epsilon_f$ this is important for Section \ref{sec:conc2}. 

However, this definition for $\ell$ is not continuous as $D_m$ changes only every integer (up to taking into account $k_1$). We will need continuity of $\ell$ for Section \ref{sec:pink2}. Note also that we showed that $D_m$ is increasing in $m$, so we define $\ell$ to be the piecewise linear interpolation between these points. We can write this explicitly as 
\begin{align} \label{betterdef}
    \ell(m) &= \ln\left(\frac{M_{NFC}\epsilon_f}{ \frac{1}{M_{NFC}M_r^rM_a} D_m^r }  \right) \\ \nonumber &+ (m-\lfloor m \rfloor ) \left(\ln\left(\frac{M_{NFC}\epsilon_f}{ \frac{1}{M_{NFC}M_r^rM_a} D_{m+1}^r }  \right) - \left(\ln\left(\frac{M_{NFC}\epsilon_f}{ \frac{1}{M_{NFC}M_r^rM_a} D_{m}^r } \right) \right) \right),
\end{align}
which simplifies to 
\begin{align}
    \label{bettererdef}
    \ell(m) = \ln\left(\frac{M_{NFC}\epsilon_f}{ \frac{1}{M_{NFC}M_r^rM_a} D_m^r }  \right) + (m-\lfloor m \rfloor ) \ln\left( \frac{D_m^r}{D_{m+1}^r}\right).
\end{align}

We make one last modification to this definition of $\ell$, to correct concerns that in the interpolation if $m$ is between two integers, technically $\ell(m)$ computed using Equation \eqref{olddef} is smaller than that of $\ell(m)$ computed using Equation \eqref{bettererdef}, we tack on a last constant. Let $M_{D} \defeq \left( e^{-(\lambda^-  +\delta)} + \frac{1}{2}C \right)$ then define

\begin{align} \label{ellrel}
    \ell(m) = \ln\left(\frac{M_{NFC}\epsilon_f}{ \frac{1}{M_{NFC}M_r^rM_aM_D^r} D_m^r }  \right) + (m-\lfloor m \rfloor ) \ln\left( \frac{D_m^r}{D_{m+1}^r}\right).
\end{align}
This expression, \eqref{ellrel}, is how we define $\ell$. Note that the second term is negative because $\ell$ decreases in $m$.

Additionally, take note of the fact that $D_m$ increases with $m$ because $m$ is large enough; see Lemma \ref{bdnded32}.

\subsubsection{Flowing by $\ell$}
\label{sec:pink2}

Now for completeness, we repeat the setup from Section \ref{sec:pink}:

 We have a `V' shape where we go backwards by $m$ (under the standard flow) and forwards by $\ell$ (under the time changed flow). Here $\ell$ is determined by $m$ and the distance $\epsilon_f$ that we want to reach (within some constant) at the end points using (modified) NFCs. We need the `legs' of the `V', i.e. the forward and backward directions to be of the same length to run the argument in Section \ref{sec:MCT2}. This is why we will introduce extra points in this section.

We fix $T$ and the points $(\tilde{x},\tilde{\omega},k_1)$ and $(\tilde{y},\tilde{\omega},k_1)$ from Section \ref{sec:green2}, and for each $m$ we pick in Section \ref{sec:green2} (from the positive measure subset of `good choices' in some interval) we get the points $(x,\omega',k_2)$ and $(y,\omega',k_2)$. Here, $\omega'\in \Omega$ is any vector such that $\omega'^- = \sigma^{-\lfloor k_1-m \rfloor}(\tilde{\omega})^-$ and $\omega'$ belongs to the positive measure subset $V$ corresponding to the points $(x,\sigma^{-\lfloor k_1-m \rfloor}(\tilde{\omega}))$ and $(y,\sigma^{-\lfloor k_1-m \rfloor}(\tilde{\omega}))$ which live in $\tilde{K}_2$.

Additionally, for each choice of $m$ we can compute $\ell(m)$ as we did in Section \ref{sec:green2}. This is independent of the choice of future $\omega'$ that we pick from some a good set. Now, starting from the points $(\tilde{x},\tilde{\omega})$ and $(\tilde{y},\tilde{\omega})$ (i.e. we are in the set $K_1$), we want to flow forwards by an amount that will allow us to complete our `V' shape.

Let us reiterate our `V' shape more definitively. Since the points $(\tilde{x},\tilde{\omega},k_1)$ and $(\tilde{y},\tilde{\omega},k_1)$ are fixed at this point, our `V' is determined entirely by a choice of $m$ (which is in a fixed interval at this point). The length of the legs is $\ell(m)$ under time change flow, with the middle points of the `V' being the points $(x,\omega',k_2)$ and $(y,\omega',k_2)$. Note that because $\tau$ in the  `time change' flow depends on which point you start at, the resulting pair of points after flowing by $\ell(m)$ will have different $\Omega$ and $[0,1[$-components compared to one another. This was never the case with the `standard' flow. We have to convert the parameter $m$ that we flowed by under the standard flow into the quantity we would flow by from $(x,\sigma^{n'}(\tilde{\omega}),k_2)$ and $(y,\sigma^{n'}(\tilde{\omega}),k_2)$ to $(\tilde{x},\tilde{\omega},k_1)$ and $(\tilde{y},\tilde{\omega},k_1)$ under the time changed flow. This quantity will differ between points. We define $m_{\tau, (\tilde{x},\tilde{\omega},k_1)}$ and $m_{\tau, (\tilde{y},\tilde{\omega},k_1)}$ to be these quantities resp. Now to complete our `V' we flow backwards from $(\tilde{x},\tilde{\omega},k_1)$ and $(\tilde{y},\tilde{\omega},k_1)$ by $d(m,(\tilde{x},\tilde{\omega},k_1)) \defeq \ell(m) + (-m)_{\tau, (\tilde{x},\tilde{\omega},k_1)} $ and $d(m, (\tilde{y},\tilde{\omega},k_1)) \defeq \ell(m) + (-m)_{\tau, (\tilde{y},\tilde{\omega},k_1)} $ (note that both functions $d$ are positive). We denote the resulting points $(\overline{x}, \overline{\omega}_x, k_x)$ and $(\overline{y},\overline{\omega}_y,k_y)$ resp. Note that $k_x$ and $k_y$ are not necessarily the same; neither are $\overline{\omega}_x$ and $\overline{\omega}_y$.

\begin{remark}
    The minus sign was included next to the $m$ before doing the time change because we want to remember to compute it using $\tau^-$, i.e. we are going backwards when flowing by $,$ and we need to remember that as we compute the time changed value.
\end{remark}

For simplicity, we will take our overall `V' and divide it into its two components, the `V' we have at the fixed point $(\tilde{x},\tilde{\omega},k_1)$ will be called `V1' and the `V' we have at the point  $(\tilde{y},\tilde{\omega},k_1)$ will be called `V2'. This allows us to drop some notation regarding time changed $m$ and the function $d$. Let us just focus on `V1' as everything we have to say will hold for `V2' as well. Rename $m_{\tau, (\tilde{x},\tilde{\omega},k_1)}$ as $m_1^{\tau}$ and $d(m,(\tilde{x},\tilde{\omega},k_0))$ as $d_1(m)$. We will have to show that $d_1$ is Bi-Lipschitz in $m$ in order to pick $m$ such that flowing backwards by $m$ from $(\tilde{x},\tilde{\omega},k_1)$ we land in the set $\tilde{K}_2$ and flowing forwards by $d_1(m)$ we land in the set $K_3$. We can use Lemma \ref{birkhoff} to do one or the other, but we have to run a different argument to ensure we can do both at the same time. A similar argument will work for `V2' where we use the notation $m_2^{\tau}$ and $d_2(m)$ for $m_{\tau, (\tilde{y},\tilde{\omega},k_1)}$ and $d(m,(\tilde{y},\tilde{\omega},k_0))$ respectively.

We run the argument to demonstrate that $d$ is Bi-Lipschitz given our new definition for $\ell$ in Equation \eqref{ellrel}, we see that the constant $r$, introduced due to Lemma \ref{omega+} (rather than Lemma \ref{dvg}) does not change anything.

\begin{lemma}\label{BiLip2}
    Take the interval $[m_1,m_2]$ from Section \ref{sec:green2}. The map $d_1 : [m_1,m_2] \to ]-\infty,0]$ is Bi-Lipschitz, (similar for $d_2$ and the same Lipschitz constant works in this case).
\end{lemma}
\begin{proof}
    First we want to show that there is $L>0$ such that $|d_1(m) - d_1(m')|\leq L|m-m'|$ for arbitrary $m$ and $m'$. 
    Recall the definition of $d_1$ from Equation \eqref{d} and plug in:
\begin{align*}
    |d_1(m) - d_1(m')| &= |(-m)_1^{\tau}-(-(m'))_1^{\tau} + \ell(m) -\ell(m')|\\
    &\leq |(-(m'))_1^{\tau}-(-m)_1^{\tau}| + |\ell(m) - \ell(m')|.
\end{align*}
By Equation \eqref{ellrel} we have that this is equal to 
\begin{align*}
    |(-m)_1^{\tau}-(-(m'))_1^{\tau}| + |\ln\left(\frac{M_{NFC}\epsilon_f}{\frac{1}{M_{NFC}M_r^rM_aM_D^r}D_{m'}^r }\right) + (m'-\lfloor m' \rfloor)\ln\left(\frac{D_{m'}^r}{D_{m'+1}^r} \right) \\ - \ln\left(\frac{M_{NFC}\epsilon_f}{\frac{1}{M_{NFC}M_r^rM_aM_D^r}D_m^r }\right) - (m-\lfloor m \rfloor)\ln\left(\frac{D_m^r}{D_{m+1}^r}\right) |
\end{align*}

Let's control the first term in the sum above. Note that $k_1$ plays a role here, we had defined $d_1$ once we fixed a point that had this $[0,1[$-component. It will show up in our computations. Making sure to keep the convention that $k_1$ is written as a positive number between 0 and 1, we can write
\begin{align}
    (-m)^{\tau}_1 &= k_1\tau^-(\tilde{x},\tilde{\omega}) + \tau^-(F^{-1}(\tilde{x},\tilde{\omega})) +\dots \tau^-(F^{\lfloor k_1-m\rfloor + 1}(\tilde{x},\tilde{\omega})) + a, \\
    (-m')^{\tau}_1 &= k_1\tau^-(\tilde{x},\tilde{\omega}) + \tau^-(F^{-1}(\tilde{x},\tilde{\omega})) +\dots \tau^-(F^{\lfloor k_1-m' \rfloor +1}(\tilde{x},\tilde{\omega})) + b.
\end{align}
Note that the first term $k_1\tau^-(\tilde{x},\tilde{\omega})$ appears because once we pass 0 we hit it with the automorphism $f_{-1}^{-1}$. For example if $k_1=0.22$, then we are already $78\%$ on our way to hitting $\tau^-(\tilde{x},\tilde{\omega})$, the remainder that we need is that $22\%$. The $a$ and $b$ terms show up because we have not quite reached the next $\tau^-$ value, i.e. $a\in [0, \tau^-(F^{\lfloor k_1-m\rfloor }(\tilde{x},\tilde{\omega}))[$ and $b\in [0, \tau^-(F^{\lfloor k_1-m'\rfloor }(\tilde{x},\tilde{\omega}))[$. We can easily compute $a$ and $b$ from what we have as clearly the following relations holds:
\begin{align*}
    \psi_m \defeq \lfloor k_1-m\rfloor - (k_1-m) +1 &= \frac{a}{\tau^-(F^{\lfloor k_1-m\rfloor }(\tilde{x},\tilde{\omega}))}>0,\\
       \psi_{m'} \defeq \lfloor k_1-m'\rfloor -(k_1-m') +1 &= \frac{b}{\tau^-(F^{\lfloor k_1-m'\rfloor }(\tilde{x},\tilde{\omega}))}>0.
\end{align*}
Without loss of generality assume $m'>m$ (remember these are positive numbers), then $(-(m'))^{\tau}_1> (-m)^{\tau}_1$. We have
\begin{align*}
    |(-(m'))^{\tau}_1 - (-m)^{\tau}_1| = \tau^-(F^{\lfloor k_1-m' \rfloor +1}(\tilde{x},\tilde{\omega})) + \dots  \tau^-(F^{\lfloor k_1-m\rfloor }(\tilde{x},\tilde{\omega})) + b-a
    \end{align*}
    \begin{align*}
    =| \tau^-(F^{\lfloor k_1-m' \rfloor +1}&(\tilde{x},\tilde{\omega})) + \dots  \tau^-(F^{\lfloor k_1-m\rfloor }(\tilde{x},\tilde{\omega}))\\ \nonumber &+ \psi_{m'}\cdot \tau^-(F^{\lfloor k_1-m'\rfloor }(\tilde{x},\tilde{\omega})) - \psi_m\cdot \tau^-(F^{\lfloor k_1-m\rfloor }(\tilde{x},\tilde{\omega}))|
\end{align*}
\begin{align}\label{expanded}
    =| \tau^-(F^{\lfloor k_1-m' \rfloor +1}(\tilde{x},\tilde{\omega}))  \dots + \tau^-(F^{\lfloor k_1-m\rfloor -1 }(\tilde{x},\tilde{\omega}))&+ 
      \tau^-(F^{\lfloor k_1-m\rfloor }(\tilde{x},\tilde{\omega}))(1-\psi_m)\\ \nonumber &+ \psi_{m'}\cdot \tau^-(F^{\lfloor k_1-m'\rfloor }(\tilde{x},\tilde{\omega}))|. 
\end{align}
Using Equation \eqref{expanded}, we then have that 
\begin{align*}
   |(-(m'))^{\tau}_1 - (-m)^{\tau}_1| &\leq | \tau^-(F^{\lfloor k_1-m' \rfloor +1}(\tilde{x},\tilde{\omega}))| + \dots + |\tau^-(F^{\lfloor k_1-m\rfloor -1 }(\tilde{x},\tilde{\omega}))|\\ \nonumber &+  |\tau^-(F^{\lfloor k_1-m\rfloor }(\tilde{x},\tilde{\omega}))|(1-\psi_m) + \psi_{m'}|\tau^-(F^{\lfloor k_1-m'\rfloor }(\tilde{x},\tilde{\omega}))|.
\end{align*}

Using the bounds on $\tau$ we get that this is bounded above by 

\begin{align*}
   &\left[\left( \left( \lfloor ( k_1-m \rfloor -1) - (\lfloor k_1-m' \rfloor +1 \right)+1\right) + (1-\psi_m) + \psi_{m'}\right] (\lambda^++\epsilon)\\ \nonumber &=(\lambda^++\epsilon)(m'-m).
\end{align*}

On the other hand, we have to lower bound this quantity too. So we consider again $m'>m$ and take Equation \eqref{expanded}. Again, the $\tau^-$-terms are lower bounded (in absolute value) by $\lambda^+-\epsilon$, so we get that Equation \eqref{expanded} is bounded below by $(\lambda^+-\epsilon)(m'-m)$.

    Now we control the second term. It reduces to 
    \begin{align*}
       |\ln\left(\frac{M_{NFC}\epsilon_f}{\frac{1}{M_{NFC}M_r^rM_aM_D^r}D_{m'}^r }\right) + (m'-\lfloor m' \rfloor)\ln\left(\frac{D_{m'}^r}{D_{m'+1}^r} \right) \\ - \ln\left(\frac{M_{NFC}\epsilon_f}{\frac{1}{M_{NFC}M_r^rM_aM_D^r}D_m^r }\right) - (m-\lfloor m \rfloor)\ln\left(\frac{D_m^r}{D_{m+1}^r}\right) |
    \end{align*}
    
    We know how $D_{m+1}$ and $D_m$ are related (see Lemma \ref{bdnded32}). When we apply an automorphism, we never grow bigger than a factor of $\ln(e^{-(\lambda^--\epsilon)}+\frac{1}{2}C)$. Remember throughout that $D_m$ is constant on an interval of length 1. 

We deal with cases:

The first case is that $m'>m$ but $(m'-m)< 1$. There are two subcases, one is that $D_m=D_{m'}$, the other is that $D_{m'}<D_m$.

Let us deal with the first subcase: $D_m = D_{m'}$, then 
\begin{align*}
    \ell(m)-\ell(m') = (m'-m')r\ln\left( \frac{D_{m+1}}{D_m} \right) \leq (m'-m) r\ln(e^{-(\lambda^--\epsilon)}+\frac{1}{2}C).
\end{align*}
    So we are done.

    Now we deal with the second subcase: $D_{m'}<D_m$. Then really $D_{m'}=D_{m+1}$.
    \begin{align*}
        \ell(m)-\ell(m') &= r\ln\left(\frac{D_{m+1}}{D_m} \right) - (m'-\lfloor m'\rfloor )r\ln\left(\frac{D_{m+1}}{D_{m+2}} \right) - (m - \lfloor m \rfloor )r\ln\left( \frac{D_{m+1}}{D_m} \right)\\ \nonumber
        &= r\left[(1-m+\lfloor m \rfloor)\ln\left( \frac{D_{m+1}}{D_m}\right) + (m'-\lfloor m' \rfloor) \ln\left( \frac{D_{m+2}}{D_{m+1}} \right)\right]\\ \nonumber
        &\leq r(1-m+\lfloor m \rfloor + m' - \lfloor m' \rfloor) \ln(e^{-(\lambda^--\epsilon)}+\frac{1}{2}C) \\ \nonumber &= (m'-m)r\ln(e^{-(\lambda^--\epsilon)}+\frac{1}{2}C).
    \end{align*}
    The last equality holds because the condition of the subcase is really that $\lfloor m \rfloor + 1 = \lfloor m' \rfloor$.

The second case is $m'>m$ but $(m'-m)\geq 1$. We could have stopped with the above and just said that since $d$ is continuous on a compact interval, $|d(m)-d(m')|$ is uniformly bounded from above and hence we are done. However, we actually want that this Lipschitz constant computed in this proof is independent of the interval $[m_1,m_2]$ that we choose except for the fact that we need $m_1$ large enough so that Lemma \ref{bdnded32} holds. Hence we do more work.

Note first that $\ell(m) = \ell(m') \leq \ell(\lfloor m \rfloor) - \ell(m')$. We will work with this quantity.
\begin{align*}
    \frac{\ell(\lfloor m \rfloor) - \ell(m')}{m'-m} = \frac{r\ln \left(\frac{D_{m'}}{D_m} \right) - (m'-\lfloor m' \rfloor)r\ln \left( \frac{D_{m'}}{D_{m'+1}}  \right)}{m'-m}.
\end{align*}
Since $(m'-m)\geq 1$ we have that $(m'-m)>m'-\lfloor m' \rfloor$. Also note that $D_{m'}<D_{m'+1}$ and so $\ln \left( \frac{D_{m'}}{D_{m'+1}}  \right)$ is negative. Then,
\begin{align*}
    \frac{\ell(\lfloor m \rfloor) - \ell(m')}{m'-m}&\leq r\left[\frac{\ln\left( \frac{D_{m'}}{D_m}\right)}{m'-m} + \ln \left( \frac{D_{m'+1}}{D_{m'}} \right)\right]\\
    &\leq r\left[\frac{\lfloor m'-m \rfloor}{m'-m}\ln(e^{-(\lambda^--\epsilon)}+\frac{1}{2}C) + \ln(e^{-(\lambda^--\epsilon)}+\frac{1}{2}C)\right] \\
    &\leq 2r\ln(e^{-(\lambda^--\epsilon)}+\frac{1}{2}C).
\end{align*}

Now we do the other side of the computation. Note that $d$ is decreasing because as $m$ increases, $(-m)_1^{\tau}$ decreases and $\ell$ decreases (in fact both do so strictly and hence so does $d$). Take $m,m'\in [m_1,m_2]$ such that $m'>m$ without loss of generality. Then,
\begin{align*}
    |d_1(m)-d_1(m')| = d_1(m)-d_1(m') &\geq (-m)_1^{\tau}-(-m')_1^{\tau} + \ell(m)-\ell(m')\\ \nonumber
    &= (\lambda^+-\epsilon)(m'-m) - \ell(m+1) + \ell(m) \\ \nonumber &\geq \lambda^+-\epsilon.
\end{align*}

We get Lipschitz constant to be 
\begin{align} \label{Lipconst2}
L \defeq \max\{ 2r\ln(e^{-(\lambda^--\epsilon)}+\frac{1}{2}C) +\lambda^++\epsilon, \frac{1}{\lambda^++\epsilon} \}.
\end{align}
This completes the proof.
\end{proof}

Now that we have established that $d_1$ is $L$-BiLipschitz, a similar argument works for $d_2$, we get the same Lipschitz constant and a similar result to Lemma \ref{dd}.

\begin{lemma}\label{dd2}
 For our choice of $m_1$ and $m_2$ as in Equation \eqref{mchoice}, and $T_{big}$ large enough, we have that the interval $[d_1(m_2),d_1(m_1)]$ inside $[0,d_1(m_1)]$ has proportion at least $(1-\frac{1}{1+\eta})$ in measure (where $\eta$ is as in Section \ref{sec:lusin} and Lemma \ref{bb2}). Similar result holds for $d_2$.
\end{lemma}
\begin{proof}
    The proof is the same as in Lemma \ref{dd}.
\end{proof}

Using Lemma \ref{dd2} we get the following lemma:

\begin{lemma}\label{landing2}
    Given the interval $[m_1,m_2]$ (where $m_1,\ m_2$ as specified above), there exists a positive measure subset $G$ such that,
    \begin{enumerate}
        \item if $m\in G$ then flowing backwards by $m$ from the points $(\tilde{x},\tilde{\omega},k_1)$ and $(\tilde{y},\tilde{\omega},k_1)$ lands us in the set $\tilde{K}_2\times [0,1[$,
        \item if $m\in G$ then flowing backwards by $d_1(m)$ and $d_2(m)$ from the points $(\tilde{x},\tilde{\omega},k_1)$ and $(\tilde{y},\tilde{\omega},k_1)$ resp. lands both points in the set $K_3$.
    \end{enumerate}
\end{lemma}
\begin{proof}
    The proof is the same as in Lemma \ref{landing}.
\end{proof}

\subsubsection{The end points}
\label{sec:MCT2}

From the points $x, y\in X$, we flow forwards by the `time change' flow for time $\ell$ under a new future to our final pair of points denoted $x',y'\in X$. We will need to ensure that we land in an appropriate set called $K_{4}'$ (see Section \ref{sec:lusin}). Throughout this section we will call $K\defeq K_4'.$

Claims \ref{Q1}, \ref{Q2} and Lemma \ref{LLJ} prove that 

\begin{lemma}\label{482}
We have that 
    \begin{align}\label{V'2}
\mu^{\bbN}(\{ \omega^+ \in \Omega^+ : F_{tc}^{\ell} (x,(\omega^-,\omega^+),k)\in K \}) \geq 1-\sqrt{Q\epsilon_4}.
\end{align}
\end{lemma}
\begin{proof}
    See the proof of Lemma \ref{48}.
\end{proof}

So there is a large measure set of appropriate new futures to choose from. Run this argument for each of the points, $(\overline{x},\overline{\omega}_x,k_x)$ and $(\overline{y},\overline{\omega}_y,k_y)$ (separately).

\subsubsection{Conclusion}
\label{sec:conc2}

From Section \ref{sec:MCT2}, Lemma \ref{482}, for each choice of $T$ and $m$ fixed (where $m$ is chosen in an interval determined by $T$) we are given $V_{T,m}$ ($\ell$ is determined and we assume it is greater than $\ell_0$ by taking $T$ large), a positive measure subset of $\Omega^+$ such that we can choose $(\omega')^+$ from.

 Additionally, in Section \ref{sec:green2} we landed in a set $\tilde{K}_2$ where the points had a large measure set of $(\omega')^+$ to choose from such that if we change our $\Omega$-component to $(\omega')^+$ we belong to the set $K_2$. This set $K_2$ had properties we needed to make arguments in Section \ref{sec:green}. We define $V = V_{(x,\sigma^{n'}(\tilde{\omega}^-))} \cap V_{(y,\sigma^{n'}(\tilde{\omega}^-))}$ as the set of good $(\omega')^+$. Now let $V' = V_{T,m}\cap V $. The parameter $Q$ is fixed at the beginning and $\epsilon_4$, and $\epsilon_2$ were chosen (at the beginning) such that $V'$ has positive measure. We have one last restriction on $\omega'$ that comes from Lemma \ref{omega+}.

Take $\ell\to \infty$, more precisely, I want the legs of the V to go to infinity. To achieve this, we take $\rho\to 0$, in turn, $T_{big}$ gets bigger, and since $\theta_1$ and $\theta_2$ are fixed, even though $T$ increases, we see that the proportion between $T$ and $T_{big}$ decreases. This implies that the amount flowing by $T$ separates the unstable component (the distance between $\tilde{y}$ and $\tilde{z}$) and the stable distance (the distance between $\tilde{x}$ and $\tilde{z}$) goes to zero. Together this constitutes the distance between the unstable planes going to zero. The parameter $m$ grows as $\rho \to 0$ but does not prevent the error between the unstable planes from going to zero. We also have that $d_1(m_1)>(1+\eta)d_1(m_2)$ and $d_2(m_1)>(1+\eta)d_2(m_2)$ as in Lemma \ref{dd} so that we can make an argument for our points along the way to stay in good sets as in Lemma \ref{landing}. Being in a good set after flowing by $d_1$ or $d_2$ allows us to use Lemma \ref{482} to get the set $V_{T,m}$ in the previous paragraph.  As $\ell$ goes to infinity, the image of the supports land in the same plane as the stable components for $\omega'$ (which is changing with $\ell$) shrink. We end up with two support points in one unstable manifold. However, these limit points by construction belong to the compact set $K_{00}\subset \calW$. This means $\calW$ must have positive measure. This is a contradiction to Lemma \ref{0.6}, so we have proven Proposition \ref{finite}.

\appendix

\section{A K3 surface with non-elementary actions but no parabolic elements}

\label{sec:example}

\phantom{\subsection{ss}}

We now demonstrate the existence of surfaces that fall outside of the results of \cite{CDinv}. Consider an algebraic K3 surface over $\bbC$, denoted $X$, and consider the N\'eron-Severi group:
$$\ns(X) \defeq H^{1,1}(X;\bbR)\cap H^2(X;\bbZ).$$
In \cite[Remark 2.4]{Csurvey}, is it shown that because $X$ is projective, $\langle |\rangle$ restricts to a non-degenerate quadratic form on $\ns(X)$. This makes $\ns(X)$ an even lattice of signature $(1,\rho-1)$ where $\rho$ is the Picard number (see \cites{Mo, ClD, Csurvey}). 

Let $O^+(NS(X)) = O(NS(X))\cap O^+(NS(X)\otimes \bbR)$ where $NS(X)\otimes \bbR$ is with its induced quadratic form. Consider $$\Delta \defeq \{\delta \in NS(X) : q_X(\delta) = -2 \}.$$
We call $\Delta$ the set of roots. For each root $\delta$ we get a reflection map $s_{\delta} \in O^+(NS(X))$
$$s_{\delta} : x \mapsto \langle x|\delta \rangle \delta . $$

We define the Weyl group $$W(X) \defeq \{s_{\delta} : q_X(\delta) = -2 \}\subset O^+(NS(X)) \subset O(NS(X)),$$ this is a normal subgroup of $O^+(NS(X))$, see \cite[Chapter 8.2]{Huy}.

Below we abstractly construct $X$ with the properties we desire.
This relies heavily on the work of \cite{N} and \cite{Mo} which we will state.

\begin{theorem}
    
    There exist K3 surfaces $X$ carrying non-elementary actions such that $\aut(X)$ contains no parabolic elements and there is no algebraic invariant curve.

\end{theorem}
\begin{proof}

First we state a result of \cite[Section 7]{PSS} (see also \cite[Chapter 15, Theorem 2.6]{Huy}); it relies heavily on the global Torelli theorem.

\begin{proposition}[\cite{PSS}, see also \cite{Huy}, Chapter 15, Theorem 2.6]\label{PS}
    The natural map $$\phi:\aut(X) \to O(NS(X))/W(X),$$ given by $f\mapsto f^*$ has finite kernel and cokernel.
\end{proposition}

Now we note a corollary of \cite[Theorem 1.14.4]{N}  which can be found in \cite[Corollary 2.9]{Mo}. It tells us that we can realize certain even lattices with small signature as N\'eron-Severi groups of algebraic K3s.

\begin{proposition}[\cite{N} Theorem 1.14.4, see also \cite{Mo} Corollary 2.9]
    For $\rho\leq 10$, every even lattice of signature $(1,\rho-1)$ occurs as the N\'eron-Severi group of an algebraic K3 surface. \label{evenlat}
\end{proposition}

In particular, we can construct an even lattice $S$ of signature $(1,1)$ and it will be the N\'eron-Severi group of a K3 surface $X$. We will construct $S$ such that the following is true:
\begin{enumerate}
    \item There are no elements of self-intersection -2,
    \item There are no null vectors,
    \item There are at least 2 distinct hyperbolic elements in $O(S)$.
\end{enumerate}

Condition (1) will give us that $W(X)$ is trivial for the K3 we get when we apply Proposition \ref{evenlat}, then Proposition \ref{PS} reduces to the map 
$$\phi:\aut(X)\to O(NS(X)),$$
having finite kernel and cokernel.

Condition (2) gives us that there are no parabolic elements inside $O(NS(X))$ because parabolic elements fix a null vector. Combined with the fact that the map $$\phi:\aut(X)\to O(NS(X))$$ has finite kernel we get that $\aut(X)$ has no parabolic elements. This is because if $f\in \aut(X)$ is parabolic, then for some $k$, $(f^*)^k$ acting on $H^{1,1}(X;\bbR)$ is unipotent, i.e. has only 1s as eigenvalues. This should be true for the action of $(f^*)^k$ on $\ns(X)$ too (making $f^*$ a parabolic element of $O(NS(X))$), the only concern would be if the action of $f^*$ on $\ns(X)$ is the identity (in which case it is not a parabolic element of $O(NS(X))$). Since the kernel of $\phi:\aut(X)\to O(NS(X))$ is finite it contains no infinite order elements, such as parabolic elements of $\aut(X)$. This means that parabolic elements $f\in \aut(X)$ map to  parabolic elements $f^*\in O(NS(X))$. Since there are no parabolic elements in $O(NS(X))$ we must have that $\aut(X)$ contains no parabolic elements.

Condition (3) combined with the fact that $\phi:\aut(X)\to O(NS(X))$ has finite cokernel (i.e. $O(NS(X))/Im(\phi)$ is finite) gives us that $\aut(X)$ is non-elementary. This is because the two distinct hyperbolic elements in $O(S) = O(NS(X))$ (which are infinite order) must be trivial in the cokernel and hence belong to $Im(\phi)$. Additionally, they will still have an eigenvector with eigenvalue $\lambda >1$ when viewed acting on $H^{1,1}(X;\bbR)$. We get that $\aut(X)$ has at least two sufficiently distinct hyperbolic elements that give us a non-elementary subgroup in $\aut(X)$.

We can take $S=\bbZ^2$, $q_X' = x^2-2y^2$; this does not represent zero because $\pm\sqrt{2}$ is not rational and it has signature $(1,1)$. We now take $q_X'$ as above and let $q_X = 7q_X'$, this gives a quadratic form that does not represent -2, hence condition (1) is satisfied. Condition (3) is also satisfied.

  Making sure that there is no invariant algebraic curve relies on the fact that we guaranteed that there are no elements with self-intersection -2, i.e. condition 1. To reiterate, condition (1) gives us that $\aut(X)\to O(NS(X))$ has finite kernel and cokernel by Proposition \ref{PS}. However, $O(NS(X))$ (and any finite index subgroup) is Zariski dense in $O(NS(X)\otimes \bbR)\cong SO(1,1)$ by the Borel density theorem. Hence, if $\aut(X)$ fixes a curve $C\subset X$, then $SO(1,1)$ fixes $[C]$ which cannot happen as this group does not fix any vector. This completes the proof.
\end{proof}

 \end{document}